\documentclass[reqno]{amsart}

\usepackage[backref=page]{hyperref}

\usepackage{longtable}

\usepackage{amssymb}

\usepackage[english,francais]{babel}

\newtheorem{theorem}{Theorem}
\newtheorem{lemma}[theorem]{Lemma}
\newtheorem{proposition}[theorem]{Proposition}
\newtheorem{corollary}[theorem]{Corollary}
\newtheorem{example}[theorem]{Example}
\theoremstyle{exercise}

\theoremstyle{definition}
\newtheorem{definition}[theorem]{Definition}
\usepackage{graphicx}
\usepackage{pstricks, enumerate, pst-node, pst-text, pst-plot}
\theoremstyle{remark}
\newtheorem{remark}[theorem]{Remark}

\numberwithin{equation}{section}

\newcommand{\intav}[1]{\mathchoice {\mathop{\vrule width 6pt height 3 pt depth  -2.5pt
\kern -8pt \intop}\nolimits_{\kern -6pt#1}} {\mathop{\vrule width
5pt height 3  pt depth -2.6pt \kern -6pt \intop}\nolimits_{#1}}
{\mathop{\vrule width 5pt height 3 pt depth -2.6pt \kern -6pt
\intop}\nolimits_{#1}} {\mathop{\vrule width 5pt height 3 pt depth
-2.6pt \kern -6pt \intop}\nolimits_{#1}}}

\newcommand{\intavl}[1]{\mathchoice {\mathop{\vrule width 6pt height 3 pt depth  -2.5pt
\kern -8pt \intop}\limits_{\kern -6pt#1}} {\mathop{\vrule width 5pt
height 3  pt depth -2.6pt \kern -6pt \intop}\nolimits_{#1}}
{\mathop{\vrule width 5pt height 3 pt depth -2.6pt \kern -6pt
\intop}\nolimits_{#1}} {\mathop{\vrule width 5pt height 3 pt depth
-2.6pt \kern -6pt \intop}\nolimits_{#1}}}

\begin{document}

\title[L'action de $SL(2,\mathbb{R})$ sur les espaces de modules de surfaces plates]{Quelques contributions \`a la th\'eorie de l'action de $SL(2,\mathbb{R})$ sur les espaces de modules de surfaces plates}

\author[Carlos Matheus]{Carlos Matheus}
\address{Universit\'e Paris 13, Sorbonne Paris Cit\'e, LAGA, CNRS (UMR 7539), F-93430, 
Villetaneuse, France}
\email{matheus@impa.br}

\dedicatory{To Jean-Christophe Yoccoz (in memoriam)}

\selectlanguage{francais}



\date{2 juin 2017}



\maketitle

\tableofcontents

\newpage

\section*{Remerciements}

Je d\'edie ce m\'emoire \`a Jean-Christophe Yoccoz: son amiti\'e et bienveillance envers moi apr\`es mon arrive en France en 2007 ont marqu\'e ma vie (math\'ematique et personnelle) \`a jamais, et je lui serai toujours reconnaissant.

Je remercie Yves Benoist, Stefano Marmi et Anton Zorich d'avoir accept\'e de relire mon m\'emoire et d'\'ecrire les rapports \`a son sujet, et Julien Barral, Henry de Th\'elin et Giovanni Forni pour me faire l'honneur d'\^etre membres de jury d'examen de ce m\'emoire.

Je remercie aussi mes coauteurs pour avoir partager avec moi la joie de la d\'ecouverte de nouveaux th\'eor\`emes.

\`A Aline et Marie-In\`es, ``la reine et la princesse de mon ch\^ateau \`a Fontainebleau--Avon'', et \`a Dominique et V\'eronique pour toute leur amiti\'e.

\newpage 

\null

\newpage

\section*{Description du m\'emoire}

Ce m\'emoire est bas\'e sur certains de mes travaux autour de la \emph{dynamique de Teichm\"uller} ou, plus pr\'ecis\'ement, la dynamique de l'action de $SL(2,\mathbb{R})$ sur les espaces de modules de surfaces plates. 

Le chapitre \ref{s.introduction} sert \`a introduire plusieurs aspects classiques de la dynamique de Teichm\"uller. Son contenu est inspir\'e par les survols de Zorich \cite{Z} et Yoccoz \cite{Y}, ainsi que les notes \cite{FM} d'un minicours donn\'e par Forni et moi-m\^eme en 2011 au Banach Center (Bedlewo, Pologne). En particulier, ce chapitre est une introduction g\'en\'erale \`a tous les chapitres post\'erieurs, de fa\c con qu'on va toujours supposer une certaine familiarit\'e avec ce chapitre dans toutes les discussions dans d'autres chapitres. 

Apr\`es avoir lu le premier chapitre, le lecteur peut choisir librement quel ordre suivre pour la lecture de chapitre restants: en fait, les r\'esultats discut\'es dans les chapitres \ref{s.AMY} \`a \ref{s.FFM} sont compl\`etement ind\'ependants les uns des autres. 

Le chapitre \ref{s.AMY} traite de la r\'egularit\'e des mesures $SL(2,\mathbb{R})$-invariantes sur les espaces de modules de surfaces plates. En g\'en\'eral, une telle mesure est dite r\'eguli\`ere si la plupart des surfaces plates dans son support poss\`edent leurs connexions de selles plus courts parall\`eles entre eux. La propri\'et\'e de r\'egularit\'e a \'et\'e utilis\'ee par Eskin, Kontsevich et Zorich \cite{EKZ} pour justifier un argument sophistiqu\'e d'int\'egration par parties intervenant dans la d\'emonstration de leur c\'el\`ebre formule pour la somme des exposants de Lyapunov non-n\'egatifs du cocycle de Kontsevich-Zorich (la partie int\'eressante de la d\'eriv\'ee de l'action de $SL(2,\mathbb{R})$ sur les surfaces plates). Dans ce m\^eme article, Eskin-Kontsevich-Zorich \cite{EKZ} ont conjectur\'e que la propri\'et\'e de r\'egularit\'e est toujours valide, de fa\c con que leur formule pour la somme des exposants de Lyapunov pourrait \^etre appliqu\'ee sans exception \`a toutes les mesures de probabilit\'e $SL(2,\mathbb{R})$-invariantes ergodiques sur les espaces de modules de surfaces plates. Let but du chapitre \ref{s.AMY} est discuter notre article \cite{AMY} avec Avila et Yoccoz contenant une r\'eponse affirmative \`a la conjecture de r\'egularit\'e d'Eskin-Kontsevich-Zorich. 

Le chapitre \ref{s.MS} est consacr\'e \`a l'\'etude de la vitesse de m\'elange du flot de Teichm\"uller.  La question de la vitesse de decroissance des correlations pour les mesures de Masur-Veech a \'et\'e r\'esolu dans un c\'el\`ebre article de Avila, Gou\"ezel et Yoccoz \cite{AGY}: la vitesse de m\'elange du flot de Teichm\"uller par rapport \`a ces mesures est toujours exponentielle. Puis, Avila et Gou\"ezel \cite{AG} ont \'etendu le r\'esultat d'Avila-Gou\"ezel-Yoccoz \`a toutes les mesures de probabilit\'e $SL(2,\mathbb{R})$-invariantes ergodiques sur les espaces de modules de surfaces plates. Une question naturelle motiv\'ee par les r\'esultats d'Avila, Gou\"ezel et Yoccoz est savoir si la vitesse exponentielle de m\'elange de ces mesures est uniforme. Le r\'esultat principal du chapitre \ref{s.MS} est un th\'eor\`eme obtenu en collaboration avec Schmith\"usen \cite{MaSc} selon lequel il n'y a pas d'uniformit\'e sur la vitesse exponentielle de m\'elange de telles mesures lorsqu'on regarde des espaces de modules de surfaces plates de genre arbitrairement grand. 

Le chapitre \ref{s.MW} aborde le probl\`eme de classification des adh\'erences des $SL(2,\mathbb{R})$-orbites dans les espaces de modules de surfaces plates. De nombreuses applications de la dynamique de Teichm\"uller \`a l'\'etude des billards math\'ematiques d\'ependent d'une connaissance pr\'ecise des  fermetures de certaines $SL(2,\mathbb{R})$-orbites de surfaces plates, ce qui explique une partie de l'int\'er\^et en classifier ces objets. Les r\'esultats remarquables d'Eskin-Mirzakhani \cite{EsMi}, Eskin-Mirzakhani-Mohammadi \cite{EMM} et Filip \cite{Fi13a} disent que les fermetures des $SL(2,\mathbb{R})$-orbites des surfaces plates sont affines dans les coordonn\'ees de p\'eriodes, quasi-projectives dans les coordonn\'ees induites par les espaces de modules de courbes et leur totalit\'e est une collection d\'enombrable: en particulier, il est raisonnable d'essayer de les classifier. Les travaux de Calta \cite{Ca} et McMullen \cite{McM03} fournissent une classification tr\`es satisfaisante des adh\'erences de $SL(2,\mathbb{R})$-orbites de surfaces plates de genre deux. Par contre, la situation en genre sup\'erieur n'est toujours pas compl\`etement comprise malgr\'e les nombreux progr\`es partiels r\'ecents. N\'eamoins, cette situation semble s'am\'erioler un peu lorsqu'on se concentre sur les courbes de Teichm\"uller (i.e., les $SL(2,\mathbb{R})$-orbites ferm\'ees): par exemple, Bainbridge, Habegger et M\"oller \cite{BHM} ont prouv\'e la finitude des courbes de Teichm\"uller alg\'ebriquement primitives engendr\'ees par des surfaces plates de genre trois. Dans le chapitre \ref{s.MW}, on discutera un r\'esultat obtenu avec Wright \cite{MW} assurant la finitude des courbes de Teichm\"uller alg\'ebriquement primitives engendr\'ees par des surfaces plates de genre $g>2$ premier ayant une seule singularit\'e conique. 

Le chapitre \ref{s.MMY} examine les exposants de Lyapunov du cocycle de Kontsevich-Zorich (la partie int\'eressante de la d\'eriv\'ee de l'action de $SL(2,\mathbb{R})$ sur les surfaces plates). Les propri\'et\'es qualitatives et/ou quantitatives des exposants de Lyapunov du KZ cocycle jouent un r\^ole important dans de nombreuses applications de la dynamique de Teichm\"uller: par exemple, Avila et Forni \cite{AF} ont explor\'e le r\'esultat de Forni \cite{F02} de hyperbolicit\'e non-uniforme du cocycle de KZ par rapport aux mesures de Masur-Veech pour montrer que les transformations d'échange d'intervalles typiques (qui ne sont pas des rotations) sont faiblement m\'elangeantes. D'un point de vue qualitatif, les exposants de Lyapunov du KZ cocycle par rapport aux mesures de Masur-Veech sont bien compris gr\^ace \`a un c\'el\`ebre article d'Avila et Viana \cite{AV} assurant la simplicit\'e (i.e., multiplicit\'e un) de ces exposants (confirmant donc une conjecture de Kontsevich et Zorich). Par contre, ceci n'est plus vrai pour d'autres mesures: Forni et moi-m\^eme (voir \cite{FM} par exemple) avons deux exemples de mesures de probabilit\'es $SL(2,\mathbb{R})$-invariantes sur les espaces de modules de surfaces plates pour lequelles les exposants de Lyapunov du cocycle KZ sont loin d'\^etre simples. Du coup, il est int\'eressant de savoir en g\'en\'eral dans quelles conditions les exposants de Lyapunov du cocycle KZ sont simples. Le point de d\'epart du chapitre \ref{s.MMY} est un r\'esultat obtenu en collaboration avec Eskin \cite{EM} assurant que les exposants de Lyapunov du KZ cocycle sur les courbes de Teichm\"uller ($SL(2,\mathbb{R})$-orbites ferm\'ees) peuvent \^etre calcul\'es \`a l'aide de produits de matrices al\'eatoires. Ensuite, ce r\'esultat et les techniques d'Avila et Viana \cite{AV} sont exploit\'es pour fournir un crit\`ere efficace (bas\'e sur la th\'eorie de Galois) obtenu en collaboration avec M\"oller et Yoccoz \cite{MMY} pour la simplicit\'e des exposants de Lyapunov de KZ cocycle sur les courbes de Teichm\"uller arithm\'etiques. Enfin, ce crit\`ere de type Galois pour la simplicit\'e  est utilis\'e dans la discussion d'un contre-exemple d\^u \`a Delecroix et moi-m\^eme \cite{DeMa} \`a une conjecture de Forni.

Le chapitre \ref{s.FFM} est d\'edi\'e \`a la structure du groupe de matrices associ\'e au cocycle de Kontsevich-Zorich (i.e., la partie int\'eressante de la d\'eriv\'ee de l'action de $SL(2,\mathbb{R})$ dans les espaces de modules de surfaces plates). Les groupes de matrices engendr\'es par le cocycle KZ m\'eritent une attention particuli\`ere car ils ont un r\^ole cl\'e dans l'\'etude de l'action de $SL(2,\mathbb{R})$ sur les surfaces plates: par exemple, le c\'el\`ebre travail d'Eskin-Mirzakhani \cite{EsMi} sur la classification des mesures $SL(2,\mathbb{R})$-invariantes dans les espaces de modules de surfaces plates est   basé sur une analyse fine de ce cocycle. Un travail r\'ecent de Filip \cite{Fi14} fournit une liste de toutes les adh\'erences de Zariski possibles (modulo facteurs compacts et \`a indice fini pr\`es) pour les groupes de matrices associ\'es au cocycle KZ: en particulier, Filip a confirm\'e une conjecture de Forni, Zorich et moi-m\^eme \cite{FMZ} sur l'origine des exposants de Lyapunov nuls pour le cocycle KZ. Cependant, la liste de Filip est  produite \`a partir de consid\'erations de variations de structures de Hodge sur les vari\'et\'es quasi-projectives et, par cons\'equent, une question naturelle est de savoir quels groupes de la liste apparaissent effectivement dans le contexte du cocycle KZ. Le r\'esultat principal du chapitre \ref{s.FFM} est un exemple de Filip, Forni et moi-m\^eme \cite{FFM} montrant que l'un des groupes de matrices quaternioniques dans la liste Filip est r\'ealis\'e dans le cadre du cocycle KZ.

Enfin, les limites habituelles de l'espace et du temps m'ont forc\'e \`a laisser quelques aspects de mon travail en dehors de ce texte. Par exemple, les sujets suivants mentionn\'es dans la liste des 
publications de l'auteur ne seront pas discut\'es ici:

\begin{itemize}
\item les articles (MY) (avec Yoccoz), (FMZi), (FMZii) et (FMZiii) (avec Forni et Zorich), (MYZ) (avec Yoccoz et Zmiaikou), (MSch) (avec Schmith\"usen) sur la dynamique de Teichm\"uller, 
\item les articles (BMMWi) et (BMMWii) (avec Burns, Masur et Wilkinson) sur la vitesse de m\'elange du flot de Weil-Petersson sur les espaces de modules de surfaces hyperboliques,
\item les articles (MMP) (avec Moreira et Pujals), (M), (CMM) (avec Cerqueira et Moreira), (LM) (avec Lima) sur la dynamique en basse dimension, et 
\item les articles (ACM) (avec Arbieto et Corcho), (M07), (AM) (avec Arbieto), (LM09) (avec Linares), (CM09) (avec Corcho) et (AMP) (avec Angulo et Pilod) sur les EDPs dispersives. 
\end{itemize}

\newpage

\section*{Liste de publications de l'auteur}

\subsubsection*{Articles de recherche (apr\`es la th\`ese de doctorat)}

\begin{enumerate}
\item[(ACM)] \textit{Rough solutions for the periodic Schr\"odinger-Korteweg-deVries system},
avec A. Arbieto et Ad\'an Corcho. \textbf{Journal of
Differential Equations}, vol.230, p.295-336 (2006).
\item[(M07)] \textit{Global well-posedness of NLS-KdV systems for periodic functions},
\textbf{Electronic Journal of Differential Equations}, vol. 2007, p.
1-20 (2007).
\item[(AM)] \textit{On the periodic Schr\"odinger-Debye equation}, avec A. Arbieto, \textbf{Communications in Pure and
Applied Analysis}, vol. 7, p. 699-713 (2008).
\item[(LM09)] \textit{Well-posedness for the 1D Zakharov-Rubenchik system}, avec Felipe Linares, \textbf{Advances in Differential Equations}, vol. 14, n.3-4, 261-288 (2009).
\item[(CM09)] \textit{Sharp bilinear estimates and well-posedness for the 1D Schr\"odinger-Debye system}, avec Ad\'an Corcho,  \textbf{Differential and Integral Equations}, vol. 22, n.3-4, 357-391 (2009).
\item[(AMP)] \textit{Global well-posedness and non-linear stability of periodic traveling waves for a Schr\"odinger-Benjamin-Ono system},
avec Jaime Angulo et Didier Pilod, \textbf{Communications in Pure and Applied Analysis}, vol. 8, n.3, 815-844 (2009).
\item[(MY)] \textit{The action of the affine diffeomorphisms on the relative homology group of certain exceptionally symmetric origamis}, avec Jean-Christophe Yoccoz, \textbf{Journal of Modern Dynamics}, vol. 4, n.3, 453-486 (2010).
\item[(FMZi)] \textit{Square-tiled cyclic covers}, avec Giovanni Forni et Anton Zorich, \textbf{Journal of Modern Dynamics}, vol. 5, n.2, 285-318 (2011).
\item[(M11)] \textit{Other relevant examples}: annexe \`a l'article ``A geometric criterion for the non-uniform hyperbolicity of the Kontsevich-Zorich cocycle'' de G. Forni, \textbf{Journal of Modern Dynamics}, vol. 5, n.2, 355-395 (2011).
\item[(M)] \textit{Some quantitative versions of Ratner's mixing estimates}, \textbf{Bulletin Brazilian Mathematical Society}, vol. 44, 469-488 (2013).
\item[(AMY)] \textit{$SL(2,\mathbb{R})$-invariant probability measures on the moduli spaces of translation surfaces are regular}, avec Artur Avila et Jean-Christophe Yoccoz, \textbf{Geometric and Functional Analysis}, v. 23, p. 1705-1729 (2013).
\item[(MMP)] \textit{Axiom A versus Newhouse phenomena for Benedicks-Carleson toy models}, avec Carlos Gustavo Moreira et Enrique Pujals, \textbf{Annales Scientifiques de l'\'Ecole Normale Sup\'erieure}, vol. 46, n. 6, p. 857-878 (2013).
\item[(MS)] \textit{Explicit Teichm\"uller curves with complementary series}, avec Gabriela Weitze-Schmith\"usen, \textbf{Bulletin de la Soci\'et\'e Math\'ematique de France}, vol. 141, n. 4, 557-602 (2013). 
\item[(FMZii)] \textit{Lyapunov spectrum of invariant subbundles of the Hodge bundle}, avec Giovanni Forni et  Anton Zorich, \textbf{Ergodic Theory and Dynamical Systems}, vol. 34, n. 2, 353-408 (2014). 
\item[(FMZiii)] \textit{Zero Lyapunov exponents of the Hodge bundle}, avec Giovanni Forni et Anton Zorich,  \textbf{Commentarii Mathematici Helvetici}, vol. 89, n. 2, 489-535 (2014). 
\item[(MYZ)] \emph{Homology of origamis with symmetries}, avec Jean-Christophe Yoccoz et David Zmiaikou, \textbf{Annales de l'Institut Fourier},  vol. 64, 1131-1176 (2014). 
\item[(DM)] \emph{Un contre-exemple \`a la r\'eciproque du crit\`ere de Forni pour la positivit\'e des exposants  de  Lyapunov du cocycle de Kontsevich-Zorich}, avec Vincent Delecroix, \textbf{Mathematical Research Letters}, vol. 22, n. 6, 1667--1678 (2015)
\item[(EM)] \emph{A coding-free simplicity criterion for the Lyapunov exponents of Teichm\"uller curves}, avec Alex Eskin, \textbf{Geometriae Dedicata}, vol. 179, n. 1,  45-67 (2015). 
\item[(MW)] \emph{Hodge-Teichm\"uller planes and finiteness results for Teichm\"uller curves}, avec Alex Wright, \textbf{Duke Mathematical Journal}, vol. 164, 1041-1077 (2015). 
\item[(MSch)] \emph{Some examples of isotropic SL(2,R)-invariant subbundles of the Hodge bundle}, avec Gabriela Weitze-Schmith\"usen, \textbf{International Mathematics Research Notices}, vol. 2015, n. 18, 8657-8679 (2015). 
\item[(MMY)] \emph{A criterion for the simplicity of the Lyapunov spectrum of square-tiled surfaces}, avec  Martin M\"oller et Jean-Christophe Yoccoz, \textbf{Inventiones mathematicae}, vol. 202, n. 1, 333-425 (2015). 
\item[(FFM)] \emph{Quaternionic covers and monodromy of the Kontsevich-Zorich cocycle in orthogonal groups}, avec Giovanni Forni et Simion Filip, \`a para\^{\i}tre dans \textbf{Journal of the European Mathematical Society }(2015).
\item[(BMMWi)] \emph{Rates of mixing for the Weil-Petersson geodesic flow I: no rapid mixing in non-exceptional moduli spaces}, avec Keith Burns, Howard Masur et Amie Wilkinson. Pr\'epublication (2015) disponible sur arXiv:1312.6012. 
\item[(CMM)] \emph{Continuity of Hausdorff dimension across generic dynamical Lagrange and Markov spectra}, avec Aline Cerqueira et Carlos Gustavo Moreira. Pr\'epublication (2016) disponible sur arXiv:1602.04649. 
\item[(BMMWii)] \emph{Rates of mixing for the Weil-Petersson geodesic flow II: exponential mixing in exceptional moduli spaces}, avec Keith Burns, Howard Masur et Amie Wilkinson. Pr\'epublication (2016) disponible sur arXiv:1605.09037. 
\item[(AMYii)] \emph{Zorich conjecture for hyperelliptic Rauzy-Veech groups}, avec Artur Avila et Jean-Christophe Yoccoz. Pr\'epublication (2016) disponible sur arXiv:1606.01227. 
\item[(LM)] \emph{Symbolic dynamics for non-uniformly hyperbolic surface maps with discontinuities}, avec  Yuri Lima. Pr\'epublication (2016) disponible sur arXiv:1606.05863.
\end{enumerate}

\subsubsection*{Articles de survol}

\begin{itemize}
\item \textit{$C^1$ density of hyperbolicity for Benedicks-Carleson toy models},
avec Carlos Gustavo Moreira et Enrique Pujals, \textbf{Oberwolfach Reports}, vol. 6, p. 1819-1823 (2009).
\item \textit{On the neutral Oseledets bundle of Kontsevich-Zorich cocycle over certain cyclic covers},
avec Giovanni Forni et Anton Zorich, \textbf{Oberwolfach Reports}, vol. 8,  p. 1361-1427 (2011).
\item \textit{Le flot g\'eod\'esique de Teichm\"uller et la g\'eom\'etrie du fibr\'e de Hodge},
\textbf{Actes du S\'eminaire de Th\'eorie Spectrale et G\'eom\'etrie du Institut Fourier \`a Grenoble}, vol. 29, p. 73-95 (2010-2011).
\item \textit{Fractal geometry of non-uniformly hyperbolic horseshoes},
\textbf{Proceedings of the Ergodic Theory Workshops at University of North Carolina at Chapel Hill, 2011-2012}, 
Ed. by Idris Assani, p. 197-240 (2013).
\item \textit{A criterium for the simplicity of Lyapunov exponents of origamis}, avec Martin Möller et  Jean-Christophe Yoccoz, \textbf{Oberwolfach Reports}, vol. 10, p. 1975–2033 (2013).
\item \textit{Introduction to Teichm\"uller theory and its applications to dynamics of interval exchange transformations, flows on surfaces and billiards}, avec Giovanni Forni, \textbf{Journal of Modern Dynamics}, vol. 8, no. 3/4, p. 271-436 (2014).
\item \textit{Lecture notes on the dynamics of the Weil-Petersson flow}, notes d'un minicours donn\'e en novembre 2013 au CIRM, Marseille, \`a para\^{\i}tre dans \textbf{CIRM Jean-Morlet Chair Subseries}, Springer (2015).
\item \textit{Variations of Hodge structures, Lyapunov exponents and Kontsevich’s formula}, annexe \`a la monographie ``Gauss-Manin Connection in Disguise (Calabi-Yau Modular Forms)'' de Hossein Movasati, \`a para\^{\i}tre dans \textbf{Surveys of Modern Mathematics}, International Press, Boston (2016).
\item \textit{Les blogs: un outil dynamique de communication en math\'ematiques},  Gaz. Math. No. 148 (2016), 46--50. 
\end{itemize}

\subsubsection*{Livres}

\begin{itemize}
\item \textit{Aspectos erg\'odicos da teoria dos n\'umeros} (Ergodic aspects of Number Theory),
joint with A. Arbieto and Carlos Gustavo Moreira, \textbf{Publica\c c\~oes Matem\'aticas do IMPA - $26^o$ Col\'oquio Brasileiro de Matem\'atica} (2007).
\item \textit{The remarkable efectiveness of Ergodic Theory in Number Theory, Parts I and II},
(Part I together with A. Arbieto and Carlos Gustavo Moreira), \textbf{Ensaios Mat\'ematicos} vol. 17, p. 1-106 (2009).
\end{itemize}

\selectlanguage{english}

\newpage 


\begin{centering}
\rule{\textwidth}{1.6pt}\vspace*{-\baselineskip}\vspace*{2.5pt}
\rule{\textwidth}{0.4pt}

\section{Introduction}\label{s.introduction}

\rule{\textwidth}{0.4pt}\vspace*{-\baselineskip}\vspace{3.2pt}
\rule{\textwidth}{1.6pt}
\end{centering}\\

This section serves as a general-purpose introduction to all other sections of this memoir. In particular, we'll always assume familiarity with the content of this section in subsequent discussions. 

The basic references for this section are the survey texts of Zorich \cite{Z}, Yoccoz \cite{Y}, and Forni and the author \cite{FM}.  

\subsection{Abelian differentials and their moduli spaces}  

Let $\mathcal{L}_g$ be the set of Abelian differentials on a Riemann surface of genus $g\geq 1$, that is, the set of pairs $(\textrm{Riemann surface structure on } M, \omega)$ where $M$ is a compact topological surface of genus $g$ and $\omega\not\equiv 0$ is a non-trivial $1$-form which is holomorphic with respect to the underlying Riemann surface structure. 

The \emph{Teichm\"uller space} of Abelian differentials of genus $g\geq 1$ is the quotient $\mathcal{TH}_g:=\mathcal{L}_g/\textrm{Diff}^+_0(M)$ and the \emph{moduli space} of Abelian differentials of genus $g\geq 1$ is the quotient $\mathcal{H}_g:=\mathcal{L}_g/\Gamma_g$. Here $\textrm{Diff}^+_0(M)$ is the set of diffeomorphisms isotopic to the identity and $\Gamma_g:=\textrm{Diff}^+(M)/\textrm{Diff}^+_0(M)$ is the mapping class group (i.e., the set of isotopy classes of orientation-preserving diffeomorphisms), and both $\textrm{Diff}^+_0(M)$ and $\Gamma_g$ act on the set of Riemann surface structure in the usual manner\footnote{By precomposition with coordinate charts.}, while they act on Abelian differentials by pull-back.

Before equipping $\mathcal{TH}_g$ and $\mathcal{H}_g$ with nice structures, let us give a \emph{concrete} description of Abelian differentials in terms of \emph{translation structures}. 

\subsection{Translation structures} Let $(M,\omega)\in\mathcal{L}_g$ and denote by $\Sigma\subset M$ the set of singularities of $\omega$, or, equivalently, the \emph{divisor} of $\omega$, i.e., the finite set 
$$\Sigma:=\textrm{div}(\omega):=\{p\in M: \omega(p)=0\}$$ 

For each $p\in M-\Sigma$, let us select a small simply-connected neighborhood $U_p$ of $p$ such that $U_p\cap\Sigma=\emptyset$. In this context, the ``period'' map $\phi_p: U_p\to\mathbb{C}$, $\phi_p(x) := \int_p^x \omega$ given by integration along \emph{any} path inside $U_p$ joining $p$ and $x$ is well-defined: in fact, any holomorphic $1$-form $\omega$ is closed and, thus, the integral $\int_p^x \omega$ does not depend on the choice of the path inside $U_p$ connecting $p$ and $x$. Furthermore, since $p\notin\Sigma$ (i.e., $\omega(p)\neq 0$), we have that, after reducing $U_p$ if necessary, this ``period'' map $\phi_p$ is a biholomorphism. 

In other words, the collection $\{(U_p, \phi_p)\}_{p\in M-\Sigma}$ of all such ``period'' maps is an atlas of $M-\Sigma$ which is compatible with the Riemann surface structure. By definition, the local expression of Abelian differential $\omega$ in these coordinates is $(\phi_p)_*(\omega) = dz$ (on $\mathbb{C}$). Also, the local equality $\int_p^x \omega = \int_p^q \omega + \int_q^x\omega$ implies that all coordinate changes are $\phi_q\circ\phi_p^{-1}(z) = z+c$ where $c=\int_q^p\omega\in\mathbb{C}$ is a constant independent of $z$. Moreover, since $\textrm{div}(\omega)$ is finite, Riemann's theorem on removable  singularities implies that this atlas of ``period'' charts on $M-\Sigma$ can be extended to $M$ in such a way that the local expression of $\omega$ in a chart around a zero $p\in\Sigma$ of $\omega$ of order $k$ is the holomorphic $1$-form $z^k dz$. 

In the literature, a maximal atlas of compatible charts on the complement $M-\Sigma$ of a finite subset $\Sigma$ of a surface $M$ whose changes of coordinates are translations $z\mapsto z+c$ of the complex plane is called a \emph{translation structure} on $M$. In this language, the discussion in the previous paragraph says that $(M,\omega)$ determines a translation structure on $M$. On the other hand, it is clear that a translation structure on $M$ determines a Riemann surface structure\footnote{Since translations are particular cases of biholomorphisms.} and an Abelian\footnote{We define $\omega$ by locally pulling-back $dz$ via the charts: this gives a globally defined Abelian differential because the changes of coordinates are translations and, hence, $dz$ is invariant under changes of coordinates.} differential $\omega$ on $M$. 

In summary, we proved the following proposition. 

\begin{proposition} The set $\mathcal{L}_g$ of all non-trivial Abelian differentials on compact Riemann surfaces of genus $g\geq 1$ is canonically identified to the set of all translation structures on the compact surfaces of genus $g\geq 1$. 
\end{proposition}

\subsection{Some examples of translation surfaces}

The notion of translation structures allows us to exhibit many concrete examples of Abelian differentials. 

\subsubsection{Abelian differentials on complex torus} We usually learn the concept of complex torii through translation structures. Indeed, a complex torus is the quotient $\mathbb{C}/\Lambda$ of the complex plane by a lattice $\Lambda=\mathbb{Z}w_1\oplus\mathbb{Z}w_2\subset\mathbb{C}$. These complex torii come equipped with Abelian differentials induced by $dz$ on $\mathbb{C}$ and they are usually depicted as a parallelogram of sides $w_1$ and $w_2$ whose parallel sides are identified via the translations $z\mapsto z+w_1$ and $z\mapsto z+w_2$: see Figure \ref{f.complex-torus}. 

\begin{figure}[htb!] 
\begin{center}
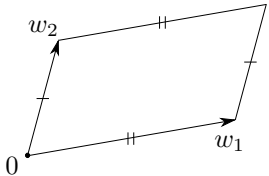
\caption{Complex torii are translation surfaces (of genus one).}\label{f.complex-torus}
\end{center}
\end{figure}

\subsubsection{Square-tiled surfaces} We can build more translation surfaces from certain coverings of the unit square torus $\mathbb{C}/(\mathbb{Z}\oplus\mathbb{Z}i)$ equipped with the Abelian differential induced by $dz$. 

More precisely, consider a finite collection $Sq$ of unit squares of the complex plane and let us glue by translations the leftmost, resp. bottomost, side of each square $Q\in Sq$ with the rightmost, resp. topmost, side of another (maybe the same) square $Q'\in Sq$. Here, we assume that, after performing the identifications, the resulting surface is connected. 

In this way, we obtain a translation surface, namely, a Riemann surface with an Abelian differetianl (equal to $dz$ on each square $Q\in Sq$). For obvious reasons, these translation surfaces are called \emph{square-tiled surfaces} and/or \emph{origamis}. 

In Figure \ref{f.L-shaped} we drew a $L$-shaped square-tiled surface built up from three unit squares by identification (via translations) of pairs of sides with the same markings. 

\begin{figure}[htb!] 
\begin{center}
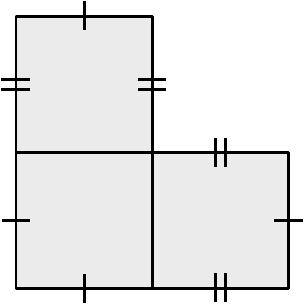
\caption{A $L$-shaped square-tiled surface.}\label{f.L-shaped}
\end{center}
\end{figure}

\begin{remark}\label{r.L-genus-computation} The translation surface $L$ in Figure \ref{f.L-shaped} has genus two. In fact, the corners of all squares are identified to a single point $p$. Moreover, this point is special when compared to any other point because we have a total angle of $6\pi$ by turning around $p$ (instead of a total angle of $2\pi$ around all other points). In other terms, a neighborhood of $p$ looks like $3$ copies of the flat complex plane stitched together, that is, the natural local coordinate around $p$ is $\zeta=z^3$. In particular, the Abelian differential $\omega$ associated to this translation surface $L$ has the form $\omega=d\zeta=3 z^2 dz$ near $p$, i.e., $\omega$ has single zero of order two on $L$. By Riemann-Hurwitz theorem, this means that $2=2g-2$ where $g$ is the genus of $L$, that is, $L$ has genus two\footnote{Alternatively, this fact can be derived from Poincar\'e-Hopf index theorem applied to the vector field given by the vertical direction at all points of $L-\{p\}$.}
\end{remark}

\subsubsection{Suspensions of interval exchange transformations} We find translation surfaces during the construction of natural extensions of one-dimensional dynamical systems called \emph{interval exchange transformations}. More concretely, recall that an interval exchange transformation (i.e.t.) of $d\geq 2$ intervals is a map $T:D_T\to D_{T^{-1}}$ where $D_T, D_{T^{-1}}\subset I$ are subsets of an open bounded interval $I$ with $\#(I-D_T) = \#(I-D_{T^{-1}})=d+1$ and the restriction of $T$ to each connected component of $I-D_T$ is a translation onto a connected component of $I-D_{T^{-1}}$: see Figure \ref{f.iets} for some examples. 

\begin{figure}[htb!] 
\includegraphics[scale=0.6]{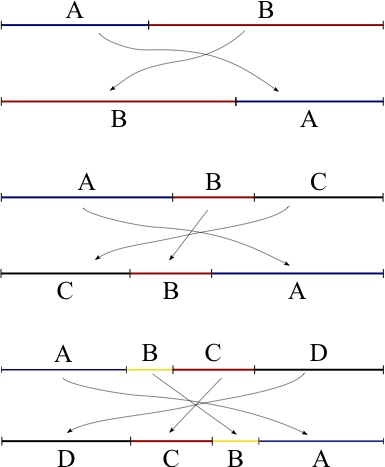}
\caption{Three examples of interval exchange transformations.}\label{f.iets}
\end{figure}

It is possible to \emph{suspend} (in several ways) any given i.e.t. $T$ to obtain \emph{translation flows}\footnote{A translation flow is obtained by moving (almost all) points of a translation surface in a fixed direction.} on translation surfaces such that $T$ is the first return map to certain transversals to such flows: for instance, Figure \ref{f.suspension-Masur} shows Masur's suspension construction applied to an i.e.t. of four intervals. 

\begin{figure}[htb!] 
\begin{center}
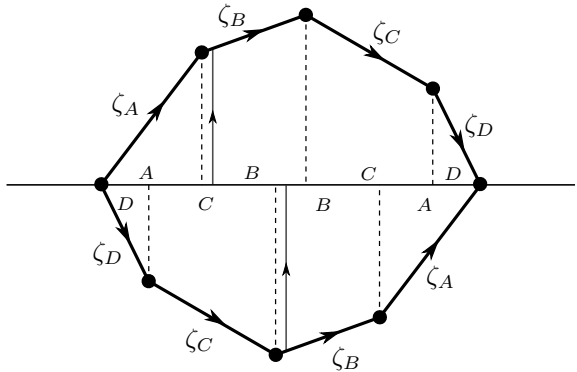
\caption{Masur's suspension of an i.e.t. of four intervals.}\label{f.suspension-Masur}
\end{center}
\end{figure}

Here, the idea of this procedure is that:

\begin{itemize}
\item the vectors $\zeta_1, \dots, \zeta_d$ have the form $\zeta_j=\lambda_j+\sqrt{-1}\,\tau_j\in\mathbb{C}$ where $\lambda_j$ are the lengths of the intervals permuted by $T$; 
\item the vectors $\zeta_j$, $1\leq j\leq d$, are organized in the plane to construct a polygon $P$ in such a way that we meet these vectors in the usual order (i.e., $\zeta_1$, $\zeta_2$, etc.) in the top part of $P$ and we meet these vectors in the order determined by $T$, i.e., using the combinatorial receipt -- a permutation $\pi$ of $d$ elements -- employed by $T$ to permute intervals, in the bottom part of $P$; 
\item gluing by translations the pairs of sides of $P$ with the same labels $\zeta_j$, we obtain a translation surface such that the unit-speed translation flow in the vertical direction has the i.e.t. $T$ as the first return map to $\mathbb{R}\times\{0\}$; 
\item finally, the \emph{suspension data} $\tau_1,\dots,\tau_d$ can be chosen ``arbitrarily'' as long as the planar figure $P$ is not degenerate, i.e., 
$$\sum\limits_{j<n}\tau_j>0 \quad \textrm{ and } \quad \sum\limits_{\pi(j)<n}\tau_j<0 \quad \forall \,\, 1\leq n\leq d$$
\end{itemize}

\begin{remark} There is no unique procedure for suspending i.e.t.'s: for example, Yoccoz's survey \cite{Y} discusses in details the so-called \emph{Veech's zippered rectangles construction}.
\end{remark}

\subsubsection{Billiards in rational polygons} Recall that a polygon is called \emph{rational} if all of its angles are rational multiples of $\pi$. Consider the billiard flow on a rational polygon $P$: the trajectory of a point in $P$ in a certain direction is a straight line until it hits the boundary $\partial P$ of the polygon; at this instant, we prolongate the trajectory by reflecting it accordingly to the usual (specular) law\footnote{I.e., the angle of reflection equals the angle of incidence.}. 

A classical \emph{unfolding} construction (due to Fox-Keshner and Katok-Zemlyakov) relates the dynamics of billiard flows on rational polygons to translation flows on translation surfaces. In a nutshell, the idea is the following: every time the billiard trajectory hits $\partial P$, we reflect the table instead of reflecting the trajectory so that the trajectory remains a straight line, see Figure \ref{f.unfolding-construction}. 

\begin{figure}[htb!] 
\includegraphics[scale=1]{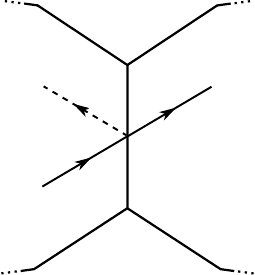}
\caption{Elementary step of the unfolding construction.}\label{f.unfolding-construction}
\end{figure} 

The group $G$ generated by the reflections about the sides of $P$ is \emph{finite} when $P$ is a \emph{rational} polygon, so that the natural surface obtained by iterating this unfolding procedure is a translation surface and the billiard flow becomes the tranlsation (straigth line) flow on this translation surface. 

In Figure \ref{f.triangle-unfolding} we drew the translation surface obtained by applying the unfolding construction to a $L$-shaped polygon and the triangle with angles $\pi/8$, $\pi/2$ and $3\pi/8$.

\begin{figure}[htb!] 
\begin{center}
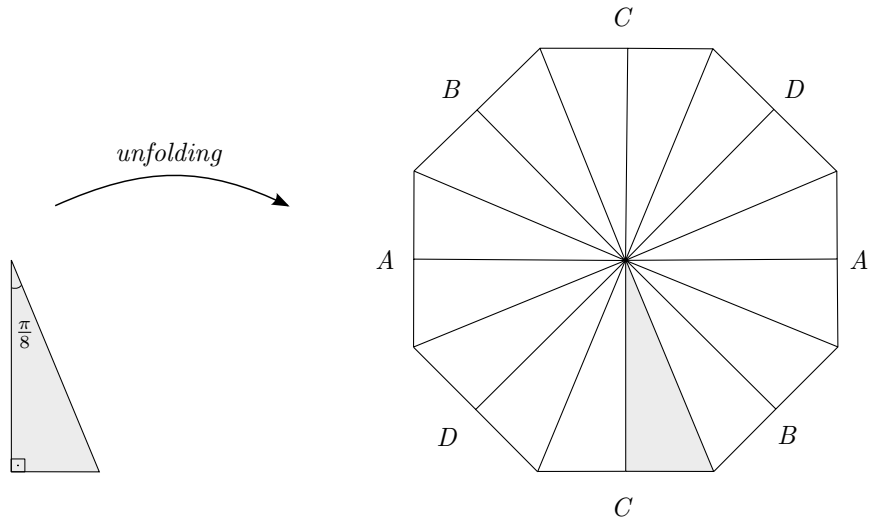
\caption{The triangle with angles $\pi/8$ and $\pi/2$ unfolds into a regular octagon.}\label{f.triangle-unfolding}
\end{center}
\end{figure}

In general, a rational polygonal $P$ of $k$ sides with angles $\pi m_i/n_i$, $1\leq i\leq N$ has a group of reflections $G$ of order $2N$ and it unfolds into a translation surface $X$ of genus $g$ given by the formula
$$2-2g = N(2-k+\sum\limits_{i=1}^N(1/n_i))$$

\subsection{Stratification of moduli spaces of translation surfaces} Once our understanding of Abelian differentials was improved thanks to the notion of translation surfaces, let us now come back to the discussion of Teichm\"uller and moduli spaces of Abelian differentials. 

Given a non-trivial Abelian differential $\omega$ on a Riemann surface $M$ of genus $g\geq 1$, we can form a list $\kappa=(k_1,\dots,k_{\sigma})$ recording the orders of the zeroes of $\omega$. Note that, by Riemann-Hurwitz theorem, this list satisfies the constraint $\sum\limits_{l=1}^{\sigma} k_l = 2g-2$. 

For each list $\kappa=(k_1,\dots,k_{\sigma})$ with $\sum\limits_{l=1}^{\sigma} k_l = 2g-2$, let $\mathcal{L}(\kappa)$ be the subset\footnote{It is possible to prove that $\mathcal{L}(\kappa)$ is non-empty whenever $\sum\limits_{l=1}^{\sigma} k_l=2g-2$.} of $\mathcal{L}_g$ consisting of all Abelian differentials whose list of orders of its zeroes coincide with $\kappa$. Since the actions of $\textrm{Diff}^+_0(M)$ and $\Gamma_g$ respect the orders of zeroes of Abelian differentials, we can take the quotients
$\mathcal{TH}(\kappa):=\mathcal{L}(\kappa)/\textrm{Diff}^+_0(M)$ and $\mathcal{H}(\kappa):=\mathcal{L}(\kappa)/\Gamma_g$.

By definition, we can write 
$$\mathcal{TH}_g:=\bigsqcup\limits_{\substack{\kappa=(k_1,\dots,k_{\sigma}) \\ k_1+\dots+ k_{\sigma}= 2g-2}} \mathcal{TH}(\kappa) \quad \textrm{ and } \quad 
\mathcal{H}_g:=\bigsqcup\limits_{\substack{\kappa=(k_1,\dots,k_{\sigma}) \\ k_1+\dots+ k_{\sigma}= 2g-2}} \mathcal{H}(\kappa)$$

In the next subsection, we will see that these decompositions of $\mathcal{TH}_g$ and $\mathcal{H}_g$ are \emph{stratifications}: the subsets $\mathcal{TH}(\kappa)$ and $\mathcal{H}(k)$ decompose  $\mathcal{TH}_g$ and $\mathcal{H}_g$ into finitely many disjoint \emph{manifolds}/\emph{orbifolds} of \emph{distinct} dimensions. For this reason, the subsets $\mathcal{TH}(\kappa)$ and $\mathcal{H}(\kappa)$ will be called \emph{strata} of the Teichm\"uller and moduli spaces of Abelian differentials (translation surfaces). 

\subsection{Period coordinates} 

Fix $\mathcal{TH}(\kappa)$ a stratum with $\kappa=(k_1,\dots, k_{\sigma})$ and $k_1+\dots+k_{\sigma} = 2g-2$. For every $\omega_0\in\mathcal{TH}(\kappa)$, one can construct an open\footnote{Here, we use the \emph{developing map} to put a natural topology on $\mathcal{TH}(\kappa)$. More concretely, given $\omega\in\mathcal{L}(\kappa)$, $p_0\in\textrm{div}(\omega)$, an universal cover $p:\widetilde{M}\to M$ and $P_1\in p^{-1}(p_0)$, we have a developing map $D_{\omega}:(\widetilde{M}, P_0)\to (\mathbb{C},0)$ determining \emph{completely} the translation structure $(M,\omega)$. The injective map $\omega\mapsto D_{\omega}$ gives a copy of $\mathcal{L}(\omega)$ inside the space $C^0(\widetilde{M},\mathbb{C})$ of complex-valued continuous functions of $\widetilde{M}$. In particular, the compact-open topology of $C^0(\widetilde{M}, \mathbb{C})$ induces natural topologies on $\mathcal{L}(\kappa)$ and $\mathcal{TH}(\kappa)$.} neighborhood $U_0\subset \mathcal{TH}(\kappa)$ such that, after naturally\footnote{Via the so-called \emph{Gauss-Manin connection}.} identifying $H_1(M,\textrm{div}(\omega),\mathbb{Z})$ and $H_1(M,\textrm{div}(\omega_0),\mathbb{Z})$ for all $\omega\in U_0$, the \emph{period map} $\Theta:U_0\to \textrm{Hom}(H_1(M,\textrm{div}(\omega_0),\mathbb{Z}),\mathbb{C})$ defined by the formula
$$\Theta(\omega):= \left(\gamma\mapsto\int_{\gamma}\omega\right)\in \textrm{Hom}(H_1(M,\textrm{div}(\omega),\mathbb{Z}), \mathbb{C})\simeq \textrm{Hom}(H_1(M,\textrm{div}(\omega_0),\mathbb{Z}), \mathbb{C})$$ 
is a \emph{local homeomorphism}. In other words, the period maps are local charts of an atlas of $\mathcal{TH}(\kappa)$. 

Recall that $\textrm{Hom}(H_1(M,\textrm{div}(\omega_0),\mathbb{Z}), \mathbb{C})\simeq H^1(M,\textrm{div}(\omega),\mathbb{C})$ is a vector space naturally isomorphic to $\mathbb{C}^{2g+\sigma-1}$: indeed, if $\{(\alpha_i,\beta_i)\}_{i=1}^g$ is a symplectic basis of $H_1(M,\mathbb{Z})$ and $\gamma_1,\dots, \gamma_{\sigma-1}$ are relative cycles connecting some fixed $p_0\in\textrm{div}(\omega_0)$ to all others $p_1,\dots, p_{\sigma-1}\in\textrm{div}(\omega_0)$, then 
$$H^1(M,\textrm{div}(\omega_0),\mathbb{C})\ni\omega\mapsto \left(\int_{\alpha_1}\omega,\int_{\beta_1}\omega,\dots,\int_{\alpha_g}\omega, \int_{\beta_g}\omega, \int_{\gamma_1}\omega, \dots, \int_{\gamma_{\sigma-1}\omega}\right)\in\mathbb{C}^{2g+\sigma-1}$$
is an isomorphism. Furthermore, by composition period maps with these isomorphisms, we see that all changes of coordinates are given by \emph{affine} transformations of $\mathbb{C}^{2g+\sigma-1}$ \emph{preserving} the Lebesgue measure. In particular, if we \emph{normalize} the Lebesgue measure so that the integral lattices $H^1(M,\textrm{div}(\omega),\mathbb{Z}\oplus\mathbb{Z}i)$ have covolume one in $H^1(M,\textrm{div}(\omega),\mathbb{C})$, then we obtain a well-defined (Lebesgue) measure $\lambda_{\kappa}$ on $\mathcal{TH}(\kappa)$. 

In summary, $\mathcal{TH}(\kappa)$ is an affine complex manifold of dimension $2g+\sigma-1$ equipped with a natural (Lebesgue) measure $\lambda_{\kappa}$ thanks to the period maps. Moreover, these structures are compatible with the action of the mapping class group $\Gamma_g$, so that $\mathcal{H}(\kappa)$ is an affine complex \emph{orbifold}\footnote{In general, $\mathcal{H}(\kappa)$ are \emph{not} manifolds: for example, the moduli space $\mathcal{H}(0)$ of flat torii is $GL^+(2,\mathbb{R})/SL(2,\mathbb{Z})$.} of dimension $2g+\sigma-1$ equipped with a natural (Lebesgue) measure $\mu_{\kappa}$. 

Geometrically, the role of period maps is easily visualized in terms of translation structures. For example, consider the polygon $Q$ depicted in Figure \ref{f.suspension-Masur} and denote by $(M,\omega_0)$ the translation surface obtained by gluing by translations the pairs of parallel sides of $Q$. Using an argument similar to Remark \ref{r.L-genus-computation}, one can show that $\omega_0\in\mathcal{TH}(2)$ and the cycles $\zeta_1$, $\zeta_2$, $\zeta_3$ and $\zeta_4$ on $M$ (i.e., the projections the sides of $Q$) form a basis of $H_1(M,\mathbb{Z})$. Hence, the period map 
$$\mathcal{TH}(2)\supset U_0\ni\omega\mapsto\left(\int_{\zeta_1}\omega, \int_{\zeta_2}\omega, \int_{\zeta_3}\omega, \int_{\zeta_4}\omega\right)\in V_0\subset \mathbb{C}^4$$ 
takes a small neighborhood $U_0$ of $\omega_0$ to a small neighborhood $V_0$ of $(\zeta_1,\zeta_2,\zeta_3,\zeta_4)\in\mathbb{C}^4$. Consequently, all $\omega\in U_0$ are described by small arbitrary perturbations (dashed red lines in Figure \ref{f.period-H2}) of the sides of the original polygon $Q$ (blue full lines in Figure \ref{f.period-H2}).  

\begin{figure}[htb!] 
\begin{center}
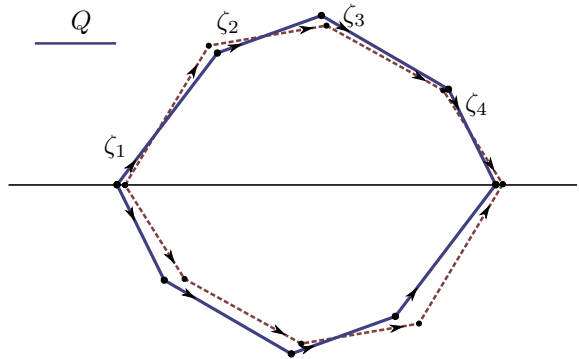
\caption{Period coordinate in $\mathcal{TH}(2)$.}\label{f.period-H2}
\end{center}
\end{figure}

\subsection{Connected components of strata} 

It might be tempting to conjecture that it is always possible to deform a given $\omega_0\in\mathcal{H}(\kappa)$ into another $\omega_1\in\mathcal{H}(\kappa)$. Nevertheless, Veech \cite{V90} discovered that $\mathcal{H}(4)$ has \emph{two} connected components: indeed, Veech distinguished these connected compoents using certain combinatorial invariants called \emph{extended Rauzy classes}\footnote{A slight modification of the notion of \emph{Rauzy classes} introduced by Rauzy \cite{R} in his study of i.e.t.'s.}.

The strategy of Veech was further pursued by Arnoux to show that $\mathcal{H}(6)$ has three connected components. However, it became clear that the classification of connected components of $\mathcal{H}(\kappa)$ via the analysis of extended Rauzy classes is a hard combinatorial problem\footnote{Rauzy classes are complicated objects: the cardinalities of the largest Rauzy classes associated to $\mathcal{H}_2$, $\mathcal{H}_3$, $\mathcal{H}_4$ and $\mathcal{H}_5$ are $15$, $2177$, $617401$ and $300296573$.}.

A complete classification of the connected components of $\mathcal{H}(\kappa)$ was obtained by Kontsevich-Zorich \cite{KZ} with the aid of algebro-geometrical invariants. Roughly speaking, they showed that the connected components can be \emph{hyperelliptic}, \emph{even spin} or \emph{odd spin}. Using these invariants of connected components, Kontsevich and Zorich proved the following result:

\begin{theorem} In genus $g=2$, both strata $\mathcal{H}(2)$ and $\mathcal{H}(1,1)$ are connected. In genus $g=3$, the strata $\mathcal{H}(4)$ and $\mathcal{H}(2,2)$ have both two connected components and all other strata are connected. In genus $g\geq 4$, we have that:
\begin{itemize}
\item the minimal stratum $\mathcal{H}(2g-2)$ has three connected components;
\item $\mathcal{H}(2l,2l)$, $l\geq 2$, has three connected components;
\item $\mathcal{H}(2l_1,\dots,2l_n)\neq\mathcal{H}(2l,2l)$, $l_i\geq 1$, has two connected components;
\item all other strata of $\mathcal{H}_g$ are connected.
\end{itemize}
\end{theorem}

\begin{remark} For later reference, let us recall the notion of \emph{parity of the spin structure} used by Kontsevich-Zorich in their definition of even spin and odd spin connected components. 

Let $(M,\omega)\in\mathcal{H}_g$ be a translation surface of genus $g\geq 1$. Given a simple smooth loop $\gamma$ in $M-\textrm{div}(\omega)$, denote by $\textrm{ind}(\gamma)$ be the index of the Gauss map of $\gamma$ and let $\phi(\gamma) = \textrm{ind}(\gamma)+1 \,\,(\textrm{mod }2)\in\mathbb{Z}/2\mathbb{Z}$. The quadratic form $\phi$ represents the symplectic intersection form $\{.,.\}$ on $H_1(M,\mathbb{Z}/2\mathbb{Z})$, i.e., $\phi(\alpha+\beta)=\phi(\alpha)+\phi(\beta)+\{\alpha,\beta\}$ for all $\alpha, \beta\in H_1(M,\mathbb{Z}/2\mathbb{Z})$. The \emph{Arf invariant} of $\phi$ is $$\Phi(M,\omega) = \sum\limits_{i=1}^g \phi(\alpha_i)\phi(\beta_i)\in\mathbb{Z}/2\mathbb{Z}$$
where $\{\alpha_i,\beta_i\}_{i=1}^g\subset H_1(M,\mathbb{Z}/2\mathbb{Z})$ is \emph{any}\footnote{It is possible to prove that the value $\Phi(M,\omega)\in\mathbb{Z}/2\mathbb{Z}$ independs of the choice.} choice of canonical symplectic basis. 

The quantity $\Phi(M,\omega)$ is the parity of the spin structure of $(M,\omega)$: by definition, $(M,\omega)$ has even, resp. odd, spin structure if $\Phi(M,\omega)=0$, resp. $1$.
\end{remark}

\subsection{$GL^+(2,\mathbb{R})$ action on $\mathcal{H}_g$} 

The correspondence between Abelian differentials and translation structures allows us to define an action of $GL^+(2,\mathbb{R})$ on $\mathcal{L}_g$. Indeed, given $(M,\omega)\in\mathcal{L}_g$, let us consider an atlas $\{\phi_{\alpha}\}_{\alpha\in I}$ of charts on $M-\textrm{div}(\omega)$ whose changes of coordinates are given by translations. A matrix $A\in GL^+(2,\mathbb{R})$ acts on $(M,\omega)$ by post-composition with the charts of this atlas, i.e., $A\cdot(M,\omega)$ is the translation surface associated to the new atlas $\{A\circ\phi_{\alpha}\}_{\alpha\in I}$. Note that this is well-defined because all changes of coordinates of this new atlas are given by translations: 
$$(A\circ\phi_{\beta})\circ(A\circ\phi_{\alpha})^{-1}(z) = A\circ(\phi_{\beta}\circ\phi_{\alpha}^{-1})(A^{-1}(z)) = A(A^{-1}(z)+c) = z+A(c)$$

Geometrically, the action of $A\in GL^+(2,\mathbb{R})$ on a translation surface $(M,\omega)$ presented by identifications by translations of pairs of parallel sides of a finite collection $\mathcal{P}$ of polygons in the plane is very simple: we apply the matrix $A$ to all polygons in $\mathcal{P}$ and we identify by translations the pairs of parallel sides as before; this operation is well-defined because the matrix $A$ respects (by linearity) the notion of parallelism in the plane. See Figure \ref{f.L-shear} for an illustration of the action of the matrix $T=\left(\begin{array}{cc} 1 & 1 \\ 0 & 1 \end{array}\right)$ on the $L$-shaped square-tiled surface from Figure \ref{f.L-shaped}. 

\begin{figure}[htb!] 
\begin{center}
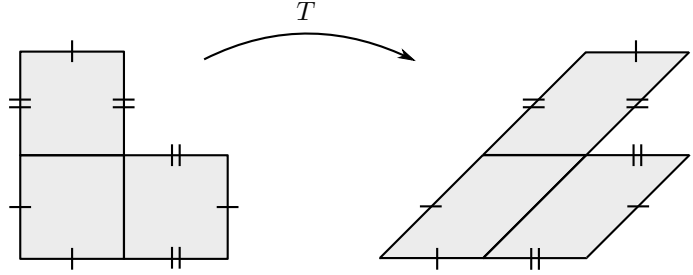
\caption{Action of a parabolic shear on a $L$-shaped origami.}\label{f.L-shear}
\end{center}
\end{figure}

This action of $GL^+(2,\mathbb{R})$ commutes with the actions of $\textrm{Diff}^+_0(M)$ and $\Gamma_g$ because $GL^+(2,\mathbb{R})$ acts by post-composition with translation charts while $\textrm{Diff}^+_0(M)$ and $\Gamma_g$ act by pre-composition with such charts. Therefore, the $GL^+(2,\mathbb{R})$-action on $\mathcal{L}_g$ descends to $\mathcal{TH}_g$ and $\mathcal{H}_g$, it respects the strata $\mathcal{TH}(\kappa)$ and $\mathcal{H}(\kappa)$, $\kappa=(k_1,\dots, k_{\sigma})$, $k_1+\dots+k_{\sigma} = 2g-2$, and the subgroup $SL(2,\mathbb{R})\subset GL^+(2,\mathbb{R})$ preserves the natural (Lebesgue) measures $\lambda_{\kappa}$ and $\mu_{\kappa}$. 

\subsection{$SL(2,\mathbb{R})$-action on $\mathcal{H}_g$}

It is not reasonable to study directly the dynamics of $GL^+(2,\mathbb{R})$ on the strata $\mathcal{H}(\kappa)$ partly because they are too large: for instance, every strata is a ``ruled space'' (foliated by the perforated complex lines $(\mathbb{C}\cdot\omega)-\{0\}$). 

For this reason, we shall restrict the action of $SL(2,\mathbb{R})$ to a fixed level\footnote{The sets $\mathcal{H}^{(a)}(\kappa)$ are ``hyperboloids'' inside $\mathcal{H}(\kappa)$: indeed, this follows from the fact that $A_{\kappa}(\omega) = \frac{i}{2}(\sum\limits_{n=1}^g (A_n \overline{B_n} - \overline{A_n} B_n)$ where $A_n=\int_{\alpha_n}\omega$ and $B_n=\int_{\beta_n}\omega$ are the periods of $\omega$ with respect to a canonical symplectic basis $\{\alpha_n, \beta_n\}_{n=1}^g$ of $H_1(M,\mathbb{R})$.} set $\mathcal{H}^{(a)}(\kappa):=A_{\kappa}^{-1}(\{a\})$, $a\in\mathbb{R}^+$, say $a=1$, of the total area function $A_{\kappa}:\mathcal{H}(\kappa)\to\mathbb{R}^+$ given by 
$$A_{\kappa}(\omega) := \frac{i}{2}\int\omega\wedge\overline{\omega}$$ 

In this way, we obtain an action of $SL(2,\mathbb{R})$ on a space $\mathcal{H}_{\kappa}^{(1)}$ supporting $SL(2,\mathbb{R})$-invariant \emph{probability} measures. In fact, a celebrated result obtained independently by Masur \cite{M82} and Veech \cite{V82} says that the disintegration of the $SL(2,\mathbb{R})$-invariant $\lambda_{\kappa}$ on $\mathcal{H}^{(1)}_{\kappa}$ has \emph{finite} mass, and, hence its normalization $\lambda^{(1)}_{\kappa}$ is a $SL(2,\mathbb{R})$-invariant probability measure on $\mathcal{H}^{(1)}(\kappa)$ (called \emph{Masur-Veech measure} in the literature). 

\subsection{Teichm\"uller flow and Kontsevich-Zorich cocycle} 

In this setting, the \emph{Teichm\"uller flow} is simply the action of the diagonal subgroup $g_t= \textrm{diag}(e^t, e^{-t})$, $t\in\mathbb{R}$, of $SL(2,\mathbb{R})$ on the strata $\mathcal{H}^{(1)}(\kappa)$ of the moduli space $\mathcal{H}_g^{(1)}$ of Abelian differentials of genus $g\geq 1$ with unit total area.

An important aspect of the Teichm\"uller flow is its role as a \emph{renormalization dynamics} for translation flows on translation surfaces. In particular, it is often the case that the dynamical features of this flow has profound consequences in the theory of interval exchange transformations, billiards in rational polygons and translation flows (see Section 6 of \cite{FM} and the references therein for more explanations). For example, Masur \cite{M82} and Veech \cite{V82} exploited the recurrence\footnote{Coming from Poincar\'e recurrence theorem.} of almost all orbits of the Teichm\"uller flow with the Masur-Veech probability measure to independently confirm a conjecture of Keane on the unique ergodicity of almost every interval exchange transformations. 

In this memoir, we will be mostly interested in the Teichm\"uller flow in itself (even though we will occasionally mention its applications to interval exchange transformations and translation flows). 

An important point in the analysis of the Teichm\"uller flow $g_t$ is the study of its derivative $Dg_t$ in period coordinates. In the sequel, we will introduce the so-called \emph{Kontsevich-Zorich} (KZ) \emph{cocycle} and we will see that the \emph{relevant} part of the $Dg_t$ is encoded by this cocycle. 

We start with the trivial bundle $\widehat{H_g^1}:=\mathcal{TH}_g^{(1)}\times H^1(M,\mathbb{R})$ and the trivial dynamical cocycle over the Teichm\"uller flow: 
$$\widehat{G_t^{KZ}}: \widehat{H_1^g} \to \widehat{H_1^g}, \quad \widehat{G_t^{KZ}}(\omega, c) := (g_t(\omega), c)$$ 
Now, we note that the mapping class group $\Gamma_g$ acts on \emph{both} factors of $\widehat{H_g^1}$, so that the quotients $H_g^1:=\widehat{H_g^1}/\Gamma_g$ and $G_t^{KZ}:= \widehat{G_t^{KZ}}/\Gamma_g$ are well-defined. In the literature, $H_g^1$ is called the real Hodge bundle over $\mathcal{H}_g^{(1)}$ and $G_t^{KZ}$ is called Kontsevich-Zorich cocycle\footnote{A similar definition can be performed over the action of $SL(2,\mathbb{R})$ and, by a slight abuse of notation, we shall also call  ``Kontsevich-Zorich cocycle'' the resulting object.}. 

\begin{remark}\label{r.KZ-lift} Strictly speaking, the KZ cocycle is not a \emph{linear} cocycle in the usual sense of Dynamical Systems because the real Hodge bundle is an \emph{orbifold} bundle. In fact, one might have \emph{ambiguities} in the definition of $G_t^{KZ}$ along $g_t$-orbits of translation surfaces $(M,\omega)$ with a non-trivial group $\textrm{Aut}(M,\omega)$ of automorphisms. In concrete terms, the fiber $H^1(M,\mathbb{R})/\textrm{Aut}(M,\omega)$ of $H_g^1$ over such $(M,\omega)$ might not be a vector space, so that the linear maps on $H^1(M,\mathbb{R})$ induced by $G_t^{KZ}$ is well-defined \emph{only} up to the cohomological action of $\textrm{Aut}(M,\omega)$. Fortunately, this ambiguity is not a serious problem as far as \emph{Lyapunov exponents} are concerned. Indeed, it is well-known that Lyapunov exponents are not affected under \emph{finite} covers, so that we can safely replace $G_t^{KZ}$ by its lift to a finite cover of $H_g^1$ obtained by taking a finite-index, torsion-free subgroup $\Gamma_g^0$ of $\Gamma_g$ (e.g., $\Gamma_g^0=\{\phi\in\Gamma_g: \phi_*=\textrm{id} \textrm{ on } H_1(M,\mathbb{Z}/3\mathbb{Z})\}$). 
\end{remark}

Contrary to its parent $\widehat{G_t^{KZ}}$, the KZ cocycle $G_t^{KZ}$ is \emph{far} from trivial: since $(\omega, c)$ is identified with $(\rho^*(\omega), \rho^*(c))$ for all $\rho\in\Gamma_g$ in the construction of $H_g^1$, the fibers of $\widehat{H_g^1}$ over $\omega$ and $\rho^*(\omega)$ are identified in a non-trivial way if $\rho\in\Gamma_g$ acts non-trivially on $H^1(M,\mathbb{R})$. Alternatively, if we fix a fundamental domain $\mathcal{D}$ of the action of $\Gamma_g$ on $\mathcal{TH}_g$, and we start with a generic  $\omega\in\textrm{int}(\mathcal{D})$ and a cohomology class $c\in H^1(M,\mathbb{R})$, then after running the Teichm\"uller flow for some long $t_0$ we eventually hit $\partial\mathcal{D}$ while pointing towards the exterior of $\mathcal{D}$. At this moment, since $\mathcal{D}$ is a fundamental domain, we have the option of applying an element $\rho\in\Gamma$ to replace $g_{t_0}(\omega)$ by a point $\rho^*(g_{t_0}(\omega))$ flowing towards $\textrm{int}(\mathcal{D})$ at the cost of replacing $c$ by $\rho^*(c)$: see Figure \ref{f.KZ-fundamental-domain}. 

\begin{figure}[htb!] 
\begin{center}
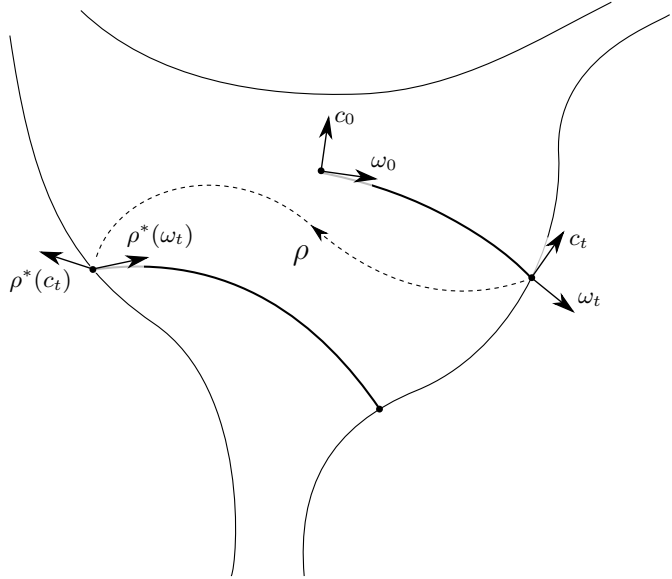
\caption{Kontsevich-Zorich cocycle on $H_1^g$.}\label{f.KZ-fundamental-domain}
\end{center}
\end{figure}

Also, $G_t^{KZ}$ is \emph{symplectic cocycle} because the action of $\Gamma_g$ on $H^1(M,\mathbb{R})$ preserves the symplectic intersection form $\{c,c'\}:=\int_M c\wedge c'$. This fact has the following consequence for the Lyapunov exponents of the KZ cocycle. Given $\mu$ an ergodic Teichm\"uller flow invariant probability measure on $\mathcal{H}_g^{(1)}$ and any choice\footnote{For example, we can take $\|.\|$ to be the so-called \emph{Hodge norm}, see \cite{F02}.} of norm $\|.\|$ with $\int \log\|G_{\pm t}^{KZ}\| d\mu<\infty$ for all $0\leq t\leq 1$, the multiplicative ergodic theorem of Oseledets guarantees the existence of real numbers (Lyapunov exponents) $\lambda_1^{\mu}>\dots>\lambda_k^{\mu}$ and a $G_t^{KZ}$-equivariant measurable decomposition $H^1(M,\mathbb{R}) = E_1(\omega)\oplus\dots\oplus E_k(\omega)$ at $\mu$-almost every $\omega$ such that 
$$\lim\limits_{t\to\pm\infty} \frac{1}{t}\log(\|G_t^{KZ}(\omega, v)\|/\|v\|) = \lambda_i^{\mu} \quad \forall \, \, v\in E_i(\omega)-\{0\}$$ 
In general, we will write the Lyapunov exponent $\lambda_i^{\mu}$ with multiplicity $\textrm{dim}E_i(\omega)$ in order to obtain a list of $2g=\textrm{dim}H^1(M,\mathbb{R})$ Lyapunov exponents
$$\lambda_1^{\mu}\geq\dots\geq\lambda_{2g}^{\mu}$$
In our setting, the symplecticity of $G_t^{KZ}$ implies that its Lyapunov exponents are symmetric\footnote{This reflects the fact that the eigenvalues of a symplectic matrix comes in pairs of the form $\theta$ and $1/\theta$.} around the origin:
$$\lambda_{2g-i+1}^{\mu} = -\lambda_{i}^{\mu} \quad \forall \, \, 1\leq i\leq g$$

By definition, $G_t^{KZ}$ acts on the \emph{tautological plane} $H_{st}^1(M,\omega):=\mathbb{R}.\textrm{Re}(\omega) \oplus \mathbb{R}.\textrm{Im}(\omega)\subset H^1(M,\mathbb{R})$ by the matrix $g_t=\textrm{diag}(e^t, e^{-t})$ (after identifying $e_1=(1,0)\simeq \textrm{Re}(\omega)$, $e_2=(0,1)\simeq \textrm{Im}(\omega)$ and $H_{st}^1(M,\mathbb{R})\simeq\mathbb{R}^2$). This means that $\pm 1$ are Lyapunov exponents of any Teichm\"uller invariant probability measure $\mu$. In fact, it is possible to prove that $1=\lambda_1^{\mu}>\lambda_2^{\mu}$: see \cite{F02}. 

Now, let us relate the KZ cocycle $G_t^{KZ}$ to the derivative $Dg_t$ of the Teichm\"uller flow. By writing $Dg_t$ in period coordinates and by writing $H^1(M,\textrm{div}(\omega),\mathbb{C}) = \mathbb{R}^2\times H^1(M,\textrm{div}(\omega),\mathbb{Z})$, we have that $Dg_t$ acts by the matrix $g_t=\textrm{diag}(e^t, e^{-t})$on the first factor $\mathbb{R}^2$ and by the natural generalization $\widetilde{G_t^{KZ}}$ of the KZ cocycle on the second factor $H^1(M,\textrm{div}(\omega),\mathbb{Z})$. In particular, the Lyapunov exponents of $Dg_t$ have the form $\pm 1 +\lambda$ where $\lambda$ are Lyapunov exponents of $\widetilde{G_t^{KZ}}$. 

Next, we observe that the ``relative part'' of $H^1(M,\textrm{div}(\omega), \mathbb{Z})$ does not contribute with interesting Lyapunov exponents. More precisely, the fact that two relative cycles in $H_1(M,\textrm{div}(\omega),\mathbb{Z})$ with the same boundaries always differ by an absolute cycle can be exploited to prove that $\widetilde{G_t^{KZ}}$ acts \emph{trivially} on the relative part, i.e., the kernel of the natural map $H^1(M,\textrm{div}(\omega),\mathbb{R})/H^1(M,\mathbb{R})$. Hence, the relative part provides $\sigma-1$ zero Lyapunov exponents of $\widetilde{G_t^{KZ}}$ and, \emph{a fortiori}, the interesting part is the restriction $G_t^{KZ}$ of $\widetilde{G_t^{KZ}}$ to $H^1(M,\mathbb{R})$. In summary, $G_t^{KZ}$ captures the most exciting part of $Dg_t$. 

The relationship between $G_t^{KZ}$ and $Dg_t$ described above allows us to recover the Lyapunov exponents of the Teichm\"uller flow from the Lyapunov exponents of the KZ cocycle: if $\mu$ is an ergodic $g_t$-invariant probability measure supported on $\mathcal{H}^{(1)}(\kappa)$, $\kappa=(k_1,\dots, k_{\sigma})$, $k_1+\dots+k_{\sigma}=2g-2$, then the Lyapunov exponents of $g_t$ with respect to $\mu$ are 
\begin{eqnarray*}
2\geq 1+\lambda_2^{\mu}\geq\dots\geq 1+\lambda_g^{\mu}\geq \overbrace{1=\dots=1}^{\sigma-1}\geq 
1-\lambda_g^{\mu}\geq\dots\geq 1-\lambda_2^{\mu}\geq 0 \\ 
\geq -1+\lambda_2^{\mu}\geq\dots\geq-1+\lambda_g^{\mu}\geq \overbrace{-1=\dots=-1}^{\sigma-1}\geq -1-\lambda_g^{\mu}\geq\dots\geq-1-\lambda_2^{\mu}\geq-2
\end{eqnarray*}
where $1>\lambda_2^{\mu}\geq\dots \geq \lambda_g^{\mu}$ are the non-negative exponents of $G_t^{KZ}$ with respect to $\mu$. 

\subsection{Teichm\"uller curves, Veech surfaces and affine homeomorphisms}\label{ss.Veech-surfaces} 

The Teichm\"uller flow and the KZ cocycle take a particularly explicit description in the case of \emph{Teichm\"uller curves}. 

By definition, a Teichm\"uller curve is a \emph{closed} $SL(2,\mathbb{R})$-orbit in $\mathcal{H}_g^{(1)}$. By a result of Smillie (see \cite{SW}), the $SL(2,\mathbb{R})$-orbit of a translation surface $X$ is a Teichm\"uller curve if and only if the stabilizer $SL(X)$ of $X$ in $SL(2,\mathbb{R})$ is a lattice. 

The group $SL(X)$ is called \emph{Veech group} of the translation surface $X$. We say that a translation surface $X$ whose Veech group is a lattice in $SL(2,\mathbb{R})$ is called \emph{Veech surface}. In this language, Smillie's result says that Teichm\"uller curves are precisely the $SL(2,\mathbb{R})$-orbits of Veech surfaces. 

The Teichm\"uller curve generated by a Veech surface $X$ is isomorphic to $SL(2,\mathbb{R})/SL(X)$, i.e., the unit cotangent of the finite-area hyperbolic surface $\mathbb{H}/SL(X)$. In particular, the Teichm\"uller flow $g_t$ on Teichm\"uller curves is simply the geodesic flow on certain finite-area hyperbolic surfaces. 

At first sight, it is not obvious that Veech surfaces exist. Nevertheless, a \emph{dense} set of Veech surfaces in any stratum $\mathcal{H}_{\kappa}^{(1)}$ can be constructed as follows. The set $\mathcal{S}$ of translation surfaces $(M,\omega)$ whose image under period maps belong to $\textrm{Hom}(M,\textrm{div}(\omega), \mathbb{Q}\oplus\mathbb{Q}i)$ is dense (because $\mathbb{Q}\oplus\mathbb{Q}i$ is dense in $\mathbb{C}$). It was shown by Gutkin and Judge \cite{GJ} that a translation surface $X$ belongs to $\mathcal{S}$ if and only if its Veech group $SL(X)$ is commensurable to $SL(2,\mathbb{Z})$ or, equivalently, $X$ is a \emph{square-tiled surface} (i.e., a translation surface obtained by finite cover of a flat square torus). Since $SL(2,\mathbb{Z})$ is a lattice of $SL(2,\mathbb{R})$, we have that any square-tiled surfaces is a Veech surface, so that $\mathcal{S}$ is the desired dense set of Veech surfaces. 

An alternative characterization of square-tiled surfaces is provided by the so-called \emph{trace field} of the corresponding Veech groups. More precisely, if $X$ is a Veech surface of genus $g\geq 1$, then its trace field $K(X):=\mathbb{Q}(\{\textrm{tr}(\gamma):\gamma\in SL(X)\})$ obtained by adjoining to $\mathbb{Q}$ all traces of elements in $SL(X)$ can be shown to be a finite extension of $\mathbb{Q}$ of degree $1\leq \textrm{deg}_{\mathbb{Q}}(K(X))\leq g$. In this setting, $X$ is a square-tiled surface if and only if its trace field is $\mathbb{Q}$. For this reason, the Teichm\"uller curves generated by square-tiled surfaces are called \emph{arithmetic Teichm\"uller curves}. 

\begin{remark} A square-tiled surface $X$ is combinatorially described by a pair of permutations $h$ and $v$ modulo simultaneous conjugations: after numbering the squares used to build up $X$ from $1$ to $N$, we define $h(i)$, resp. $v(i)$ as the square to the right, resp. on the top, of $i$. Since our choice of numbering is arbitrary, $(\phi h \phi^{-1}, \phi v \phi^{-1})$ and $(h,v)$ determine the same square-tiled surface. 

Moreover, all square-tiled surfaces in a given Teichm\"uller curve can be found by the following algorithm. We fix a pair of permutations $(h,v)$ associated to a square-tiled surface $X$ in our preferred Teichm\"uller curve. All square-tiled surfaces in the $SL(2,\mathbb{R})$-orbit of $X$ belong to the $SL(2,\mathbb{Z})$-orbit of $X$. Since $SL(2,\mathbb{Z})$ is generated by the parabolic matrices $T=\left(\begin{array}{cc} 1 & 1 \\ 0 & 1 \end{array}\right)$ and $S=\left(\begin{array}{cc}1 & 0 \\ 1 & 1\end{array}\right)$, we can algorithmatically compute $SL(2,\mathbb{Z})$-orbits of square-tiled surfaces by determining how the matrices $T$ and $S$ act on pairs $(h,v)$ of permutations. As it turns out, a direct inspection shows that $T(h,v) = (h,vh^{-1})$ and $S(h,v)=(hv^{-1},v)$. 
\end{remark}

The KZ cocycle over a Teichm\"uller curve is described by the cohomological action of \emph{affine homeomorphisms} of a Veech surface. 

More concretely, an affine homeomorphism of a translation surface $(M,\omega)$ is an orientation-preserving homeomorphism of $M$ preserving $\textrm{div}(\omega)$ whose local expressions in translation charts of $M-\textrm{div}(\omega)$ are affine transformations of the plane. 

Any affine homeomorphism $f$ has a well-defined linear part $Df\in SL(2,\mathbb{R})$ because the change of coordinates in $M-\textrm{div}(\omega)$ are translations. Therefore, we have a natural homomorphism 
$$D:\textrm{Aff}(M,\omega)\to SL(2,\mathbb{R})$$ 
from the group $\textrm{Aff}(M,\omega)$ of affine homeomorphisms to $SL(2,\mathbb{R})$. By definition, the kernel of $D$ is the group $\textrm{Aut}(M,\omega)$ of automorphisms of $(M,\omega)$. Also, it is not hard to check that the image of $D$ coincides with the Veech group of $(M,\omega)$. In particular, we have a short exact sequence 
$$1\to \textrm{Aut}(M,\omega)\to \textrm{Aff}(M,\omega)\to SL(M,\omega)\to 1$$ 

The stabilizer of the $SL(2,\mathbb{R})$-orbit of $(M,\omega)\in\mathcal{TH}_g$ in $\Gamma_g$ is precisely the group $\textrm{Aff}(M,\omega)$ of its affine homeomorphisms. In particular, the KZ cocycle over the $SL(2,\mathbb{R})$-orbit of $(M,\omega)$ is the quotient of the trivial cocycle 
$$g_t\times\textrm{id}: SL(2,\mathbb{R})(M,\omega)\times H^1(M,\mathbb{R})\to SL(2,\mathbb{R})(M,\omega)\times H^1(M,\mathbb{R})$$ 
by the natural action of $\textrm{Aff}(M,\omega)$ on both factors. 

This interpretation of the KZ cocycle in terms of affine homeomorphisms is useful to produce concrete matrices of this cocycle. For example, let us consider the $L$-shaped square-tiled surface $(M,\omega)$ from Figure \ref{f.L-shaped}. This translation surface decomposes into two horizontal \emph{cylinders}, i.e., two maximal collections of closed geodesics parallel to the horizontal direction: see Figure \ref{f.L-shaped-horizontal-cylinders}. 

\begin{figure}[htb!] 
\begin{center}
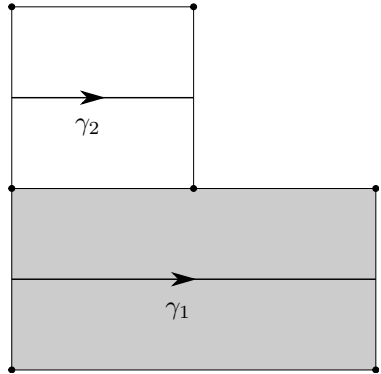
\caption{Horizontal cylinders of $(M,\omega)$ and their waist curves $\gamma_1$ and $\gamma_2$.}\label{f.L-shaped-horizontal-cylinders}
\end{center}
\end{figure}

This collection of horizontal cylinders can be used to define a special type of affine homeomorphism of $(M,\omega)$ called \emph{Dehn multitwist}. 

Suppose that $C$ is a maximal horizontal cylinders of height $h$ and widths $w$. By definition, we can cut and paste by translation the image of $C$ under any power $T_{w/h}^n$, $n\in\mathbb{N}$, of the parabolic matrix $T_{w/h}:=\left(\begin{array}{cc} 1 & w/h \\ 0 & 1 \end{array}\right)$ in order to recover $C$: in other words, $T_{w/h}^n$ stabilizes $C$. Also, $T_{w/h}^n$ fixes the waist curve of $C$ while adding $n$ times the waist curve of $C$ to any cycle crossing $C$ upwards. The matrices $T_{w/h}^n$ are a particular example of a Dehn multitwist. 

In the case of the $L$-shaped square-tiled surface $(M,\omega)$, we have two horizontal cylinders $C_1$ and $C_2$ whose waist curves $\gamma_1$ and $\gamma_2$ are depicted in Figure \ref{f.L-shaped-horizontal-cylinders}. Note that $C_1$ has width two, $C_2$ has width one, and both $C_i$, $i=1, 2$, have height one. Thus, the parabolic matrix $T_2=\left(\begin{array}{cc} 1 & 2 \\ 0 & 1 \end{array}\right)$ stabilize both $C_1$ and $C_2$, and, \emph{a fortiori}, $T_2$ defines an affine homeomorphism of $(M,\omega)$. Furthermore, our description of the effect of Dehn multitwists on the waist curves and cycles crossing cylinders says that $T_2$ acts on the basis $\{\sigma,\mu, \zeta,\nu\}$ of $H_1(M,\mathbb{R})$ in Figure \ref{f.L-cycle-basis} via: 
$$(T_2)_*(\sigma) = \sigma, \quad (T_2)_*(\mu) = \mu, \quad (T_2)_*(\zeta) = \zeta+2\mu, \quad (T_2)_*(\nu) = \nu + \sigma+\mu$$

\begin{figure}[htb!] 
\begin{center}
\input{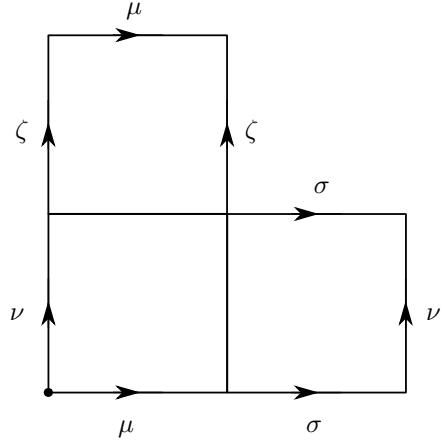}
\caption{A choice of basis of homology of a $L$-shaped origami.}\label{f.L-cycle-basis}
\end{center}
\end{figure}

Hence, the KZ cocycle matrix corresponding to the action of $T_2$ on the basis $\{\sigma,\mu, \zeta,\nu\}$ is

$$(T_2)_*=\left(\begin{array}{cccc} 1 & 0 & 0 & 1 \\ 0 & 1 & 2 & 1 \\ 0 & 0 & 1 & 0 \\ 0 & 0 & 0 & 1 \end{array}\right)$$

\begin{remark} Strictly speaking, we compute a matrix of the \emph{dual} of the KZ cocycle: indeed, this cocycle was defined in terms of the action on cohomology groups $H^1(M,\mathbb{R})$, but our calculations were in homology groups $H_1(M,\mathbb{R})$, i.e., the duals of $H^1(M,\mathbb{R})$ (by Poincar\'e duality). Of course, this is a minor detail that is usually not very important. 
\end{remark}

\newpage 

\null

\newpage


\begin{centering}
\rule{\textwidth}{1.6pt}\vspace*{-\baselineskip}\vspace*{2.5pt}
\rule{\textwidth}{0.4pt}

\section{Proof of the Eskin-Kontsevich-Zorich regularity conjecture}\label{s.AMY}

\rule{\textwidth}{0.4pt}\vspace*{-\baselineskip}\vspace{3.2pt}
\rule{\textwidth}{1.6pt}
\end{centering}\\

In 1980, the physicists J. Hardy and J. Weber conjectured that the diffusion rate of typical trajectories in $\mathbb{Z}^2$-periodic Ehrenfest wind-tree models of Lorenz gases is abnormal: more precisely, if  $\phi_t^{\theta}$ is the billiard flow in direction $\theta\in S^1$ in the billiard table $T(a,b)\subset\mathbb{R}^2$, $0<a,b<1$, obtained by putting rectangular obstacles of dimensions $a\times b$ at each $(m,n)\in\mathbb{Z}^2$, then Hardy-Weber conjecture predicts that 
$$\limsup\limits_{t\to\infty}\frac{\log d_{\mathbb{R}^2}(\phi_t^{\theta}(x), x)}{\log t} > \frac{1}{2}$$
for Lebesgue almost every $\theta\in S^1$ and $x\in T(a,b)$. 

In a recent work, Delecroix, Hubert and Leli\`evre \cite{DHL} confirmed this conjecture by proving the following stronger result: the rate of diffusion of a typical\footnote{I.e., for Lebesgue almost every $\theta$ and $x$.} trajectory $\phi_t^{\theta}(x)$ in $T(a,b)$ is 
$$\limsup\limits_{t\to\infty}\frac{\log d_{\mathbb{R}^2}(\phi_t^{\theta}(x), x)}{\log t} = \frac{2}{3}$$ 

\subsection{Eskin-Kontsevich-Zorich formula} 

Among several important ingredients used by Delecroix-Hubert-Leli\`evre \cite{DHL}, we find a remarkable formula of Eskin-Kontsevich-Zorich \cite{EKZ} for the sum of non-negative Lyapunov exponents of the KZ cocycle with respect to $SL(2,\mathbb{R})$-invariant probability measures. In fact, the diffusion rate in Delecroix-Hubert-Leli\`evre theorem is a Lyapunov exponent $\lambda$ of the KZ cocycle with respect to a certain $SL(2,\mathbb{R})$-invariant probability measure on the moduli space $\mathcal{H}_5$ of Abelian differentials of genus five, and its explicit value $\lambda=2/3$ was computed thanks to Eskin-Kontsevich-Zorich formula. 

In a nutshell, Eskin-Kontsevich-Zorich formula relates sums of Lyapunov of the KZ cocycle to the flat geometry of translation surfaces in the following way. Given an ergodic $SL(2,\mathbb{R})$-invariant probability measure $m$ on the moduli space $\mathcal{H}_g^{(1)}$ of Abelian differentials of genus $g\geq 1$ with total area one, Kontsevich \cite{K} and Forni \cite{F02} proved that the sum of the non-negative Lyapunov exponents of $m$ can be expressed in terms of the integral of the curvature of the determinant of the Hodge bundle with respect to $m$. In general, it is not always easy to work directly with the curvature $\Theta$ of the Hodge bundle and, for this reason, Eskin-Kontsevich-Zorich used the Riemann-Roch-Hirzebruch-Grothendieck theorem to convert the integral of $\Theta$ into the sum of a combinatorial term $\frac{1}{12}\sum\limits_{l=1}^{\sigma} \frac{k_l(k_l+2)}{k_l+1}$ depending on the orders $k_1,\dots, k_{\sigma}$ of the zeroes of $\omega\in\textrm{supp}(m)$ and a certain integral expression $I$ depending on the flat geometry of the translation surfaces in $\textrm{supp}(m)$. Finally, Eskin-Kontsevich-Zorich derive their formula by relating $I$ to the so-called \emph{Siegel-Veech constants} associated to counting problems of flat cylinders in translation surfaces in $\textrm{supp}(m)$. 

An important point in Eskin-Kontsevich-Zorich's proof of their formula is the fact that most arguments use only the $SL(2,\mathbb{R})$-invariance of $m$: indeed, there is just a single place in their paper (namely,  \cite[Section 9]{KZ}) where a certain \emph{regularity} assumption on $m$ is required in order to justify an integration by parts argument. 

The regularity condition on $m$ is defined in \cite{EKZ} as follows. Recall that a cylinder $C$ in a translation surface $(M,\omega)$ is a maximal collection of parallel closed geodesic in $(M,\omega)$ and the modulus $\textrm{mod}(C)$ of a cylinder is the quotient $\textrm{mod}(C)=h(C)/w(C)$, where $h(C)$ is the height of $C$ and $w(C)$ is the width of $C$. A $SL(2,\mathbb{R})$-invariant probability measure $m$ on $\mathcal{H}_g^{(1)}$ is \emph{regular} if there exists a constant $K>0$ such that 
$$\lim\limits_{\rho\to 0}\frac{m(\mathcal{H}_g^{cyl}(K,\rho))}{\rho^2} = 0 \quad (\textrm{i.e., } m(\mathcal{H}_g^{cyl}(K,\rho)) = o(\rho^2))$$ 
where $\mathcal{H}_g^{cyl}(K,\rho)$ is the set of Abelian differentials $(M,\omega)\in\mathcal{H}_g^{(1)}$ possessing two non-parallel cylinders $C_1$ and $C_2$ with moduli $\textrm{mod}(C_i)\geq K$ and widths $w(C_i)\leq\rho$ for $i=1,2$. 

\subsection{Statement of the Eskin-Kontsevich-Zorich regularity conjecture} 

By the time that Eskin-Kontsevich-Zorich wrote their paper \cite{EKZ}, the regularity of all \emph{known} examples of $SL(2,\mathbb{R})$-invariant probability measures on moduli spaces of translation surfaces was established by \emph{ad-hoc} methods: in particular, Eskin-Kontsevich-Zorich formula could be applied in many contexts. 

Nevertheless, it is natural to ask what is the exact range of applicability of Eskin-Kontsevich-Zorich formula. In this direction, Eskin-Kontsevich-Zorich \cite{EKZ} made the conjecture that \emph{all} $SL(2,\mathbb{R})$-invariant probability measures in moduli spaces of translation surfaces are regular. 

In our joint work \cite{AMY} with Avila and Yoccoz, we confirmed Eskin-Kontsevich-Zorich regularity conjecture by showing the following (slightly stronger) result. 

\begin{theorem}\label{t.AMY} Let $m$ be an ergodic $SL(2,\mathbb{R})$-invariant probability measure on a connected component $\mathcal{C}$ of a stratum of the moduli space of unit area translation surfaces of genus $g\geq 1$. 

Denote by $\mathcal{C}_{(2)}(\rho)$ the set of translation surfaces $(M,\omega)\in\mathcal{C}$ possessing two non-parallel saddle connections of lengths $\leq \rho$. Then, 
$$m(\mathcal{C}_{(2)}(\rho)) = o(\rho^2)$$  
\end{theorem}

\begin{remark} Recall that a \emph{saddle connection} of a translation surface $(M,\omega)$ is a geodesic segment $\gamma\subset M$ such that $\partial\gamma\subset\textrm{div}(\omega)$ and $\textrm{int}(\gamma)\cap\textrm{div}(\omega)=\emptyset$. 

Since the boundary of cylinder $C$ is the union of (finitely many) saddle connections, the existence of a cylinder $C$ of width $w(C)$ implies the existence of a saddle connection of length $\leq w(C)$. In particular, this justifies our claim that Theorem \ref{t.AMY} is a slightly stronger conclusion than the statement predicted in Eskin-Kontsevich-Zorich regularity conjecture. 
\end{remark}

Intuitively, Theorem \ref{t.AMY} says that if $\rho>0$ is small, then $\mathcal{C}_{(2)}(\rho)$ occupies a small fraction of the set $\{M\in\mathcal{C}:\textrm{sys}(M)\leq\rho\}$
of translation surfaces $M\in\mathcal{C}$ whose \emph{systole} $\textrm{sys}(M)$ (i.e., the length of the shortest saddle connections of $M$) is at most $\rho$. In fact, our theorem asserts that $m(\mathcal{C}_{(2)}(\rho)) = o(\rho^2)$, while the following lemma of Veech \cite{V98} and Eskin-Masur \cite{EMas} ensures that the set $\{M\in\mathcal{C}:\textrm{sys}(M)\leq\rho\}$ has $m$-mass of order $\sim\rho^2$: 

\begin{lemma}\label{l.Eskin-Masur} Let $m$ be an ergodic $SL(2,\mathbb{R})$-invariant probability measure on a connected component $\mathcal{C}$ of a stratum of the moduli space of unit area translation surfaces of genus $g\geq 1$. Then, 
$$m(\{M\in\mathcal{C}: \textrm{sys}(M)\leq\rho\}) = O(\rho^2)$$
\end{lemma}

\begin{proof}[Sketch of proof] The key idea in the proof of this lemma is the so-called \emph{Siegel-Veech formula}. 

By following Eskin-Masur \cite{EMas}, let us first discuss the general version of the Siegel-Veech formula (which has little to do with moduli spaces, but rather the action of $SL(2,\mathbb{R})$ on $\mathbb{R}^2$). 

Suppose that $SL(2,\mathbb{R})$ acts on a space $X$. Let us fix $\mu$ a $SL(2,\mathbb{R})$-invariant probability measure on $X$ and a function $V$ assigning a subset $V(x)\subset\mathbb{R}^2-\{(0,0)\}$ of non-zero vectors in $\mathbb{R}^2$ with weights/multiplicities to each $x\in X$. Later in the proof of this lemma, $X=\mathcal{C}$ and $V(x)$ is the discrete subset of holonomy vectors $\int_{\gamma}\omega$ of saddle connections $\gamma$ in $x=(M,\omega)\in X$. 

In general, the Siegel-Veech formula concerns functions $V$ with the following properties:
\begin{itemize}
\item $V$ is $SL(2,\mathbb{R})$-equivariant, i.e., $V(gx) = g(V(x))$ for all $x\in X$ and $g\in SL(2,\mathbb{R})$; 
\item there exists a constant $c(x)>0$ for each $x\in X$ such that $N_V(x,R):=\#(V(x)\cap B(0,R))$ is at most $c(x) R^2$ for all $R>0$ (where $B(0,R)\subset\mathbb{R}^2$ is the Euclidean ball of radius $R$ centered at the origin); moreover, $c(x)$ can be chosen uniformly on compact subsets of $X$; 
\item there are $R>0$ and $\varepsilon>0$ such that $N_V(x,\mathbb{R})\in L^{1+\varepsilon}(X,\mu)$.
\end{itemize}
The non-trivial fact that these conditions hold for the particular case $X=\mathcal{C}$ of the moduli space of unit area translation surfaces and $V$ is the function assigning the set of holonomies of saddle connections was proved by Eskin-Masur \cite{EMas}. 

Coming back to the general setting, let $f\in C^{\infty}_0(\mathbb{R}^2)$ be a real-valued function with  compact support. We define its \emph{Siegel-Veech transform} $\widehat{f}:X\rightarrow\mathbb{R}$ as
$$\widehat{f}(x)=\sum\limits_{v\in V(x)}f(v)$$

In this language, the Siegel-Veech formula asserts that 
$$\int_X \widehat{f}(x) \, d\mu(x) = c(\mu) \int_{\mathbb{R}^2} f(v) \, d\textrm{Leb}_{\mathbb{R}^2}(v)$$
where $c(\mu)=c_V(\mu)\geq 0$ is the so-called \emph{Siegel-Veech constant} of $\mu$ (with respect to $V$). At first sight, the Siegel-Veech formula looks tricky to prove because $X$,$\mu$ and $V$ are ``arbitrary''. Nevertheless, this formula becomes easy to derive if we notice that
$$ f\in C^{\infty}_0(\mathbb{R}^2)\mapsto \int_X \widehat{f}(x) \, d\mu(x)$$
is a non-negative linear functional on $C^{\infty}_0(\mathbb{R}^2)$, i.e., a measure on $\mathbb{R}^2$: indeed, this linear functional is well-defined because $\widehat{f}$ is finite, bounded on compact sets and $\widehat{f}\in L^{1+\varepsilon}(X,\mu)\subset L^1(X,\mu)$ by our assumptions on $V$. Furthermore, the $SL(2,\mathbb{R})$-equivariance of $V$ implies that this measure on $\mathbb{R}^2$ is $SL(2,\mathbb{R})$-invariant. Since the sole $SL(2,\mathbb{R})$-invariant measures on $\mathbb{R}^2$ are linear combinations of the Dirac measure at the origin $(0,0)\in\mathbb{R}^2$ and the Lebesgue measure $\textrm{Leb}_{\mathbb{R}^2}$, it follows that this measure has the form
$$\int_X \widehat{f}(x) \, d\mu(x) = a f(0,0)+ b\int f \, d\textrm{Leb}_{\mathbb{R}^2}$$

Finally, since $V(x)\subset\mathbb{R}^2-\{(0,0)\}$, it is possible to check that $a=0$, so that the Siegel-Veech formula holds (with $b=c_V(\mu)$).

Once we know the Siegel-Veech formula, we can deduce that $\mu(\mathcal{C}_1(\rho))=O(\rho^2)$ by applying this formula to a ``smooth version''$f_{\rho}$ of the characteristic function of the ball $B(0,\rho)\subset\mathbb{R}^2$:
$$m(\mathcal{C}_1(\rho))\leq \int \widehat{f_{\rho}} \, d\mu = c(\mu) \int f_{\rho} \, d\textrm{Leb}_{\mathbb{R}^2}=O(\rho^2).$$ 
This proves the lemma. 
\end{proof}

The remainder of this section is devoted to the proof of Eskin-Kontsevich-Zorich regularity conjecture (or, more precisely, Theorem \ref{t.AMY}). 

\subsection{Idea of the proof of Theorem \ref{t.AMY}} 

The basic idea behind the proof of Theorem \ref{t.AMY} is to use a conditional measure argument to reduce the global estimate on $m$ to an orbit by orbit estimates saying that the $SL(2,\mathbb{R})$-Haar measures of the intersections of $\mathcal{C}_{(2)}(\rho)$ with certain pieces of $SL(2,\mathbb{R})$-orbits are $o(\rho^2)$. 

More precisely, given $\rho>0$, let $X(\rho)=\{M\in\mathcal{C}:\textrm{sys}(M)=\rho\}$. Inside the $\rho$-level $X(\rho)$ of the systole function $\textrm{sys}$, we consider the subsets 
$$X_0^*(\rho):=\{M\in X(\rho): \textrm{ all non-vertical saddle-connections have length }>\rho\}$$
and 
$$X^*(\rho):=\bigcup\limits_{-\pi/2<\theta\leq\pi/2}R_{\theta}(X_0^*(\rho))$$
where $R_{\theta}\in SO(2,\mathbb{R})$ denotes the rotation by $\theta$. 

Starting from $X_0^*(\rho)$, we can access deeper levels of the systole function via the set 
$$Y^*(\rho)=\bigcup\limits_{|\theta|<\pi/4}\bigcup\limits_{0\leq t<\log\cot|\theta|}g_t R_\theta(X_0^*(\rho))$$
Indeed, the choice of $\theta$ and $t$ is guided by the fact that the vector $g_t R_{\theta} e_2$ is shorter than the (unit) vector $e_2=(0,1)\in\mathbb{R}^2$ for $0\leq t<\log\cot|\theta|$, $|\theta|<\pi/4$, so that the systole of $g_t R_{\theta} M_0$ is smaller than the systole of $M_0\in X_0^*(\rho)$. 

Furthermore, $Y^*(\rho)$ is an \emph{interesting} way to access $\{M\in\mathcal{C}:\textrm{sys}(M)\leq\rho\}$ because the sets $g_t R_{\theta}(X_0^*(\rho))$ for $|\theta|<\pi/4$ and $0\leq t<\log\cot|\theta|$ form a \emph{measurable partition} (in Rokhlin's sense) of $Y^*(\rho)$. In particular, by the $SL(2,\mathbb{R})$-invariance of $m$, we will be able to compute the $m$-measure of subsets of $Y^*(\rho)$ in terms of the Lebesgue measure $dt$ on $\mathbb{R}$, the Lebesgue measure $\cos 2\theta d\theta$ on the circle and a certain \emph{flux measure} $m_0=m_0^{\rho}$ on $X_0^*(\rho)$. 

Using this disintegration, we can \emph{transfer} mass from $X_0^*(\rho)$ to deep levels $\{M\in\mathcal{C}:\textrm{sys}(M)\leq\rho\exp(-T)\}$, $T>0$, as follows. First, we will show that, for $|\sin2\theta|<\exp(-2T)$, there is an open interval $J(T,\theta)$ of $t's$ (whose length is explicitly computable) such that $\textrm{sys}(g_t R_{\theta}(M_0))\leq\rho\exp(-T)$ for all $M_0\in X_0^*(\rho)$. Geometrically, the set $Y(\rho, T)$ of $g_tR_{\theta}M_0$ for $M_0\in X_0^*(\rho)$, $|\sin2\theta|<\exp(-2T)$, $t\in J(T,\theta)$ correspond to the pieces of hyperbolas below the threshold $\rho\exp(-T)$. 
%
%
%
%
Secondly, we use the disintegration results to show that the $m$-measure of $\{M\in\mathcal{C}:\textrm{sys}(g_t R_{\theta}(M))\leq\rho\exp(-T)\}$ is \emph{at least} 
$$m_0^{\rho}(X_0^*(\rho))\int_{|\sin2\theta|<\exp(-2T)} |J(T,\theta)| \cos2\theta d\theta=\frac{\pi}{2}(\exp(-T))^2 m_0^{\rho}(X_0^*(\rho))$$

At this point, the idea to derive Theorem \ref{t.AMY} is very simple. We will show that there is a (positive) constant $c(m)$ such that: 
\begin{itemize}
\item as $s\to 0$, the $m$-measure of $\{M\in\mathcal{C}:\textrm{sys}(g_t R_{\theta}(M))\leq s\}$ is $\frac{1}{2}(c(m)+o(1))s^2$, and
\item there exists a sequence $(\rho_n)_{n\in\mathbb{N}}$ with $\rho_n\to 0$ as $n\to\infty$ such that the densities $\pi m_0^{\rho_n}(X_0^*(\rho_n))$ are $(c(m)-o(1))\rho_n^2$.
\end{itemize}
Intuitively, this says that the flux through $X_0^*(\rho_n)$ is almost maximal.\footnote{At first sight, the factor of $1/2$ might seem strange, but, as we will show, in general, the flux through $X_0^*(\rho)$ equals $F'(\rho)/\rho$ where $F(\rho)=m(\{M\in\mathcal{C}: \textrm{sys}(M)\leq \rho\})$. In particular, by L'H\^opital rule, we \emph{expect} that 
$$\limsup\limits_{\rho\to0}F'(\rho)/\rho=c(m)$$ 
if $\lim\limits_{\rho\to0} F(\rho)/\rho^2=(1/2)c(m)$.}

In any case, putting these facts together, we deduce that 
\begin{eqnarray*}
\frac{1}{2}(c(m)+o(1))\rho_n^2\exp(-2T)&\geq& m(\{M\in\mathcal{C}:\textrm{sys}(M)\leq\rho_n\exp(-T)\}) \\ 
&\geq& m(Y(\rho_n, T))=\frac{\pi}{2}(\exp(-T))^2 m_0^{\rho_n}(X_0^*(\rho_n)) \\
&\geq& \frac{1}{2}(c(m)-o(1))(\rho_n\exp(-T))^2
\end{eqnarray*}

From this, we get that the set $Y(\rho_n, T)$ of translation surfaces with systole $\leq \rho_n\exp(-T)$ ``accessed'' from $X_0^*(\rho_n)$ occupies most of $\{M\in\mathcal{C}:\textrm{sys}(M)\leq \rho_n\exp(-T)\}$ in the sense that its complement has $m$-measure $o(1)(\rho_n\exp(-T))^2$ for all $T>0$.

Finally, once we know that most translation surfaces with systole $\leq \rho_n\exp(-T)$ ``come'' from $X_0^*(\rho)$, we complete the proof of Theorem \ref{t.AMY} by showing that the translation surfaces $M_0\in X_0^*(\rho_n)$ leading to translation surfaces $M=g_t R_{\theta} M_0\in Y(\rho_n, T)\cap \mathcal{C}_2(\rho_n\exp(-T))$ are (essentially) those $M_0$ with two non-parallel saddle-connections of lengths comparable to $\rho_n$ making a very small\footnote{Here, ``very small angle'' means that $\theta_0$ becomes close to zero for $T$ is sufficiently large (depending on $\rho_n$).} angle $\theta_0$. Then, since the $m_0^{\rho_n}$-density of the set of those $M_0$ is small, say $o(1)\rho_n^2$, for $\theta_0$ small, i.e., $T$ large, we can use again that $m$ disintegrates as $dt\times \cos2\theta d\theta\times m_0^{\rho_n}$ to conclude that the $m$-measure of $Y(\rho_n,T)\cap\mathcal{C}_2(\rho_n\exp(-T))$ is $o(1)(\rho_n\exp(-T))^2$ for $T$ large, as desired. 

Of course, there are plenty of details to check in this scheme and the next subsections serve to formalize the ideas above. 

\subsection{Reduction of Theorem \ref{t.AMY} to Propositions \ref{p.AMY1} and \ref{p.AMY2}}

Given a connected component $\mathcal{C}$ of a stratum of the moduli space of unit area translation surfaces of genus $g\geq 1$, let us denote by $\mathcal{C}(A,\rho)$ the subset of $(M,\omega)\in\mathcal{C}$ with a minimizing (i.e., length $\textrm{sys}(M)$) saddle-connection $\gamma$ of size $\rho$ and another saddle-connection $\delta$ of length $\leq A\cdot \textrm{sys}(M)$ which is not parallel to $\gamma$. 

\begin{lemma}\label{l.AMY} Suppose that, for each $A>1$, one has $m(\mathcal{C}(A,\rho))=o(\rho^2)$. Then, $m(\mathcal{C}_{(2)}(\rho))=o(\rho^2)$. 
\end{lemma}

\begin{proof} By Lemma \ref{l.Eskin-Masur}, we know that 
\begin{equation}\label{e.Eskin-Masur}
m(\{M\in\mathcal{C}: \textrm{sys}(M)\leq s\})\leq C(m) s^2
\end{equation}
for some constant $C(m)>1$ and for all $s>0$. 

Given $0<\eta<1$, it follows from \eqref{e.Eskin-Masur} that  
$$m(\{M\in\mathcal{C}: \textrm{sys}(M)\leq\rho/A\})\leq\frac{\eta}{2}\rho^2$$
for $A:=A(\eta):=\sqrt{2C(m)/\eta}$ and for all $\rho>0$. 

On the other hand, our hypothesis imply the existence of $\rho_0=\rho_0(A(\eta))>0$ such that 
$$m(\mathcal{C}(A,\rho))\leq \frac{\eta}{2}\rho^2$$
for all $0<\rho<\rho_0$. 

Since $\mathcal{C}_{(2)}(\rho)\subset \{M\in\mathcal{C}: \textrm{sys}(M)\leq\rho/A\}\cup\mathcal{C}(A,\rho)$ for any $A>1$, we deduce from the previous two estimates that 
$$m(\mathcal{C}_{(2)}(\rho))\leq \eta\rho^2$$
for all $0<\rho<\rho_0(A(\eta))$. 

Because $0<\eta<1$ was arbitrary, the proof of the lemma is complete. 
\end{proof}

This lemma reduces the proof of Theorem \ref{t.AMY} to the following result:

\begin{theorem}\label{t.AMY'} For each fixed $A>1$, one has $m(\mathcal{C}(A,\rho))=o(\rho^2)$. 
\end{theorem} 

Our proof of Theorem \ref{t.AMY'} is naturally divided into two statements. First, we will show that a large portion of $\{M\in\mathcal{C}:\textrm{sys}(M)\leq \rho_0\exp(-T)\}$, $T>0$, can be captured with the aid of the $SL(2,\mathbb{R})$ by pushing certain translation surfaces $M_0$ with $\textrm{sys}(M_0)=\rho_0$ for an adequate choice of the level $\rho_0$ of the systole function. 

\begin{proposition}\label{p.AMY1} Given $\eta>0$, there exists $\rho_0=\rho_0(\eta)>0$ with the following property. Let  $X_0^*$ be the set of translation surfaces $M\in\mathcal{C}$ with $\textrm{sys}(M)=\rho_0$ whose non-vertical saddle-connections have lengths $>\rho_0$, and, for each $T>0$, $\omega_0\in (0,\pi/2]$, and $B\subset X_0^*$ a Borel subset, denote by 
$$Y(T,\omega_0,B):=\{M=g_t R_{\theta}M_0\in\mathcal{C}: M_0\in B, |\sin 2\theta|<\exp(-T) \sin\omega_0, \|g_tR_{\theta}e_2\|<\exp(-T)\}$$
Then, for all $T>0$, the subset $Y(T, \pi/2, X_0^*)$ of $\{M\in\mathcal{C}: \textrm{sys}(M)<\rho_0 \exp(-T)\}$ has almost full $m$-measure, i.e.,   
$$m(\{M\in\mathcal{C}: \textrm{sys}(M)<\rho_0 \exp(-T)\} - Y(T,\pi/2,X_0^*))<\frac{\eta}{2}\rho_0^2 \exp(-2T)$$
\end{proposition} 

Secondly, for each fixed $A>1$, we will exploit the geometry of saddle-connections of translation surfaces $M\in Y(T,\pi/2,X_0^*)$ for $T\gg 1$ sufficiently large to prove that the $m$-measure of $Y(T,\pi/2,X_0^*)\cap\mathcal{C}(A,\rho_0\exp(-T))$ is small. 

\begin{proposition}\label{p.AMY2} Given $A>1$, $\rho_0>0$ and $\eta>0$, there exists $T_0=T_0(A,\rho_0,\eta)$ such that 
$$m(Y(T,\pi/2,X_0^*)\cap\mathcal{C}(A,\rho_0\exp(-T)))\leq \frac{\eta}{2}\rho_0^2\exp(-2T)$$
for all $T\geq T_0$. 
\end{proposition}

Of course, these propositions imply Theorem \ref{t.AMY'}. 

\begin{proof}[Proof of Theorem \ref{t.AMY'}] Fix $A>1$. Given $\eta>0$, we choose $\rho_0=\rho_0(\eta)>0$ as in Proposition \ref{p.AMY1} and $T_0=T_0(A,\rho_0(\eta),\eta)=T_0(A,\eta)$ as in Proposition \ref{p.AMY2}. By writing $\rho=\rho_0\exp(-T)$, the conclusions of Propositions \ref{p.AMY1} and \ref{p.AMY2} tell us that 
\begin{eqnarray*}
m(\mathcal{C}(A,\rho))&\leq& m(\{M\in\mathcal{C}: \textrm{sys}(M)<\rho\} - Y(T,\pi/2,X_0^*)) + m(Y(T,\pi/2,X_0^*)\cap\mathcal{C}(A,\rho)) \\ 
&\leq& \frac{\eta}{2}\rho_0^2\exp(-2T) + \frac{\eta}{2}\rho_0^2\exp(-2T) = \eta\rho^2 
\end{eqnarray*}
for all $0<\rho=\rho_0\exp(-T)\leq\rho_0\exp(-T_0)$. Since $\eta>0$ was arbitrary, the proof is complete.
\end{proof}

In the sequel, we shall reduce Propositions \ref{p.AMY1} and \ref{p.AMY2} to the following facts about the measure $m$ (whose proofs are postponed to Subsections \ref{ss.p.mY*} and \ref{ss.flux}. First, the $SL(2,\mathbb{R})$-invariance of $m$, Rokhlin disintegration theorem and the features of the Haar measure of $SL(2,\mathbb{R})$ will be exploited to show the following result.

\begin{proposition}\label{p.mY*} Given $\rho_0>0$ such that $\{M\in\mathcal{C}:\textrm{sys}(M)>\rho_0\}$ has positive $m$-measure, denote by $X_0^*$ the set of $M\in\mathcal{C}$ with $\textrm{sys}(M)=\rho_0$ such that all non-vertical saddle-connections of $M$ have length $>\rho_0$. 

Then, the set 
$$Y^*:=\{M\in\mathcal{C}: M=g_tR_{\theta}M_0, M_0\in X_0^*, |\theta|<\pi/4, \|g_tR_{\theta}e_2\|<1\}$$ 
has positive $m$-measure and the restriction of $m$ to $Y^*$ has the form 
$$m|_{Y^*}=dt\times\cos(2\theta)d\theta\times m_0$$
where $m_0$ is a finite measure on $X_0^*$. 

In particular, for each $T>0$, $\omega_0>0$, $B\subset X_0^*$ Borel, the $m$-measure of the set 
$$Y(T,\omega_0,B):=\{M=g_tR_{\theta}M_0\in\mathcal{C}: M_0\in B, |\sin 2\theta|<\exp(-2T)\sin\omega_0, \|g_tR_{\theta}e_2\| < \exp(-T)\}$$
equals to 
$$m(Y(T,\omega_0,B)) = \frac{1}{4}\exp(-2T)m_0(B)\int_{-\omega_0}^{\omega_0}\log\frac{1+\cos\omega}{1-\cos\omega} \cos\omega d\omega$$
\end{proposition}

Also, we will show that the total mass of the measure $m_0$ introduced above can be interpreted as a flux of the measure $m$ through the level set $\{M\in\mathcal{C}:\textrm{sys}(M)=\rho_0\}$ of the systole function. 

\begin{proposition}\label{p.slice} For any $\rho_0>0$ with $m(\{M\in\mathcal{C}:\textrm{sys}(M)>\rho_0\})>0$, one has 
$$\lim\limits_{\tau\to 0}\frac{1}{\tau}m(\{M\in\mathcal{C}:\rho_0\exp(-\tau)\leq \textrm{sys}(M)\leq\rho_0\}) = \pi m_0(X_0^*)$$
\end{proposition}

\subsection{Proof of Proposition \ref{p.AMY1} (modulo Propositions \ref{p.mY*} and \ref{p.slice})} 

Denote by $F(\rho):=m(\{M\in\mathcal{C}: \textrm{sys}(M)\leq\rho\})$. Note that $F(\rho)$ is a non-decreasing function of $\rho$. 

\begin{lemma} The function $F(\rho)$ is continuous, i.e., $m(\{M\in\mathcal{C}:\textrm{sys}(M)=\rho\})=0$ for all $\rho>0$.
\end{lemma} 

\begin{proof} Fix $\rho>0$. By Fubini's theorem, 
$$m(\{M\in\mathcal{C}:\textrm{sys}(M)=\rho\}) = \int_{\mathcal{C}} \mu_L(\{g\in SL(2,\mathbb{R}): \textrm{sys}(gx)=\rho\}) \, dm(x)$$
where $\mu_L$ is the normalized restriction of the Haar measure on $SL(2,\mathbb{R})$ to the compact subset $L:=\{g\in SL(2,\mathbb{R}): \|g\|\leq 2\}$. 

On the other hand, by a result of Masur (see \cite{M90}), the number of length-minimizing saddle-connections on a translation surface $X\in\mathcal{C}$ with $\textrm{sys}(X)=\rho$ is uniformly bounded in terms of a constant depending only on $\rho$ and the genus $g$ of $X$. 

It follows that, for each $x\in\mathcal{C}$, the $\nu_L$-measure of $\{g\in SL(2,\mathbb{R}): \textrm{sys}(gx)=\rho\}$ is zero, and, \emph{a fortiori}, $m(\{M\in\mathcal{C}:\textrm{sys}(M)=\rho\}) = 0$.
\end{proof}

By Proposition \ref{p.slice}, the function $F(\rho)$ has a left-derivative $F'(\rho)$ at every $\rho_0$ with $F(\rho_0)<1$ and  
$$\rho_0 F'(\rho_0) = \pi m_0(X_0^*)$$ 

By Proposition \ref{p.mY*} and the elementary fact\footnote{The change of variables $u=\tan(\omega/2)$ gives that $\int_0^{\pi/2}\log\frac{1+\cos\omega}{1-\cos\omega}\cos\omega\,d\omega = 4\int_0^1 \log(1/u)\frac{1-u^2}{(1+u^2)}du$. It follows that $\int_0^{\pi/2}\log\frac{1+\cos\omega}{1-\cos\omega}\cos\omega\,d\omega =\pi$ because $\frac{1-u^2}{(1+u^2)^2}=\sum\limits_{n\geq 0}(-1)^n(2n+1)u^{2n}$, $\int_0^1\log(1/u) u^n du=1/(n+1)^2$ and $\sum\limits_{n\geq 0}\frac{(-1)^n}{2n+1}=\pi$  (cf. Lemma 3.5 in \cite{AMY}).} that $\int_{-\pi/2}^{\pi/2}\log\frac{1+\cos\omega}{1-\cos\omega}\cos\omega\,d\omega=2\pi$, we deduce that 
\begin{equation}\label{e.p.6.1}
F(\rho_0\exp(-T))\geq m(Y(T,\pi/2,X_0^*)) = \frac{1}{2}\exp(-2T)\rho_0 F'(\rho_0)
\end{equation}
for all $T>0$ and $\rho_0>0$ (with $F(\rho_0)<1$).

\begin{lemma}\label{l.p.6.2} The function $F$ is absolutely continuous and its left-derivative verifies 
$$F'(\rho)=O(\rho)$$

Moreover, the constant 
$$c(m):=\frac{1}{2}\sup\limits_{F(\rho)<1}\frac{F'(\rho)}{\rho}$$
satisfies 
$$\lim\limits_{\rho\to 0}\frac{F(\rho)}{\rho^2} = c(m) = \frac{1}{2}\limsup\limits_{\rho\to 0}\frac{F'(\rho)}{\rho}$$ 
\end{lemma}

\begin{proof} By Lemma \ref{l.Eskin-Masur}, we know that $F(\rho)=O(\rho^2)$. Hence, 
$$F'(\rho)\leq 2 \frac{F(\rho)}{\rho} = O(\rho)$$ 
thanks to \eqref{e.p.6.1}. In particular, $F'$ is bounded. 

We affirm that $F$ is absolutely continuous. Actually, this follows immediately from the general claim: if a continuous function $f$ on an interval $[\rho_0,\rho_1]$ whose left-derivative is bounded by $C$, then 
$$|f(\rho_1)-f(\rho_0)|\leq C(\rho_1-\rho_0)$$ 
The proof of this claim is not difficult. For each $C'>C$, denote by 
$$I(C')=\{\rho\in [\rho_0,\rho_1]: |f(\rho_1)-f(\rho)|\leq C'|\rho_1-\rho|\}$$
Note that $I(C')$ is not empty (because $I(C')\ni\rho_1$), $I(C')$ is closed (by continuity of $f$), and $I(C')$ is open to the left\footnote{This means that if $\rho_*\in I(C')\cap (\rho_0,\rho_1]$, then there exists $\delta_*=\delta_*(\rho_*)$ such that $[\rho_* - \delta_*, \rho_*]\subset I(C')$.} (because the left-derivative of $f$ is bounded by $C<C'$). By connectedness, it follows that $I(C')=[\rho_0,\rho_1]$. Since $C'>C$ was arbitrary, the claim is proved. 

The absolute continuity of $F$ implies that $F$ is the integral of its almost everywhere derivative: 
$$F(\rho)=\int_0^{\rho} F'(s) \, ds = \int_0^{\rho} \frac{F'(s)}{s} s \, ds$$ 
Therefore, 
$$\limsup\limits_{\rho\to 0} \frac{F(\rho)}{\rho^2} \leq \frac{1}{2}\limsup\limits_{\rho\to 0} \frac{F'(\rho)}{\rho}\leq \frac{1}{2}\sup_{F(\rho)<1}\frac{F'(\rho)}{\rho}:=c(m)$$ 
Moreover, the estimate \eqref{e.p.6.1} says that 
$$\liminf\limits_{\rho\to 0} \frac{F(\rho)}{\rho^2}\geq \frac{1}{2}\sup_{F(\rho)<1}\frac{F'(\rho)}{\rho}:=c(m)$$ 
It follows from these inequalities that 
$$c(m)\leq \liminf\limits_{\rho\to 0} \frac{F(\rho)}{\rho^2}\leq \limsup\limits_{\rho\to 0} \frac{F(\rho)}{\rho^2} \leq \frac{1}{2}\limsup\limits_{\rho\to 0} \frac{F'(\rho)}{\rho}\leq c(m)$$
This completes the proof of the lemma. 
\end{proof}

At this point, we are ready to prove Proposition \ref{p.AMY1}. Indeed, given $\eta>0$, we use Lemma \ref{l.p.6.2} to select $\rho_0=\rho_0(\eta)>0$ (with $F(\rho_0)<1$) such that 
$$\frac{1}{2}\frac{F'(\rho_0)}{\rho_0} > c(m)-\frac{\eta}{4}$$
and 
$$\frac{F(\rho)}{\rho^2} < c(m)+\frac{\eta}{4} \quad \textrm{ for all } 0<\rho<\rho_0$$ 
By plugging these estimates into \eqref{e.p.6.1} and by writing $\rho=\rho_0\exp(-T)$, $T>0$, we deduce that 
\begin{eqnarray*}
F(\rho) - m(Y(T,\pi/2,X_0^*))  
&=& F(\rho) - \frac{1}{2}\frac{F'(\rho_0)}{\rho_0}\rho^2 \\ 
&<& \left(c(m)+ \frac{\eta}{4}\right)\rho^2 - \left(c(m)- \frac{\eta}{4}\right)\rho^2 \\ 
&=& \frac{\eta}{2}\rho^2 
\end{eqnarray*}
Since $m(\{M\in\mathcal{C}: \textrm{sys}(M)<\rho\} - Y(T,\pi/2,X_0^*)):=F(\rho) - m(Y(T,\pi/2,X_0^*))$, the proof of Proposition \ref{p.AMY1} is complete. 

\subsection{Proof of Proposition \ref{p.AMY2} (modulo Proposition \ref{p.mY*})} The basic idea behind the proof of Proposition \ref{p.AMY2} is very simple: given $A>1$ and $M\in Y(T,\pi/2,X_0^*)$, i.e., $M=g_t R_{\theta} M_0$ with $$|\sin 2\theta|<\exp(-2T), \quad \|g_t R_{\theta} e_2\|<\exp(-T), \quad M_0\in X_0^*,$$ 
we will show that $M$ can not have saddle-connections with length $\leq A\cdot\textrm{sys}(M)$ which are not parallel to length-minimizing ones unless $M_0$ and $\theta$ satisfy some severe constraints; by Proposition \ref{p.mY*}, these constraints imply that the $m$-measure of $Y(T,\pi/2,X_0^*)\cap \mathcal{C}(A,\rho_0\exp(-T))$ must be small.

More precisely, we start with the following result saying that if $\theta$ is not too small, then the long saddle-connections of $M_0\in X_0^*$ can not give rise to a saddle-connection of $M = g_t R_{\theta} M_0$ of length $\leq A\cdot\textrm{sys}(M)$. 

\begin{lemma}\label{l.AMY.6.5} Given $\omega_0>0$, the constant $K=K(\omega_0):= \sqrt{1 + \frac{4}{\sin^2\omega_0}}$ has the following property. For all $T>0$,  $\exp(-2T)\sin\omega_0<|\sin(2\theta)|<\exp(-2T)$,
and $t\in\mathbb{R}$ with $\|g_tR_{\theta} e_2\|<\exp(-T)$, one has 
$$\|g_t R_{\theta}e_2\|\leq K\exp(-t)$$
In particular, for such $T$, $\theta$ and $t$, we have 
$$\|g_t R_{\theta} v\| > A K \exp(-t)\geq A\|g_t R_{\theta} e_2\|$$
for all vector $v\in\mathbb{R}^2$ with $\|v\|\geq A K$. Thus, in this setting, a saddle-connection of $M=g_t R_{\theta} M_0$ of length $\leq A\cdot\textrm{sys}(M)$ does not come from a saddle-connection of $M_0\in X_0^*$ of length $>AK\rho_0$. 
\end{lemma}

\begin{proof} By definition 
$$\|g_t R_{\theta} e_2\|^2 = e^{-2t}(\cos^2\theta + e^{4t}\sin^2\theta) \leq e^{-2t} (1+e^{4t}\sin^2\theta)$$ 
Moreover, the fact that $\|g_t R_{\theta} e_2\|<\exp(-T)$ implies that 
$$e^{2t}\sin^2\theta\leq\|g_t R_{\theta} e_2\|^2<\exp(-2T)$$ 
It follows from these estimates that 
\begin{equation}\label{e.AMY.6.5.I}
\|g_t R_{\theta} e_2\|^2\leq e^{-2t}(1+\exp(-2T)e^{2t})
\end{equation}

On the other hand, the hypothesis $e^{-2t}\cos^2\theta+e^{2t}\sin^2\theta=\|g_tR_{\theta}e_2\|^2<\exp(-2T)$ becomes 
$$x^2\sin^2\theta - \exp(-2T) x + \cos^2\theta <0$$
after the change of variables $x=e^{2t}$. Since the largest root of this second degree inequality is 
$$x_+:=\frac{\exp(-2T)+\sqrt{\exp(-4T)-\sin^2(2\theta)}}{2\sin^2\theta},$$
we deduce that 
$$e^{2t}=x\leq x_+=\frac{\exp(-2T)+\sqrt{\exp(-4T)-\sin^2(2\theta)}}{2\sin^2\theta} = \frac{\exp(-2T)(1+\cos\omega)}{2\sin^2\theta}
$$
after the change of variables $\sin(2\theta):=\exp(-2T)\sin\omega$ (with $\cos\omega>0$). Because $\omega_0<|\omega|<\pi/2$ (thanks to our assumption that $\exp(-2T)\sin\omega_0<|\sin(2\theta)|<\exp(-2T)$), we deduce from this last inequality that 
\begin{equation}\label{e.AMY.6.5.II}
\exp(-2T)e^{2t}\leq \exp(-4T)\frac{1}{\sin^2\theta}\leq \frac{4}{\sin\omega_0}
\end{equation}

By combining \eqref{e.AMY.6.5.I} and \eqref{e.AMY.6.5.II}, we obtain that 
$$\|g_t R_{\theta} e_2\|^2\leq \exp(-2t)\left(1+\frac{4}{\sin^2\omega_0}\right)=:\exp(-2t) K(\omega_0)^2$$
This completes the proof of the lemma. 
\end{proof}

Next, we show that given $A>1$, all saddle-connections of $M=g_t R_{\theta} M_0\in Y(T,\pi/2,X_0^*)$ of length $\leq A\cdot\textrm{sys}(M)$ comes exclusively from saddle-connections of $M_0\in X_0^*$ of length $\leq A\cdot\textrm{sys}(M_0)$ making a small angle with the vertical direction whenever $T$ is sufficiently large.

\begin{lemma}\label{l.AMY.6.7} Given $A>1$ and $\overline{\theta_0}>0$, there exists $T_0=T_0(A,\overline{\theta_0})\geq 1$ such that 
$$\|g_t R_{\theta+\theta'} e_2\| > A \|g_t R_{\theta} e_2\|$$
for all $T\geq T_0$, $|\sin 2\theta|<\exp(-2T)$, $\|g_tR_{\theta}e_2\|<\exp(-T)$, and $\overline{\theta_0}<|\theta'|<\pi/2$. 

In particular, in this setting, a saddle-connection of $M=g_t R_{\theta} M_0\in Y(T,\pi/2,X_0^*)$ of length $\leq A\cdot \textrm{sys}(M)$ does not come from a saddle-connection of $M_0\in X_0^*$ making an angle $>\overline{\theta_0}$ with the vertical direction. 
\end{lemma} 

\begin{proof} Since $|\sin(2\theta)|\leq \exp(-2T)$ (by hypothesis), we have $|\theta|<\overline{\theta_0}/2$ for all $T$ sufficiently large depending on $\overline{\theta_0}$, say $T\geq T_0(\overline{\theta_0})$. Hence, 
$$\|g_t R_{\theta+\theta'} e_2\|^2 = e^{2t}\sin^2(\theta+\theta') + e^{-2t}\cos^2(\theta+\theta')\geq e^{2t}\sin^2(\overline{\theta_0}/2)$$ 

On the other hand, our assumption that $e^{2t}\sin^2\theta + e^{-2t}\cos^2\theta = \|g_tR_{\theta}e_2\|^2<\exp(-2T)$ implies that $x=e^{2t}$ solves the second degree inequality 
$$x^2\sin^2\theta - \exp(-2T) x + \cos^2\theta <0$$
whose smallest root is 
$$x_-:=\frac{\exp(-2T)-\sqrt{\exp(-4T)-\sin^2(2\theta)}}{2\sin^2\theta}=\exp(-2T)\frac{1-\cos\omega}{2\sin^2\theta}$$ 
where $\sin 2\theta:=\exp(-2T)\sin\omega$ and $\cos\omega>0$. Thus, 
$$e^{2t}\geq \exp(-2T)\frac{1-\cos\omega}{2\sin^2\theta}$$

It follows from this discussion that 
$$\|g_t R_{\theta+\theta'} e_2\|^2\geq \sin^2(\overline{\theta_0}/2)\exp(-2T)\frac{1-\cos\omega}{2\sin^2\theta} > \sin^2(\overline{\theta_0}/2)\frac{1-\cos\omega}{2\sin^2\theta}\|g_tR_{\theta}e_2\|^2$$
for all $T\geq T_0(\overline{\theta_0})$.  

Next, we notice $|\cos\theta|\geq 1/2$ whenever $T$ is larger than an absolute constant (because $|\sin(2\theta)|<\exp(-2T)$). Hence, 
$$2(1-\cos\omega) \geq 1-\cos^2\omega = \sin^2\omega = \exp(4T)\sin^2(2\theta)\geq \exp(4T)\sin^2\theta$$ 

By combining the previous two inequalities, we conclude that 
$$\|g_t R_{\theta+\theta'} e_2\|^2 > \sin^2(\overline{\theta_0}/2)\frac{\exp(4T)}{4}\|g_tR_{\theta}e_2\|^2\geq A^2\|g_tR_{\theta}e_2\|^2$$ 
for all $T\geq T_0=T_0(A,\overline{\theta_0})$. This proves the lemma. 
\end{proof}

These lemmas have the following consequence for the study of $Y(T,\pi/2,X_0^*)\cap \mathcal{C}(A,\rho_0\exp(-T))$: 

\begin{corollary}\label{c.} Fix $A>1$ and $\rho_0>0$ (with $m(\{M\in\mathcal{C}: \textrm{sys}(M)>\rho_0\})>0$). Given $\omega_0>0$, let $K=K(\omega_0)=\sqrt{1+4\sin^{-2}\omega_0}$ and, for each $M_0\in X_0^*$, denote by $\overline{\theta}_{\omega_0}(M_0)$ the minimal angle between a non-vertical saddle-connection of $M_0$ of length $A K(\omega_0)\textrm{sys}(M_0) = A K \rho_0$ and the vertical direction (with the convention that $\overline{\theta}_{\omega_0}(M_0)=\pi/2$ whenever such saddle-connections do not exist). 

Then, for each $\omega_0>0$ and $\overline{\theta_0}>0$, one has 
$$Y(T,\pi/2,X_0^*)\cap\mathcal{C}(A,\rho_0\exp(-T))\subset Y(T,\omega_0,X_0^*)\cup Y(T,\pi/2,B_{\omega_0}(\overline{\theta_0})) \quad \forall \, T\geq T_0(A,\overline{\theta_0})$$ 
where $T_0(A,\overline{\theta_0})$ is the constant provided by Lemma \ref{l.AMY.6.7} and $B_{\omega_0}(\overline{\theta_0}):=\{M_0\in X_0^*: \overline{\theta}_{\omega_0}(M_0)\leq\overline{\theta_0}\}$.
\end{corollary}

\begin{proof} Let $M\in Y(T,\pi/2,X_0^*)$. Our task is to show that if $M\notin Y(T,\omega_0,X_0^*)\cup Y(T,\pi/2,B_{\omega_0}(\overline{\theta_0}))$, then $M\notin\mathcal{C}(A,\rho_0\exp(-T))$. 

For this sake, we note that if $M\notin Y(T,\omega_0,X_0^*)\cup Y(T,\pi/2,B_{\omega_0}(\overline{\theta_0}))$, then $M=g_t R_{\theta} M_0$ with $\sin\omega_0\exp(-2T)<|\sin2\theta|<\exp(-2T)$, $\|g_t R_{\theta} e_2\|<\exp(-T)$, $\overline{\theta}_{\omega_0}(M_0)>\overline{\theta_0}$ and $T\geq T_0(A,\overline{\theta_0})$. It follows from Lemmas \ref{l.AMY.6.5} and \ref{l.AMY.6.7} that:
\begin{itemize} 
\item no saddle-connection of $M_0$ of length $>AK\rho_0$ gives rise to a saddle-connection of $M=g_t R_{\theta} M_0$ of length $\leq A\cdot\textrm{sys}(M)$; 
\item all non-vertical saddle connections of $M_0$ of length $\leq AK\rho_0$ make an angle $\geq \overline{\theta}_{\omega_0}(M_0)>\overline{\theta_0}$ with the vertical direction and, thus, they do not give rise to saddle-connections of $M=g_t R_{\theta} M_0$ of length $\leq A\cdot\textrm{sys}(M)$. 
\end{itemize} 
This means that all saddle-connections of $M=g_t R_{\theta} M_0$ of length $\leq A\cdot\textrm{sys}(M)$ are parallel to the length-minimizing ones, i.e., $M\notin\mathcal{C}(A,\rho_0\exp(-T))$. This proves the corollary. 
\end{proof} 

At this point, it is fairly easy to complete the proof of Proposition \ref{p.AMY2}. Indeed, the previous corollary says that 
\begin{equation}\label{e.p.AMY2}
m(Y(T,\frac{\pi}{2},X_0^*)\cap\mathcal{C}(A,\rho_0\exp(-T)))\leq m(Y(T,\omega_0,X_0^*)) + m(Y(T,\frac{\pi}{2},B_{\omega_0}(\overline{\theta_0})))
\end{equation}
for all $\omega_0>0$, $\overline{\theta_0}>0$ and $T\geq T_0(A,\overline{\theta_0})$. Also, the Proposition \ref{p.mY*} tells us that 
$$m(Y(T,\omega_0,X_0^*)) = \frac{1}{4}\exp(-2T)m_0(X_0^*)\int_{-\omega_0}^{\omega_0}\log\frac{1+\cos\omega}{1-\cos\omega} \cos\omega d\omega$$
and 
$$m(Y(T,\frac{\pi}{2},B_{\omega_0}(\overline{\theta_0}))) = \frac{1}{4}\exp(-2T)m_0(B_{\omega_0}(\overline{\theta_0}))\int_{-\pi/2}^{\pi/2}\log\frac{1+\cos\omega}{1-\cos\omega} \cos\omega d\omega$$ Therefore, given $\eta>0$, if we choose $\omega_0=\omega_0(\rho_0,\eta)>0$ small so that 
$$m_0(X_0^*)\int_{-\omega_0}^{\omega_0}\log\frac{1+\cos\omega}{1-\cos\omega} \cos\omega d\omega<\eta\rho_0^2$$
and $\overline{\theta_0}(\rho_0,\eta)>0$ small so that 
$$m_0(B_{\omega_0}(\overline{\theta_0}))\int_{-\pi/2}^{\pi/2}\log\frac{1+\cos\omega}{1-\cos\omega} \cos\omega d\omega<\eta\rho_0^2,$$ 
it follows from this discussion that 
$$m(Y(T,\omega_0,X_0^*))<\frac{\eta}{4}\rho^2\exp(-2T) \quad \textrm{ and } \quad m(Y(T,\frac{\pi}{2},B_{\omega_0}(\overline{\theta_0}))) <\frac{\eta}{4}\rho^2\exp(-2T)$$
for all $T\geq T_0(A,\overline{\theta_0}(\rho_0,\eta))=T_0(A,\rho_0,\eta)$. By plugging these inequalities into \eqref{e.p.AMY2}, we obtain  
$$m(Y(T,\frac{\pi}{2},X_0^*)\cap\mathcal{C}(A,\rho_0\exp(-T)))\leq\frac{\eta}{2}\rho_0^2\exp(-2T)$$
for all $T\geq T_0(A,\rho_0,\eta)$. This proves Proposition \ref{p.AMY2} (modulo Proposition \ref{p.mY*}). 

\subsection{Proof of Proposition \ref{p.mY*} via Rokhlin's disintegration theorem}\label{ss.p.mY*} Fix $\rho_0>0$ with $m(\{M\in\mathcal{C}:\textrm{sys}(M)>\rho_0\})>0$. Denote by $X_0^*$ the set of $M\in\mathcal{C}$ with $\textrm{sys}(M)=\rho_0$ such that all non-vertical saddle-connections of $M$ have length $>\rho_0$.

Let $X^*:=\bigcup_{\theta} R_{\theta}(X_0^*)$. Note that $R_{\theta}(X_0^*) = R_{\theta+\pi}(X_0^*)$ and $R_{\theta_0}(X_0^*)\cap R_{\theta_1}(X_0^*)=\emptyset$ for $-\frac{\pi}{2}<\theta_0<\theta_1\leq \frac{\pi}{2}$. In particular, 
$$X^*=\bigsqcup\limits_{-\frac{\pi}{2}<\theta\leq\frac{\pi}{2}} R_{\theta}(X_0^*)$$
By definition, $X^*$ and $X_0^*$ are submanifolds of $\mathcal{C}$ of codimensions one and two. 

Observe that 
$$e^{2t}\sin^2\theta + e^{-2t}\cos^2\theta=\|g_t R_{\theta}e_2\|<\|e_2\|=1$$ 
for $0<t<\log\cot|\theta|$ and $|\theta|<\pi/4$. Thus,  $g_t R_{\theta}(X_0^*)$ is disjoint from $\{M\in\mathcal{C}:\textrm{sys}(M)=\rho_0\}$ for such $t$ and $\theta$. This means that 
\begin{eqnarray*}
Y^* &:=& \{M\in\mathcal{C}: M=g_tR_{\theta}M_0, M_0\in X_0^*, |\theta|<\pi/4, \|g_tR_{\theta}e_2\|<1\} \\ 
&=& \bigsqcup\limits_{|\theta|<\frac{\pi}{4}}\bigsqcup\limits_{0<t<\log\cot|\theta|} g_tR_{\theta}(X_0^*)
\end{eqnarray*}
is a disjoint union of certain pieces of $SL(2,\mathbb{R})$-orbits. In particular, we can identify  
\begin{equation}\label{e.Y*coordinates}
Y^*\simeq \{(t,\theta,M)\in \mathbb{R}\times(-\pi/4,\pi/4)\times X_0^*: 0<t<\log\cot|\theta|\}
\end{equation}

We want to use this information to study the restriction of $m$ to $Y^*$. In this direction, the first step is the following lemma:

\begin{lemma}\label{l.mY*>0} The $m$-measure of $Y^*$ is positive. 
\end{lemma}

The proof of this lemma goes along the following lines. By Fubini's theorem and the $SL(2,\mathbb{R})$-invariance of $m$, we have  
$$m(Y^*) = \int_{\mathcal{C}} \mu(\{\gamma\in SL(2,\mathbb{R}): \gamma x\in Y^*\}) dm(x)$$
where $\mu$ is any Borel probability measure on $SL(2,\mathbb{R})$. 

Since $m(\{x\in\mathcal{C}:\textrm{sys}(x)>\rho_0\})>0$, this reduces our task to show that, given $x\in\mathcal{C}$ with $\textrm{sys}(x)>\rho_0$, the set of $\gamma\in SL(2,\mathbb{R})$ such that $\gamma x\in Y^*$ has non-empty interior (and hence positive Haar measure). 

Let $\omega$ be an angle such that $R_{\omega} x$ has a length-minimizing saddle-connection in the vertical direction. By definition, the quantity $s=\log(\textrm{sys}(x)/\rho_0)>0$ has the property that $g_s R_{\omega} x\in X_0^*$, i.e., $\gamma_0 x\in X_0^*$ where $\gamma_0:=g_s R_{\omega}\in SL(2,\mathbb{R})$. 

Observe that $n_u$ fixes the basis vector $e_2=(0,1)\in\mathbb{R}^2$, $n_u\gamma_0 x\in X_0^*$ whenever $|u|$ is sufficiently small, say $|u|<u_0$. Thus, $g_t R_{\theta} n_u \gamma_0 x\in Y^*$ for $|u|$ small, $|\theta|<\pi/4$ and $0<t<\log\cot|\theta|$. 

Therefore, our proof of lemma \ref{l.mY*>0} is reduced to prove that the set of $\gamma=g_t R_{\theta} n_u\gamma_0\in SL(2,\mathbb{R})$ with $|u|<u_0$, $|\theta|<\pi/4$ and $0<t<\log\cot|\theta|$ has non-empty interior in $SL(2,\mathbb{R})$. As it turns out, this is an immediate consequence of the following elementary fact about $SL(2,\mathbb{R})$: 

\begin{lemma}\label{l.p.3.1} The map $(t,\theta,u)\mapsto g_t R_{\theta} n_u$ is a diffeomorphism from $\mathbb{R}\times (-\frac{\pi}{4}, \frac{\pi}{4})\times \mathbb{R}$ to 
$$W:=\left\{\left(\begin{array}{cc} a & b \\ c & d \end{array}\right)\in SL(2,\mathbb{R}): d>0, |bd|<1/2\right\}$$
\end{lemma}

\begin{proof} The matrix $n_u$ fixes $e_2$ and the vector $g_t R_{\theta} e_2 = (b,d)$ satisfies $d>0$ and $|bd|<1/2$. Conversely, given $(b,d)\in\mathbb{R}^2$ with $d>0$ and $|bd|<1/2$, there exists an unique $(t,\theta)\in\mathbb{R}\times(-\frac{\pi}{4}, \frac{\pi}{4})$ depending smoothly on $(b,d)$ such that $(b,d)=g_t R_{\theta} e_2$. In fact, this happens because $g_t$ moves a non-zero vector $(x_0,y_0)$ along the hyperbola $\{(x,y)\in\mathbb{R}^2: xy=x_0y_0\}$ and $R_{\theta}$, $|\theta|<\pi/4$, moves $e_2$ along the arc of unit circle $\{(-\sin\theta,\cos\theta):|\theta|<\pi/4\}$ located between the hyperbolas $\{(x,y)\in\mathbb{R}^2: xy=-1/2\}$ and $\{(x,y)\in\mathbb{R}^2: xy=1/2\}$, see Figure \ref{f.AMY.p.3.1} below.

\begin{figure}[h!]
\begin{center}
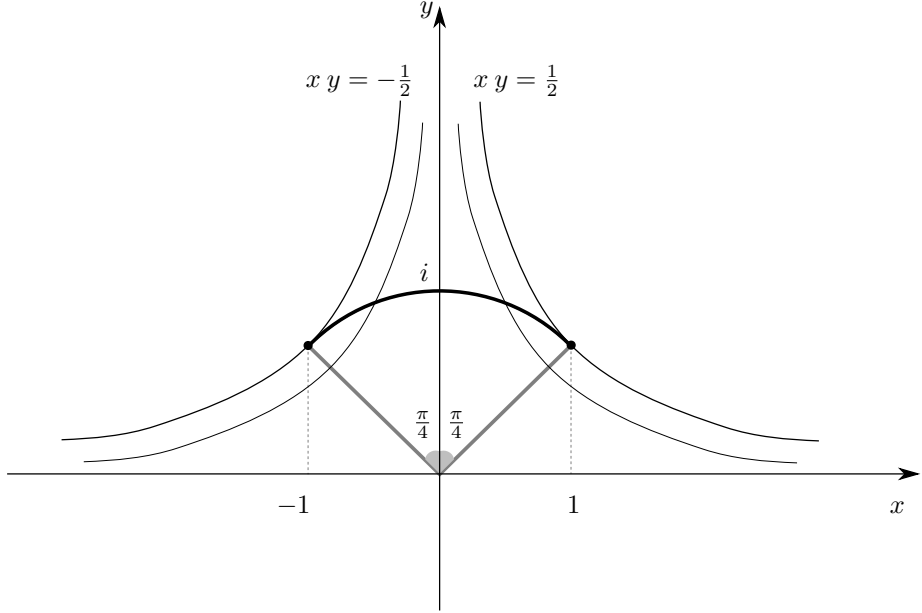
\end{center}
\caption{The region $\{(b,d)\in\mathbb{R}^2: d>0 \textrm{ and }|bd|<1/2\}$.}\label{f.AMY.p.3.1}
\end{figure}

This proves the lemma.  
\end{proof}

The second step is the study of $m|_{Y^*}$ via Rokhlin's disintegration theorem:

\begin{lemma}\label{l.p.4.1} There exists a finite measure $m_0$ on $X_0^*$ such that 
$$m|_{Y^*} = dt\times cos(2\theta)d\theta\times m_0$$
under the identification $Y^*\simeq \{(t,\theta,M)\in \mathbb{R}\times(-\pi/4,\pi/4)\times X_0^*: 0<t<\log\cot|\theta|\}$ in \eqref{e.Y*coordinates}.
\end{lemma}  

\begin{proof} We define the measure $m_0$ on $X_0^*$ as follows. Since $n_u$ fixes $e_2=(0,1)\in\mathbb{R}^2$, the vector field $\mathfrak{n}$ generating $n_u$ is tangent to $X_0^*$ at any of its points. 

Given a small smooth codimension-one submanifold $\Sigma$ of $X_0^*$ which is transverse to $\mathfrak{n}$, we can find $u_0>0$ such that $n_u(\Sigma)\subset X_0^*$ for all $|u|<u_0$ and $n_u(\Sigma)\cap \Sigma=\emptyset$ for all $0<u<2u_0$. In this setting, the map 
$$\Psi_0(u,M):=n_u M$$
is a smooth diffeomorphism from $(-u_0,u_0)\times \Sigma$ onto an open subset $B\subset X_0^*$. Furthermore, $X_0^*$ has a locally finite covering by such subsets $B$, so that it suffices to define the measure $m_0$ on $X_0^*$ by its restrictions to such subsets $B$. 

For this sake, given $B=\Psi_0((-u_0,u_0)\times\Sigma)$, consider 
$$U:=\{g_tR_{\theta}n_u: |u|<u_0, |\theta|<\pi/4, 0<t<\log\cot|\theta|\}\subset W$$
where $W=\left\{\left(\begin{array}{cc} a & b \\ c & d \end{array}\right)\in SL(2,\mathbb{R}): d>0, |bd|<1/2\right\}$ is the set from Lemma \ref{l.p.3.1}. The map 
$$\Psi(g,M)=gM$$
is a smooth diffeomorphism from $U\times \Sigma$ onto an open subset $V\subset\mathcal{C}$.

If $m(V)=0$, then $U$ is disjoint from the support of $m_0$.

If $m(V)>0$, we note that the Borel probability measure $\Psi^*(m|_V)$ on $U\times\Sigma$ is \emph{invariant}, that is, for any measurable subset $Z\subset U\times\Sigma$ and for any $h\in SL(2,\mathbb{R})$ such that $hZ:=\{(hg,M): (g,M)\in Z\}\subset U\times\Sigma$, we have $\Psi^*(m|_V)(hZ) = \Psi^*(m|_V)(Z)$. Indeed, this is a direct consequence of the $SL(2,\mathbb{R})$-invariance of $m$. 

In this context, an elementary variant\footnote{The one-page proof of this statement can be found in Proposition 2.6 of \cite{AMY}.} of Rokhlin's disintegration theorem says that 
$$\Psi^*(m|_V)=\mu|_{U}\times\nu$$
where $\nu=p_*(\Psi^*(m|_V))$, $p:U\times\Sigma\to \Sigma$ is the natural projection, and $\mu|_{U}$ is the normalized restriction to $U$ of the Haar measure $\mu$ of $SL(2,\mathbb{R})$.

In terms of the diffeomorphism $\Phi(t,\theta,u):= g_t R_{\theta} n_u$ from Lemma \ref{l.p.3.1}, the restriction $\mu|_{W}$ of the Haar measure to $W$ has the form $\gamma(t,\theta,u) dt d\theta du$ for some positive function $\gamma$ on $\mathbb{R}\times(-\frac{\pi}{4}, \frac{\pi}{4})\times\mathbb{R}$. Since the Haar measure $\mu$ of $SL(2,\mathbb{R})$ is left-invariant and right-invariant, $\gamma(t,\theta,u)=\gamma(\theta)$. Therefore, $\Phi_*(dt\times\gamma(\theta) d\theta\times du) = \mu|_{W}$.

Since $U\subset W$, it follows from this discussion that if we define
$$m_0|_{U}:=(\Phi_0)_*(du\times \nu),$$
then the corresponding finite measure $m_0$ on $X_0^*$ satisfies 
$$m|_{Y^*} = dt\times\gamma(\theta) d\theta\times m_0$$

In other words, it remains only to show that $\gamma(\theta)=\cos(2\theta)$ to complete the proof of the lemma. For this sake, fix $\theta_0\in(-\pi/4,\pi/4)$ and consider tiny open sets around the origin $(t,\theta,u)=(0,0,0)$ and their respective images under $R_{\theta_0}$. In terms of matrices, this amounts to consider the equation:
\begin{equation}\label{e.Haar-density}
R_{\theta_0}g_t R_{\theta} n_u = g_T R_{\Theta} n_U
\end{equation}
for $t,\theta, u$ close to $0$, and $T=T_{\theta_0}(t,\theta,u)$, $\Theta=\Theta_{\theta_0}(t,\theta,u)$, $U=U_{\theta_0}(t,\theta,u)$. For sake of simplicity, since $\theta_0$ is fixed, we will omit the dependence of the functions $T, \Theta, U$ on $\theta_0$ in what follows. From the $R_{\theta_0}$-invariance of Haar measure and the change of variables formula. one has  
\begin{equation}\label{e.Jacobian-formula-I} 
\gamma(\theta_0)=\gamma(0)/J_{\theta_0}(0,0,0)
\end{equation} 
where $J_{\theta_0}(0,0,0)$ is the determinant of the Jacobian matrix $D(T,\Theta,U)=\frac{\partial(T,\Theta,U)}{\partial(t,\theta,u)}$ at the origin $(0,0,0)$. So, our task is to show that $J_{\theta_0}(0,0,0)=1/\cos2\theta_0$. Keeping this goal in mind, note that $t=0$ implies that $T=0$, $\Theta=\theta+\theta_0$ and $U=u$ in \eqref{e.Haar-density}. Thus, at the origin $(0,0,0)$, one has 
$$\frac{\partial T}{\partial \theta}=\frac{\partial T}{\partial u}=0, \quad \frac{\partial \Theta}{\partial \theta}=1, \quad \frac{\partial \Theta}{\partial u}=0, \quad \frac{\partial U}{\partial \theta}=0, \quad \textrm{ and } \quad \frac{\partial T}{\partial u}=1$$
In particular, $\frac{\partial(T,\Theta,U)}{\partial(t,\theta,u)}(0,0,0)=\left(\begin{array}{ccc}\frac{\partial T}{\partial t}(0,0,0) & 0 & 0 \\ \frac{\partial \Theta}{\partial t}(0,0,0) & 1 & 0 \\ \frac{\partial U}{\partial t}(0,0,0) & 0 & 1\end{array}\right)$ and, \emph{a fortiori}, 
\begin{equation}\label{e.Jacobian-formula-II} 
J_{\theta_0}(0,0,0)=\frac{\partial T}{\partial t}(0,0,0).
\end{equation} 
In order to compute $\frac{\partial T}{\partial t}(0,0,0)$, we apply both matrices in \eqref{e.Haar-density} to the vertical basis vector $e_2$ to get the equations 
\begin{equation}\label{e.eTsin}
-e^T \sin\Theta = -e^t \cos\theta_0 \sin\theta - e^{-t} \sin\theta_0 \cos\theta
\end{equation} 
and 
\begin{equation}\label{e.e-Tcos}
e^{-T} \cos\Theta = -e^t \sin\theta_0\sin\theta+e^{-t}\cos\theta_0\cos\theta
\end{equation}

By taking the partial derivative with respect to $t$ in \eqref{e.e-Tcos}, we get 
$$-\frac{\partial T}{\partial t}e^{-T}\cos\Theta - e^{-T}\sin\Theta \frac{\partial \Theta}{\partial t} = -e^t\sin\theta_0\sin\theta-e^{-t}\cos\theta_0\cos\theta$$
Because $(T(0,0,0),\Theta(0,0,0), U(0,0,0))=(0,\theta_0,0)$, we obtain that 
\begin{equation}\label{e.dTdTheta}
-\frac{\partial T}{\partial t}(0,0,0)\cos\theta_0 - \sin\theta_0 \frac{\partial \Theta}{\partial t}(0,0,0) = -\cos\theta_0
\end{equation}

On the other hand, by multiplying together \eqref{e.eTsin} and \eqref{e.e-Tcos}, we get the relation:
$$\frac{1}{2}\sin 2\Theta=\sin\Theta \cos\Theta = 
 (e^t \cos\theta_0 \sin\theta + e^{-t} \sin\theta_0 \cos\theta) (-e^t \sin\theta_0\sin\theta+e^{-t}\cos\theta_0\cos\theta)$$
By taking the partial derivative with respect to $t$, we deduce that 
\begin{eqnarray*}
\cos 2\Theta \frac{\partial \Theta}{\partial t} &=& (e^t \cos\theta_0 \sin\theta - e^{-t} \sin\theta_0 \cos\theta) (-e^t \sin\theta_0\sin\theta+e^{-t}\cos\theta_0\cos\theta) \\ 
&+& (e^t \cos\theta_0 \sin\theta + e^{-t} \sin\theta_0 \cos\theta) (-e^t \sin\theta_0\sin\theta-e^{-t}\cos\theta_0\cos\theta)
\end{eqnarray*}
Since $(T(0,0,0),\Theta(0,0,0), U(0,0,0))=(0,\theta_0,0)$, we have $\cos 2\theta_0 \frac{\partial \Theta}{\partial t} (0,0,0)= -2\sin\theta_0\cos\theta_0$, i.e.,  
\begin{equation}\label{e.partialThetapartialt}
\frac{\partial \Theta}{\partial t}(0,0,0)=-\tan 2\theta_0
\end{equation}

By combining \eqref{e.dTdTheta} and \eqref{e.partialThetapartialt}, we conclude that 
\begin{eqnarray*}
\frac{\partial T}{\partial t}(0,0,0)&=&1+\tan\theta_0\tan2\theta_0 = 1+ \frac{\sin\theta_0}{\cos\theta_0}\frac{\sin2\theta_0}{\cos2\theta_0}  = \frac{1}{\cos 2\theta_0} \left(\cos 2\theta_0 + \frac{\sin\theta_0}{\cos\theta_0}\sin2\theta_0\right) \\ 
&=& \frac{1}{\cos 2\theta_0}(\cos^2\theta_0-\sin^2\theta_0+2\sin^2\theta_0)  
= \frac{1}{\cos 2\theta_0}(\cos^2\theta_0+\sin^2\theta_0) \\ 
&=& \frac{1}{\cos 2\theta_0}
\end{eqnarray*}

By \eqref{e.Jacobian-formula-I} and \eqref{e.Jacobian-formula-II}, this means that $\gamma(\theta_0)=\frac{1}{\cos2\theta_0}$. The proof of the lemma is now complete. 
\end{proof}

\begin{remark} We computed all entries of the Jacobian matrix $\frac{\partial(T,\Theta,U)}{\partial(t,\theta,u)}$ at the origin except for $\frac{\partial U}{\partial t}(0,0,0)$. Even though this particular entry plays no role in our calculation of $J_{\theta_0}(0,0,0)$ above, the curious reader is invited to compute this entry along the following lines. By applying both matrices in \eqref{e.Haar-density} to the horizontal basis vector $e_1=(1,0)$, one gets two relations:
$$e^T\cos\Theta-Ue^T\sin\Theta=e^t\cos\theta_0\cos\theta-e^{-t}\sin\theta_0\sin\theta -u(e^t\cos\theta_0\sin\theta+e^{-t}\sin\theta_0\cos\theta)$$
and
$$e^{-T}\sin\Theta+Ue^{-T}\cos\Theta=e^t\sin\theta_0\cos\theta + e^{-t}\cos\theta_0\sin\theta -u(e^t\sin\theta_0\sin\theta-e^{-t}\cos\theta_0\cos\theta)$$
By taking the partial derivative of the second relation above with respect to $t$ at the origin and by plugging the values $\frac{\partial \Theta}{\partial t}(0,0,0)=-\tan2\theta_0$ and $\frac{\partial T}{\partial t}(0,0,0)=1/\cos2\theta_0$ just computed, one has 
$$-\frac{\sin\theta_0}{\cos2\theta_0}-\cos\theta_0\tan2\theta_0+\cos\theta_0\frac{\partial U}{\partial t}(0,0,0)=\sin\theta_0,$$
i.e.,
\begin{eqnarray*}
\frac{\partial U}{\partial t}(0,0,0)&=&\tan2\theta_0+\frac{\sin\theta_0}{\cos\theta_0}\left(1+\frac{1}{\cos2\theta_0}\right) 
= \tan2\theta_0+\frac{\sin\theta_0}{\cos\theta_0}
\left(\frac{\cos^2\theta_0-\sin^2\theta_0+1}{\cos2\theta_0}\right) \\
&=& \tan2\theta_0+\frac{\sin\theta_0}{\cos\theta_0}
\left(\frac{2\cos^2\theta_0}{\cos2\theta_0}\right) = \tan2\theta_0+\frac{2\sin\theta_0\cos\theta_0}{\cos2\theta_0} \\
&=& 2 \tan 2\theta_0
\end{eqnarray*}
\end{remark}

At this stage, the proof of Proposition \ref{p.mY*} is almost complete: thanks to Lemmas \ref{l.mY*>0} and \ref{l.p.4.1}, we just need to show the following result. 

\begin{lemma}\label{l.c.4.2} For any $T>0$, $\omega_0>0$ and $B$ a Borel subset of $X_0^*$, the set 
$$Y(T,\omega_0, B)=\{g_tR_{\theta} M: |\sin2\theta|<\exp(-2T)\sin\omega_0, \|g_tR_{\theta}e_2\|<\exp(-T), M\in B\}$$
has $m$-measure 
$$m(Y(T,\omega_0,B)) = \frac{1}{4}\exp(-2T)m_0(B)\int_{-\omega_0}^{\omega_0}\log\frac{1+\cos\omega}{1-\cos\omega} \cos\omega d\omega$$
\end{lemma}

\begin{proof} Denote by 
$$J(T,\theta):=\{t\in\mathbb{R}:\|g_tR_{\theta}e_2\|<\exp(-T)\}$$
From Lemma \ref{l.p.4.1}, one has that the $m$-measure of the subset $Y(T,\omega_0,B)\subset Y^*$ equals to 
\begin{equation}\label{e.c.4.2}
m(Y(T,\omega,B)) = m_0(B) \int_{|\sin2\theta|<\exp(-2T)\sin\omega_0} \left(\int_{t\in J(T,\theta)} dt \right) \cos 2\theta d\theta
\end{equation}

Let us compute the length of the interval $J(T,\theta)$. For this sake, we observe that the condition 
$e^{2t}\sin^2\theta + e^{-2t}\cos^2\theta = \|g_tR_{\theta}e_2\|^2<\exp(-2T)$ is equivalent to the requirement that $x=e^{2t}$ solves the second degree inequality 
$$x^2\sin^2\theta - \exp(-2T) x + \cos^2\theta <0$$
Hence, $t\in J(T,\theta)$ if and only if $x_-< x=e^{2t} <x_+$ where 
$$x_{\pm}:=\frac{\exp(-2T)\pm\sqrt{\exp(-4T)-\sin^2(2\theta)}}{2\sin^2\theta}$$
In other terms, using the change of variables $\sin2\theta:=\exp(-2T)\sin\omega$ with $\cos\omega>0$, we have that $x_{\pm}=\exp(-2T)\frac{1\pm\cos\omega}{2\sin^2\theta}$. This means that the length of $J(T,\theta)$ is 
$$\int_{t\in J(T,\theta)} dt = \frac{1}{2}(\log x_+ - \log x_-) = \frac{1}{2}\log\left(\frac{x_+}{x_-}\right) = \frac{1}{2}\log\frac{1+\cos\omega}{1-\cos\omega}$$
By plugging this formula in \eqref{e.c.4.2} while keeping the change of variables $\sin2\theta:=\exp(-2T)\sin\omega$ in mind, we deduce that 
$$m(Y(T,\omega,B)) = \frac{1}{4} m_0(B) \int_{-\omega_0}^{\omega_0}  \log\frac{1+\cos\omega}{1-\cos\omega}\cos \omega d\omega$$
This proves the lemma. 
\end{proof}

This ends our discussion of Proposition \ref{p.mY*}. 

\subsection{Proof of Proposition \ref{p.slice} via Rokhlin's disintegration theorem}\label{ss.flux}

We want to interpret the total mass of the measure $m_0$ constructed above as a flux of the measure $m$ through $\{M\in\mathcal{C}: \textrm{sys}(M)=\rho_0\}$. For this sake, we will follow the same strategy used in the proof of Theorem \ref{t.AMY'}, namely: 

\begin{itemize}
\item we will use pieces of $SL(2,\mathbb{R})$-orbits to capture a portion (called \emph{regular part}) of the slice $\{M\in\mathcal{C}: \rho_0\exp(-\tau)\leq \textrm{sys}(M)\leq \rho_0\}$ whose $m$-measure is not hard to compute, and 
\item we will prove that the portion of the slice that was not captured by this procedure (called \emph{singular part}) has negligible $m$-measure.
\end{itemize}

Let us start by formalizing the first item. Given $M\in X^*=\bigsqcup\limits_{-\frac{\pi}{2}<\theta\leq\frac{\pi}{2}} R_{\theta}(X_0^*)$ and $t\geq 0$, we define the following ``pseudo Teichm\"uller flow'':
$$\Phi_t(M) = R_{\theta} g_t R_{-\theta}(M) \quad \textrm{when} \quad M\in R_{\theta}(X_0^*)$$ 

The systole of $\Phi_t(M)$ is $\rho_0\exp(-t)$. Also, for $M\in R_{\theta}(X_0^*)$, all length-minimizing saddle-connections of $\Phi_t(M)$ make angle $\theta$ with the vertical direction. Therefore, $\Phi_t$ is injective and $\Phi_t(X^*)\cap\Phi_{t'}(X^*)=\emptyset$ for $t\neq t'$. 

Given $\tau>0$, we say that the \emph{regular part} $Reg(\tau)$ of the slice 
$$S(\tau):=\{M\in\mathcal{C}:\rho_0\exp(-\tau)\leq\textrm{sys}(M)\leq \rho_0\}$$
is the set 
$$Reg(\tau):=\bigsqcup\limits_{0\leq t\leq\tau} \Phi_t(X^*)$$ 

The $m$-measure of $Reg(\tau)$ is provided by the following lemma:

\begin{lemma}\label{l.regular-part} Let $\tilde{m}_{\tau}$ be the measure on $X^*$ given by  
$$\tilde{m}_{\tau}(B)=\frac{2}{1-\exp(-2\tau)} \, m\left(\bigsqcup\limits_{0\leq t\leq \tau}\Phi_t(B)\right)$$
for $B\subset X^*$ a Borel subset. Then, $\tilde{m}_{\tau}$ is independent of $\tau$ and $\tilde{m}_{\tau} = d\theta\times m_0$. In particular, $m(Reg(\tau))=\frac{1-\exp(-2\tau)}{2}\pi m_0(X_0^*)$ and 
$$\lim\limits_{\tau\to 0} \frac{1}{\tau} m(Reg(\tau)) = \pi m_0(X_0^*)$$
\end{lemma}

\begin{proof} By definition, $\tilde{m}_{\tau}$ is invariant under the group $SO(2,\mathbb{R})$ of rotations. Thus, by an elementary variant\footnote{Cf. Proposition 2.6 of \cite{AMY}.} of Rokhlin's disintegration theorem, one has $\tilde{m}_{\tau} = d\theta\times m_{\tau}$. 

This reduces our task to show that $m_{\tau}=m_0$ for all $\tau>0$. In this direction, consider $\Sigma$ a small codimension one submanifold of $X_0^*$ which is transverse to the infinitesimal generator $\mathfrak{n}$ of $n_u$ and take $u_0>0$ so that $n_u(\Sigma)\subset X_0^*$ for all $|u|<u_0$ and $n_u(\Sigma)\cap\Sigma=\emptyset$ for all $0<u<2u_0$. 

Note that $\Phi_t(R_{\theta} n_u M) = R_{\theta} g_t n_u M$ for any $|u|<u_0$, $M\in\Sigma$, $\theta$ and $t\geq 0$. Also, observe that the set $W$ from Lemma \ref{l.p.3.1} contains $R_{\theta} g_t\in W$ for $|\theta|<\pi/4$ and $t\geq 0$, so that 
$$\left(\begin{array}{cc} e^t\cos\theta & -e^{-t}\sin\theta \\ e^t\sin\theta & e^{-t}\cos\theta \end{array}\right) = R_{\theta} g_t = g_T R_{\Theta} n_U=\left(\begin{array}{cc} e^T(\cos\Theta - U \sin\Theta) & -e^T\sin\Theta \\ e^{-T}(\sin\Theta - U\cos\Theta) & e^{-T}\cos\Theta \end{array}\right)$$
for some smooth functions $T=T(t,\theta)$, $\Theta=\Theta(t,\theta)$, $U=U(t,\theta)$. It follows that 
$$T(t,\theta) = t + O(\theta), \quad \Theta(t,\theta) = e^{-2t}\theta + O(\theta^2), \quad U(t,\theta) = O(\theta)$$
for $\theta$ close to zero. 

This information can be combined with the expression for $m|_{Y^*}$ in $g_T R_{\Theta} n_{U}$-coordinates in Lemma \ref{l.p.4.1} in order to compute $m_{\tau}$ in the following way. If $B_0=(u_1, u_2)\times B$ is a Borel subset of $(-u_0,u_0)\times\Sigma$, then 
\begin{eqnarray*}
\tilde{m}_{\tau}([0,\theta_0]\times B_0) &=& \frac{2}{1-\exp(-2\tau)} \int_{0}^{\tau}\int_{u_1}^{u_2}\int_{B} e^{-2t}\theta_0 dt\,du\,d\nu + O(\theta_0^2) \\ 
&=& m_0(B_0)\theta_0 + O(\theta_0^2)
\end{eqnarray*} 
for $\theta_0$ close to zero. This implies that $m_{\tau} = m_0$, so that the proof of the lemma is complete. 
\end{proof}

Next, we will study the $m$-measure of the \emph{singular part} 
$$Sing(\tau):= S(\tau)-Reg(\tau)$$ 
of the slice $S(\tau)$. 

\begin{lemma}\label{l.p.5.3} For $\tau>0$ small, the $m$-measure of $Sing(\tau)$ is $o(\tau)$, i.e., 
$$\lim\limits_{\tau\to 0} \frac{1}{\tau}m(Sing(\tau)) = 0$$
\end{lemma} 

We introduce the set $Z(\tau)$ of translation surfaces $M\in S(\tau)$ possessing a saddle-connection of length $\leq\rho_0\exp(\tau)$ which is not parallel to a minimizing one. By definition, 
\begin{equation}\label{e.Ztau}
Sing(\tau)\subset Z(\tau),
\end{equation} 
so that proof of Lemma \ref{l.p.5.3} is reduced to prove that $m(Z(\tau)) = o(\tau)$. The proof of this fact is divided into two parts depending on the size of the angle between short saddle-connections of $M\in Z(\tau)$. More concretely, for $M\in S(\tau)$, denote by $\hat{\theta}(M)$ the smallest angle between two saddle-connections of lengths $\leq 3\rho_0$ which are not parallel (with the convention that $\hat{\theta}(M)=\pi/2$ when such connections do not exist). 

We begin by estimating the $m$-measure of the subset of $S(\tau)$ consisting of translation surfaces $M$ with $\hat{\theta}(M)$ small. 

\begin{lemma}\label{l.small.angle} Given $\eta>0$, there exists $\hat{\theta}_0=\hat{\theta}_0(\eta)>0$ such that 
$$m(\{M\in S(\tau): \hat{\theta}(M)<\hat{\theta}_0\})<\eta\tau$$
for all $\tau>0$ small enough. 
\end{lemma} 

\begin{proof} Let $S_1(\tau)$ be the subset of $M\in S(\tau)$ possessing a length-minimizing saddle-connection making an angle $\leq\pi/6$ with the vertical direction. 

The $SO(2,\mathbb{R})$-invariance of $m$ tells us that 
$$m(\mathcal{S})\leq 3 \, m(\mathcal{S}\cap S_1(\tau))$$
for any $SO(2,\mathbb{R})$-invariant subset $\mathcal{S}\subset S(\tau)$. In particular, for any $\hat{\theta}_0>0$, one has 
\begin{equation}\label{e.small-angle}
m(\{M\in S(\tau): \hat{\theta}(M)<\hat{\theta}_0\})\leq 3 \, m(\{M\in S_1(\tau): \hat{\theta}(M)<\hat{\theta}_0\})
\end{equation}

In order to estimate the right-hand side of this inequality, we claim that, for any $M\in S_1(\tau)$ and $j\in\mathbb{N}-\{0\}$ with $\exp(3(j+1)\tau)<\cot\frac{\pi}{6}=\sqrt{3}$, the systole of $g_{3j\tau}M$ is \begin{equation}\label{e.small-angle-sys}
\textrm{sys}(g_{3j\tau}M)<\rho_0\exp(-\tau)
\end{equation} 
Indeed, this happens whenever the estimate $e^{6j\tau}\sin^2\theta + e^{-6j\tau}\cos^2\theta=\|g_{3j\tau}R_{\theta}e_2\|^2< e^{-2\tau}$ holds for all $|\theta|\leq\pi/6$. Since this second degree inequality on $x=e^{6j\tau}$ is satisfied when 
$$e^{6j\tau}< \frac{e^{-2\tau}+\sqrt{e^{-4\tau}-\sin^2(2\theta)}}{2\sin^2\theta}$$
for all $|\theta|\leq \pi/6$ and 
$$2(e^{-2\tau}+\sqrt{e^{-4\tau}-3/4})\leq \frac{e^{-2\tau}+\sqrt{e^{-4\tau}-\sin^2(2\theta)}}{2\sin^2\theta}$$
for all $|\theta|\leq\pi/6$, the proof of our claim is reduced to check that 
$$e^{6j\tau} < 2(e^{-2\tau}+\sqrt{e^{-4\tau}-3/4}).$$
This last inequality follows easily from our assumption that $e^{3(j+1)\tau}<\sqrt{3}$, i.e., $e^{6j\tau}<3 e^{-6\tau}$: in fact, this is an immediate consequence of the fact that the inequality 
$$3 e^{-\kappa\tau} < 2(e^{-2\tau}+\sqrt{e^{-4\tau} - 3/4})$$ is equivalent to $3e^{-\kappa\tau} - 2e^{-2\tau} < 2\sqrt{e^{-4\tau}-3/4}$, that is, $9 e^{(2-\kappa)\tau} + 3 e^{(\kappa+2)\tau}<12$, and this estimate is true for any $\kappa>4$ and $\tau>0$ small enough because the derivative at $\tau=0$ of the function $9 e^{(2-\kappa)\tau} + 3 e^{(\kappa+2)\tau}$ is $9(2-\kappa)+3(\kappa+2)=24-6\kappa<0$. 

Now, we observe that \eqref{e.small-angle-sys} implies the disjointness of $g_{3j\tau}(S_1(\tau))$ and $g_{3j'\tau}(S_1(\tau))$ for all $0<j<j'<\frac{\log 3}{6\tau}-1$. In particular, 
\begin{equation}\label{e.small-angle-sys-cor}
\frac{1}{6\tau} m(\{M\in S_1(\tau): \hat{\theta}(M)<\hat{\theta}_0\})\leq m\left(\bigcup\limits_{0<j<\frac{\log 3}{6\tau} - 1} g_{3j\tau}(\{M\in S_1(\tau): \hat{\theta}(M)<\hat{\theta}_0\})\right)
\end{equation}
because the number of $j\in\mathbb{N}$ with $0 < j < \frac{\log 3}{6\tau} - 1$ is $\geq 1/6\tau$. 

On the other hand, if $0<j<\frac{\log 3}{6\tau}-1$, then, for any $M\in S(\tau)$, the systole of $M'=g_{3j\tau}M$ is $\rho_0/2<\textrm{sys}(M')<3\rho_0$ and  $M'$ has a pair of saddle-connections of lengths $\leq 3\sqrt{3}\rho_0$ with angle $\leq 10\cdot\hat{\theta}(M)$. Therefore, for each $\hat{\theta}_0>0$, the set $\mathcal{C}_{\hat{\theta}_0}(\rho_0)$ consisting of translation surfaces $M'$ with $\textrm{sys}(M')\in (\frac{\rho_0}{2}, 3\rho_0)$ and a pair of non-parallel saddle-connections of lengths $\leq 3\sqrt{3}\rho_0$ with angle $\leq 10\hat{\theta_0}$ contains 
$$\bigcup\limits_{0<j<\frac{\log 3}{6\tau} - 1} g_{3j\tau}(\{M\in S_1(\tau): \hat{\theta}(M)<\hat{\theta}_0\})$$ 

It follows from \eqref{e.small-angle} and \eqref{e.small-angle-sys-cor} that 
$$m(\{M\in S(\tau): \hat{\theta}(M)<\hat{\theta}_0\}) \leq 18\tau \, m(\mathcal{C}_{\hat{\theta}_0}(\rho_0))$$ 

This completes the proof of the lemma: indeed, given $\eta>0$, if we take $\hat{\theta}_0=\hat{\theta}(\eta)>0$ small enough so that $m(\mathcal{C}_{\hat{\theta}_0}(\rho_0))<\eta/18$, then $m(\{M\in S(\tau): \hat{\theta}(M)<\hat{\theta}_0\}) <\eta\tau$. 
\end{proof} 

Next, we estimate the $m$-measure of the subset of $S(\tau)$ consisting of translation surfaces $M\in Z(\tau)$ with $\hat{\theta}(M)$ large. 

\begin{lemma}\label{l.big-angle} For any $\hat{\theta}_0>0$, one has 
$$m(\{M\in Z(\tau):\hat{\theta}(M)\geq\hat{\theta}_0\}) = O(\tau^{3/2})$$
where the implied constant depends on $\hat{\theta}_0$, $\rho_0$ and the genus $g$ of the translation surfaces in $\mathcal{C}$. 
\end{lemma}

\begin{proof} By Fubini's theorem and the $SL(2,\mathbb{R})$-invariance of $m$, we have  
$$m(\{M\in Z(\tau): \hat{\theta}(M)\geq\hat{\theta}_0\}) = \int_{\mathcal{C}}\mu_L(\{\gamma\in SL(2,\mathbb{R}): \gamma x\in Z(\tau), \hat{\theta}(\gamma x)\geq\hat{\theta}_0\}) \, dm(x)$$
where $\mu_L$ is the normalized restriction of the Haar measure of $SL(2,\mathbb{R})$ to the compact subset 
$$L:=\{\gamma\in SL(2,\mathbb{R}): \|\gamma\|\leq 2\}$$

This reduces our task to prove the following claim: for each $x\in\mathcal{C}$, one has 
$$\mu_L(\{\gamma\in SL(2,\mathbb{R}): \gamma x\in Z(\tau), \hat{\theta}(\gamma x)\geq\hat{\theta}_0\}) = O(\tau^{3/2})$$

Fix $x\in\mathcal{C}$. If the set $B_{\hat{\theta}_0, \tau}(x):=\{\gamma\in L: \gamma x\in Z(\tau), \hat{\theta}(\gamma x)\geq\hat{\theta}_0\}$ is empty, we are done. So, we can assume that this set is not empty. This imposes a constraint on the systole of $x$. Indeed, if $B_{\hat{\theta}_0, \tau}(x) \neq\emptyset$, then one has $\rho_0\exp(-\tau)\leq\textrm{sys}(\gamma_0 x)\leq\rho_0$ for some $\gamma_0\in SL(2,\mathbb{R})$ with $\|\gamma\|\leq 2$. Since $\|\gamma_0^{-1}\|=\|\gamma_0\|$, we have 
$$\frac{\rho_0}{2}\exp(-\tau)\leq\textrm{sys}(x)\leq 2\rho_0$$ 

Moreover, the matrices $\gamma\in B_{\hat{\theta}_0,\tau}(x)$ satisfy some severe restrictions. In fact, given $\gamma\in B_{\hat{\theta}_0,\tau}(x)$, there are non-parallel holonomy vectors $v, v'\in\mathbb{R}^2$ of saddle-connections of $x$ such that the angle between $\gamma v$ and $\gamma v'$ is $\geq\hat{\theta}_0$ and 
$$\rho_0\exp(-\tau)\leq \|\gamma v\|\leq \rho_0\exp(\tau) \textrm{ and } \rho_0\exp(-\tau)\leq \|\gamma v\|\leq \rho_0\exp(\tau).$$ 
In other words, $\gamma\in E(\frac{v}{\rho_0}, \frac{v'}{\rho_0},\tau)$ where 
$$E(w,w',\tau):=\{\gamma\in L: \rho_0\exp(-\tau)\leq \|\gamma w\|, \|\gamma w'\|\leq \rho_0\exp(\tau)\}$$ 
Since $\|\gamma\|=\|\gamma^{-1}\|\leq 2$, we have that the angle between $v$ and $v'$ is $\geq\hat{\theta}_0/10$ and $\|v\|, \|v'\|\leq 3\rho_0$ (for $\tau>0$ small enough). In particular, the quantity $\rho_0^{-1}\|v\pm v'\|$ is uniformly bounded away from zero by a constant $c=c(\hat{\theta}_0)>0$:
$$\left\|\frac{v}{\rho_0}\pm \frac{v'}{\rho_0}\right\|\geq c$$

In summary, if we denote by $v_1,\dots, v_N$ the holonomy vectors of saddle-connections of $x$ of length $\leq 3\rho_0$, then 
\begin{equation}\label{e.big-angle} 
\{\gamma\in SL(2,\mathbb{R}): \gamma x\in Z(\tau), \hat{\theta}(\gamma x)\geq\hat{\theta}_0\} \subset \bigcup\limits_{\|\frac{v_i}{\rho_0}\pm \frac{v_j}{\rho_0}\|\geq c} E(\frac{v_i}{\rho_0}, \frac{v_j}{\rho_0},\tau)
\end{equation}

By a result of Masur (see \cite{M90}), the number $N$ of saddle-connections of lengths $\leq 3\rho_0$ on the translation surface $x$ with $\textrm{sys}(x)\geq\rho_0\exp(-\tau)/2>\rho_0/3$ is bounded by a constant $N(\rho_0, g)$. Also, an elementary computation\footnote{Cf. Proposition 3.3 in \cite{AMY} for a one-page proof of this fact.} with the Iwasawa decomposition of $SL(2,\mathbb{R})$ says that 
$$\mu_L(E(w,w',\tau))=O(\tau^{3/2})$$
where the implied constant depends only on $\|w\pm w'\|$. In other terms, given a pair of non-collinear vectors $w, w'\in\mathbb{R}^2$, the (Haar) probability that a matrix $\gamma\in SL(2,\mathbb{R})$ with $\|\gamma\|\leq 2$ takes both of them to vectors $\gamma w, \gamma w'$ inside a ``$\tau$-thin'' annulus $\{v\in\mathbb{R}^2: e^{-\tau}\leq\|v\|\leq e^{\tau}\}$ around the unit circle has order $\tau^{3/2}$. 

By combining the information in the previous paragraph with \eqref{e.big-angle}, we conclude that 
$$\mu_L(\{\gamma\in SL(2,\mathbb{R}): \gamma x\in Z(\tau), \hat{\theta}(\gamma x)\geq\hat{\theta}_0\}) = O(\tau^{3/2})$$
where the implied constant depends only on $\hat{\theta}_0$, $\rho_0$ and $g$. This proves our claim. 
\end{proof}

At this point, the proof of Proposition \ref{p.slice} is complete. Indeed, Lemmas \ref{l.small.angle} and \ref{l.big-angle} imply Lemma \ref{l.p.5.3} saying that $m(Sing(\tau))=o(\tau)$. By combining this fact with Lemma \ref{l.regular-part}, we conclude the desired formula   
$$m(S(\tau))=m(Reg(\tau))+m(Sing(\tau)) = \pi m_0(X_0^*)\tau + o(\tau)$$
for the $m$-measure of the slices $S(\tau):=Reg(\tau)\sqcup Sing(\tau)$. 

\newpage 

\null

\newpage


\begin{centering}
\rule{\textwidth}{1.6pt}\vspace*{-\baselineskip}\vspace*{2.5pt}
\rule{\textwidth}{0.4pt}

\section{Arithmetic Teichm\"uller curves with complementary series}\label{s.MS}

\rule{\textwidth}{0.4pt}\vspace*{-\baselineskip}\vspace{3.2pt}
\rule{\textwidth}{1.6pt}
\end{centering}\\

Let $\mathcal{C}$ be a connected component of a stratum of the moduli space of unit area translation surfaces of genus $g\geq 1$. 

It is well-known\footnote{See Subsection \ref{ss.Ratner-mixing}.} that the theory of unitary representations of $SL(2,\mathbb{R})$ and the fact that the Teichm\"uller flow $g_t$ is part of an action of $SL(2,\mathbb{R})$ on $\mathcal{C}$ can be used to prove that any ergodic $SL(2,\mathbb{R})$-invariant probability measure $\mu$ on $\mathcal{C}$ is actually \emph{mixing}, i.e., for all $u, v\in L^2(\mathcal{C},\mu)$, the \emph{correlation function} $C_t(f,g):=\int_{\mathcal{C}} (u \cdot v\circ g_t) d\mu - \int_{\mathcal{C}} u d\mu \cdot \int_{\mathcal{C}} v d\mu$ decays to zero: 
$$\lim\limits_{t\to\infty} C_t(u,v) = 0$$ 

In general, the speed of decay of correlation functions of a mixing measure depends on the features of the dynamics at hand: for example, the presence of hyperbolicity usually tends to accelerate the rate of convergence of $C_t(u,v)$ to zero for significant classes of observables $u$ and $v$. 

In particular, the non-uniform hyperbolicity properties of Teichm\"uller flow established by Veech \cite{V86} and Forni \cite{F02} indicate that $SL(2,\mathbb{R})$-invariant probability measures on $\mathcal{C}$ exhibit a fast decay of correlations. 

\subsection{Exponential mixing of the Teichm\"uller flow} The rate of mixing of the \emph{Masur-Veech measure} $\mu_{\mathcal{C}}$ of $\mathcal{C}$ was computed in the celebrated work of Avila, Gou\"ezel and Yoccoz \cite{AGY}: 

\begin{theorem}\label{t.AGY} The Teichm\"uller flow $g_t$ is exponentially mixing with respect to $\mu_{\mathcal{C}}$, i.e., $C_t(u,v)$ converges exponentially fast to zero as $t\to\infty$ for all sufficiently smooth observables $u, v\in L^2(\mathcal{C},\mu_{\mathcal{C}})$. 
\end{theorem}

The proof of Theorem \ref{t.AGY} is based on the (mostly \emph{combinatorial}) analysis of a symbolic model of $(g_t,\mu_{\mathcal{C}})$ called \emph{Rauzy-Veech induction} and a criterion (based on \emph{Dolgopyat-like estimates}) for the exponential mixing of certain suspension flows. 

The strategy outline above is hard to extend to \emph{arbitrary} $SL(2,\mathbb{R})$-invariant probability measures on $\mathcal{C}$: indeed, the symbolic models provided by the Rauzy-Veech induction are somehow tailor-made for the Masur-Veech measures. 

Nevertheless, Avila and Gou\"ezel \cite{AG} managed to compute the rate of mixing of an arbitrary $SL(2,\mathbb{R})$-invariant probability measure $\mu$ on $\mathcal{C}$:

\begin{theorem}\label{t.AG} The Teichm\"uller flow $g_t$ is exponentially mixing with respect to any $SL(2,\mathbb{R})$-invariant probability measure $\mu$ on $\mathcal{C}$. In particular, there exists $\delta(\mu)>0$ such that 
$$|C_t(u,v)|\leq e^{-\delta(\mu) t}\|u\|_{L^2(\mathcal{C},\mu)}\|v\|_{L^2(\mathcal{C},\mu)}$$
for all $SO(2,\mathbb{R})$-invariant observables $u, v\in L^2(\mathcal{C},\mu)$. 
\end{theorem}

The proof of Theorem \ref{t.AG} is based on the delicate construction of \emph{anisotropic} Banach spaces adapted to the spectral analysis of certain transfer operators. 

\subsection{Teichm\"uller curves with complementary series} The particularly nice features of the $SL(2,\mathbb{R})$-action on moduli spaces of translation surfaces led Avila and Gou\"ezel \cite{AG} to ask if there can be some sort of uniformity in the way that the Teichm\"uller flow mixes the phase space: for instance, is it possible to take the constant $\delta(\mu)>0$ in the statement of Theorem \ref{t.AG} uniformly bounded away from zero as $\mu$ varies? 

This question is still open (to the best of our knowledge) if the $SL(2,\mathbb{R})$-invariant probability measures $\mu$ are only allowed to vary within a \emph{fixed} connected component $\mathcal{C}$ of a stratum of the moduli space of translation surfaces of genus $g\geq 2$. 

Also, it was conjectured\footnote{In the language of unitary $SL(2,\mathbb{R})$-representations (see Subsection \ref{ss.Ratner-mixing}), the precise statement of Yoccoz's conjecture is: ``the regular $SL(2,\mathbb{R})$ representation $L^2(\mathcal{C},\mu_{\mathcal{C}})$ has no complementary series whenever $\mu_{\mathcal{C}}$ is a Masur-Veech measure''. So far, the validity of this conjecture is known only for the moduli space of unit area flat torii $\mathcal{C}=SL(2,\mathbb{R})/SL(2,\mathbb{Z})$ thanks to a classical theorem of Selberg.} by Yoccoz that $\delta(\mu_{\mathcal{C}})$ can be taken arbitrarily close to one when $\mu_{\mathcal{C}}$ is a Masur-Veech measure. 

On the other hand, if we allow the support of $\mu$ to vary among \emph{all} strata of moduli spaces of translation surfaces, then it was proved by Schmith\"usen and the author \cite{MaSc} that no uniform lower bound on $\delta(\mu)>0$ is possible. 

\begin{theorem}\label{t.MaSc} For each $k\geq 3$, there exists an explicit square-tiled surface $Z_{2k}$ (of genus $48k+3$ tiled by $192k$ squares) generating an arithmetic Teichm\"uller curve $\mathcal{S}_{2k}$ such that there is no uniform lower bound on the exponential rate of mixing of the $SL(2,\mathbb{R})$-invariant probability $\mu_{2k}$ supported on $\mathcal{S}_{2k}$, i.e., $\lim\limits_{k\to\infty}\delta(\mu_{2k})=0$ (where $\delta(.)$ is the best constant in the statement of Theorem \ref{t.AG}).
\end{theorem} 

We devote the rest of this section to discuss this theorem.

\subsection{Idea of proof of Theorem \ref{t.MaSc}} From the abstract point of view, an old procedure\footnote{Selberg used cyclic covers to show that there is no uniform \emph{spectral gap} for the Laplacian of $\mathbb{H}/\Gamma$ (where $\Gamma\subset SL(2,\mathbb{Z})$ is a lattice). As it turns out, this is equivalent to our assertion on rates of mixing: see Subsection \ref{ss.Ratner-mixing}.} due to Selberg allows one to build a sequence  $(\mathbb{H}/\Gamma^{(k)})_{k\in\mathbb{N}}$, $\Gamma^{(k)}\subset SL(2,\mathbb{Z})$, of arithmetic finite-area hyperbolic surface such that there is no uniform lower bound on the exponential rate of mixing of the Lebesgue measures $\mu_k$ of $\mathbb{H}/\Gamma^{(k)}$ by taking appropriate \emph{cyclic} covers of a fixed finite-area hyperbolic surface of positive genus (see Figure \ref{f.Selberg-cyclic-cover}). 

On the other hand, it is not obvious at all that the lattices $\Gamma^{(k)}\subset SL(2,\mathbb{Z})$ provided by Selberg's procedure are useful for our purposes of proving Theorem \ref{t.MaSc}: in fact, there is no reason for $\Gamma^{(k)}$ to correspond to Veech groups of origamis $Z_{2k}$ or, equivalently, it is not clear that $SL(2,\mathbb{R})/\Gamma^{(k)}$ is realizable as an arithmetic Teichm\"uller curve. 

Nevertheless, Avila, Yoccoz and the author noticed during some conversations that the results of Ellenberg and McReynolds \cite{EM} on the realizability of certain lattices $\Gamma\subset SL(2,\mathbb{Z})$ as Veech groups of origamis could be combined with Selberg's argument to show the \emph{existence} of a sequence $Z_{2k}$ of square-tiled surfaces satisfying the conclusions of Theorem \ref{t.MaSc}. 

In principle, it is not easy to build the explicit examples of origamis in Theorem \ref{t.MaSc} \emph{directly} from the arguments of Avila, Yoccoz and the author mentionned in the previous paragraph, but Schmith\"usen and the author were able to \emph{adapt} them to obtain Theorem \ref{t.AG}. 

\subsection{Quick review of representation theory of $SL(2,\mathbb{R})$}\label{ss.Ratner-mixing} Before starting the proof of Theorem \ref{t.MaSc}, it is useful to recall the relationship between the spectral properties of the regular $SL(2,\mathbb{R})$-representation $L^2(\mathcal{C},\mu)$ and the exponential rate of mixing of the Teichm\"uller flow with respect to an ergodic $SL(2,\mathbb{R})$-invariant probability measure $\mu$ supported on Teichm\"uller curves. For this reason, we shall review in this subsection some basic aspects of the theory of unitary $SL(2,\mathbb{R})$-representation and the results of Ratner \cite{Rt92} on rates of mixing. 

\subsubsection{Spectrum of unitary $SL(2,\mathbb{R})$-representations} Let $\rho:SL(2,\mathbb{R})\to U(\mathcal{H})$ be a unitary \emph{representation} of $SL(2,\mathbb{R})$, i.e., $\rho$ is a homomorphism from $SL(2,\mathbb{R})$ into the group $U(\mathcal{H})$ of \emph{unitary} transformations of the \emph{complex} separable Hilbert space $\mathcal{H}$. We say that a vector $v\in\mathcal{H}$ is a $C^k$-vector of $\rho$ if $g\mapsto\rho(g)v$ is a $C^k$ function on $SL(2,\mathbb{R})$ . Recall that the subset of $C^{\infty}$-vectors is \emph{dense} in $\mathcal{H}$.

The \emph{Lie algebra}\footnote{I.e., the tangent space of $SL(2,\mathbb{R})$ at the identity.} $sl(2,\mathbb{R})$ of $SL(2,\mathbb{R})$ is the set of all $2\times2$ matrices with zero trace. Given a $C^1$-vector $v$ of the representation $\rho$ and $X\in sl(2,\mathbb{R})$, the \emph{Lie derivative} $L_X v$ is
$$L_X v := \lim\limits_{t\to0}\frac{\rho(\exp(tX))\cdot v - v}{t}\,,$$
where $\exp(X)$ is the \emph{exponential map} (of matrices).

An important basis of $sl(2,\mathbb{R})$ is 
$$W:=\left(\begin{array}{cc}0&1\\-1&0\end{array}\right), \quad Q:=\left(\begin{array}{cc}1&0\\0&-1\end{array}\right),  \quad V:=\left(\begin{array}{cc}0&1\\1&0\end{array}\right)$$

These vectors are the infinitesimal generators of the following subgroups of $SL(2,\mathbb{R})$:
$$\exp(tW)=\left(\begin{array}{cc}\cos t&\sin t\\-\sin t&\cos t\end{array}\right), \quad \exp(tQ) = \left(\begin{array}{cc}e^t&0\\0&e^{-t}\end{array}\right), \quad \exp(tV) = \left(\begin{array}{cc}\cosh t&\sinh t\\-\sinh t&\cosh t\end{array}\right)$$ Furthermore, $[Q,W]=2V$, $[Q,V]=2W$, $[W,V]=2Q$, where $[.,.]$ is the Lie bracket\footnote{I.e., $[A,B]:= AB-BA$ is the commutator.} of $sl(2,\mathbb{R})$.

The \emph{Casimir operator} $\Omega_{\rho}$  is $\Omega_{\rho}:=(L_V^2+L_Q^2-L_W^2)/4$ on the dense subspace of $C^2$-vectors of $\rho$. It is known that $\Omega_{\rho}$ is symmetric\footnote{That is, 
$\langle \Omega_{\rho}v,w\rangle = \langle v,\Omega_{\rho}w\rangle$ for any $C^2$-vectors $v,w\in\mathcal{H}$}, its closure is a \emph{self-adjoint} operator, and it commutes with $L_X$ on $C^3$-vectors and with $\rho(g)$ on 
$C^2$-vectors (for all $X\in sl(2,\mathbb{R})$ and $g\in SL(2,\mathbb{R})$). 

In addition, when the representation $\rho$ is \emph{irreducible}, $\Omega_{\rho}$ is a scalar multiple of the identity operator, i.e., $\Omega_{\rho}v = \lambda(\rho)v$ for some $\lambda(\rho)\in\mathbb{R}$ and for all $C^2$-vectors $v\in\mathcal{H}$ of $\rho$. In general, as we're going to see below, the spectrum $\sigma(\Omega_{\rho})$ of the Casimir operator $\Omega_{\rho}$ is a fundamental object.

\subsubsection{Bargmann's classification}

We introduce the following notation:
\begin{equation*}r(\lambda):=\left\{\begin{array}{cc}-1 & \quad\quad\textrm{if } \lambda\leq -1/4, \\ -1+\sqrt{1+4\lambda} & \quad\quad\quad\,\,\,\,\textrm{ if } -1/4<\lambda<0 \\ -2 & \textrm{if } \lambda\geq 0\end{array}\right.
\end{equation*}
Note that $r(\lambda)$ satisfies the quadratic equation $x^2+2x-4\lambda=0$ when $-1/4<\lambda<0$. 

Bargmann's classification of \emph{irreducible} unitary $SL(2,\mathbb{R})$ says that the eigenvalue $\lambda(\rho)$ of the Casimir operator $\Omega_{\rho}$ has the form
$$\lambda(\rho) = (s^2-1)/4$$
where $s\in\mathbb{C}$ falls into one of the following three categories:
\begin{itemize}
\item \emph{Principal series}: $s$ is \emph{purely imaginary}, i.e., $s\in\mathbb{R}i$;
\item \emph{Complementary series}: $s\in (0,1)$ and $\rho$ is \emph{isomorphic} to the representation 
$$\rho_s\left(\begin{array}{cc}a&b\\c&d\end{array}\right) f(x):= (cx+d)^{-1-s} f\left(\frac{ax+b}{cx+d}\right),$$ where $f$ belongs to the Hilbert space $\mathcal{H}_s:=\left\{f:\mathbb{R}\to\mathbb{C}: \iint\frac{f(x)\overline{f(y)}}{|x-y|^{1-s}}dx\,dy<\infty\right\}$;
\item \emph{Discrete series}: $s\in\mathbb{N}-\{0\}$.
\end{itemize}
In other words, $\rho$ belongs to the \emph{principal series} when $\lambda(\rho)\in(-\infty,-1/4]$, $\rho$ belongs to the \emph{complementary series} when $\lambda(\rho)\in (-1/4,0)$ and $\rho$ belongs to the \emph{discrete series} when $\lambda(\rho)=(n^2-1)/4$ for some natural number $n\geq 1$.
Note that, when $-1/4<\lambda(\rho)<0$ (i.e., $\rho$ belongs to the complementary series), we have $r(\lambda(\rho))=-1+s$.

\subsubsection{Hyperbolic surfaces and examples of regular unitary $SL(2,\mathbb{R})$-representations}

Recall that $SL(2,\mathbb{R})$ is naturally identified with the unit cotangent bundle of the upper half-plane $\mathbb{H}$. Indeed, the quotient $SL(2,\mathbb{R})/SO(2,\mathbb{R})$ is diffeomorphic to $\mathbb{H}$ via
$$\left(\begin{array}{cc}a&b\\c&d\end{array}\right)\cdot SO(2,\mathbb{R})\mapsto \frac{ai+b}{ci+d}$$
Let $\Gamma$ be a lattice of $SL(2,\mathbb{R})$ and denote by $\mu$ the $SL(2,\mathbb{R})$-invariant probability measure on $\Gamma\backslash SL(2,\mathbb{R})$ induced from the Haar measure of $SL(2,\mathbb{R})$. Note that, in this situation, $M:=\Gamma\backslash SL(2,\mathbb{R})$ is naturally identified with the unit cotangent bundle $T_1 S$ of the hyperbolic surface $S:=\Gamma\backslash SL(2,\mathbb{R})\slash SO(2,\mathbb{R}) = \Gamma\backslash \mathbb{H}$ of finite area with respect to the natural measure $\nu$.

Since the actions of $SL(2,\mathbb{R})$ on $M:=\Gamma\backslash SL(2,\mathbb{R})$ and $S:=\Gamma\backslash\mathbb{H}$ preserve $\mu$ and $\nu$, we obtain the following \emph{regular} unitary $SL(2,\mathbb{R})$ representations:
$$\rho_{\Gamma}(g)f(\Gamma z):=f(\Gamma z\cdot g) \quad \forall\, f\in L^2(M,\mu)$$
and
$$\rho_S(g)f(\Gamma z SO(2,\mathbb{R})) := f(\Gamma z\cdot g SO(2,\mathbb{R})) \quad \forall\, f\in L^2(S,\nu).$$
Observe that $\rho_S$ is a subrepresentation of $\rho_M$ because the space $L^2(S,\nu)$ can be identified with the subspace $\mathcal{H}_{\Gamma}:=\{f\in L^2(M,\mu): f \textrm{ is constant along } SO(2,\mathbb{R})-\textrm{orbits}\}$. Nevertheless, it is possible to show that the Casimir operator $\Omega_{\rho_M}$ restricted to $C^2$-vectors of $\mathcal{H}_{\Gamma}$ \emph{coincides} with the Laplacian $\Delta=\Delta_S$ on $L^2(S,\nu)$. Also, we have that a number $-1/4<\lambda<0$ belongs to the spectrum of the Casimir operator $\Omega_{\rho_M}$ (on $L^2(M,\mu)$) if and only if $-1/4<\lambda<0$ belongs to the spectrum of the Laplacian $\Delta=\Delta_S$ on $L^2(S,\nu)$.

\subsubsection{Rates of mixing and spectral gap}

Recall that the action of the $1$-parameter subgroup $g(t):= \textrm{diag}(e^t,e^{-t})$, $t\in\mathbb{R}$, of diagonal matrices of $SL(2,\mathbb{R})$ on $M=\Gamma\backslash SL(2,\mathbb{R})$ is identified with the geodesic flow on a hyperbolic surface of finite area $S=\Gamma\backslash\mathbb{H}$.

Ratner \cite{Rt92} showed that the Bargmann's series of the irreducible factors of the regular $SL(2,\mathbb{R})$-representation $\rho_{\Gamma}$ on $L^2(M,\mu)$ can be deduced from the \emph{rates of mixing} of the geodesic flow $g(t)$ along a certain class of observables. More concretely, let $c(\Gamma)=\sigma(\Delta_S)\cap (-1/4,0)$ be the intersection of the spectrum of the Laplacian $\Delta_S$ with the open interval $(-1/4,0)$. We denote 
$$\beta(\Gamma) = \sup c(\Gamma)$$
with the convention $\beta(c(\Gamma))=-1/4$ when $c(\Gamma)=\emptyset$ and
$$\sigma(\Gamma)=r(\beta(\Gamma)):=-1+\sqrt{1+4\beta(\Gamma)}\,.$$
Observe that the subset $c(\Gamma)$ detects the presence of complementary series in the decomposition of $\rho_{\Gamma}$ into irreducible representations. Also, since $\Gamma$ is a lattice, it is possible to show that $c(\Gamma)$ is finite and, \emph{a fortiori}, $\beta(\Gamma)<0$. Since $\beta(\Gamma)$ essentially measures the distance between zero and the first eigenvalue of $\Delta_S$ on $\mathcal{H}_{\Gamma}$, it is natural to call $\beta(\Gamma)$ the \emph{spectral gap}.  

\begin{theorem}[Ratner]\label{t.Ratner}For any $u,v\in\mathcal{H}_{\Gamma}$ and $|t|\geq 1$, we have
\begin{itemize}
\item $|\langle u, \rho_{\Gamma}(g(t))v\rangle|\leq C_{\beta(\Gamma)}\cdot e^{\sigma(\Gamma)t}\cdot \|u\|_{L^2}\|v\|_{L^2}$ when $\mathcal{C}(\Gamma)\neq\emptyset$;
\item $|\langle u, \rho_{\Gamma}(g(t))v\rangle|\leq C_{\beta(\Gamma)}\cdot e^{\sigma(\Gamma)t}\cdot \|u\|_{L^2}\|v\|_{L^2} = C_{\beta(\Gamma)}\cdot e^{-t}\cdot \|u\|_{L^2}\|v\|_{L^2}$ when $\mathcal{C}(\Gamma)=\emptyset$, $\sup(\sigma(\Delta_S)\cap(-\infty,-1/4))<-1/4$ and $-1/4$ is not an eigenvalue of the Casimir operator $\Omega_{\rho_{\Gamma}}$;
\item $|\langle u, \rho_{\Gamma}(g(t))v\rangle|\leq C_{\beta(\Gamma)}\cdot t\cdot e^{\sigma(\Gamma)t}\cdot \|u\|_{L^2}\|v\|_{L^2} = C_{\beta(\Gamma)}\cdot t\cdot e^{-t}\cdot \|u\|_{L^2}\|v\|_{L^2}$ otherwise, i.e., when $\mathcal{C}(\Gamma)=\emptyset$ and either $\sup(\sigma(\Delta_S)\cap(-\infty,-1/4))=-1/4$ or $-1/4$ is an eigenvalue of the Casimir operator $\Omega_{\rho_{\Gamma}}$.
\end{itemize}
The above constants\footnote{The original arguments of Ratner allow one to explicitly these constants: see our paper \cite{Ma-Rt} for more details.} $C_{\mu}$ are uniformly bounded when $\mu$ varies on compact subsets of $(-\infty,0)$.
\end{theorem}

In other words, Ratner's theorem relates the (exponential) rate of mixing of the geodesic flow $g(t)$ with the spectral gap: indeed, the quantity $|\langle f,\rho_{\Gamma}(a(t))g\rangle|$ roughly measures how fast the geodesic flow $g(t)$ mixes different places of phase space\footnote{This is more clearly seen when $f$ and $g$ are characteristic functions of Borelian sets.}, so that Ratner's result says that the exponential rate $\sigma(\Gamma)$ of mixing of $g(t)$ is an explicit function of the spectral gap $\beta(\Gamma)$ of $\Delta_S$.

\subsection{Explicit hyperbolic surfaces $\mathbb{H}/\Gamma_6(2k)$ with complementary series} After this brief revision of Ratner's work \cite{Rt92}, let us discuss now one of the key ingredients in the proof of Theorem \ref{t.MaSc}, namely, Selberg's construction of cyclic covers with arbitrarily small spectral gap. 

We want to define a sequence of lattices $\Gamma_6(2k)\subset SL(2,\mathbb{Z})$ in such a way that $\mathbb{H}/\Gamma_6(2k)$ is a family of cyclic covers of a fixed genus one (finite area) hyperbolic surface. 

Evidently, the first step is to find an appropriate genus one hyperbolic surface serving as ``base surface'' of the cyclic cover construction. For this sake, we denote by   
$$\Gamma_N:=\left\{ \left(\begin{array}{cc} a & b \\ c & d \end{array}\right)\in SL(2,\mathbb{Z}): a\equiv d\equiv 1, \, b\equiv c\equiv 0 \,\, (\textrm{mod } N)\right\}$$ 
the principal congruence subgroup of level $N\in\mathbb{N}$ of $SL(2,\mathbb{Z})$, and we observe that $\mathbb{H}/\Gamma_6$ can be used as the base surface in the cyclic cover construction thanks to the following classical fact\footnote{This proposition fails for $N<6$: indeed, $\mathbb{H}/\Gamma_n$ has genus $0$ for $1\leq n\leq 5$.}: 

\begin{proposition}\label{p.Gamma6} $\mathbb{H}/\Gamma_6$ is a genus one hyperbolic surface with 12 cusps. Moreover, $\rho(c_1) = \left(\begin{array}{cc} 29 & 12 \\ 12 & 5\end{array}\right)\in \Gamma_6$ represents a non-peripheral, homotopically non-trivial closed geodesic of $\mathbb{H}/\Gamma_6$. 
\end{proposition}

\begin{proof} A complete proof of this proposition can be found in \cite[Subsection 2.1]{MaSc}. For the sake of convenience of the reader, let us give a brief sketch of the argument.

The group $P\Gamma_6:=\Gamma_6/\{\pm\textrm{Id}\}$ is a normal subgroup of index 12 of $P\Gamma_2:=\Gamma_2/\{\pm\textrm{Id}\}$: this is so because the exact sequence 
$$1\to P\Gamma_6\to PSL(2,\mathbb{Z})\to PSL(2,\mathbb{Z}/6\mathbb{Z}) \simeq PSL(2,\mathbb{Z}/2\mathbb{Z}) \times PSL(2,\mathbb{Z}/3\mathbb{Z}) \to 1$$
restricts to the exact sequence 
$$1\to P\Gamma_6\to P\Gamma_2\to PSL(2,\mathbb{Z}/3\mathbb{Z})\to 1,$$
so that the quotient $P\Gamma_2/P\Gamma_6$ is isomorphic to the finite group $PSL(2,\mathbb{Z}/3\mathbb{Z})$ of order 12. Moreover, the matrices 
\begin{eqnarray*}
& & A_1=\textrm{Id}, A_2=x, A_3=x^2, A_4 = y^{-1}x, A_5=y^{-1}, A_6=y \\ 
& & A_7 = y x^{-1}, A_8=y^{-1}x^{-1}, A_9=yx, A_{10}=y^{-1}x^{-1}y, A_{11}=yxy^{-1}, A_{12}= yxy^{-1}x^{-1}
\end{eqnarray*}
form a system of representatives of the cosets of $P\Gamma_2/P\Gamma_6$. 

The group $P\Gamma_2$ is isomorphic to the free group on 
$$x= \left(\begin{array}{cc} 1 & 2 \\ 0 & 1 \end{array}\right) \quad \textrm{ and } \quad y = \left(\begin{array}{cc} 1 & 0 \\ 2 & 1 \end{array}\right)$$ 
and a fundamental domain for the action of $P\Gamma_2$ on $\mathbb{H}$ is given by 
$$\mathcal{F}_2 = \bigcup\limits_{l=1}^6\alpha_l(\{z\in\mathbb{H}: |\textrm{Re}(z)|\leq 1/2, |z|\geq 1\})$$ 
where $\alpha_1=\textrm{Id}$, $\alpha_2:=\left(\begin{array}{cc} 1 & 1 \\ 0 & 1 \end{array}\right)$, $\alpha_3:=\left(\begin{array}{cc} 0 & -1 \\ 1 & 0 \end{array}\right)$, $\alpha_4=\alpha_2\alpha_3$, $\alpha_5 = \alpha_3\alpha_2$ and $\alpha_6 = \alpha_2^{-1}\alpha_3\alpha_2$. Note that $\mathcal{F}_2$ is an ideal quadrilateral in $\mathbb{H}$ whose edges are paired, so that $\mathbb{H}/P\Gamma_2$ is a genus 0 curve with three cusps.  

It follows from this discussion that the cover $\mathbb{H}/P\Gamma_6\to\mathbb{H}/P\Gamma_2$ has a tesselation into $12$ quadrilaterals whose dual graph is $P\Gamma_2/P\Gamma_6=\{A_m\cdot  P\Gamma_6:m=1, \dots, 12\}$ with respect to the generators $x$ and $y$: see Figure \ref{f.cayley-Gamma-6}. 

\begin{figure}[htb!]
\includegraphics[scale=0.4]{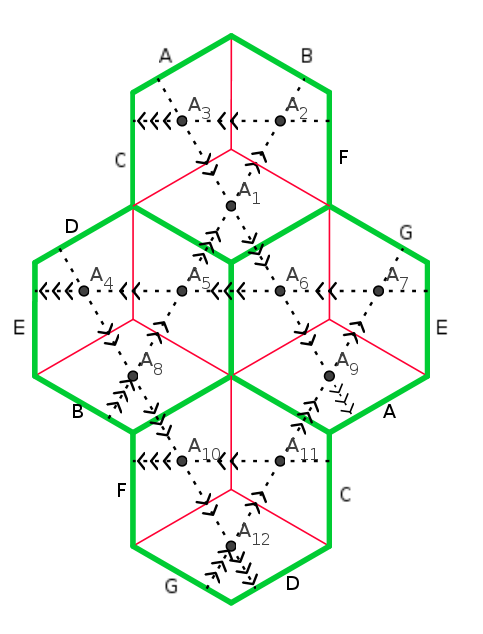}
\caption{Cayley graph of $P\Gamma_2/P\Gamma_6$.}\label{f.cayley-Gamma-6}
\end{figure}

In Figure \ref{f.cayley-Gamma-6}, the edges labelled by the same letter are identified and, thus, we have that $\mathbb{H}/P\Gamma_6$ has genus 1. Moreover, all vertices of the quadrilaterals are cusps, so that $\mathbb{H}/P\Gamma_6$ has 12 cusps. 

Finally, the fundamental group $P\Gamma_6$ of $\mathbb{H}/P\Gamma_6$ is generated by the paths $A, \dots, G$ and the small loops around the cusps $L_1,\dots, L_6$ indicated in Figure \ref{f.fundamental-group-Gamma-6}. 

\begin{figure}[htb!]
\includegraphics[scale=0.4]{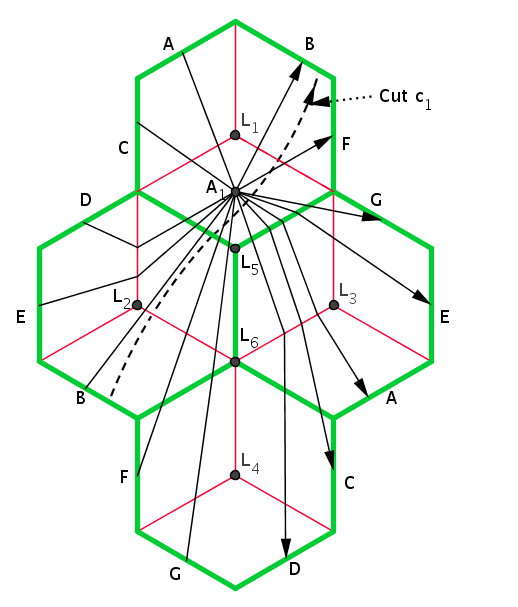}
\caption{Generators of $\pi_1(\mathbb{H}/\Gamma_6)$.}\label{f.fundamental-group-Gamma-6}
\end{figure}

In particular,  the path $c_1$ (connecting $B$-sides) in Figure \ref{f.fundamental-group-Gamma-6} is a homotopically non-trivial, non-peripheral, closed curve whose geodesic representative corresponds to the matrix 
$$\rho(c_1)=xyxy=\left(\begin{array}{cc} 29 & 12 \\ 12 & 5\end{array}\right)$$ 
This completes the proof of the proposition. 
\end{proof} 

Using this proposition, we construct a family $\mathbb{H}/\Gamma_6(2k)$ of cyclic covers of $\mathbb{H}/\Gamma_6$ as follows. We slit $\mathbb{H}/\Gamma_6$ along $c_1=\alpha$, we take $2k$ copies of the resulting slitted surface and we glue them in a cyclic order as in Figure \ref{f.Selberg-cyclic-cover} below. 

\begin{figure}[h!]
\begin{center}
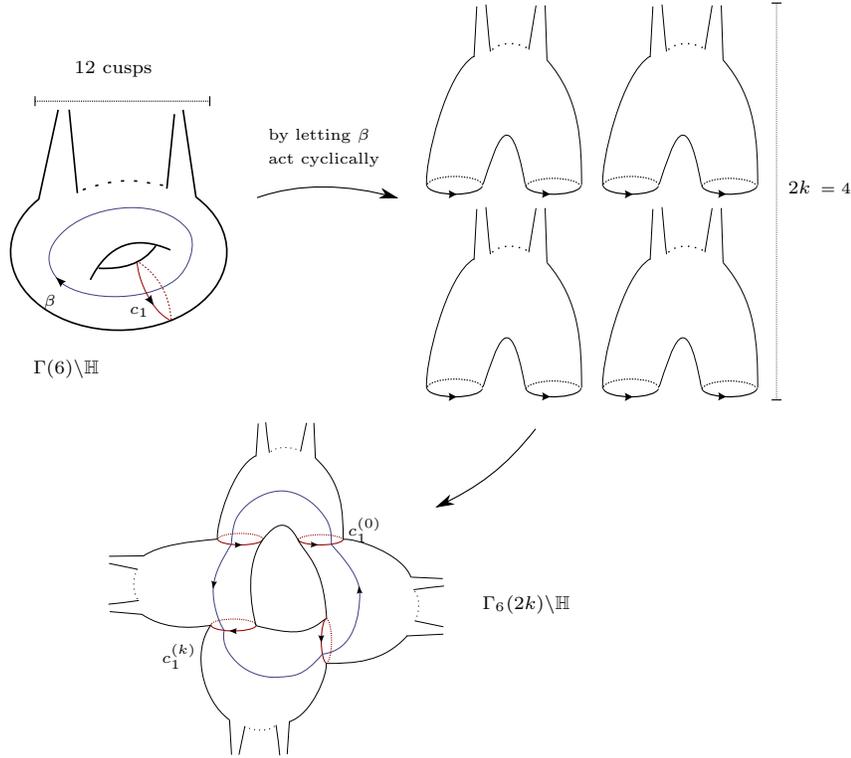
\end{center}
\caption{Selberg's cyclic cover construction.}\label{f.Selberg-cyclic-cover}
\end{figure}

Algebraically, we can describe the cyclic covers $\mathbb{H}/\Gamma_6(2k)$ in the following way. Consider the generators $\{A,\dots, G, L_1,\dots, L_6\}$ (in Figure \ref{f.fundamental-group-Gamma-6}) of the fundamental group $P\Gamma_6$ of $\mathbb{H}/P\Gamma_6$. The homomorphism $m:P\Gamma_6\to\mathbb{Z}$ defined on the generators $\{A,\dots, G, L_1, \dots, L_6\}$ of given by 
$$m(A)=m(C)=1=m(D)=m(E), \quad m(B)=m(F)=0=m(G)=m(L_n),$$
$n=1,\dots, 6$, is precisely the homomorphism assigning to elements of the fundamental group of $\mathbb{H}/P\Gamma_6$ their oriented intersection numbers with $c_1$. For each $k\in\mathbb{N}$, the kernel of the composition of $m$ with the reduction modulo $2k$ is denoted $P\Gamma_6(2k)$ and its inverse image in $SL(2,\mathbb{Z})$ is $\Gamma_6(2k)$. 

The presence of complementary series for $\mathbb{H}/\Gamma_6(2k)$ is easily detectable thanks to the so-called \emph{Buser's inequality}:

\begin{proposition}\label{p.2kGamma6} For every $k\geq 3$, the first eigenvalue $\lambda_{2k}>0$ of hyperbolic Laplacian of $\mathbb{H}/\Gamma_6(2k)$ satisfies 
$$\lambda_{2k}<\frac{1}{2k}$$ 
In particular, $\mathbb{H}/\Gamma_6(2k)$ exhibits complementary series for all $k\geq 3$ (because $\lambda_{2k}<1/4$). 
\end{proposition} 

\begin{proof} Given a hyperbolic surface $\mathbb{H}/\Gamma$ of finite area, Buser's inequality (cf. Buser \cite{Bu} and Lubotzky \cite[p. 44]{Lu}) says that the first eigenvalue $\lambda(\Gamma)>0$ of the hyperbolic Laplacian of $\mathbb{H}/\Gamma$ verifies the following estimate:
$$\sqrt{10\lambda(\Gamma)+1}\leq 10h(\Gamma)+1$$
where 
$$h(\Gamma):=\min\limits_{\substack{\gamma \textrm{ multicurve of } \mathbb{H}/\Gamma \\ \textrm{separating it into } \\ \textrm{ two connected components } A, B}} \frac{\textrm{length}(\gamma)}{\min\{\textrm{area}(A), \textrm{area}(B)\}}$$ 
is the Cheeger constant of $\mathbb{H}/\Gamma$. 

In the case of $\mathbb{H}/\Gamma_6(2k)$, we can bound its Cheeger constant $h_{2k}:=h(\Gamma_6(2k))$ as follows. Consider the multicurve in $\mathbb{H}/\Gamma_6(2k)$ consisting of the disjoint union of the copies $c_1^{(0)}$ and $c_1^{(k)}$ of $c_1$ (as indicated in Figure \ref{f.Selberg-cyclic-cover}). By definition, this multicurve separates $\mathbb{H}/\Gamma_6(2k)$ into two connected components, each of them formed of $k$ copies of $\mathbb{H}/\Gamma_6$. Thus, 
$$h_{2k}\leq \frac{2\cdot\textrm{length}(c_1)}{k\cdot \textrm{area}(\mathbb{H}/\Gamma_6)}$$ 
Since $c_1$ is represented by the matrix $\left(\begin{array}{cc} 29 & 12 \\ 12 & 5\end{array}\right)\in \Gamma_6$ (cf. Proposition \ref{p.Gamma6}), we have 
$$\textrm{length}(c_1) = 2\textrm{ arc cosh} \left(\frac{|\textrm{tr}(\rho(c_1))|}{2}\right) = 2 \textrm{ arc cosh} (17) $$ 
Also, the area\footnote{Because $\Gamma_6$ has index $72$ in $SL(2,\mathbb{Z})$ and the fundamental domain $\mathcal{F}_1=\{z\in\mathbb{H}: |\textrm{Re}(z)|\leq 1/2, |z|\geq 1\}$ of $\mathbb{H}/SL(2,\mathbb{Z})$ has hyperbolic area $\int_{-1/2}^{1/2}\int_{\sqrt{1-x^2}}^{\infty}\frac{dxdy}{y^2}=\pi/3$} of $\mathbb{H}/\Gamma_6$ is $24\pi$. By plugging this into the previous inequality, we deduce that 
$$h_{2k}\leq \frac{\textrm{arc cosh}(17)}{6k\pi}$$

By Buser's inequality, this means that the first eigenvalue $\lambda_{2k}$ of the hyperbolic Laplacian of $\mathbb{H}/\Gamma_6(2k)$ satisfies 
$$\sqrt{10\lambda_{2k}+1}\leq \frac{5\textrm{ arc cosh}(17)}{3k\pi}+1$$
i.e., 
$$\lambda_{2k}\leq \left(\frac{5\textrm{ arc cosh}(17)^2}{9k\pi^2}+\frac{2\textrm{ arc cosh}(17)}{3\pi}\right)\frac{1}{2k}$$
Since $\textrm{arc cosh}(17)<3.5255$, it follows that 
$$\lambda_{2k}<\frac{1}{2k}$$ 
for all $k\geq 3$. This proves the proposition. 
\end{proof}

\begin{remark}\label{r.2kGamma6} In general, the first eigenvalue $\lambda(\Gamma)$ of the Laplacian of $\mathbb{H}/\Gamma$ does not increase under finite covers: if $\Gamma'$ is a finite index subgroup of $\Gamma$, then $\lambda(\Gamma')\leq\lambda(\Gamma)$. Therefore, $\lambda(\Gamma)\leq\lambda_{2k}<1/k$ for any finite-index subgroup $\Gamma$ of $\Gamma_{6}(2k)$. 
\end{remark}

\begin{remark}
$\Gamma_6(2k)$, $k\geq 3$, is \emph{not} a congruence\footnote{$\Gamma\subset SL(2,\mathbb{Z})$ is a congruence subgroup if $\Gamma$ contains the principal congruence subgroup $\Gamma_N$ for some $N\in\mathbb{N}$.} subgroup of $SL(2,\mathbb{Z})$. Indeed, Selberg's $3/16$ theorem says that the first eigenvalue $\lambda(\Gamma)$ of the Laplacian of $\mathbb{H}/\Gamma$ satisfies $\lambda(\Gamma)\geq 3/16$ when $\Gamma$ is congruence, but we know from Proposition \ref{p.2kGamma6} that $\lambda(\Gamma_6(2k))=\lambda_{2k}<1/6$ for all $k\geq 3$. 
\end{remark}

\subsection{Arithmetic Teichm\"uller curves $\mathcal{S}_{2k}$ birational to $\mathbb{H}/\Gamma_6(2k)$} 

In view of Remark \ref{r.2kGamma6} and Ratner's theorem \ref{t.Ratner}, the proof of Theorem \ref{t.MaSc} is reduced to the following statement. 

\begin{theorem}\label{t.Z2k} For each $k\geq 3$, there exists an explicit square-tiled surface $Z_{2k}$ (of genus $48k+3$ tiled by $192k$ unit squares) whose Veech group is contained in $\Gamma_6(2k)$. In particular, the hyperbolic surface $\mathbb{H}/\Gamma_6(2k)$ is finitely covered by the arithmetic Teichm\"uller curve $\mathcal{S}_{2k}$ generated by the $SL(2,\mathbb{R})$-orbit of $Z_{2k}$. 
\end{theorem}

The construction of $Z_{2k}$ is based on the ideas of Ellenberg-McReynolds \cite{EM} and Schmith\"usen \cite{Sc}. Unfortunately, the implementation of these ideas is somewhat technical and, for this reason, we'll just give a sketch of the construction of $Z_{2k}$ while referring to Sections 3 and 4 of \cite{MaSc} for more details.

We build $Z_{2k}$ with the aid of ramified covers of translation surfaces. For this sake,  let us recall that if $h:X_1\to X_2$ is a finite covering of Riemann surfaces, then the \emph{ramification data} $\textrm{rm}(P,h)$ of a point $P\in X_2$ consists of the ramification indices of all preimages of $P$ counted with multiplicities. 

This notion is important for our purposes because any affine homeomorphism $\hat{f}$ on $X_1$ descending to an affine homeomorphism $f$ of $X_2$ under a translation covering $h:X_1\to X_2$ must respect the ramification data, i.e., $\textrm{rm}(f(P), h) = \textrm{rm}(P,h)$. In particular, we can force such an affine homeomorphism $f$ to respect certain partitions of the branch points of $h$ by prescribing distinct ramification data to them and, as it turns out, this information can be used to put constraints on the linear part $Df$ of $f$, that is, on the element $Df$ of the Veech group of $X_1$. 

The simplest example illustrating the ideas in the previous paragraph is the translation covering from $E[2]:=\mathbb{C}/(2\mathbb{Z}\oplus 2\mathbb{Z}i)$ to $E:=\mathbb{C}/(\mathbb{Z}\oplus\mathbb{Z}i)$ (given by the natural isogeny). Indeed, the affine homeomorphisms of $E$ fixing the origin $P:=(0,0)$ correspond to $SL(2,\mathbb{Z})$ and the affine homeomorphisms of $E[2]$ fixing the $2$-torsion points $P$, $Q=(1,0)$, $S=(1,1)$ correspond to the principal congruence subgroup $\Gamma_2$ of level $2$.

Next, for each $k\in\mathbb{N}$, one builds a translation covering $q_{2k}:Y_{2k}\to E[2]$ of degree $[\Gamma_2:\Gamma_6(2k)]=24k$ such that: 
\begin{itemize}
\item $q_{2k}$ is ramified precisely over $P$, $Q$ and $S$;  
\item all affine homeomorphisms of $E[2]$ fixing $P$, $Q$ and $S$ lift to affine homeomorphisms of $Y_{2k}$: in particular, the image of the group $\textrm{Aff}_*^{q_{2k}}(Y_{2k})$ of such lifts under the derivative homomorphism $D$ is $\Gamma_2$; 
\item the fiber $q_{2k}^{-1}(R)$ of the $2$-torsion point $R=(0,1)$ is bijectively mapped to $\Gamma_2/\Gamma_6(2k)$ by a map $\theta$ which is equivariant with respect to the derivative homomorphism $D:\textrm{Aff}_*^{q_{2k}}(Y_{2k})\to\Gamma_2$, i.e., $Df\cdot \theta(R_i) = \theta(f(R_i))$ for all $R_i\in q_{2k}^{-1}(R)$ and $f\in \textrm{Aff}_*^{q_{2k}}(Y_{2k})$. 
\item the Veech group of $Y_{2k}$ is $SL(2,\mathbb{Z})$. 
\end{itemize} 

Finally, the square-tiled surface $Z_{2k}$ is obtained from a (double) cover $r_{2k}:Z_{2k}\to Y_{2k}$ such that the ramification data of $P$, $Q$ and $S$ with respect to $r_{2k}\circ q_{2k}$ are pairwise distinct and the ramification data of the point $R_{\textrm{id}}:=\theta^{-1}(\textrm{id}\cdot \Gamma_6(2k))\in q_{2k}^{-1}(R)$ is different from the ramification data of all points in $q_{2k}^{-1}(R)$. 

In this way, we have that any affine homeomorphism $\hat{f}$ of $Z_{2k}$ \emph{descending} to an affine homeomorphism $f$ of  $Y_{2k}$ has linear part $D\hat{f}\in\Gamma_6(2k)$. Indeed, our condition on the ramification data forces $f\in\textrm{Aff}_*^{q_{2k}}(Y_{2k})$ to fix $R_{\textrm{id}}$. In particular, 
$$Df\cdot\Gamma_6(2k) = Df\cdot\theta(R_{\textrm{id}})=\theta(f(R_{\textrm{id}})) = \theta(R_{\textrm{id}}) = \textrm{id}\cdot\Gamma_6(2k),$$
that is, $D\hat{f}=Df\in\Gamma_6(2k)$. 

Therefore, the proof of Theorem \ref{t.Z2k} will be complete if we have that all affine homeomorphisms of $Z_{2k}$ descend to $Y_{2k}$. Here, one exploits the action of $\textrm{Aff}(Y_{2k})$ on $q_{2k}^{-1}(\{P, Q, R, S\})$ in order to detect a partition of $q_{2k}^{-1}(\{P, Q, R, S\})$ with the following property: if the ramification data of $r$ on the atoms of this partition are distinct, then all affine homeomorphisms of $Z_{2k}$ descend to $Y_{2k}$. Since it is not hard to produce a branched cover $r$ with this feature (for any given partition of $q_{2k}^{-1}(\{P, Q, R, S\})$), this finishes our sketch of proof of Theorem \ref{t.Z2k}.

\begin{remark} The first member $Z_6$ of the family $Z_{2k}$, $k\geq 3$, is a square-tiled surface associated to the following pair of permutations $\sigma_h$ and $\sigma_v$ (on $576$ unit squares). 

\begin{longtable}{lcl}
  $\sigma_h$ &=& 
  $(1, 13, 193, 207, 243, 253)
  (2, 14, 194, 208, 244, 254)
  (3, 15, 195, 209, 245, 255)$ \\
  &&
  $(4, 16, 196, 210, 246, 256)
  (5, 17, 197, 211, 247, 257)
  (6, 18, 198, 212, 248, 258)$\\
  &&
  $(7, 19, 199, 213, 249, 259)
  (8, 20, 200, 214, 250, 260)
  (9, 21, 201, 215, 251, 261)$\\
  &&
  $(10, 22, 202, 216, 252, 262)
  (11, 23, 203, 205, 241, 263)
  (12, 24, 204, 206, 242, 264)$\\
  &&
  $(25, 38, 266, 280, 220, 230, 26, 37, 265, 279, 219, 229)
  (27, 39, 267, 281, 221, 231)$\\
  &&
  $(28, 40, 268, 282, 222, 232)
  (29, 41, 269, 283, 223, 233)
  (30, 42, 270, 284, 224, 234)$\\
  &&
  $(31, 43, 271, 285, 225, 235)
  (32, 44, 272, 286, 226, 236)
  (33, 45, 273, 287, 227, 237)$\\ 
  &&
  $(34, 46, 274, 288, 228, 238) 
  (35, 47, 275, 277, 217, 239)
  (36, 48, 276, 278, 218, 240)$\\
  &&
  $(49, 61, 433, 445, 337, 349)
  (50, 62, 434, 446, 338, 350)
  (51, 63, 435, 447, 339, 351)$\\ 
  &&
  $(52, 64, 436, 448, 340, 352)
  (53, 65, 437, 449, 341, 353)
  (54, 66, 438, 450, 342, 354)$\\
  &&
  $(55, 67, 439, 451, 343, 355)
  (56, 68, 440, 452, 344, 356)
  (57, 69, 441, 453, 345, 357)$\\
  &&
  $(58, 70, 442, 454, 346, 358)
  (59, 71, 443, 455, 347, 359)
  (60, 72, 444, 456, 348, 360)$\\
  &&
  $(73, 85, 361, 373, 457, 469)
  (74, 86, 362, 374, 458, 470)
  (75, 87, 363, 375, 459, 471)$\\
  &&
  $(76, 88, 364, 376, 460, 472)
  (77, 89, 365, 377, 461, 473)
  (78, 90, 366, 378, 462, 474)$\\
  &&
  $(79, 91, 367, 379, 463, 475)
  (80, 92, 368, 380, 464, 476)
  (81, 93, 369, 381, 465, 477)$\\
  &&
  $(82, 94, 370, 382, 466, 478)
  (83, 95, 371, 383, 467, 479)
  (84, 96, 372, 384, 468, 480)$\\
  &&
  $(97, 109, 385, 397, 481, 493)
  (98, 110, 386, 398, 482, 494)
  (99, 111, 387, 399, 483, 495)$\\
  &&
  $(100, 112, 388, 400, 484, 496)
  (101, 113, 389, 401, 485, 497)
  (102, 114, 390, 402, 486, 498)$\\
  &&
  $(103, 115, 391, 403, 487, 499)
  (104, 116, 392, 404, 488, 500)
  (105, 117, 393, 405, 489, 501)$\\
  &&
  $(106, 118, 394, 406, 490, 502)
  (107, 119, 395, 407, 491, 503)
  (108, 120, 396, 408, 492, 504)$\\
  &&
  $(121, 133, 505, 517, 409, 421)
  (122, 134, 506, 518, 410, 422)
  (123, 135, 507, 519, 411, 423)$\\
  &&
  $(124, 136, 508, 520, 412, 424)
  (125, 137, 509, 521, 413, 425)
  (126, 138, 510, 522, 414, 426)$\\
  &&
  $(127, 139, 511, 523, 415, 427)
  (128, 140, 512, 524, 416, 428)
  (129, 141, 513, 525, 417, 429)$\\
  &&
  $(130, 142, 514, 526, 418, 430)
  (131, 143, 515, 527, 419, 431)
  (132, 144, 516, 528, 420, 432)$\\
  &&
  $(145, 167, 539, 551, 299, 301)
  (146, 168, 540, 552, 300, 302)
  (147, 157, 529, 541, 289, 303)$\\
  &&
  $(148, 158, 530, 542, 290, 304)
  (149, 159, 531, 543, 291, 305)
  (150, 160, 532, 544, 292, 306)$\\
  &&
  $(151, 161, 533, 545, 293, 307)
  (152, 162, 534, 546, 294, 308)
  (153, 163, 535, 547, 295, 309)$\\
  &&
  $(154, 164, 536, 548, 296, 310)
  (155, 165, 537, 549, 297, 311)
  (156, 166, 538, 550, 298, 312)$\\
  &&
  $(169, 183, 315, 325, 553, 565)
  (170, 184, 316, 326, 554, 566)
  (171, 185, 317, 327, 555, 567)$\\
  &&
  $(172, 186, 318, 328, 556, 568)
  (173, 187, 319, 329, 557, 569)
  (174, 188, 320, 330, 558, 570)$\\
  &&
  $(175, 189, 321, 331, 559, 571)
  (176, 190, 322, 332, 560, 572)
  (177, 191, 323, 333, 561, 573)$\\
  &&
  $(178, 192, 324, 334, 562, 574)
  (179, 181, 313, 335, 563, 575)
  (180, 182, 314, 336, 564, 576)$,\\
  $\sigma_v$ &=&
  $(1, 265, 289, 553, 433, 73)
  (2, 266, 290, 554, 434, 74)
  (3, 267, 291, 555, 435, 75)$\\
  &&
  $(4, 268, 292, 556, 436, 76)
  (5, 269, 293, 557, 437, 77)
  (6, 270, 294, 558, 438, 78)$\\
  &&
  $(7, 271, 295, 559, 439, 79)
  (8, 272, 296, 560, 440, 80)
  (9, 273, 297, 561, 441, 81)$\\
  &&
  $(10, 274, 298, 562, 442, 82)
  (11, 275, 299, 563, 443, 83, 12, 276, 300, 564, 444, 84)$\\
  &&
  $(13, 229, 157, 565, 493, 133)
  (14, 230, 158, 566, 494, 134)
  (15, 231, 159, 567, 495, 135)$\\
  &&
  $(16, 232, 160, 568, 496, 136)
  (17, 233, 161, 569, 497, 137)
  (18, 234, 162, 570, 498, 138)$\\
  &&
  $(19, 235, 163, 571, 499, 139)
  (20, 236, 164, 572, 500, 140)
  (21, 237, 165, 573, 501, 141)$\\
  &&
  $(22, 238, 166, 574, 502, 142)
  (23, 239, 167, 575, 503, 143)
  (24, 240, 168, 576, 504, 144)$\\
  &&
  $(25, 97, 505, 529, 169, 193, 26, 98, 506, 530, 170, 194)
  (27, 99, 507, 531, 171, 195)$\\
  &&
  $(28, 100, 508, 532, 172, 196)
  (29, 101, 509, 533, 173, 197)
  (30, 102, 510, 534, 174, 198)$\\
  &&
  $(31, 103, 511, 535, 175, 199)
  (32, 104, 512, 536, 176, 200)
  (33, 105, 513, 537, 177, 201)$\\
  &&
  $(34, 106, 514, 538, 178, 202)
  (35, 107, 515, 539, 179, 203)
  (36, 108, 516, 540, 180, 204)$\\
  &&
  $(37, 61, 469, 541, 325, 253)
  (38, 62, 470, 542, 326, 254)
  (39, 63, 471, 543, 327, 255)$\\
  &&
  $(40, 64, 472, 544, 328, 256)
  (41, 65, 473, 545, 329, 257)
  (42, 66, 474, 546, 330, 258)$\\
  &&
  $(43, 67, 475, 547, 331, 259)
  (44, 68, 476, 548, 332, 260)
  (45, 69, 477, 549, 333, 261)$\\
  &&
  $(46, 70, 478, 550, 334, 262)
  (47, 71, 479, 551, 335, 263)
  (48, 72, 480, 552, 336, 264)$\\
  &&
  $(49, 361, 145, 313, 385, 121)
  (50, 362, 146, 314, 386, 122)
  (51, 363, 147, 315, 387, 123)$\\
  &&
  $(52, 364, 148, 316, 388, 124)
  (53, 365, 149, 317, 389, 125)
  (54, 366, 150, 318, 390, 126)$\\
  &&
  $(55, 367, 151, 319, 391, 127)
  (56, 368, 152, 320, 392, 128)
  (57, 369, 153, 321, 393, 129)$\\
  &&
  $(58, 370, 154, 322, 394, 130)
  (59, 371, 155, 323, 395, 131)
  (60, 372, 156, 324, 396, 132)$\\
  &&
  $(85, 109, 421, 301, 181, 349)
  (86, 110, 422, 302, 182, 350)
  (87, 111, 423, 303, 183, 351)$\\
  &&
  $(88, 112, 424, 304, 184, 352)
  (89, 113, 425, 305, 185, 353)
  (90, 114, 426, 306, 186, 354)$\\
  &&
  $(91, 115, 427, 307, 187, 355)
  (92, 116, 428, 308, 188, 356)
  (93, 117, 429, 309, 189, 357)$\\
  &&
  $(94, 118, 430, 310, 190, 358)
  (95, 119, 431, 311, 191, 359)
  (96, 120, 432, 312, 192, 360)$\\
  &&
  $(205, 277, 397, 517, 445, 373)
  (206, 278, 398, 518, 446, 374)
  (207, 279, 399, 519, 447, 375)$\\
  &&
  $(208, 280, 400, 520, 448, 376)
  (209, 281, 401, 521, 449, 377)
  (210, 282, 402, 522, 450, 378)$\\
  &&
  $(211, 283, 403, 523, 451, 379)
  (212, 284, 404, 524, 452, 380)
  (213, 285, 405, 525, 453, 381)$\\
  &&
  $(214, 286, 406, 526, 454, 382)
  (215, 287, 407, 527, 455, 383)
  (216, 288, 408, 528, 456, 384)$\\
  &&
  $(217, 337, 457, 481, 409, 241)
  (218, 338, 458, 482, 410, 242)
  (219, 339, 459, 483, 411, 243)$\\
  &&
  $(220, 340, 460, 484, 412, 244)
  (221, 341, 461, 485, 413, 245)
  (222, 342, 462, 486, 414, 246)$\\
  &&
  $(223, 343, 463, 487, 415, 247)
  (224, 344, 464, 488, 416, 248)
  (225, 345, 465, 489, 417, 249)$\\
  &&
  $(226, 346, 466, 490, 418, 250)
  (227, 347, 467, 491, 419, 251)
  (228, 348, 468, 492, 420, 252)$
\end{longtable}
\end{remark}

\newpage 

\null

\newpage


\begin{centering}
\rule{\textwidth}{1.6pt}\vspace*{-\baselineskip}\vspace*{2.5pt}
\rule{\textwidth}{0.4pt}

\section{Some finiteness results for algebraically primitive Teichm\"uller curves}\label{s.MW}

\rule{\textwidth}{0.4pt}\vspace*{-\baselineskip}\vspace{3.2pt}
\rule{\textwidth}{1.6pt}
\end{centering}\\

Many applications of the dynamics of $SL(2,\mathbb{R})$ on moduli spaces of translation surfaces to the investigation of translation flows and billiards rely on the features of the closure of certain $SL(2,\mathbb{R})$-orbits. For example, Delecroix-Hubert-Leli\`evre \cite{DHL} exploited the properties of the closure of certain $SL(2,\mathbb{R})$-orbits of translation surfaces of genus five in order to confirm a conjecture of Hardy and Weber on the abnormal rate of diffusion of trajectories in $\mathbb{Z}^2$-periodic Ehrenfest wind-tree models. 

Partly motivated by potential further applications, the problem of classifying closures of $SL(2,\mathbb{R})$-orbits in moduli spaces of translation surfaces received a considerable attention in recent years. 

\subsection{Some classification results for the closures of $SL(2,\mathbb{R})$-orbits in moduli spaces}\label{ss.classication-survey} 

The quest of listing all $SL(2,\mathbb{R})$-orbit closures in moduli spaces of translation surfaces became a reasonable goal after the groundbreaking works of Eskin and Mirzakhani \cite{EsMi}, Eskin, Mirzakhani and Mohammadi \cite{EMM} and Filip \cite{Fi13a}. Indeed, these results say that such $SL(2,\mathbb{R})$-orbit closures have many good properties including: they are affine in period coordinates, there are only countably many of them, and they are quasi-projective varieties with respect to the natural algebraic structure on moduli spaces. 

Despite the absence of a complete classification of $SL(2,\mathbb{R})$-orbit closures of translation surfaces, the current literature on the subject contains many papers. For this reason, instead of trying to give an exhaustive list of articles on this topic, we shall restrict ourselves to the discussion of the smallest possible $SL(2,\mathbb{R})$-orbit closures -- namely, \emph{Teichm\"uller curves} -- while refereeing to the introduction of the paper of Apisa \cite{Ap} and the references therein for more details on higher-dimensional $SL(2,\mathbb{R})$-orbit closures. 

\emph{Arithmetic} Teichm\"uller curves are always abundant: they form a dense subset in any connected component of any stratum of the moduli space of translation surfaces. This indicates that a complete classification of these objects is a challenging task and, indeed, we are able to list all arithmetic Teichm\"uller curves only in the case of the minimal stratum $\mathcal{H}(2)$ thanks to the works of Hubert and Leli\`evre \cite{HL} and McMullen \cite{McM05}. 

\emph{Non-arithmetic} Teichm\"uller curves seem less abundant and we dispose of many partial results towards their classification. In fact, Calta \cite{Ca} and McMullen \cite{McM03}, \cite{McM06} obtained a complete classification of all $SL(2,\mathbb{R})$-orbit closures of translation surfaces of genus two: it follows from their results that the minimal stratum $\mathcal{H}(2)$ contains infinitely many non-arithmetic Teichm\"uller curves, but the principal stratum $\mathcal{H}(1,1)$ contains just one non-arithmetic Teichm\"uller curve (generated by a regular decagon). In higher genera $g\geq 3$, we have many results establishing the \emph{finiteness} of \emph{algebraically primitive} Teichm\"uller curves, i.e., Teichm\"uller curves whose trace field\footnote{See Subsection \ref{ss.Veech-surfaces} above.} has the largest possible degree $g$ over $\mathbb{Q}$. For example:
\begin{itemize}
\item M\"oller \cite{Mo08} showed that $\mathcal{H}(g-1,g-1)^{hyp}$ contains only finitely many algebraically primitive Teichm\"uller curves, and  
\item Bainbridge and M\"oller \cite{BM} established the finiteness of algebraically primitive Teichm\"uller curves in $\mathcal{H}(3,1)$.
\end{itemize} 

The main theorem of this section (namely, Theorem \ref{t.MW-A}) is a result due to Wright and the author \cite{MW} showing the finiteness of algebraically primitive Teichm\"uller curves in the minimal stratum $\mathcal{H}(2g-2)$ when $g>2$ is a prime number. 

Before giving the precise statement of the main result of \cite{MW} (and sketching its proof), let us mention that a recent work of Bainbridge, Habbeger and M\"oller \cite{BHM} proved the finiteness of algebraically primitive Teichm\"uller curves in all strata of the moduli space of translation surfaces of genus three: similarly to the work of Bainbridge and M\"oller \cite{BM} mentionned above, Bainbridge, Habegger and M\"oller rely mostly on algebro-geometrical arguments, even though their treatment of the particular of the case $\mathcal{H}(2,2)^{odd}$ build upon the techniques of our joint work \cite{MW} with Wright. 

\subsection{Statement of the main results}\label{ss.HTplane}

The main result of our paper \cite{MW} with A. Wright is: 

\begin{theorem}\label{t.MW-A} Let $\mathcal{C}$ be a connected component of a stratum $\mathcal{H}(k_1,\dots, k_{\sigma})$ of the moduli space of translation surfaces of genus $g$ ($=1+\sum\limits_{j=1}^{\sigma} k_j/2$). 
\begin{itemize}
\item[(a)] If $g\geq 3$, then the (countable) union $A=\bigcup C_i$ of all algebraically primitive Teichm\"uller curves $C_i$ contained in $\mathcal{C}$ is not dense, i.e., $\overline{A}\neq\mathcal{C}$.
\item[(b)] If $g\geq 3$ is prime and $\mathcal{C}$ is a connected component of the minimal stratum $\mathcal{H}(2g-2)$, then there are only finitely many algebraically primitive Teichm\"uller curves contained in $\mathcal{C}$.
\end{itemize}
\end{theorem} 

\begin{remark} Apisa \cite{Ap} recently improved item (b) for the hyperelliptic component $\mathcal{C}=\mathcal{H}(2g-2)^{\textrm{hyp}}$ of the minimal stratum by removing the constraint ``$g$ is prime''. 
\end{remark}

\begin{remark} The technique of proof of this theorem is ``flexible'': for example, it was used (beyond the context of algebraic primitivity) by Nguyen, Wright and the author (cf. \cite[Theorem 1.6]{MW}) to show that the hyperelliptic component $\mathcal{H}(4)^{\textrm{hyp}}$ of the minimal stratum in genus $3$ contains\footnote{In fact, it was conjectured by Bainbridge-M\"oller \cite{BM} that it contains exactly two non-arithmetic Teichm\"uller curves, namely, the algebraically primitive closed $SL(2,\mathbb{R})$-orbit generated by the regular $7$-gon and the non-algebraically primitive closed $SL(2,\mathbb{R})$-orbit generated by the $12$-gon.} only finitely many non-arithmetic Teichm\"uller curves. 
\end{remark}

A key idea in the proof of Theorem \ref{t.MW-A} is the study of \emph{Hodge-Teichm\"uller planes}:

\begin{definition} Let $M$ be a translation surface. We say that $P\subset H^1(M,\mathbb{R})$ is a \emph{Hodge-Teichm\"uller plane} if the (Gauss-Manin) parallel transport\footnote{Technically speaking, this parallel transport might be well-defined only on an adequate finite cover of $\mathcal{C}$ (getting rid of all ambiguities coming from eventual automorphisms of $M$): see Remark \ref{r.KZ-lift}. Of course, this minor point does not affect the arguments in this section and, for this reason, we will skip in all subsequent discussion.} of $P$ along the $SL(2,\mathbb{R})$-orbit of $M$ respect the Hodge decomposition\footnote{Recall that Hodge's decomposition theorem says that  $H^1(M,\mathbb{C}) = H^{1,0}(M)\oplus H^{0,1}(M)$ where $H^{1,0}(M)$, resp. $H^{0,1}(M)$, is the space of holomorphic, resp. anti-holomorphic, forms.}, i.e., 
$$\textrm{dim}_{\mathbb{C}}((hP\otimes\mathbb{C})\cap H^{1,0}(hM)) = 1$$
for all $h\in SL(2,\mathbb{R})$. 
\end{definition}

\begin{example}\label{ex.HT-canonical} Any translation surface $M=(X,\omega)$ possesses a canonical Hodge-Teichm\"uller plane, namely its tautological plane $\textrm{span}_{\mathbb{R}}(\textrm{Re}(\omega), \textrm{Im}(\omega)) \subset H^1(M,\mathbb{R})$. 
\end{example}

\begin{example}\label{ex.HT-Galois-conjugates} Let $M$ be a Veech surface whose trace field $k(M)=\mathbb{Q}(\{\textrm{tr}(\gamma):\gamma\in SL(M)\})$ associated to its Veech group $SL(M)$ has degree $k$ over $\mathbb{Q}$. The $k$ embeddings of $k(M)$ can be used to construct $k$ planes $\mathbb{L}_1, \dots, \mathbb{L}_k\subset H^1(M,\mathbb{R})$ obtained from the tautological plane $\mathbb{L}_1$ by Galois conjugation. As it was observed by M\"oller \cite[Proposition 2.4]{Mo06}, we have a decomposition 
$$H^1(M,\mathbb{R})=\mathbb{L}_1\oplus\dots\oplus \mathbb{L}_k\oplus \mathbb{M}$$ 
of \emph{variation of Hodge structures}\footnote{I.e., this is a $SL(2,\mathbb{R})$-equivariant decomposition such that the complexification of each $\mathbb{L}_j$ is the sum of its $(1,0)$ and $(0,1)$ parts: $(\mathbb{L}_j)_{\mathbb{C}}:=\mathbb{L}_j\otimes\mathbb{C}$ equals $\mathbb{L}_j^{1,0}\oplus \mathbb{L}_j^{0,1}$ where $\mathbb{L}_j^{a,b} := (\mathbb{L}_j)_{\mathbb{C}}\cap H^{a,b}(X)$.} whose summands are symplectically orthogonal. By definition, this means that $\mathbb{L}_1,\dots,\mathbb{L}_k$ are symplectically orthogonal Hodge-Teichm\"uller planes. 
\end{example}

In fact, these planes are important for our purposes because of the following features highlighted in the next two theorems (compare with Theorems 1.2 and 1.3 in \cite{MW}).

\begin{theorem}\label{t.MW-B} Suppose that $\mathcal{M}$ is an affine invariant submanifold in the moduli space of genus $g$ translation surfaces containing a dense set of algebraically primitive Teichm\"uller curves. Then, every translation surface in $\mathcal{M}$ has $g$ symplectically orthogonal Hodge-Teichm\"uller planes. 
\end{theorem}

\begin{theorem}\label{t.MW-C} Let $\mathcal{C}$ be a connected component of a stratum of translation surfaces of genus $g\geq 3$. Then, there exists a translation surface $M_{\mathcal{C}}\in\mathcal{C}$ which does not have $g-1$ symplectically orthogonal Hodge-Teichm\"uller planes.
\end{theorem}

In other words, Theorem \ref{t.MW-B} says that algebraically primitive Teichm\"uller curves support many Hodge-Teichm\"uller planes and, moreover, these planes pass to the closure of any sequence of algebraically primitive Teichm\"uller curves. On the other hand, Theorem \ref{t.MW-C} says that the presence of many symplectically orthogonal Hodge-Teichm\"uller planes is not satisfied by all translation surfaces in any given stratum. 

Note that Theorems \ref{t.MW-B} and \ref{t.MW-C} trivially imply the item (a) of Theorem \ref{t.MW-A}. Furthermore, these two theorems also imply immediately the item (b) of Theorem \ref{t.MW-A} when they are combined with the following result of A. Wright (cf. \cite[Corollary 8.1]{W}):

\begin{theorem}[Wright]\label{t.Wr} Let $m\geq 2$ be a prime number. Denote by $\mathcal{M}$ an affine invariant submanifold of a connected component $\mathcal{C}$ the minimal stratum $\mathcal{H}(2m-2)$. If $\mathcal{M}$ properly contains an algebraically primitive Teichm\"uller curve, then $\mathcal{M}=\mathcal{C}$.
\end{theorem}  

\begin{proof}[Sketch of proof of Theorem \ref{t.Wr}] Denote by $k(\mathcal{M})$ the \emph{field of definition} of $\mathcal{M}$, i.e., the smallest extension of $\mathbb{Q}$ containing the coefficients of all affine equations in period coordinates defining $\mathcal{M}$. 

The field of definition has the following three general properties: 
\begin{itemize}
\item the field of definition of a Teichm\"uller curve coincides with its trace field; 
\item it has a hereditary property: if $\mathcal{N}$ and $\mathcal{M}$ are affine invariant submanifolds and $\mathcal{N}\subset\mathcal{M}$, then $k(\mathcal{M})\subset k(\mathcal{N})$; 
\item $\textrm{dim}_{\mathbb{C}}p(T\mathcal{M})\cdot\textrm{deg}_{\mathbb{Q}}(k(\mathcal{M}))\leq 2m$ whenever $\mathcal{M}$ is an affine invariant submanifold in a stratum of genus $m$ translation surfaces whose tangent space $T\mathcal{M}\subset H^1(M,\textrm{div}(\omega),\mathbb{C})$ at a point $(M,\omega)\in\mathcal{M}$ projects to a subspace $p(T\mathcal{M})\subset H^1(M,\mathbb{C})$ under the natural projection $p: H^1(M,\textrm{div}(\omega),\mathbb{C})\to H^1(M,\mathbb{C})$. 
\end{itemize}

Let $\mathcal{M}$ be an affine invariant submanifold of a connected component $\mathcal{C}$ of $\mathcal{H}(2m-2)$. Suppose that $m\geq 2$ is a prime number 
and $\mathcal{M}$ properly contains an algebraically primitive Teichm\"uller curve $C$. Then, the first two properties above of the field of definition imply that the degree $\textrm{deg}_{\mathbb{Q}}(k(\mathcal{M}))$ divides the degree of the trace field of $C$, i.e., $\textrm{deg}_{\mathbb{Q}}(k(\mathcal{M}))$ divides $m$. Since $m$ is a prime number, this means that $\textrm{deg}_{\mathbb{Q}}(k(\mathcal{M}))$ equals $1$ or $m$. We affirm that $\textrm{deg}_{\mathbb{Q}}(k(\mathcal{M}))=1$: indeed, if $\textrm{deg}_{\mathbb{Q}}(k(\mathcal{M}))=m$, then the third property of the field of definition would imply that $\textrm{dim}_{\mathbb{C}}p(T\mathcal{M})\leq 2$, a contradiction with the fact that $\mathcal{M}$ \emph{properly} contains a Teichm\"uller curve. Once we know that $k(\mathcal{M})=\mathbb{Q}$, it is not hard to see that the tangent space to $\mathcal{M}$ has complex dimension at least $2m$: in fact, since $\mathcal{M}$ is defined over $\mathbb{Q}$, the space $p(T\mathcal{M})$ at any point $(M,\omega)\in C$ contains the tangent space (tautological plane) to the algebraically primitive Teichm\"uller curve $C$ and all of its $m$ Galois conjugates. Because $\mathcal{M}\subset\mathcal{H}(2m-2)$ and the minimal stratum $\mathcal{H}(2m-2)$ has complex dimension $2m\leq\textrm{dim}_{\mathbb{C}}(T\mathcal{M})$, it follows that $\mathcal{M}$ is an open $GL^+(2,\mathbb{R})$-invariant subset of $\mathcal{H}(2m-2)$. By the ergodicity theorem of Masur and Veech, this implies that $\mathcal{M}$ is a connected component of the stratum $\mathcal{H}(2m-2)$.
\end{proof}

In the sequel, we will discuss the proofs of Theorems \ref{t.MW-B} and \ref{t.MW-C}. More precisely, we will establish Theorem \ref{t.MW-B} in Subsection \ref{ss.t.MW-B} below by studying some continuity properties of Hodge-Teichm\"uller planes, and we will provide a \emph{sketch} of proof of Theorem \ref{t.MW-C}  together with an elementary proof of a particular case of this theorem in Subsection \ref{ss.t.MW-C} below. 

\subsection{Proof of Theorem \ref{t.MW-B}}\label{ss.t.MW-B} Recall that the most basic example of Hodge-Teichm\"uller plane associated to any given translation surface $(M,\omega)$ is the tautological plane $\mathbb{L}_1:=\textrm{span}_{\mathbb{R}}(\textrm{Re}(\omega), \textrm{Im}(\omega)) \subset H^1(M,\mathbb{R})$ (cf. Example \ref{ex.HT-canonical}). 

If the $SL(2,\mathbb{R})$-orbit of a translation surface $X=(M,\omega)$ of genus $g$ generates an algebraically Teichm\"uller curve $\mathcal{C}$, then we have $g$ symplectically orthogonal Hodge-Teichm\"uller planes $\mathbb{L}_1, \dots, \mathbb{L}_g$ (cf. Example \ref{ex.HT-Galois-conjugates}). 

Therefore, the proof of Theorem \ref{t.MW-B} is reduced to the following continuity property of Hodge-Teichm\"uller planes: 

\begin{proposition}\label{p.HT-continuity} Let $\mathcal{C}$ be a connected component of a stratum of the moduli space of  translation surfaces. Suppose that $X_n\in\mathcal{C}$ is a sequence of translation surfaces converging to $X\in\mathcal{C}$ such that, for some fixed $k\in\mathbb{N}$, each $X_n$ possesses $k$ symplectically orthogonal Hodge-Teichm\"uller planes, say $P_n^{(1)}, \dots, P_n^{(k)}$. Then, $X$ possesses $k$ symplectically orthogonal Hodge-Teichm\"uller planes.  
\end{proposition} 

\begin{proof} By definition, $X_n=(M_n,\omega_n)$ converges to $X=(M,\omega)$ whenever we can find diffeomorphisms $f_n: M_n\to M$ such that $(f_n)_*(\omega_n)\to\omega$. 

By extracting an appropriate subsequence if necessary, we can assume that, for each $1\leq j\leq k$,  $(f_n)_*(P_n^{(j)})$ converges to a plane $P^{(j)}$ in the Grassmanian of planes of $H^1(M,\mathbb{R})$. 

We claim that $P^{(j)}$ is a Hodge-Teichm\"uller plane (for each $1\leq j\leq k$). In fact, given any $h\in SL(2,\mathbb{R})$, we have that 
\begin{equation}\label{e.hPn}
(\phi_h\circ f_n\circ \phi_h^{-1})_*(h P_n^{(j)})\to h P^{(j)},
\end{equation} 
where $\phi_h$ is the affine homeomorphism induced by $h$. On the other hand, we know that  $$(hP_n^{(j)}\otimes\mathbb{C})\cap H^{1,0}(hM_n)\neq\{0\}$$ for each $n\in\mathbb{N}$ and $1\leq j\leq k$ (because $P_n^{(j)}$ are Hodge-Teichm\"uller planes) and, in general,  $H^{1,0}(N)$ varies continuously with $N$ (see, e.g., \cite[Chapitre 9]{Voisin}). Thus, it follows from \eqref{e.hPn} that $(hP^{(j)}\otimes\mathbb{C})\cap H^{1,0}(hM)\neq\{0\}$, i.e., $P^{(j)}$ is a Hodge-Teichm\"uller plane. 

Finally, we affirm that $P^{(j)}$ are $k$ pairwise distinct symplectically orthogonal planes. Indeed, the continuity of the symplectic intersection form implies that $P^{(j)}$ are mutually symplectically orthogonal. Nevertheless, this is not sufficient to obtain that $P^{(j)}$ are pairwise distinct. For this sake, we observe that, by definition, a Hodge-Teichm\"uller plane is Hodge-star\footnote{Recall that the Hodge-star operator $\ast:H^1(M,\mathbb{R})\to H^1(M,\mathbb{R})$ is defined by the fact that the form $c+i(\ast c)$ is holomorphic for all $c\in H^1(M,\mathbb{R})$.} invariant (because its complexification is the sum of its $(1,0)$ and $(0,1)$ parts). Hence, the symplectic orthogonal of a Hodge-Teichm\"uller plane coincides\footnote{See Lemma 3.4 of \cite{FMZ-ETDS} for more details.} with its orthogonal for the Hodge inner product\footnote{The Hodge norm $\|.\|$ is $\|c\|^2:=\int_M c\wedge \ast c$.} and, \emph{a fortiori}, $P^{(j)}$ are mutually orthogonal with respect to the Hodge inner product. In particular, $P^{(j)}$ are pairwise distinct as it was claimed. 
\end{proof}

\subsection{Sketch of proof of Theorem \ref{t.MW-C}}\label{ss.t.MW-C} Let $X=(M,\omega)$ be a translation surface. The group $\textrm{Aff}(X)$ of affine homeomorphisms of $X$ acts on $H^1(M,\mathbb{R})$ via symplectic matrices. Let $\Gamma(X)$ be the image of the natural representation $\textrm{Aff}(X)\to \textrm{Sp}(H^1(M,\mathbb{R}))$ and denote by $\overline{\Gamma(X)}$ the Zariski closure of $\Gamma(X)$. 

The proof of Theorem \ref{t.MW-C} starts with the following fact: 

\begin{proposition}\label{p.HT-mon-invariance} The set of Hodge-Teichm\"uller planes of $X$ is $\overline{\Gamma(X)}$-invariant.  
\end{proposition} 

\begin{proof} For each $h\in SL(2,\mathbb{R})$, let $H^{1,0}_h:=h^{-1}(H^{1,0}(hX))\subset H^1(X,\mathbb{C})$. By definition, a plane $P\subset H^1(X,\mathbb{R})$ is Hodge-Teichm\"uller if and only if $P_{\mathbb{C}}:=P\otimes\mathbb{C}$ intersects $H^{1,0}_h$ non-trivially for all $h\in SL(2,\mathbb{R})$. 

Given $h\in SL(2,\mathbb{R})$, the condition that $P_{\mathbb{C}}$ intersects $H^{1,0}$ non-trivially corresponds to a finite number of polynomial equations on the Grassmanian of planes in $H^1(X,\mathbb{R})$: indeed, if we form a matrix $M_P$ by listing a basis of $P$ next to a basis of $H^{1,0}_h$, then $P_{\mathbb{C}}\cap H^{1,0}_h \neq \{0\}$ is equivalent to $\textrm{rank}(M_P)\leq g+1$, i.e., all $(g+2)\times (g+2)$ minors of $M_P$ vanish. (Here, $g$ is the genus of $X$.)

In particular, the set $\mathcal{P}$ of Hodge-Teichm\"uller planes is a subvariety of the Grassmanian of planes in $H^1(X,\mathbb{R})$. It follows that the stabilizer of $\mathcal{P}$ contains the Zariski closure $\overline{\Gamma(X)}$ (because $\mathcal{P}$ is clearly $\Gamma(M)$-invariant). This proves the proposition. 
\end{proof}

Given a translation surface $X=(M,\omega)$ (of genus $g$), let $H^1(X,\mathbb{R})^{\perp}$ be the ($(2g-2)$-dimensional) symplectic orthogonal of the tautological plane $\textrm{span}_{\mathbb{R}}(\textrm{Re}(\omega), \textrm{Im}(\omega))$. Denote by $\Gamma_{\perp}(X)$ the restriction of $\Gamma(X)$ to $H^1(X,\mathbb{R})^{\perp}$ and let $\overline{\Gamma_{\perp}(X)}$ its Zariski closure. 

We use Proposition \ref{p.HT-mon-invariance} to show that a translation surface has few Hodge-Teichm\"uller planes when $\overline{\Gamma_{\perp}(X)}$ is large: 

\begin{proposition}\label{p.few-HT-criterion} Let $X=(M,\omega)$ be a translation surface of genus $g\geq 3$ such that $\overline{\Gamma_{\perp}(X)} = \textrm{Sp}(H^1(X,\mathbb{R})^{\perp})$. Then, the sole Hodge-Teichm\"uller plane of $X$ is its tautological plane. 
\end{proposition}

\begin{proof} By contradiction, suppose that $X$ has a Hodge-Teichm\"uller plane $P_0$ which is not the tautological plane. Since the tautological plane is precisely the kernel of the natural projection $\pi: H^1(X,\mathbb{R})\to H^1(X,\mathbb{R})^{\perp}$, we have that $P:=\pi(P_0)$ is dimension $\geq 1$ and, moreover, the complexification $P_{\mathbb{C}}:=P\otimes\mathbb{C}$ intersects $H^{1,0}(X)$ or $H^{0,1}(X)$ (because the complexification of $\pi$ respects $H^{1,0}(X)$ and $H^{0,1}(X)$). 

Taking into account that $P$ is a real subspace and $H^{0,1}(X)$ is the complex conjugate of $H^{1,0}(X)$, we have that $P_{\mathbb{C}}$ intersects both $H^{1,0}(X)$ and $H^{0,1}(X)$ and, hence, $P$ has dimension two. Note that the same argument applies to $hP$ for all $h\in SL(2,\mathbb{R})$. This means that $P=\pi(P_0)$ is a Hodge-Teichm\"uller plane whenever $P_0$ is a non-tautological Hodge-Teichm\"uller plane. 

By Proposition \ref{p.HT-mon-invariance}, our hypothesis $\overline{\Gamma_{\perp}(X)} = \textrm{Sp}(H^1(X,\mathbb{R}))^{\perp}$ implies that $\gamma(P)$ is a Hodge-Teichm\"uller plane for all $\gamma\in \textrm{Sp}(H^1(X,\mathbb{R})^{\perp})$. 

This is a contradiction because $\textrm{Sp}(H^1(X,\mathbb{R})^{\perp})$ acts transitively on the set of symplectic planes in $H^1(X,\mathbb{R})^{\perp}$, but there are\footnote{For instance, the set of symplectic planes is open in the Grassmannian of planes while the set of planes whose complexification intersects $H^{1,0}$ has positive codimension when $g\geq 3$.} symplectic planes which are not Hodge-Teichm\"uller when $g\geq 3$. 
\end{proof}

This proposition allows us to establish some low-genus cases of Theorem \ref{t.MW-C}: 

\begin{proposition}\label{p.MW-C-g3} Let $\mathcal{C}$ be a connected component of $\mathcal{H}(4)$. Then, there exists a square-tiled surface $M_{\mathcal{C}} \in \mathcal{C}$ such that $\overline{\Gamma_{\perp}(M_{\mathcal{C}})}\simeq \textrm{Sp}(4,\mathbb{R})$. In particular, $M_{\mathcal{C}}$ has only one Hodge-Teichm\"uller plane. 
\end{proposition}

\begin{proof}
The minimal stratum $\mathcal{H}(4)$ of the moduli space of translation surfaces of genus $3$ has two connected components $\mathcal{H}(4)^{\textrm{hyp}}$ and $\mathcal{H}(4)^{\textrm{odd}}$. As it is explained in \cite{KZ}, these connected components are distinguished by the \emph{parity of the spin structure}: $M\in\mathcal{H}(4)^{\textrm{hyp}}$, resp. $\mathcal{H}(4)^{\textrm{odd}}$, if and only if $\Phi(M)=0$, resp. $1$, where $\Phi(M)\in\mathbb{Z}/2\mathbb{Z}$ is the so-called \emph{Arf invariant}\footnote{Recall that if $\{\alpha_i,\beta_i: i=1,\dots, g\}$ is a canonical symplectic basis on a genus $g$ translation surface $(M,\omega)$, then $\Phi(M):=\sum\limits_{i=1}^g(\textrm{ind}_{\omega}(\alpha_i)+1)(\textrm{ind}_{\omega}(\beta_i)+1)$ where $\textrm{ind}_{\omega}(\gamma)$ is the degree of the Gauss map associated to the tangents of a curve $\gamma$ not intersecting the set $\textrm{div}(\omega)$ of zeroes of $\omega$.}/parity of spin structure. 

From a direct inspection of the definitions, one can check that:
\begin{itemize} 
\item the square-tiled surface $M_{\ast}$ associated to the permutations $h_{\ast}=(1)(2,3)(4,5,6)$, $v_{\ast} = (1,4,2)(3,5)(6)$ belongs to $\mathcal{H}(4)^{\textrm{odd}}$, and 
\item the square-tiled surface $M_{\ast\ast}$ associated to the permutations $h_{\ast\ast}=(1)(2,3)(4,5,6)$, $v_{\ast\ast} = (1,2)(3,4)(5)(6)$ belongs to $\mathcal{H}(4)^{\textrm{hyp}}$.
\end{itemize}

We affirm that $\overline{\Gamma_{\perp}(M_{\ast})} \simeq \overline{\Gamma_{\perp}(M_{\ast\ast})} \simeq \textrm{Sp}(4,\mathbb{R})$. For the sake of exposition, we treat only the case of $M_{\ast}$ (while referring to \cite[Lemmas 4.6 and 4.7]{MW} for the case of $M_{\ast\ast}$). 

By Poincar\'e duality, our task is equivalent to show that $\textrm{Aff}(M_{\ast})$ acts on the annihilator $H_1^{\perp}(M_{\ast},\mathbb{R})\subset H_1(M,\mathbb{R})$ of the tautological plane in $H^1(M_{\ast},\mathbb{R})$ through a Zariski dense subgroup of $\textrm{Sp}(H_1^{\perp}(M_{\ast},\mathbb{R}))$. In this direction, we shall compute the action of some elements of $\textrm{Aff}(M_{\ast})$ and we will prove that they generate a Zariski dense group. 

Note that $M_{\ast}$ decomposes into three horizontal, resp. vertical, cylinders whose waist curves $\sigma_0$, $\sigma_1$, $\sigma_2$, resp. $\zeta_0$, $\zeta_1$, $\zeta_2$, have lengths $1$, $2$ and $3$. Also, $M_{\ast}$ decomposes into two cylinders in the slope $1$ direction whose waist curves $\delta_1$ and $\delta_2$ are given by the property that $\delta_1$ intersects $\sigma_0$ and $\delta_2$ intersects $\zeta_2$. See Figure \ref{f.M-ast}. 

\begin{figure}[h!]
\begin{center}
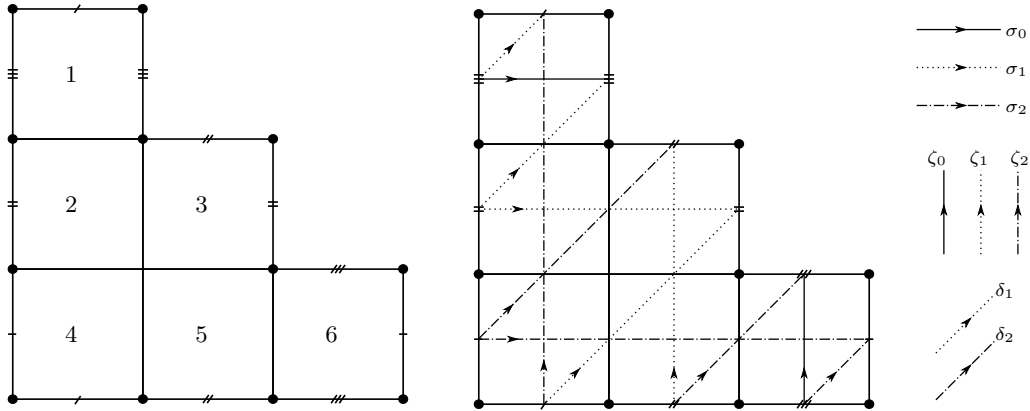
\end{center}
\caption{The geometry of the translation surface $M_{\ast}$.}\label{f.M-ast}
\end{figure}
Let us consider Dehn multitwists $A, B, C\in\textrm{Aff}(M_{\ast})$ in the horizontal, vertical and slope $1$ directions with linear parts 
$$dA=\left(\begin{array}{cc} 1 & 6 \\ 0 & 1 \end{array}\right), \quad dB=\left(\begin{array}{cc} 1 & 0 \\ 6 & 1 \end{array}\right), \quad dC=\left(\begin{array}{cc} -2 & 3 \\ -3 & 4 \end{array}\right)$$ 

It is not hard to see that the actions of $A$, $B$ and $C$ on $H_1(M_{\ast},\mathbb{R})$ are given by: 
$$A(\sigma_i) = \sigma_i \,\,\forall\,i = 1, 2, 3, \quad A(\zeta_0)=\zeta_0+2\sigma_2, \quad A(\zeta_1)=\zeta_1+3\sigma_1+2\sigma_2, \quad A(\zeta_2) = \zeta_2+3\sigma_1+2\sigma_2+6\sigma_0,$$ 
$$B(\sigma_0)=\sigma_0+2\zeta_2, \quad B(\sigma_1) = \sigma_1 + 3\zeta_1 + 2\zeta_2, \quad B(\sigma_2) = \sigma_2 + 3\zeta_1+2\zeta_2+6\zeta_0, \quad B(\zeta_i)=\zeta_i \,\,\forall\,i = 1, 2, 3,$$
$$C(\sigma_0) = \sigma_0-\delta_1 \quad C(\sigma_1) = \sigma_1-\delta_1-\delta_2 \quad 
C(\sigma_2) = \sigma_2-\delta_1-2\delta_2,$$
$$C(\zeta_0)=\zeta_0+\delta_2 \quad C(\zeta_1)=\zeta_1+\delta_1+\delta_2 \quad C(\zeta_2)=\zeta_2+2\delta_1+\delta_2$$

Thus, we can get matrices for the actions of $A$, $B$ and $C$ on $H_1^{\perp}(M_{\ast},\mathbb{R})$ after we fix a basis of this vector space. In this direction, we observe that $H_1^{\perp}(M_{\ast},\mathbb{R})$ is the subspace of $H_1(M_{\ast},\mathbb{R})$ consisting of cycles with trivial intersection with $\sigma:=\sigma_0+\sigma_1+\sigma_2$ and $\zeta:=\zeta_0+\zeta_1+\zeta_2$. In particular, the cycles $\overline{\sigma_1}:=\sigma_1-2\sigma_0$, $\overline{\sigma_2}:=\sigma_2-3\sigma_0$, $\overline{\zeta_1}:=\zeta_1-2\zeta_0$, $\overline{\zeta_2}:=\zeta_2-3\zeta_0$ form a basis of $H_1^{\perp}(M_{\ast},\mathbb{R})$ because these cycles are linearly independent, $H_1^{\perp}(M_{\ast},\mathbb{R})$ has dimension $2g-2$ and $M_{\ast}$ has genus $g=3$. 

Since $\delta_1=\sigma_1+\sigma_0+\zeta_2$ and $\delta_2=\sigma_2+\zeta_1+\zeta_0$, we conclude that the matrices $A_{\ast}$, $B_{\ast}$, $C_{\ast}$ of $A$, $B$, $C$ with respect to the basis $\{\overline{\sigma_1}, \overline{\sigma_2}, \overline{\zeta_1}, \overline{\zeta_2}\}$ of $H_1^{\perp}(M_{\ast},\mathbb{R})$ are: 
$$A_{\ast} = \left(\begin{array}{cccc} 1 & 0 & 3 & 3 \\ 0 & 1 & -2 & -4 \\ 0 & 0 & 1 & 0 \\ 0 & 0 & 0 & 1 \end{array}\right), \quad 
B_{\ast}=\left(\begin{array}{cccc} 1 & 0 & 0 & 0 \\ 0 & 1 & 0 & 0 \\ 3 & 3 & 1 & 0 \\ -2 & -4 & 0 & 1 \end{array}\right), \quad 
C_{\ast} = \left(\begin{array}{cccc} 2 & 2 & 1 & 2 \\ -1 & -1 & -1 & -2 \\ -1 & -2 & 0 & -2 \\ 1 & 2 & 1 & 3 \end{array}\right)$$

Once we computed these matrices, it suffices to check the Zariski closure $G$ of the group $\langle A_{\ast}, B_{\ast}, C_{\ast} \rangle$ is $\textrm{Sp}(4,\mathbb{R})$. As it turns out, this fact can be proved as follows. The Lie algebra $\mathfrak{g}$ of $G$ contains 
$$\log A_{\ast}=\left(\begin{array}{cccc} 0 & 0 & 3 & 3 \\ 0 & 0 & -2 & -4 \\ 0 & 0 & 0 & 0 \\ 0 & 0 & 0 & 0 \end{array}\right)\in\mathfrak{g}$$ and, \emph{a fortiori}, $\mathfrak{g}$ also contains the nine conjugates of $\log A_{\ast}$ by the matrices 
$$B_{\ast}, \quad B_{\ast}^2, \quad A_{\ast} B_{\ast}, \quad A_{\ast}^2 B_{\ast}, \quad 
B_{\ast} A_{\ast} B_{\ast}, \quad C_{\ast}, \quad C_{\ast}^2, \quad A_{\ast} C_{\ast}, \quad 
B_{\ast} C_{\ast}$$
On the other hand, a direct computation reveals that $\log A_{\ast}$ and these nine conjugates are linearly independent. Since $\textrm{Sp}(4,\mathbb{R})$ has dimension 10, this shows that $G=\textrm{Sp}(4,\mathbb{R})$. 

In summary, we showed that $\overline{\Gamma_{\perp}(M_{\ast})} = \textrm{Sp}(H^1(M_{\ast},\mathbb{R})^{\perp})\simeq \textrm{Sp}(4,\mathbb{R})$. In particular, $M_{\ast}$ has only one Hodge-Teichm\"uller plane (by Proposition \ref{p.few-HT-criterion}). This completes the proof of the proposition. 
\end{proof}

At this point, the idea of the proof of Theorem \ref{t.MW-C} can be explained as follows. If $\mathcal{C}$ were a connected component of a stratum of the moduli space of translation surfaces of genus $g\geq 3$ such that all $M\in\mathcal{C}$ has $(g-1)$ Hodge-Teichm\"uller planes, then all translation surfaces in all ``adjacent'' strata to $\mathcal{C}$ would have ``many'' Hodge-Teichm\"uller planes thanks to a ``continuity argument''. However, this is impossible because $\mathcal{C}$ is ``adjacent'' to a connected component of $\mathcal{H}(4)$, but  Proposition \ref{p.MW-C-g3} says that all connected components of $\mathcal{H}(4)$ contain some translation surfaces with few Hodge-Teichm\"uller planes.  

More concretely, we formalize this idea in \cite{MW} in two steps. First, we use an \emph{elementary} continuity argument (similar to Proposition \ref{p.HT-continuity}) and the notion of \emph{adjacency}\footnote{More precisely, we need the fact stated in \cite[Corollary 4]{KZ} that the boundary of any connected component $\mathcal{C}$ of any stratum of the moduli space $\mathcal{H}_g$ of translation surfaces of genus $g$ contains a connected component $\mathcal{C}'$ of the minimal stratum $\mathcal{H}(2g-2)$.} of strata from \cite{KZ} to establish the following result (cf. \cite[Proposition 5.1]{MW}): 

\begin{proposition}\label{p.MW-prop-5-1} Let $\mathcal{C}$ be a connected component of $\mathcal{H}(k_1,\dots,k_s)$, $s>1$, $\sum\limits_{l=1}^s k_l = 2g-2$. Suppose that all translation surfaces in $\mathcal{C}$ possess $m\geq 1$ symplectically orthogonal Hodge-Teichm\"uller planes. Then, there exists a connected component $\mathcal{C}'$ of the minimal stratum $\mathcal{H}(2g-2)$ such that all translation surfaces in $\mathcal{C}'$ also possess $m\geq 1$ symplectically orthogonal Hodge-Teichm\"uller planes. 
\end{proposition}

Secondly, we use a \emph{sophisticated} version of the previous continuity argument to move Hodge-Teichm\"uller planes across minimal strata (cf. \cite[Proposition 5.3]{MW}): 

\begin{proposition}\label{p.MW-prop-5-3} Let $\mathcal{C}$ be a connected component of $\mathcal{H}(2g-2)$. Suppose that every translation surface in $\mathcal{C}$ has $m\geq 1$ symplectically orthogonal Hodge-Teichm\"uller planes. Then, there exists a connected component $\mathcal{C}'$ of $\mathcal{H}(2g-4)$ such that every translation surface in $\mathcal{C}'$ has $(m-1)$ symplectically orthogonal Hodge-Teichm\"uller planes.
\end{proposition}

The basic idea behind the proof of this proposition is not difficult, but a complete argument (provided in Sections 5 and 6 of \cite{MW}) is somewhat technical partly because it requires a discussion of the so-called \emph{Deligne-Mumford compactification}. For this reason, we will content ourselves with the outline of proof of this proposition.

\begin{proof}[Sketch of proof of Proposition \ref{p.MW-prop-5-3}] In their study of connected components of strata of the moduli space of translation surfaces, Kontsevich and Zorich \cite{KZ} introduced a local surgery of Abelian differentials called \emph{bubbling a handle}. This surgery increases the genus by one and it is defined in two steps, namely \emph{splitting a zero} and \emph{gluing a torus}. 

Roughly speaking, one splits a zero of order $m$ by a certain (local) cutting and pasting operation which produces a pair of zeroes of orders $m'$ and $m''$ with $m'+m''=m$ joined by a saddle connection with holonomy $v\in\mathbb{R}^2$. After splitting a zero, one can cut the saddle connection to obtain a slit. Then, one can \emph{bubble a handle} by gluing a cylinder/torus/handle into this slit. In what follows, we will be interested in gluing a \emph{square} torus/handle to the slit. See Figures \ref{f.split-zero} and \ref{f.square-slit} for an illustration of these procedures. 

\begin{figure}[htb!]
\includegraphics[scale=0.5]{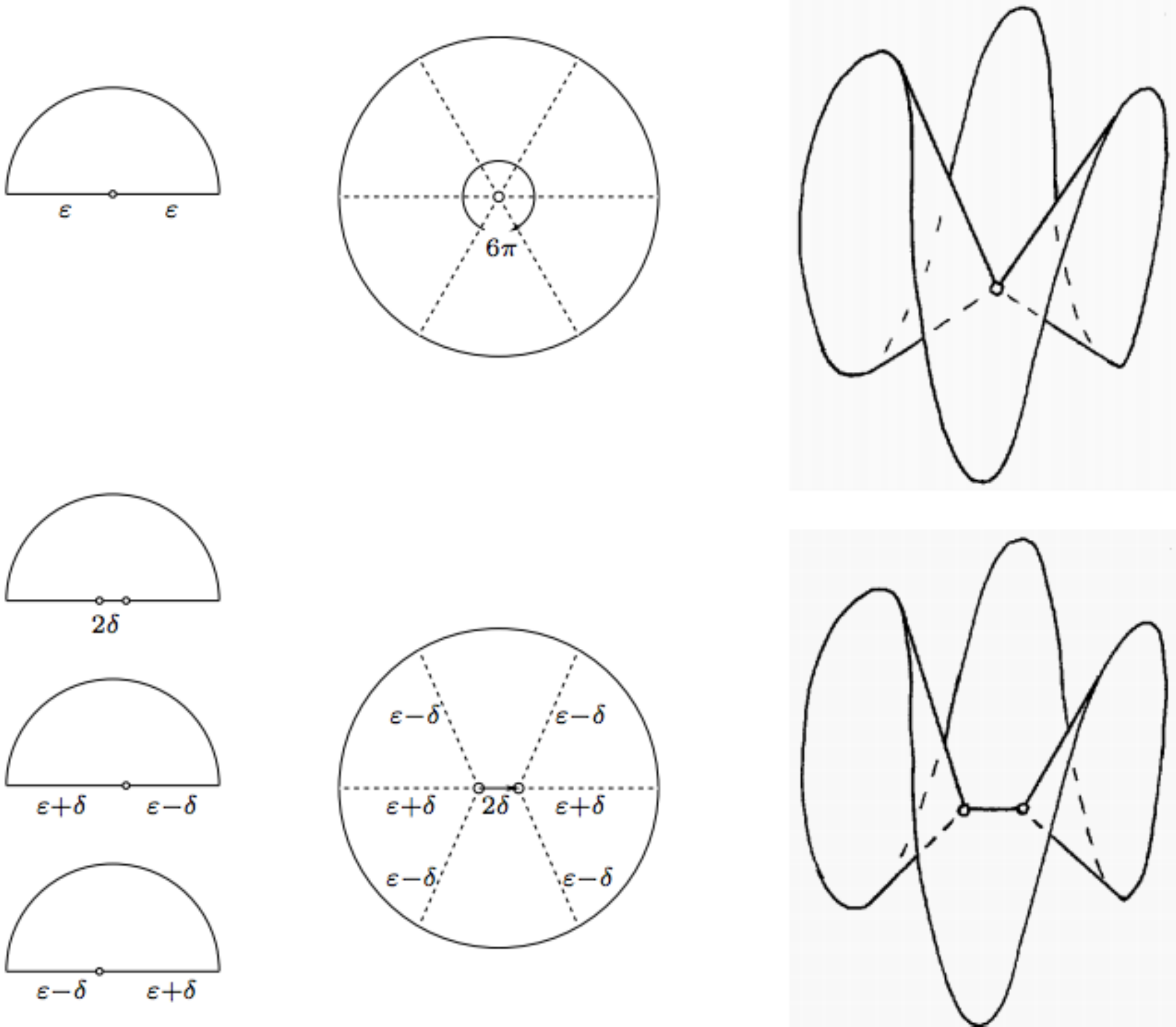}
\caption{Splitting a zero of an Abelian differential (after Eskin-Masur-Zorich).}\label{f.split-zero}
\end{figure}

\begin{figure}[htb!]
\includegraphics[scale=0.5]{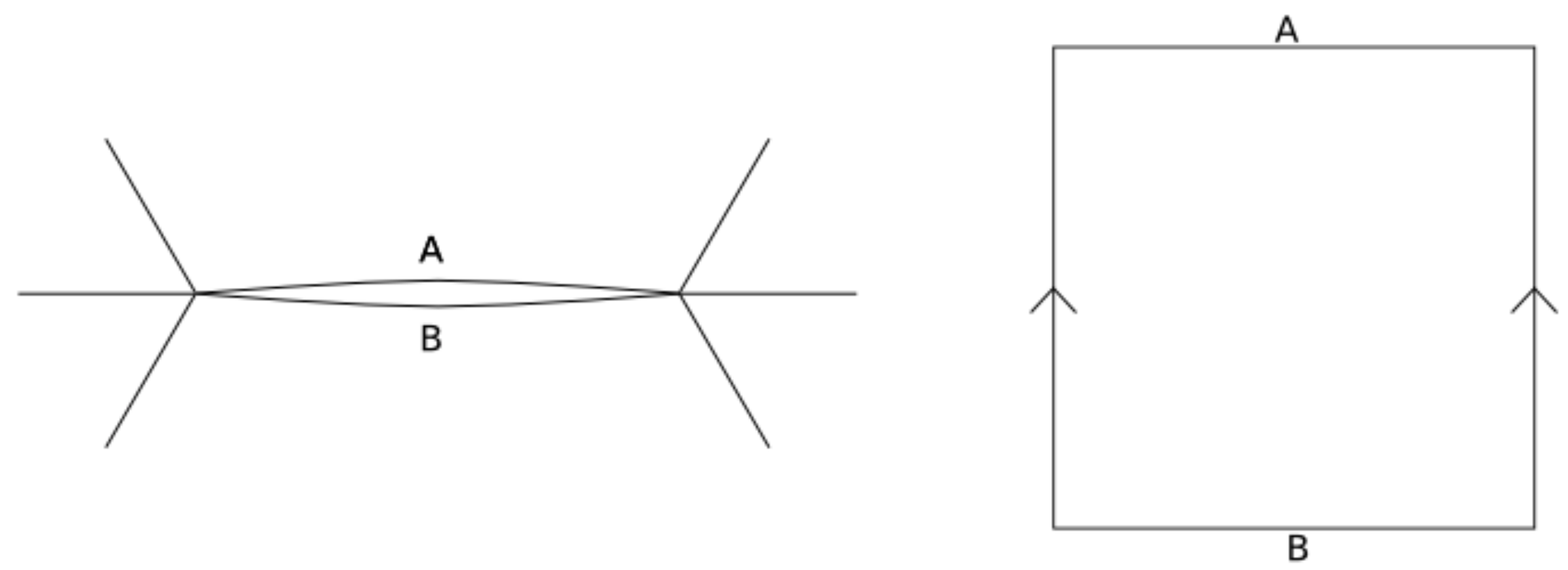}
\caption{Bubbling a square handle.}\label{f.square-slit}
\end{figure}

The operation of bubbling a handle allows to understand the \emph{adjacencies} of strata. For example, we claim that if $\mathcal{C}$ is a connected component of $\mathcal{H}(2g-2)$, then there exists a connected component $\mathcal{C}'\subset \mathcal{H}(2g-4)$ such that we can bubble a handle on every translation surface in $\mathcal{C}'$ in order to obtain a translation surface in $\mathcal{C}$ (compare with \cite[Lemma 14]{KZ}). 

In fact, Lemma 20 in \cite{KZ} says that $\mathcal{C}$ contains a translation surface $X$ given by the suspension of an interval exchange transformation associated to a \emph{good} \emph{standard} permutation $\pi$, i.e., a permutation of the form $\pi=\left(\begin{array}{ccc} A & \pi'_t & Z \\ Z & \pi'_b & A \end{array}\right)$ such that the permutation $\pi'=\left(\begin{array}{c} \pi'_t \\ \pi'_b \end{array}\right)$ derived from $\pi$ by erasing the letters $A$ and $Z$ is \emph{irreducible}. Concretely, the irreducibility of $\pi'$ means that we can use it to build (through suspension) a translation surface $X'$ belonging to $\mathcal{H}(2g-4)$ (in this context). See Figure \ref{f.MW-pre-bdry} for an illustration of $X$ together with a ``copy'' of $X'$ inside it. 

\begin{figure}[h!]
\begin{center}
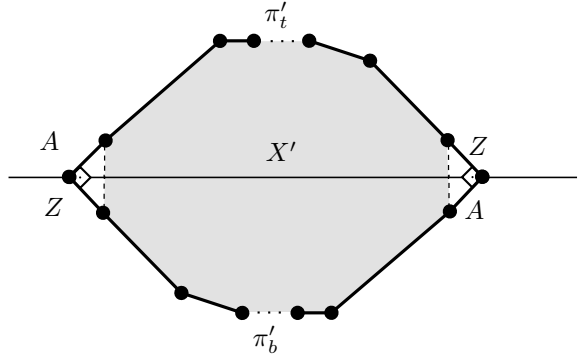
\end{center}
\caption{The surface $X$ and the subsurface $X'$.}\label{f.MW-pre-bdry}
\end{figure}

Denote by $\mathcal{C}'$ the connected component of $\mathcal{H}(2g-4)$ containing $X'$. Observe that, by definition, $X\in\mathcal{C}$ is obtained from $X'\in\mathcal{C}'$ by bubbling a square handle. Since the operation of bubbling a handle can be performed in a \emph{continuous} way (cf. \cite[Lemma 10]{KZ}), we obtain the desired claim that one can bubble a (square) handle in any $M'\in\mathcal{C}'$ to obtain a translation surface in $M\in\mathcal{C}$. 

By exploiting this relation between $\mathcal{C}\subset\mathcal{H}(2g-2)$ and $\mathcal{C}'\subset \mathcal{H}(2g-4)$ through the operation of bubbling a square handle, one can complete the proof of the proposition along the following lines (cf. Propositions 5.5 and 5.6 in \cite{MW}). Suppose that every $M\in\mathcal{C}$ has $m$ symplectically orthogonal Hodge-Teichm\"uller planes. Given $M'\in\mathcal{C}'$, let us bubble a square handle to obtain $M\in\mathcal{C}$. By degenerating the square handle (i.e., by letting the length of the sizes of the square tend to zero) in $M$ to reach $M'$, we have\footnote{Here, the fact that the angles of the torus do \emph{not} degenerate (because it is always a square) is crucial.} that $M$ converges to a \emph{stable nodal curve}\footnote{I.e., a point in the boundary of Deligne-Mumford compactification of the moduli space of curves.} of the form depicted in Figure \ref{f.MW-bdry}. 

\begin{figure}[htb!]
\includegraphics[scale=0.5]{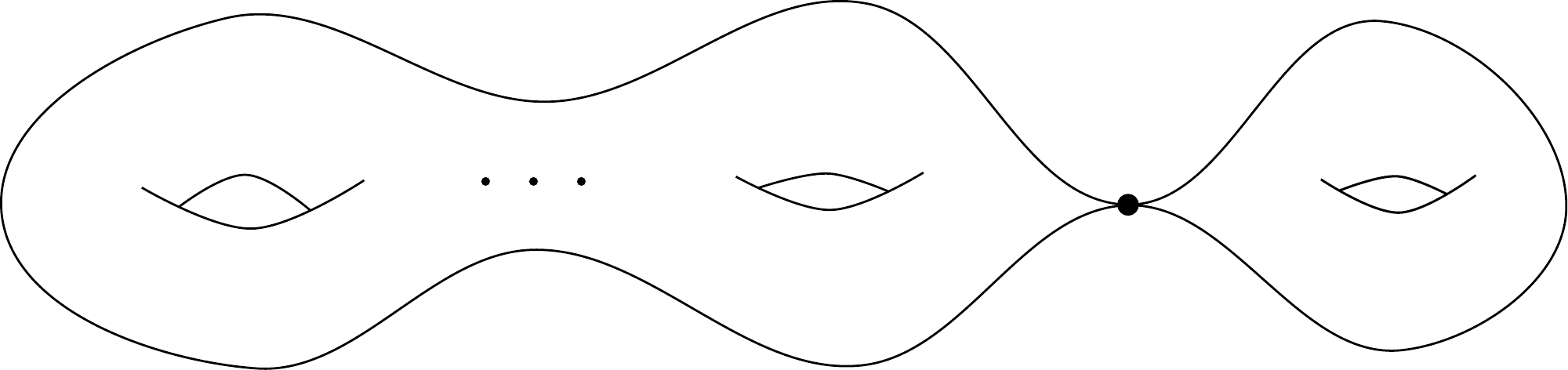}
\caption{The surface $M'$ (on the left) is attached to a torus (on the right) at a node.}\label{f.MW-bdry}
\end{figure}

In this process, one of the $m$ Hodge-Teichm\"uller planes of $M$ might be ``lost'' (in case it converges to the plane $H^1(Y,\mathbb{R})$ associated to the degenerating square handle), but at least $(m-1)$ Hodge-Teichm\"uller planes of $M$ will converge to $(m-1)$ (symplectically orthogonal) Hodge-Teichm\"uller planes of $M'$. This completes our sketch of proof of the proposition. 
\end{proof}

At this point, Theorem \ref{t.MW-C} is a simple consequence of the following argument: by contradiction, suppose that $\mathcal{C}$ is a connected component of stratum of the moduli space of translation surfaces of genus $g\geq 3$ such that every $M\in\mathcal{C}$ has $(g-1)$ symplectically orthogonal Hodge-Teichm\"uller planes. By Proposition \ref{p.MW-prop-5-1}, there would be a connected component $\mathcal{C}'$ of $\mathcal{H}(2g-2)$ such that every $M\in\mathcal{C}'$ would also possess $(g-1)$ symplectically orthogonal Hodge-Teichm\"uller planes. By applying $(g-3)$ times Proposition \ref{p.MW-prop-5-3}, we would find a connected component $\mathcal{C}''$ of 
$\mathcal{H}(4)$ such that \emph{every} $M\in\mathcal{C}''$ would have $2$ Hodge-Teichm\"uller planes, a contradiction with Proposition \ref{p.MW-C-g3}. 

Closing this subsection, let us give an \emph{elementary}\footnote{Here, by ``elementary'' we mean that, instead of proving the existence of $M_{\mathcal{C}}$ by indirect methods (including the use of Deligne-Mumford compactification), we will build $M_{\mathcal{C}}$ directly for certain $\mathcal{C}$'s.} proof of the following \emph{particular}\footnote{Our discussion of these particular cases follows an argument described in a blog post entitled ``\emph{Hodge-Teichm\"uller planes and finiteness results for Teichm\"uller curves}'' in my mathematical blog \cite{DM}.} cases of  Theorem \ref{t.MW-C}: 

\begin{proposition}\label{p.partial-MW-C-cover} For each $d\geq 1$ odd, there exists a translation surface $M_{\ast}(d)$ in the odd connected component $\mathcal{H}(5d-1)^{\textrm{odd}}$ of the minimal stratum $\mathcal{H}(5d-1)$ which does not possess $(5d-1)/2$ symplectically orthogonal Hodge-Teichm\"uller planes.
\end{proposition} 

\begin{proof} We affirm that a square-tiled surface $N$ of genus $g$ covering the square-tiled surface $M_{\ast}$ constructed in the proof of Proposition \ref{p.MW-C-g3} can not have $(g-1)$ symplectically orthogonal Hodge-Teichm\"uller planes. In fact, if we denote by $p:N\to M_{\ast}$ the translation covering defining the translation structure of $N$, then $H^1(N,\mathbb{R}) = E\oplus H^1(N,\mathbb{R})^{(p)}$
where $H_1(N,\mathbb{R})^{(p)}$ consists of all cycles projecting to zero under $p$ and $E\simeq H_1(M_{\ast},\mathbb{R})$ is the symplectic orthogonal of $H_1(N,\mathbb{R})^{(p)}$. By appplying an argument similar to the proof of Proposition \ref{p.few-HT-criterion}, one can check that if $N$ had $(g-1)$ symplectically orthogonal Hodge-Teichm\"uller planes, then two of them would be contained in $E\simeq H^1(M_{\ast},\mathbb{R})$ and, \emph{a fortiori}, $M_{\ast}$ would have two Hodge-Teichm\"uller planes, a contradiction with Proposition \ref{p.MW-C-g3}. 

In view of the discussion in the previous paragraph, the proof of the proposition will be complete once we construct a degree $d$ branched cover $M_{\ast}(d)\in\mathcal{H}(5d-1)^{\textrm{odd}}$ of $M_{\ast}$. 

For technical reasons, it is desirable to ``simplify'' the geometry of $M_{\ast}$ before studying its covers. In this direction, let us replace $M_{\ast}$ by the square-tiled surface $\overline{M}_{\ast}$ associated to the permutations $\overline{h}_{\ast}=(1,2,3,4,5,6)$, 
$\overline{v}_{\ast}=(1)(2,5,4)(3,6)$: there is no harm in doing so because $\overline{M}_{\ast}$ belongs to the $SL(2,\mathbb{Z})$-\emph{orbit} of $M_{\ast}$ (indeed, $\overline{M}_{\ast}=J\cdot T^2(M_{\ast})$ where 
$J=\left(\begin{array}{cc} 0 & -1 \\ 1 & 0 \end{array}\right)$ and $T=\left(\begin{array}{cc} 1 & 1 \\ 0 & 1 \end{array}\right)$). 

\begin{figure}[h!]
\begin{center}
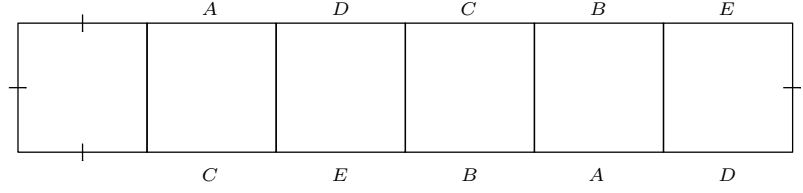
\end{center}
\caption{The one-cylinder surface $\overline{M}_{\ast}$ in $SL(2,\mathbb{Z})\cdot M_{\ast}$.}
\end{figure}

Given $d\geq 3$ an odd integer, let $\overline{M}_{\ast}(d)$ be the square-tiled surface\footnote{The ``shape'' of the covering $\overline{M}_{\ast}(d)$ was ``guessed'' with the help of the computer program Sage. In fact, I tried a few simple-minded finite coverings of $\overline{M}_{\ast}$ (including $\overline{M}_{\ast}(d)$ for $d=3, 5, \dots, 13$) and I asked Sage to determine their connected components. Then, once we got the ``correct'' connected components (in minimal strata), I looked at the permutations corresponding to these square-tiled surfaces and I found the ``partner'' leading to the expressions for the permutations $\overline{h}_{\ast}(d)$ and $\overline{v}_{\ast}(d)$.} associated to the pair of permutations
$$\overline{h}_{\ast}(d):=(1_1,2_1,3_1,4_1,5_1,6_1)(1_2,2_2,3_2,4_2,5_2,6_2,1_3,2_3,3_3,4_3,5_3,6_3)\dots$$
$$(1_{d-1},2_{d-1},3_{d-1},4_{d-1},5_{d-1},6_{d-1},1_d,2_d,3_d,4_d,5_d,6_d),$$
$$\overline{v}_{\ast}(d):=(1_1,1_2,\dots,1_d)(2_1,5_1,4_1)(3_1,6_1)\dots(2_d,5_d,4_d)(3_d,6_d)$$
By definition, $\overline{M}_{\ast}(d)$ is a degree $d$ cover of $\overline{M}_{\ast}$ belonging to the stratum $\mathcal{H}(5d-1)$. Therefore, it remains only to verify that $\overline{M}_{\ast}(d)\in \mathcal{H}(5d-1)^{\textrm{odd}}$ in order to complete the proof of the proposition. 

Since $\overline{M}_{\ast}(d)$ is not hyperelliptic, we have that $\overline{M}_{\ast}(d)\in \mathcal{H}(5d-1)^{\textrm{odd}}$ if and only if $\Phi(\overline{M}_{\ast}(d))=1$, where the parity $\Phi(\overline{M}_{\ast}(d))\in\mathbb{Z}_2$ of the spin structure of $\overline{M}_{\ast}(d)$ is defined as follows (see \cite{KZ} for more details). Let $a_1, b_1, a_2, b_2, \dots, a_g, b_g$ (where $g=(5d+1)/2$) be a canonical symplectic basis of $H_1(\overline{M}_{\ast}(d),\mathbb{Z}_2)$. Then, 
$$\Phi(\overline{M}_{\ast}(d))=\sum\limits_{i=1}^{(5d+1)/2} \phi(a_i)\phi(b_i) \textrm{ (mod } 2),$$
where $\phi:H_1(\overline{M}_{\ast}(d),\mathbb{Z}_2)\to\mathbb{Z}_2$ satisfies the following properties:
\begin{itemize}
\item $\phi(\gamma)=\textrm{ind}(\gamma)+1$ whenever $\gamma$ is a simple smooth closed curve in $\overline{M}_{\ast}(d)$ not passing through singular points whose Gauss map has degree $\textrm{ind}(\gamma)$;
\item $\phi$ is a quadratic form representing the intersection form $(.,.)$ in the sense that $\phi(\alpha+\beta)=\phi(\alpha)+\phi(\beta)+(\alpha,\beta)$.
\end{itemize} 

In other words, we need to build a canonical symplectic basis of $H_1(\overline{M}_{\ast}(d),\mathbb{R})$ in order to compute $\Phi(\overline{M}_{\ast}(d))$. For this sake, we begin by fixing a basis of $H_1(\overline{M}_{\ast}(d),\mathbb{R})$. We think of it as $(d+1)/2$ horizontal cylinders determined by the permutation $\overline{h}_{\ast}(d)$ whose top and bottom boundaries are glued accordingly to the permutation $\overline{v}_{\ast}(d)$. Using this geometrical representation, we define the following cycles in $H_1(\overline{M}_{\ast}(d),\mathbb{Z})$ (see Figure \ref{f.M*d-cycles}). 
\begin{itemize}
\item for each $i=1,\dots,d$, let $c_{23}^{(i)}$, $c_4^{(i)}$, $c_5^{(i)}$ and $c_6^{(i)}$ be the homology classes of the vertical cycles within the horizontal cylinders connecting the middle of the bottom side of the squares $5_i$, $2_i$, $4_i$ and $3_i$ (resp.) to the top side of the squares $2_i$, $4_i$, $5_i$ and $6_i$ (resp.);
\item for each $j=2l$, $l=1,\dots, (d-1)/2$, let $c_1^{(j)}$ be the homology classes of the vertical cycles within the horizontal cylinders connecting the middles of the bottom side of the square $1_{j+1}$ to the top side of the squares $1_{j}$; 
\item let $c_1^{(1)}$ be the homology class of the concatenation of the vertical cycles within the horizontal cylinders connecting the middles of the bottom side of the square $1_{2l}$ to the top side of the square $1_{2l+1}$ for $l=1,\dots, (d-1)/2$, the bottom side of the square $1_2$ to the left side of the square $1_2$, and the right side of the square $6_d$ to the top side of the square $1_d$;
\item $\sigma_1$ is the horizontal cycle connecting the left vertical side of the square $1_1$ to the right vertical side of the square $6_1$, and, for $j=2l$, $l=1,\dots, (d-1)/2$, $\sigma_{j}$ is the horizontal cycle connecting the left vertical side of the square $1_{j}$ to the right vertical side of the square $6_{j+1}$.  
\end{itemize}

\begin{figure}[htb!]
\begin{center}
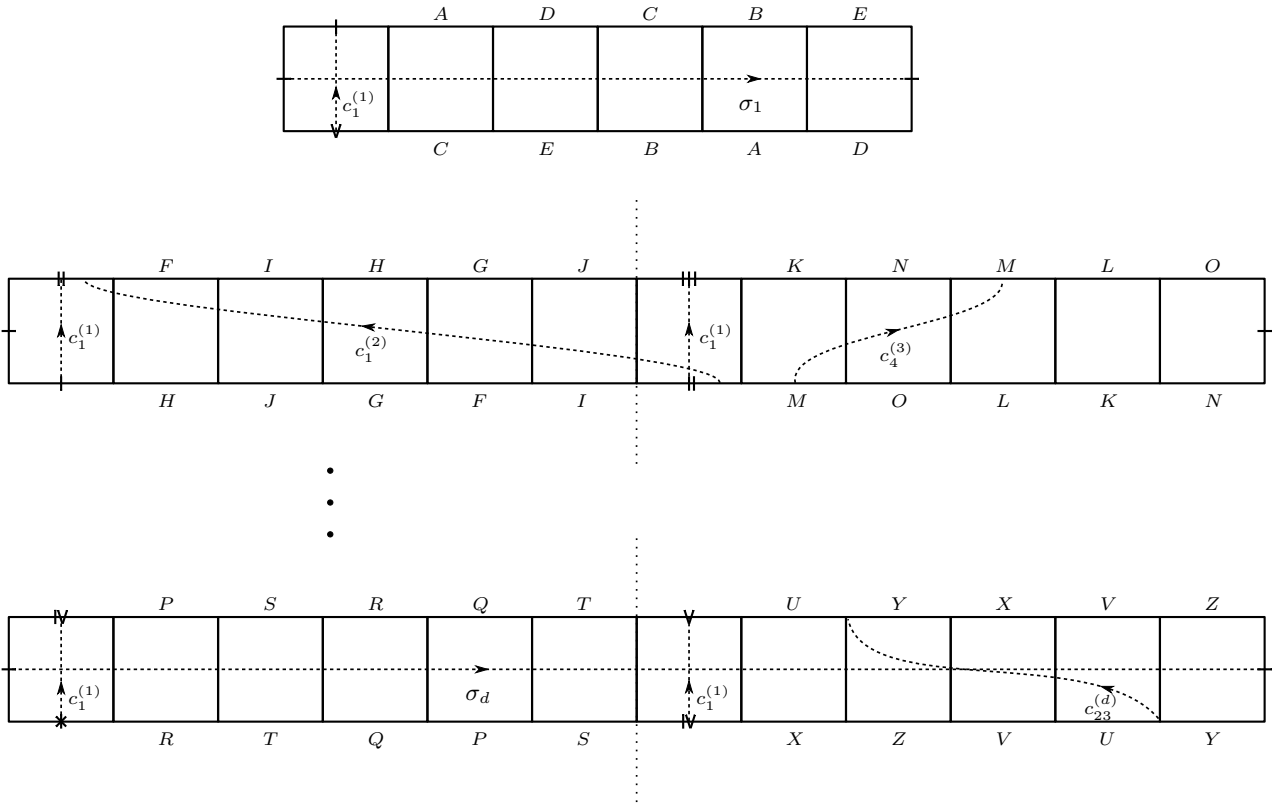
\end{center}
\caption{Some homology cycles in $\overline{M}_{\ast}(d)$}\label{f.M*d-cycles}
\end{figure}

It is not hard to check that these $5d+1$ cycles form a basis of $H_1(\overline{M}_{\ast}(d),\mathbb{Z})$. From this basis, we can produce a canonical basis of homology cycles of
$\overline{M}_{\ast}(d)$ using the orthogonalization procedure\footnote{A Gram-Schmidt orthogonalization process to produce canonical basis of homology modulo $2$.} described in \cite[Appendix C]{Z08}. More precisely, we start with the cycles $a_1:=\sigma_2$ and $b_1:=c_1^{(2)}$. Then, by induction, we use the cycles $a_i$, $b_i$ to successively render the cycles 
$$c_{23}^{(2)}, c_4^{(2)}, c_5^{(2)}, c_6^{(2)}, \dots, \sigma_1, c_1^{(1)}, c_{23}^{(1)}, c_4^{(1)}, c_5^{(1)}, c_6^{(1)}$$ into a canonical basis $a_1, b_1, a_2, b_2, \dots, a_g, b_g$ (where $g=(5d+1)/2$). 

Using these cycles, we are ready to compute $\Phi(\overline{M}_{\ast}(d))$ inductively. More concretely, we affirm that $\Phi(\overline{M}_{\ast}(d+2))=\Phi(\overline{M}_{\ast}(d))$ for each $d\geq 1$ odd. Indeed, we note that the $10$ cycles 
$$\sigma_{d+1}, c_1^{(d+1)}, c_{23}^{(d+1)}, c_4^{(d+1)}, c_5^{(d+1)}, c_6^{(d+1)}, c_{23}^{(d+2)}, c_4^{(d+2)}, c_{5}^{(d+2)}, c_6^{(d+2)}$$
of $\overline{M}_{\ast}(d+2)$ do not intersect the other $5d-1$ cycles of $\overline{M}_{\ast}(d+2)$ defined above except for $c_1^{(1)}$, and, moreover, they have trivial intersection with the cycle 
\begin{eqnarray*}
\widetilde{c}_{1}^{(1)}&:=& c_1^{(1)}-c_1^{(d+1)}+c_{23}^{(d+1)}-c_4^{(d+1)}+2c_5^{(d+1)}-2c_6^{(d+1)} \\ &+& c_4^{(d+2)}-c_{23}^{(d+2)}+c_6^{(d+2)}-c_5^{(d+2)}
\end{eqnarray*}
satisfying $\phi(\widetilde{c}_{1}^{(1)})=\phi(c_1^{(1)})=1$. Thus, by replacing $c_1^{(1)}$ by $\widetilde{c}_1^{(1)}$ and by applying the orthogonalization procedure to these two sets of $10$ and $5d-1$ cycles in an independent way, we deduce that 
$\Phi(\overline{M}_{\ast}(d+2))$ differs from $\Phi(\overline{M}_{\ast}(d))$ by a term of the form 
$$\Phi(\overline{M}_{\ast}(d+2))-\Phi(\overline{M}_{\ast}(d))=\sum\limits_{i=1}^{5}\phi(a_i^{(d+1)})\phi(b_i^{(d+1)})$$
where $a_i^{(d+1)}, b_i^{(d+1)}$, $i=1,\dots, 5$ is an orthogonalization of the $10$ cycles 
$$\sigma_{d+1}, c_1^{(d+1)}, c_{23}^{(d+1)}, c_4^{(d+1)}, c_5^{(d+1)}, c_6^{(d+1)}, c_{23}^{(d+2)}, c_4^{(d+2)}, c_{5}^{(d+2)}, c_6^{(d+2)}.$$ 
On the other hand, by a direct computation, one can check that 
\begin{itemize}
\item $a_1^{(d+1)}=\sigma_{d+1}$, $b_1^{(d+1)}=c_{1}^{(d+1)}$, 
\item $a_2^{(d+1)}=c_4^{(d+1)}-a_1^{(d+1)}-b_1^{(d+1)}$, $b_2^{(d+1)} = c_{23}^{(d+1)}-a_1^{(d+1)}-b_1^{(d+1)}$, 
\item $a_3^{(d+1)}=c_6^{(d+1)}-a_1^{(d+1)}-b_1^{(d+1)}+a_2^{(d+1)}$, $b_3^{(d+1)}=c_5^{(d+1)}-a_1^{(d+1)}-b_1^{(d+1)}+a_2^{(d+1)}$, 
\item $a_4^{(d+1)}=c_4^{(d+2)}-b_1^{(d+1)}+b_2^{(d+1)}-a_2^{(d+1)}+2b_3^{(d+1)}-2a_3^{(d+1)}$, $b_4^{(d+1)}=c_{23}^{(d+2)}-b_1^{(d+1)}+b_2^{(d+1)}-a_2^{(d+1)}+2b_3^{(d+1)}-2a_3^{(d+1)}$, 
\item $a_5^{(d+1)}=c_5^{(d+2)}-c_4^{(d+2)}$, $b_5^{(d+1)}=c_6^{(d+2)}-c_4^{(d+2)}$
\end{itemize}
is an orthogonalization of the $10$ cycles 
$$\sigma_{d+1}, c_1^{(d+1)}, c_{23}^{(d+1)}, c_4^{(d+1)}, c_5^{(d+1)}, c_6^{(d+1)}, c_{23}^{(d+2)}, c_4^{(d+2)}, c_{5}^{(d+2)}, c_6^{(d+2)}.$$
Moreover, $\phi(a_i^{(d+1)})=\phi(b_i^{(d+1)})=1$ for $i=1,4$, and 
$\phi(a_j^{(d+1)})=\phi(b_j^{(d+1)})=0$ for $j=2,3,5$. Hence,  
$$\Phi(\overline{M}_{\ast}(d+2))-\Phi(\overline{M}_{\ast}(d))=\sum\limits_{i=1}^{5}\phi(a_i^{(d+1)})\phi(b_i^{(d+1)}) = 0,$$
as it was claimed. Therefore, $\Phi(\overline{M}_{\ast}(d))=\Phi(\overline{M}_{\ast}(1)):=\Phi(\overline{M}_{\ast})=1\in\mathbb{Z}_2$.  
\end{proof}

\newpage

\begin{centering}
\rule{\textwidth}{1.6pt}\vspace*{-\baselineskip}\vspace*{2.5pt}
\rule{\textwidth}{0.4pt}

\section{Simplicity of Lyapunov exponents of arithmetic Teichm\"uller curves}\label{s.MMY}

\rule{\textwidth}{0.4pt}\vspace*{-\baselineskip}\vspace{3.2pt}
\rule{\textwidth}{1.6pt}
\end{centering}\\

The circle of ideas developed for the study of Lyapunov exponents of the Kontsevich-Zorich cocycle over the $SL(2,\mathbb{R})$-action on moduli spaces of translation surfaces was fruitfully used in many contexts:
\begin{itemize}
\item Zorich \cite{Z97} and Forni \cite{F02} related the Lyapunov exponents of the KZ cocycle with respect to Masur-Veech measures to obtain a complete description of the deviations of ergodic averages of typical interval exchange transformations and translation flows; 
\item Avila and Forni \cite{AF} used the positivity of the second largest Lyapunov exponent of the KZ cocycle with respect to Masur-Veech measures to establish the weak mixing property for typical interval exchange transformations (not of rotation type) and translation flows (on higher genus surfaces); 
\item Delecroix, Hubert and Leli\`evre \cite{DHL} exploited the precise values of the Lyapunov exponents of the KZ cocycle over a certain closed $SL(2,\mathbb{R})$-invariant locus of translation surfaces of genus five to confirm a conjecture of Hardy and Weber on the abnormal rate of diffusion of typical trajectories in $\mathbb{Z}^2$-periodic Ehrenfest wind-tree models of Lorenz gases; 
\item Kappes and M\"oller \cite{KM} employed some invariants (inspired of the Lyapunov exponents of the KZ cocycle) the exact number (nine) of commensurability classes of the non-arithmetic lattices of $PU(1,n)$ constructed by Deligne and Mostow. 
\end{itemize} 

\subsection{Kontsevich-Zorich conjecture and Veech's question} 

In some applications of the Lyapunov exponents of the KZ cocycles, it is important to know whether they have a qualitative property called \emph{simplicity}, that is, if all of them appear with multiplicity one. For instance, the most complete picture for the deviations of ergodic averages of typical interval exchange transformations and translation flows depends on the simplicity of the Lyapunov exponents of the KZ cocycle with respect to Masur-Veech measures (see \cite{Z97} and \cite{F02}).  

In the case of Masur-Veech measures, the simplicity of the Lyapunov exponents of the KZ cocycle was conjectured by Kontsevich and Zorich after several numerical experiments. This conjecture was fully confirmed in a celebrated work of Avila and Viana \cite{AV} after an important partial result of Forni \cite{F02}. 

In the case of other $SL(2,\mathbb{R})$-invariant probability measures on moduli spaces of translation surfaces, Veech asked\footnote{Veech knew that the analog of his question for arbitrary Teichm\"uller flow invariant probability measures was false: for example, theKZ cocycle might have zero Lyapunov exponents along certain periodic orbits of the Teichm\"uller flow in $\mathcal{H}(2)$ (see \cite[Appendix B]{FM}).} if the Lyapunov exponents of the KZ cocycle are always non-zero and/or simple. This question was negatively answered by Forni and the author with two examples called \emph{Eierlegende Wollmilchsau} and \emph{Ornithorynque} (see, e.g., \cite{FM}): loosely speaking, these are examples of arithmetic Teichm\"uller curves such that the KZ cocycle over them has many zero Lyapunov exponents. 

Of course, this answer to Veech's question motivates the problem of finding criteria for the simplicity of the Lyapunov exponents of the KZ cocycle with respect to $SL(2,\mathbb{R})$-invariant probability measures in moduli spaces of translation surfaces. 

In the remainder of this section, we shall offer an answer to this problem in the context of the KZ cocycle over Teichm\"uller curves. 

\subsection{Lyapunov exponents of Teichm\"uller curves and random products of matrices}

In this subsection, we will show that the Lyapunov exponents of the KZ cocycle with respect to $SL(2,\mathbb{R})$-invariant probability measures supported on Teichm\"uller curves can be computed via random products of matrices. The relevance of this fact for our purposes is explained by the vast literature (cf. Furstenberg \cite{Fu63}, Guivarc'h-Raugi \cite{GR}, \cite{GR2},  Goldsheid-Margulis \cite{GM}, Avila-Viana \cite{AV}, etc.) on the simplicity of Lyapunov exponents of random products of matrices. 

\begin{remark} The analog of this statement for Masur-Veech measures is discussed in Avila-Viana paper \cite{AV} in their reduction of the Kontsevich-Zorich conjecture to the study of ``random'' products of matrices in Rauzy-Veech monoids. 
\end{remark}

More concretely, let $X=(M,\omega)$ be a translation surface whose $SL(2,\mathbb{R})$-orbit is closed in the moduli space of translation surfaces. Recall that, by a result of Smillie (see, e.g., \cite{SW}), $SL(2,\mathbb{R}) X$ is closed if and only if $X$ is a \emph{Veech surface}, i.e., the stabilizer $SL(X)$ of $X$ is a lattice of $SL(2,\mathbb{R})$. 

We have a short exact sequence $\textrm{Aut}(X)\to \textrm{Aff}(X)\to SL(X)$ where $\textrm{Aut}(X)$ is the group of automorphisms of $X$ and $\textrm{Aff}(X)$ is the group of affine homeomorphisms of $X$. 

For the sake of exposition, we will assume that the group of automorphisms of $X$ is trivial: $\textrm{Aut}(X)=\{\textrm{Id}\}$. In this case, the groups $\textrm{Aff}(X)$ and $SL(X)$ are isomorphic. 

Since the KZ cocycle $G_t^{KZ}$ over the Teichm\"uller flow $g_t=\textrm{diag}(e^t, e^{-t})\in SL(2,\mathbb{R})$ is the quotient of the trivial cocycle 
$$SL(2,\mathbb{R})X\times H_1(X,\mathbb{R})\to SL(2,\mathbb{R})X\times H_1(X,\mathbb{R})$$
$$(\eta,[c])\mapsto (g_t(\eta), [c])$$
by the diagonal action of the mapping class group and the stabilizer of $SL(2,\mathbb{R})X$ in the mapping class group is precisely $\textrm{Aff}(X)$, we have that the KZ cocycle is given by the actions on homology of the elements of $\textrm{Aff}(X)$ appearing\footnote{I.e., the elements $\phi_n\in\textrm{Aff}(X)$ used to bring $g_t(\eta)$ close to $\eta$.} along the orbits of $g_t$. 

In this setting, Eskin and the author \cite{EM} proved that the Lyapunov exponents of the KZ cocycle are seen by random products of the matrices of actions on homology of affine homeomorphisms:

\begin{theorem}\label{t.EM} Let $X$ be a Veech surface of genus $g\geq 1$. Denote by 
$$1=\lambda_1>\lambda_2\geq\dots\geq\lambda_g \geq-\lambda_g\geq\dots\geq -\lambda_2>-\lambda_1=-1$$ 
the Lyapunov exponents of the KZ cocycle with respect to the unique $SL(2,\mathbb{R})$-invariant probability measure on $SL(2,\mathbb{R})X$. 

Then, there exists a probability measure on (the countable group) $\textrm{Aff}(X)$ assigning non-zero mass to every element and a constant $\lambda>0$ such that the Lyapunov exponents of random products of matrices in $Sp(2g,\mathbb{Z})$ with respect to the law $\rho(\nu)$ where $\rho:\textrm{Aff}(X)\to Sp(2g,\mathbb{Z})$ is the representation induced by the action of affine homeomorphisms on $H_1(X,\mathbb{Z})$ are precisely 
$$\lambda>\lambda\cdot\lambda_2\geq\dots\geq\lambda\cdot\lambda_g \geq-\lambda\cdot\lambda_g\geq\dots\geq -\lambda\cdot\lambda_2>-\lambda$$ 
\end{theorem} 

\begin{proof} Before discussing Lyapunov exponents, we need to ``replace'' the typical orbits of the geodesic flow $g_t$ on the hyperbolic plane $\mathbb{H}$ by certain random walks in $SL(X)\cdot i$. 

In general, the random walk on $SL(X)\cdot i$ with a law $\nu$ of full support on $SL(X)$ is tracked by a geodesic ray in $\mathbb{H}$ up to a sublinear error: for some $\lambda>0$ and for $\nu^{\mathbb{N}}$-almost all sequences $(\gamma_n)_{n\in\mathbb{N}}\subset SL(X)$, there exists a unit-speed geodesic ray $\{\alpha(t):t\in\mathbb{R}\}\subset\mathbb{H}$ such that 
$$d_{\mathbb{H}}(\gamma_n\dots\gamma_1\cdot i, \alpha(\lambda n))=o(n)$$
as $n\to \infty$. Indeed, this is a direct consequence of Oseledets theorem saying that the random product $\gamma_n\dots\gamma_1$ in $SL(2,\mathbb{R})$ is ``close'' to the matrix $g_{\lambda t} r_{\theta(\overline{\gamma})}$, where $\lambda>0$ is the top Lyapunov exponent associated to $\nu$ and $\theta(\underline{\gamma})\in [0,2\pi)$ is an angle depending on $\underline{\gamma}=(\gamma_1,\gamma_2,\dots)\in SL(X)^{\mathbb{N}}$. See, e.g., Lemma 4.1 of \cite{CE} for more details. 

The distribution of the angles $\theta(\underline{\gamma})$ \emph{depend} on $\nu$. In particular, it is \emph{far from obvious} that one can choose $\nu$ in such a way that a typical geodesic rays in $\mathbb{H}$ are tracked by  random walks with law $\nu$ up to sublinear error: in concrete terms, this basically amounts to choose $\nu$ so that the distribution of angles $\theta(\underline{\gamma})\in [0,2\pi)$ is given by the Lebesgue measure. 

Fortunately, a profound theorem of Furstenberg \cite{Fu71} ensures the existence of a probability measure $\nu$ on the lattice $SL(X)$ with full support and the desired (Lebesgue) distribution of angles\footnote{Actually, Furstenberg proves that the so-called \emph{Poisson boundary} of $(SL(X),\nu)$ is $(SO(2,\mathbb{R}), \textrm{Lebesgue})$}. 

Once we know that a typical orbit of the geodesic flow is tracked by a random walk with sublinear error, let us come back to the discussion of Lyapunov exponents. Consider a typical geodesic ray $g_t r_{\theta}(\omega)$ tracked by a random walk $(\gamma_n\dots\gamma_1)_{n\in\mathbb{N}}$ with sublinear error. 

Forni \cite[\S 2]{F02} proved that the following general growth estimate for the KZ cocycle:  
$$\frac{d}{dt}\log\|G_t^{KZ}(\eta)\|\leq 1$$
for all translation surface $\eta$, where $\|.\|$ denotes the \emph{Hodge norm}. From this estimate, we deduce: 
$$\log\|(\rho(\gamma_n\dots\gamma_1)\cdot G_{\lambda n}^{KZ}(r_{\theta}\omega)^{-1})^{\pm1}\|\leq d_{\mathbb{H}}(\gamma_n\dots\gamma_1\cdot i, g_{\lambda n}(r_{\theta}\omega)) = o(n)$$ 

Hence,  
$$\lim\limits_{n\to\infty}\frac{1}{n}\log\frac{\|\rho(\gamma_n\dots \gamma_1)(w)\|}{\|G_{\lambda n}^{KZ}(w)\|}=0$$
for any vector $w\in H_1(X,\mathbb{R})-\{0\}$. By definition of Lyapunov exponents, this equality means that 
$$\lambda_{\rho(\nu)}(w) = \lambda\cdot \lambda_{KZ}(w)$$ 
where $\lambda_{\rho(\nu)}(w)$, resp. $\lambda_{KZ}(w)$, is the Lyapunov exponent of $w$ with respect to random products of matrices in $Sp(2g,\mathbb{Z})$ with law $\rho(\nu)$, resp. KZ cocycle $G_t^{KZ}$. This proves the theorem. 
\end{proof}

This theorem together with Avila-Viana criterion \cite{AV} for the simplicity of Lyapunov exponents of random products of matrices yield the following statement (cf. \cite[Theorem 1]{EM}):

\begin{corollary}\label{c.EM} Let $X=(M,\omega)$ be a Veech surface of genus $g\geq 1$. Suppose that $\textrm{Aff}(X)$ contains two elements $\phi$ and $\psi$ whose actions on the annihilator $H_1^{(0)}(X,\mathbb{R})$ of the tautological plane 
$$\mathbb{R}\cdot\textrm{Re}(\omega) \oplus \mathbb{R}\cdot\textrm{Im}(\omega) \subset H^1(X,\mathbb{R})$$ 
are given by two matrices $A$ and $B$ (in $Sp(H_1^{(0)}(X,\mathbb{R}))\simeq Sp(2g-2,\mathbb{R})$) such that:
\begin{itemize}
\item[(i)] $A$ is pinching, i.e., all eigenvalues of $A$ are real, simple and their moduli are distinct; 
\item[(ii)] $B$ is twisting with respect to $A$, i.e., for any $1\leq k\leq g-1$, any isotropic $A$-invariant $k$-plane $F$ and any coisotropic $A$-invariant $(2g-2-k)$-plane $F'$, we have $B(F)\cap F' = \{0\}$.
\end{itemize}
Then, the Lyapunov exponents of the KZ cocycle with respect to the Lebesgue measure on $SL(2,\mathbb{R}) X$ are simple, i.e., the Lyapunov spectrum has the form 
$$1=\lambda_1>\lambda_2>\dots>\lambda_g>-\lambda_g>\dots>-\lambda_2>-\lambda_1=-1$$
\end{corollary}

\begin{proof} By Theorem \ref{t.EM}, our task is equivalent to establish the simplicity of the Lyapunov spectra of the random products with law $\rho_0(\nu)$ of matrices in $Sp(H_1^{(0)}(X,\mathbb{R}))$, where $\rho_0:\textrm{Aff}(X)\to H_1^{(0)}(X,\mathbb{R})$ is the natural representation. 

By the simplicity criterion\footnote{See also Theorem 2.17 in \cite{MMY}.} of Avila-Viana \cite{AV}, a random product with law $\theta$ of matrices in $Sp(2d,\mathbb{R})$ has simple Lyapunov spectrum whenever the support of $\eta$ contains pinching and twisting matrices (in the sense of items (i) and (ii) above). 

Since $\rho_0(\nu)$ gives positive mass to $A$ and $B$ (because $\nu$ assigns positive masses to $\phi$ and $\psi$), the proof of the corollary is complete. 
\end{proof}

\begin{remark}
Theorem \ref{t.EM}, Corollary \ref{c.EM} and its variants were applied by Eskin and the author to study the Lyapunov spectra of certain Teichm\"uller curves in genus four (of Prym type) and certain variations of Hodge structures of weight three associated to $14$ families of Calabi-Yau threefolds (including \emph{mirror quintics}). See Sections 3 and 4 of \cite{EM} for more explanations. 
\end{remark}

\subsection{Galois-theoretical criterion for simplicity of exponents of origamis} 

From the practical point of view, Corollary \ref{c.EM} is not quite easy to apply. Indeed, the verification of the pinching and twisting properties might be tricky: for example, Avila-Viana \cite{AV} performed a somewhat long inductive procedure in order to establish the pinching and twisting properties in their context (of Rauzy-Veech monoids). 

As it turns out, M\"oller, Yoccoz and the author \cite{MMY} found an effective version of Corollary \ref{c.EM} in the case of square-tiled surfaces thanks to some combinatorial arguments involving Galois theory. 

In the sequel, we will state and prove a Galois-theoretical simplicity criteria and we will discuss its applications to square-tiled surfaces of genus three. 

\subsubsection{Galois-pinching matrices} 

Let $A\in Sp(2d,\mathbb{R})$ be a $2d\times 2d$ symplectic matrix. The characteristic polynomial of $A$ is a monic reciprocal polynomial $P$ of degree $2d$. Denote by $\widetilde{R}=\{\lambda_i, \lambda_i^{-1}: 1\leq i\leq d\}=P^{-1}(0)$ the set of roots of $P$ and let $R=p(\widetilde{R})$ where $p(\lambda)=\lambda+\lambda^{-1}$. 

We say that a matrix $A\in Sp(2d,\mathbb{Z})$ is \emph{Galois-pinching} if its characteristic polynomial $P$ is irreducible over $\mathbb{Q}$, its eigenvalues are real ($\widetilde{R}\subset \mathbb{R}$), and the Galois group $Gal$ of $P$ is the largest possible, i.e., $Gal\simeq S_d \rtimes (\mathbb{Z}/2\mathbb{Z})^d$ acts by the full permutation group on $R=p(\widetilde{R})$ and the subgroup fixing $R$ pointwise acts by independent transpositions of each of the $d$ pairs $\{\lambda_i, \lambda_i^{-1}\}$. 

For each $\lambda\in\widetilde{R}$, let us choose an eigenvector $v_{\lambda}$ of $A$ associated to $\lambda$ with coordinates in the field $\mathbb{Q}(\lambda)$ in such a way that $v_{g.\lambda} = g.v_{\lambda}$ for all $g\in Gal$. 

The nomenclature ``Galois-pinching'' is justified by the following proposition:

\begin{proposition}\label{p.Galois-pinching} A Galois-pinching matrix is pinching. 
\end{proposition}

\begin{proof} By definition, all eigenvalues of a Galois-pinching matrix $A$ are real and simple. Hence, our task is to show these eigenvalues have distinct moduli. 

Suppose that $\lambda$ and $-\lambda$ are eigenvalues of $A$. An element of the Galois group $Gal$ fixing $\lambda$ must also fix $-\lambda$, a contradiction with the fact that $Gal$ is the largest possible. 
\end{proof}

An important point about Galois-pinching matrices is the fact that they can be detected in an \emph{effective} way. 

In order to illustrate this, let us consider the prototype $P$ of characteristic polynomial of a matrix in $Sp(4,\mathbb{Z})$, i.e., $P(x)=x^4+ax^3+bx^2+ax+1$ is a monic reciprocal integral polynomial of degree four. 

The following elementary proposition characterizes the polynomials $P$ with real, simple and positive roots (and, \emph{a fortiori}, of distinct moduli).

\begin{proposition} The polynomial $P(x)=x^4+ax^3+bx^2+ax+1$ has real, simple and positive roots if and only if 
$$\Delta_1:=a^2-4b+8 > 0, \quad t:=-a-4>0 \quad \textrm{ and } \quad d:=b+2+2a>0$$
\end{proposition}

\begin{proof} A simple calculation shows that $\lambda$ is a root of $P$ if and only if $\mu:=\lambda+\lambda^{-1}-2$ is a root of the quadratic polynomial
$$Q(y):=y^2-ty+d$$
of discriminant $t^2-4d=a^2-4b+8:=\Delta_1$.

Since the roots $\lambda$ of $P$ are real, simple and positive if and only if the roots $\mu$ of $Q$ have the same properties, the desired proposition follows. 
\end{proof}

The next two propositions provide a criterion for the irreducibility of $P$ over $\mathbb{Q}$. 

\begin{proposition} The polynomial $P(x)=x^4+ax^3+bx^2+ax+1\in\mathbb{Z}[x]$ is a product of two reciprocal quadratic rational polynomials if and only if $\sqrt{\Delta_1}\in\mathbb{Q}$. 
\end{proposition}

\begin{proof} The quadratic polynomial $Q(y)=y^2-ty+d$ with roots $\mu$ related to the roots $\lambda$ of $P$ via the formula $\mu=\lambda+\lambda^{-1}-2$ is reducible over $\mathbb{Q}$ if and only if $\sqrt{\Delta_1}\in\mathbb{Z}$.  
\end{proof}

\begin{proposition} Let $P(x)=x^4+ax^3+bx^2+ax+1$ be a monic reciprocal integral polynomial of degree four. Suppose that $\Delta_1:=a^2-4b+8$ is not a square (i.e., $\sqrt{\Delta_1}\notin\mathbb{Z}$) and $P$ is reducible over $\mathbb{Q}$. Then, 
$$\Delta_2:=(b+2-2a)(b+2+2a)$$
is a square, i.e., $\sqrt{\Delta_2}\in\mathbb{Z}$. 
\end{proposition}

\begin{proof} Since $\Delta_1$ is not a square, $P$ has no rational root. Because $P$ is reducible over $\mathbb{Q}$ (by assumption), the previous proposition implies that $P=P' P''$ where $P', P''\in \mathbb{Q}[x]$ are monic irreducible quadratic polynomials which are not reciprocal. Thus, we can relabel the roots of $P$ in such a way that 
$$P'(x) = (x-\lambda_1)(x-\lambda_2), \quad P''(x) = (x-\lambda_1^{-1})(x-\lambda_2^{-1})$$
Note that $P'\in\mathbb{Q}[x]$ implies that $\lambda_1\lambda_2, \lambda_1+\lambda_2\in\mathbb{Q}$ and, \emph{a fortiori}, $\lambda_1^2+\lambda_2^2\in\mathbb{Q}$. Therefore, 
$$(\lambda_1-\lambda_1^{-1})(\lambda_2-\lambda_2^{-1}) = \lambda_1\lambda_2 - \frac{1-\lambda_1^2-\lambda_2^2}{\lambda_1\lambda_2}\in\mathbb{Q}$$
It follows that $\Delta_2=(b+2-2a)(b+2+2a) = (\lambda_1-\lambda_1^{-1})^2(\lambda_2-\lambda_2^{-1})^2$ is a square. 
\end{proof}

Furthermore, it is not hard to decide whether an irreducible monic reciprocal integral polynomial of degree four has the largest possible Galois group: 

\begin{proposition}\label{p.Galois-degree4} Let $P(x)=x^4+ax^3+bx^2+ax+1\in\mathbb{Z}[x]$ be irreducible over $\mathbb{Q}$. The Galois group $Gal$ of $P$ is the largest possible (i.e., $Gal\simeq S_2\rtimes(\mathbb{Z}/2\mathbb{Z})^2$ has order eight) if and only if 
$$\sqrt{\Delta_1}, \sqrt{\Delta_2}, \sqrt{\Delta_1\Delta_2}\notin\mathbb{Z},$$ 
(where $\Delta_1:=a^2-4b+8$ and $\Delta_2:=(b+2-2a)(b+2+2a)$). 

Moreover, in this case we have that the splitting field of $P$ contains exactly three quadratic subfields, namely, $\mathbb{Q}(\sqrt{\Delta_1})$, $\mathbb{Q}(\sqrt{\Delta_2})$, $\mathbb{Q}(\sqrt{\Delta_1\Delta_2})$. 
\end{proposition}

\begin{proof} The solution of this elementary exercise in Galois theory is explained in Lemmas 6.12 and 6.13 of \cite{MMY}. For the sake of completeness, let us sketch the proof of this proposition. 

Since $P$ is irreducible, $Gal$ acts transitively on the roots $\lambda_1, \lambda_1^{-1}, \lambda_2, \lambda_2^{-1}$ of $P$. Thus, if $Gal$ can permute $\lambda_i$ and $\lambda_i^{-1}$ independently for $i=1, 2$, then $Gal$ has order eight and, \emph{a fortiori}, $Gal$ is the largest possible.

This reduces our task to show that if $Gal$ can \emph{not} permute $\lambda_i$ and $\lambda_i^{-1}$ independently for $i=1, 2$, then either $\sqrt{\Delta_2}$ or $\sqrt{\Delta_1\Delta_2}$ is an integer. 

For this sake, we observe that if if $Gal$ doesn't permute $\lambda_i$ and $\lambda_i^{-1}$ independently for $i=1, 2$, then $Gal$ permutes simultaneously $\lambda_1, \lambda_1^{-1}$ and $\lambda_2, \lambda_2^{-1}$. In this case, there are two possibilities: 
\begin{itemize}
\item[(a)] either $Gal$ is generated by the permutations $(\lambda_1,\lambda_2)(\lambda_1^{-1},\lambda_2^{-1})$ and $(\lambda_1,\lambda_1^{-1})(\lambda_2,\lambda_2^{-1})$, 
\item[(b)] or $Gal$ is generated by the four cycle $(\lambda_1,\lambda_2,\lambda_1^{-1},\lambda_2^{-1})$.
\end{itemize}

These cases can be distinguished as follows. The expression $(\lambda_1-\lambda_1^{-1})(\lambda_2-\lambda_2^{-1})$ is invariant in case (a) (i.e., $Gal \simeq \mathbb{Z}/2\mathbb{Z}\times \mathbb{Z}/2\mathbb{Z}$ is a Klein group) but not in case (b) (i.e., $Gal\simeq \mathbb{Z}/4\mathbb{Z}$ is a cyclic group). Similarly, the expression $(\lambda_1+\lambda_1^{-1}-\lambda_2-\lambda_2^{-1}) (\lambda_1-\lambda_1^{-1}) (\lambda_2-\lambda_2^{-1})$ is invariant in case (b) but not in case (a). Therefore: 
\begin{itemize}
\item (a) occurs if and only if $(\lambda_1-\lambda_1^{-1})(\lambda_2-\lambda_2^{-1})\in\mathbb{Q}$;
\item (b) occurs if and only if $(\lambda_1+\lambda_1^{-1}-\lambda_2-\lambda_2^{-1}) (\lambda_1-\lambda_1^{-1}) (\lambda_2-\lambda_2^{-1})\in\mathbb{Q}$.
\end{itemize}

Since $(\lambda_1-\lambda_1^{-1})^2(\lambda_2-\lambda_2^{-1})^2=\Delta_2$ and $(\lambda_1+\lambda_1^{-1}-\lambda_2-\lambda_2^{-1})^2 (\lambda_1-\lambda_1^{-1})^2 (\lambda_2-\lambda_2^{-1})^2 =\Delta_1\Delta_2$, we conclude that 
\begin{itemize}
\item (a) occurs if and only if $\sqrt{\Delta_2}\in\mathbb{Z}$;
\item (b) occurs if and only if $\sqrt{\Delta_1\Delta_2}\in\mathbb{Z}$.
\end{itemize}
In any case, we show that either $\sqrt{\Delta_2}$ or $\sqrt{\Delta_1\Delta_2}$ is an integer when $Gal$ can not permute $\lambda_i$ and $\lambda_i^{-1}$ independently for $i=1, 2$. This completes our sketch of proof. 
\end{proof}

In summary, the previous four propositions allow us to test whether a matrix $A\in Sp(4,\mathbb{Z})$ is  Galois-pinching by studying three discriminants $\Delta_1$, $\Delta_2$ and $\Delta_1\Delta_2$ naturally attached to its characteristic polynomial. 

\subsubsection{Twisting with respect to Galois-pinching matrices I: statements of results} 

After discussing the Galois-pinching property, let us study the twisting property with respect to Galois-pinching matrices. Our main result in this direction is:

\begin{theorem}\label{t.MMY-twisting} Let $A\in Sp(2d,\mathbb{Z})$ be a Galois-pinching matrix. Suppose that $B\in Sp(2d,\mathbb{Z})$ has the property that $A$ and $B^2$ share no common proper invariant subspace. Then, there exists $m\geq 1$ and, for any $\ell^{\ast}$, there are integers $\ell_i\geq\ell^{\ast}$, $1\leq i\leq m-1$, such that the product 
$$B A^{\ell_1} \dots B A^{\ell_{m-1}} B$$ 
is twisting with respect to $A$, i.e., for all $1\leq k\leq d$, for any $A$-invariant isotropic subspace $F$ of dimension $k$ and for any $A$-invariant coisotropic subspace $F'$ of dimension $2d-k$, we have $B A^{\ell_1} \dots B A^{\ell_{m-1}} B(F)\cap F'=\{0\}$. 
\end{theorem} 

This theorem constitutes the main ingredient in the simplicity criteria in \cite{MMY}. Before starting its somewhat long proof, let us make some comments on its statement and applicability. 

First, we observe that Theorem \ref{t.MMY-twisting} becomes false if $B^2$ is replaced by $B$ in the hypothesis ``$A$ and $B^2$ share no common proper invariant subspace'': for example, 
$A=\left(\begin{array}{cc} 2 & 1 \\ 1 & 1 \end{array}\right)\in SL(2,\mathbb{Z})$ is Galois-pinching, $B=\left(\begin{array}{cc} 0 & -1 \\ 1 & 0 \end{array}\right)\in SL(2,\mathbb{Z})$ has no invariant subspaces, but $B A^{\ell_1} \dots B A^{\ell_{m-1}} B$ is
\emph{never} twisting with respect to $A$ (because $B$ permutes the eigenspaces of $A$).\footnote{Logically, there is no contradiction to Theorem \ref{t.MMY-twisting} here: $A$ and $B$ don't fit the assumptions of this theorem since $A$ and $B^2 = -\textrm{Id}$ share two common proper invariant subspaces, namely, the eigenspaces of $A$.}

Second, the condition ``$A$ and $B^2$ have no common proper invariant subspace'' might not be easy to check in general. For this reason, the following two propositions (cf. Proposition 4.7, Lemma 5.1 and Lemma 5.5 in \cite{MMY}) are useful in certain applications of Theorem \ref{t.MMY-twisting}.

\begin{proposition}\label{p.unipotent-twisting} Let $A\in Sp(2d,\mathbb{Z})$ be a Galois-pinching matrix. Suppose that $B\in Sp(2d,\mathbb{Z})$ is unipotent and $B\neq\textrm{Id}$. If $A$ and $B^2$ share a common proper invariant subspace, then $(B-\textrm{Id})(\mathbb{R}^{2d})$ is a Lagrangian
subspace of $\mathbb{R}^{2d}$.
\end{proposition}

\begin{proof} We begin by noticing that our assumptions imply that any subspace which is invariant under $B^m$ for some $m>0$ is also invariant under $B$. Indeed, if we write $B=\textrm{Id}+N$ and $B^m=\textrm{Id}+N'$, then the binomial formula says that $N'$ is nilpotent whenever $N$ is nilpotent. Hence, $B=(\textrm{Id}+N')^{1/m}$ can be calculated by truncating the binomial series\footnote{I.e., since $N'$ is nilpotent, the formal binomial series $(I+N')^a=\sum\limits_{k=0}^{\infty}\binom{a}{k}(N')^k$ (where $a\in\mathbb{C}$ and $\binom{a}{k}:=a(a-1)\dots(a-k+1)/k!$) can be interpreted as a polynomial function
of $N'$.} and, therefore, $B$ is a polynomial function of $N'=B^m-\textrm{Id}$. In particular, any subspace invariant under $B^m$ is also invariant under $B$.  

In particular, our assumption that $A$ and $B^2$ share a common proper invariant subspace imply (in our current setting) that $A$ and $B$ share a common proper invariant subspace. 

Since $A$ is Galois-pinching, any $A$-invariant subspace is spanned by eigenvectors. Denote by $P$ the characteristic polynomial of $A$, let $\widetilde{R}=P^{-1}(0)$ and, for each $\lambda\in\widetilde{R}$, take $v_{\lambda}$ an eigenvector of $A$ with eigenvalue $\lambda$ whose coordinates belong to $\mathbb{Q}(\lambda)$ in such a way that $v_{g.\lambda}=g.v_{\lambda}$ for all $g$ in the Galois group $Gal$ of $P$. 

Next, let us fix $R_B\subset\widetilde{R}$ a proper subset of \emph{minimal} cardinality such that the subspace $E(R_B)$ spanned by the vectors $\{v_{\lambda}:\lambda\in R_B\}$ is $B$-invariant. Because $B$ has integral entries, for all $\sigma\in Gal$ the subspaces $E(\sigma(R_B))$ are also $B$-invariant. Thus, the minimality of the cardinality of $R_B$ implies that, for each $\sigma\in Gal$, either $\sigma(R_B)\cap R_B=\emptyset$ or $\sigma(R_B)=R_B$. This property together with the fact that $Gal$ is the largest possible is a severe constraint on the proper subset $R_B\subset \widetilde{R}$:
\begin{itemize}
\item either $R_B=\{\lambda\}$ has cardinality one,  
\item or $R_B$ has the form $R_B=\{\lambda,\lambda^{-1}\}$
\end{itemize}

The first possibility does not occur in our context: if $B(v_{\lambda})=c v_{\lambda}$, then $c=1$ (because $B$ is unipotent); hence, using the action of $Gal$, we would deduce that $B$ fixes all eigenvectors of $A$, so that $B=\textrm{Id}$, a contradiction. 

Therefore, $B$ preserves \emph{some} subspace $E(\lambda,\lambda^{-1})$ of the form $E(\lambda,\lambda^{-1}) = \mathbb{R} v_{\lambda} \oplus \mathbb{R} v_{\lambda^{-1}}$. Using again the action of $Gal$, we deduce that $B$ preserves \emph{all} subspaces $E(\lambda,\lambda^{-1})$, $\lambda\in\widetilde{R}$, and the restrictions of $B$ to such planes are Galois-conjugates. As $B\neq\textrm{Id}$ is unipotent, $(B-\textrm{Id})E(\lambda,\lambda^{-1})$ is an one-dimensional subspace  of $E(\lambda,\lambda^{-1})$ for all $\lambda\in\widetilde{R}$. Since the planes $E(\lambda,\lambda^{-1})$ are mutually symplectically orthogonal, it follows that $(B-\textrm{Id})(\mathbb{R}^{2d})$ is a Lagrangian subspace. 
\end{proof}

\begin{proposition}\label{p.Galois-twisting} Let $A\in Sp(2d,\mathbb{Z})$ be a Galois-pinching matrix. Suppose that $B\in Sp(2d,\mathbb{Z})$ has minimal polynomial of degree $>2$ with no irreducible factor of even degree and a splitting field disjoint from the splitting field of the characteristic polynomial of $A$. Then, $A$ and $B^2$ do not share a common proper invariant subspace.  
\end{proposition} 

\begin{proof} We affirm that a $B^2$-invariant subspace is also $B$-invariant. In fact, $B$ and $B^2$ have the same characteristic subspaces because $\lambda$ and $-\lambda$ can't be both eigenvalues of $B$ (thanks to our assumption that the minimal polynomial of $B$ has no irreducible factor of even degree). Since an invariant subspace is the sum of its intersections with characteristic subspaces, it suffices to check that a $B^2$-invariant subspace contained in a characteristic subspace of $B$ is also $B$-invariant. Since $B$ is unipotent up to scalar factors in its characteristic subspaces, the desired fact follows from the argument used in the beginning of the proof of Proposition \ref{p.unipotent-twisting}. 

Let us now assume by contradiction that $A$ and $B^2$ share a common proper invariant subspace. From the discussion of the previous paragraph, this means that $A$ and $B$ share a common proper invariant subspace. By repeating the analysis in the proof of Proposition \ref{p.unipotent-twisting}, we have that all planes $E(\lambda,\lambda^{-1}) = \mathbb{R} v_{\lambda} \oplus \mathbb{R} v_{\lambda^{-1}}$, $\lambda\in\widetilde{R}$, are invariant under $A$ and $B$. 

Observe that $E(\lambda,\lambda^{-1})$ is defined over $\mathbb{Q}(\lambda+\lambda^{-1})$ and the trace and determinant of $B|_{E(\lambda,\lambda^{-1})}$. From our hypothesis of disjointness of the splitting fields of $A$ and $B$, the trace and determinant of $B|_{E(\lambda,\lambda^{-1})}$ are rational. Thus, the minimal polynomial of $B$ has degree $\leq 2$, a contradiction with our hypotheses. 
\end{proof}

After these comments on the statement of Theorem \ref{t.MMY-twisting}, let us now prove this result. 

\subsubsection{Twisting with respect to Galois-pinching matrices II: proof of Theorem \ref{t.MMY-twisting}}

Consider the Galois-pinching matrix $A$ and denote by $\widetilde{R}$ the set of its eigenvalues. For each $\lambda\in\widetilde{R}$, we select eigenvectors $v_{\lambda}$ of $A$ associated to $\lambda$ with coordinates in the field $\mathbb{Q}(\lambda)$ behaving coherently with respect to the Galois group $Gal$ of the characteristic polynomial of $A$, i.e., $v_{g\lambda}=gv_{\lambda}$ for all $g\in Gal$. 

The twisting propertyfor a matrix $C$ with respect to $A$ can be translated in terms of combinatorial properties of certain graphs naturally attached to its exterior powers $\bigwedge^k C$. 

More concretely, for each $1\leq k\leq d$, let $\widetilde{R}_k$ be the collection of all subsets of $\widetilde{R}$ with cardinality $k$. Note that, for each $\underline{\lambda}=\{\lambda_1<\dots<\lambda_k\}\in\widetilde{R}_k$, we can associated a multivector $v_{\underline{\lambda}} = v_{\lambda_1}\wedge\dots\wedge v_{\lambda_k}\in\bigwedge^k\mathbb{R}^{2d}$. By definition, the set $\{v_{\underline{\lambda}}\vert \underline{\lambda}\in\widetilde{R}_k\}$ is a basis of $\bigwedge^k\mathbb{R}^{2d}$ diagonalizing $\bigwedge^k A$: indeed, $(\bigwedge^k A)(v_{\underline{\lambda}}) = N(\underline{\lambda})v_{\underline{\lambda}}$ where $N(\underline{\lambda}):=\prod\limits_{\lambda_i\in\underline{\lambda}}\lambda_i$.

In this setting, the $A$-invariant isotropic subspaces are easy to characterize. If $p(\lambda)=\lambda+\lambda^{-1}$ and $\widehat{R}_k$ is the collection of $\underline{\lambda}\in\widetilde{R}_k$ such that $p|_{\underline{\lambda}}$ is injective, then  the subspace generated by $v_{\lambda_1}, \dots, v_{\lambda_k}$ is \emph{isotropic} if and only if $\underline{\lambda}=\{\lambda_1,\dots,\lambda_k\}\in\widehat{R}_k$. Using this fact and an elementary (linear algebra) computation (cf. \cite[Lemma 4.8]{MMY}), it is not hard to check that:
\begin{lemma} A matrix $C$ is twisting with respect to $A$ if and only if for every $1\leq k\leq d$ the coefficients $C^{(k)}_{\underline{\lambda}, \underline{\lambda}'}$
of the matrix $\bigwedge^k C$ in the basis $\{v_{\underline{\lambda}}\}$ satisfy the condition
\begin{equation}\label{e.condition-k}
C^{(k)}_{\underline{\lambda}, \underline{\lambda}'}\neq 0\, \quad \forall\, 
\underline{\lambda}, \underline{\lambda}'\in\widehat{R}_k\,.
\end{equation}
In other words, if $\Gamma_k(C)$ is oriented graph with set of vertices $\textrm{Vert}(\Gamma_k(C))=\widehat{R}_k$ and set of arrows $\{\underline{\lambda}_0\to\underline{\lambda}_1: C^{(k)}_{\underline{\lambda}_0,\underline{\lambda}_1}\neq 0\}$, then $C$ is twisting with respect to $A$ if and only if $\Gamma_k(C)$ is a complete graph for each $1\leq k\leq d$. 
\end{lemma}

In general, it is not easy to apply this lemma because the verification of the completeness of  
$\Gamma_k(C)$ might be tricky. For this reason, we introduce the following (classical) notion:
\begin{definition}\label{d.mixing} The graph $\Gamma_k(C)$ is \emph{mixing} if there exists $m\geq 1$ such that for all $\underline{\lambda}_0,\underline{\lambda}_1\in\widehat{R}_k$ we can find an oriented path in $\Gamma_k(C)$ of length $m$ going from $\underline{\lambda}_0$ to $\underline{\lambda}_1$.
\end{definition} 

\begin{remark} 
In this definition, it is important to connect two arbitrary vertices by a path of length \emph{exactly} $m$ (and not only of length $\leq m$). For instance, the graph in Figure \ref{f.no-mixing} is not mixing because all paths connecting $A$ to $B$ have \emph{odd} length while all paths connecting $A$ to $C$ have \emph{even} length. 
\end{remark}

\begin{figure}[htb!]
\includegraphics[scale=0.7]{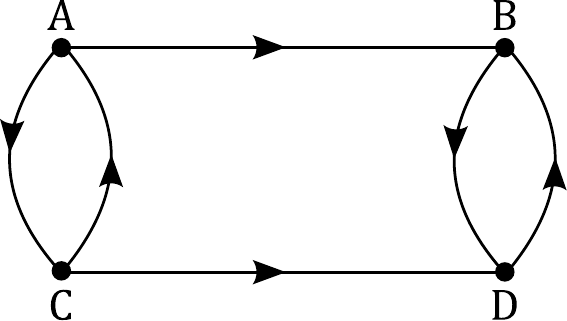}
\caption{A non-mixing oriented graph}\label{f.no-mixing}
\end{figure}

The relevance of this notion is explained by the following proposition (saying that if $\Gamma_k(C)$ is mixing, then $\Gamma_k(D)$ is complete for certain products $D$ of the matrices $C$ and $A$).   
\begin{proposition}\label{p.mixing/twisting} Let us assume that the graph $\Gamma_k(C)$ is mixing with respect to an integer $m\geq 1$.  Then there exists a finite family of hyperplanes $V_1,\dots,V_t$ of 
$\mathbb{R}^{m-1}$ such that the following holds. For any $\underline{\ell}=(\ell_1,\dots,\ell_{m-1})\in \mathbb{Z}^{m-1}-(V_1\cup\dots\cup V_{m-1})$, the graph $\Gamma_k(D(n))$ associated to the matrix 
$$D(n):=C A^{n\ell_1}\dots C A^{n\ell_{m-1}} C$$
is complete for all sufficiently large $n\in \mathbb N$.
\end{proposition}

\begin{proof} Let us denote $D:=D(n)$. By definition, 
$$
D^{(k)}_{\underline{\lambda}_0,\underline{\lambda}_m} =
\sum_{ \substack{\gamma \textrm{ path of length }m \\  \textrm{ in } \Gamma_k(C) \textrm{ from } \underline{\lambda}_0  \textrm{ to }\underline{\lambda}_m}} 
C^{(k)}_{\underline{\lambda}_0,\underline{\lambda}_1}N(\underline{\lambda}_1)^{n\ell_1} C^{(k)}_{\underline{\lambda}_1,\underline{\lambda}_2}\dots 
N(\underline{\lambda}_{m-1})^{n\ell_{m-1}} C^{(k)}_{\underline{\lambda}_{m-1},
\underline{\lambda}_m} \,.
$$
Consider the linear forms:
$$
L_\gamma(\underline{\ell}) = \sum\limits_{i=1}^{n-1} \ell_i\left(\sum\limits_{\lambda\in\underline{\lambda}_i} \log|\lambda|\right)\,.
$$
Our assumption on $C$ ensures that there are coefficients $c_\gamma \not =0$ such that
$$
D^{(k)}_{\underline{\lambda}_0,\underline{\lambda}_m}  = \sum\limits_{\gamma} c_{\gamma} \exp(n L_{\gamma}(\underline{\ell}))\,.
$$
We want to prove that $D^{(k)}_{\underline{\lambda}_0,\underline{\lambda}_m}\neq0$. Since  $D^{(k)}_{\underline{\lambda}_0,\underline{\lambda}_m}$ is a sum of exponentials with non-vanishing coefficients, our task is to show that $\underline{\ell}$ can be taken so that potential cancelations among these terms can be avoided. 

Here, the key idea is to prove that, for $\gamma\neq\gamma'$, the linear forms $L_{\gamma}$ and $L_{\gamma'}$ are \emph{distinct}. Indeed, suppose that this is the case and let us define $V(\gamma,\gamma')=\{\underline{\ell}: L_{\gamma}(\underline{\ell}) = L_{\gamma'}(\underline{\ell})\}$. By hypothesis, each $V(\gamma,\gamma')$ is a \emph{hyperplane}. Since there are only finitely many paths 
$\gamma,\gamma'$ of length $m$ on $\Gamma_k(C)$, the collection of $V(\gamma,\gamma')$ corresponds to a finite family of hyperplanes $V_1,\dots,V_t$.
Hence, if we take $\underline{\ell}\notin V_1\cup\dots\cup V_t$, then 
$$D^{(k)}_{\underline{\lambda}_0,\underline{\lambda}_m}=\sum\limits_{\gamma} c_{\gamma} \exp(n L_{\gamma}(\underline{\ell}))\neq 0$$
for $n\to\infty$ sufficiently large (because the coefficients $L_{\gamma}(\underline{\ell})$ are mutually distinct).  

Let us now complete the proof of the proposition by showing that $L_{\gamma}\neq L_{\gamma'}$ for $\gamma\neq\gamma'$. Given $\underline{\lambda}\in\widehat{R}_k$ and $\underline{\lambda}'\in\widetilde{R}_k$, $\underline{\lambda}'\neq\underline{\lambda}$, we claim that the following coefficients of $L_{\gamma}$ and $L_{\gamma'}$ differ:
$$\sum\limits_{\lambda\in\underline{\lambda}} \log|\lambda| \neq \sum\limits_{\lambda'\in\underline{\lambda}'} \log|\lambda'|\,.$$
In fact, an equality between these coefficients would imply a relation 
$$\prod\limits_{\lambda\in\underline{\lambda}}\lambda = \pm\prod\limits_{\lambda'\in\underline{\lambda}'}\lambda' :=\phi\,.$$
On the other hand, if $\lambda\in\underline{\lambda}$ then $\lambda^{-1}\notin\underline{\lambda}$ (since $\underline{\lambda}\in\widehat{R}_k$). So, if we take  $\lambda(0)\in\underline{\lambda}-\underline{\lambda}'$ and $g\in Gal$ with $g(\lambda(0))=\lambda(0)^{-1}$ and $g(\lambda)=\lambda$ otherwise, then 
$$\lambda(0)^{-2}\phi = \prod\limits_{\lambda\in\underline{\lambda}}g\lambda = g\phi  = \pm\prod\limits_{\lambda'\in\underline{\lambda}'}g\lambda'=\pm\left\{\begin{array}{cc}\lambda(0)^2\phi, & \textrm{ if } \lambda(0)^{-1}\in\underline{\lambda}'  \, \\ \phi, & \textrm{ otherwise}\end{array}\right. .$$ 
Thus, $\lambda(0)^{-2}\phi=\pm\lambda(0)^2\phi$ or $\pm\phi$, a contradiction in any case (because $A$ Galois-pinching implies that $\lambda(0)\in\mathbb{R}-\{\pm 1\}$).
\end{proof}

This proposition suggests the following strategy of proof of Theorem~\ref{t.MMY-twisting}:

\begin{itemize}
\item{\em Step 0}: Show that the graphs $\Gamma_k(C)$ are always non-trivial, i.e., there is at least one arrow starting at each of its vertices.
\item {\em Step 1:} Starting from $A$ and $B$ in the statement of Theorem \ref{t.MMY-twisting}, we will show that $\Gamma_1(B)$ is mixing. By Proposition~\ref{p.mixing/twisting}, there is a product $C$ of powers of $A$ and $B$ such that $\Gamma_1(C)$ is complete. In particular, this settles the case $d=1$ of Theorem \ref{t.MMY-twisting}. 
\item Let us now consider the cases $d\geq 2$ of Theorem \ref{t.MMY-twisting}. Unfortunately, there is no ``unified'' argument to deal with all cases and we are obliged to separate the case $d=2$ from $d\geq 3$.
\item {\em Step 2:} In the case $d\geq 3$, we will show that $\Gamma_k(C)$ (with $C$ as in Step 1) is mixing for all $2\leq k<d$. By Proposition~\ref{p.mixing/twisting}, there is a product $D$ of powers of $A$ and $C$ such that $\Gamma_k(D)$ is complete for all $1\leq k<d$. Using this information, we will prove that $\Gamma_d(D)$ is mixing. By 
Proposition~\ref{p.mixing/twisting}, a certain product $E$ of powers of $A$ and $D$ is twisting with respect to $A$, so that this completes the argument in this case.
\item {\em Step 3:} In the special case $d=2$, we will show that either $\Gamma_2(C)$ or a closely related graph $\Gamma_2^*(C)$ is mixing and we will see that this is sufficient to construct $D$ such that $\Gamma_2(D)$ is complete.
\end{itemize} 

During the implementation of this strategy, the following easy remarks will be repeatedly used:

\begin{remark}\label{r.galois-invariance} If $C\in Sp(2d,\mathbb{Z})$, then $\Gamma_k(C)$ is invariant under the action of Galois group $Gal$ on the set $\widehat{R}_k\times \widehat{R}_k$ (parametrizing all possible arrows of $\Gamma_k(C)$). In particular, since the Galois group $Gal$ is the largest possible, whenever an arrow 
$\underline{\lambda}\to\underline{\lambda}'$ belongs to $\Gamma_k(C)$, the inverse arrow $\underline{\lambda}'\to\underline{\lambda}$ also belongs to $\Gamma_k(C)$. Consequently, $\Gamma_k(C)$ always contains loop of even length. 
\end{remark}

\begin{remark}\label{r.mixing} A connected graph $\Gamma$ is not mixing if and only if there exists an integer $m\geq 2$ such that the lengths of all of its loops 
are multiples of $m$.
\end{remark}

The next lemma deals with Step 0 of the strategy of proof of Theorem \ref{t.MMY-twisting}: 

\begin{lemma}\label{l.step0} Let $C\in Sp(2d,\mathbb{R})$. Then, each $\underline{\lambda}\in\widehat{R}_k$ is the start of at least one arrow of $\Gamma_k(C)$.
\end{lemma}

\begin{remark} The symplecticity of $C$ is really used in this lemma: the analogous statement for general invertible (i.e., $GL$) matrices is false.
\end{remark}

\begin{proof} Every $1$-dimensional subspace is isotropic. Hence, $\widehat{R}_1=\widetilde{R}$ and the lemma follows in the case $k=1$ from the invertibility of $C$. 

So, we can assume that $k\geq 2$. The invertibility of $C$ ensures that, for each $\underline{\lambda}\in\widehat{R}_k$, there exists $\underline{\lambda}'\in\widetilde{R}_k$ with $C^{(k)}_{\underline{\lambda}, \underline{\lambda}'}\neq 0$.
If $\underline{\lambda}'\in\widehat{R}_k$, we are done. If $\underline{\lambda}'\in\widetilde{R}_k - \widehat{R}_k$, i.e., $\# p(\underline{\lambda}')<k$, our task is to ``convert'' $\underline{\lambda}'$ into some $\underline{\lambda}''\in\widehat{R}_k$ with $C^{(k)}_{\underline{\lambda}, \underline{\lambda}''}\neq 0$. For this sake, it suffices to prove that if $\#p(\underline{\lambda}')<k$ and $C^{(k)}_{\underline{\lambda}, \underline{\lambda}'}\neq 0$, then there exists $\underline{\lambda}''$ with 
$C^{(k)}_{\underline{\lambda}, \underline{\lambda}''}\neq 0$ and $\#p(\underline{\lambda}'')=\#p(\underline{\lambda}')+1$. 

Keeping this goal in mind, note that if $\underline{\lambda}'\notin\widehat{R}_k$, then we can write $\underline{\lambda}'=\{\lambda_1',\lambda_2',\dots,\lambda_k'\}$ with $\lambda_1'\cdot \lambda_2'=1$. Also, $C^{(k)}_{\underline{\lambda}, \underline{\lambda}'}\neq0$ if and only if the $k\times k$ minor of $C$ associated to $\underline{\lambda}$ and $\underline{\lambda}'$ is invertible. 

Hence, we can choose bases to convert invertible minors of $C$ into the $k\times k$ identity matrix: if $\underline{\lambda}=\{\lambda_1,\lambda_2,\dots,\lambda_k\}$, then we can find $w_1,\dots, w_k\in\mathbb{R}^{2d}$ such that $\textrm{span}\{w_1,\dots,w_k\}=\textrm{span}\{v_{\lambda_1},\dots, v_{\lambda_k}\}$ and 
$$
C(w_i)=v_{\lambda_i'}+\sum\limits_{\lambda\notin\underline{\lambda}'}C_{i\lambda}^* v_{\lambda}\,.
$$
 
Denote by $\{,.\}$ the symplectic form. Observe that $\{w_1, w_2\}=0$ because $w_1$ and $w_2$ belong to the span $v_{\lambda_i}$ (an isotropic subspace as $\underline{\lambda}\in\widehat{R}_k$). Thus, the simplecticity of $C$ implies that 
$$0=\{w_1,w_2\}=\{C(w_1), C(w_2)\} = \{v_{\lambda_1'}, v_{\lambda_2'}\} + \sum_{\substack{\lambda', \lambda''\notin \underline{\lambda}' \\ \lambda'\cdot\lambda''=1}} C_{1\lambda'}^* C_{2\lambda''}^* \{v_{\lambda'}, v_{\lambda''}\}\,.$$
Since $\{v_{\lambda_1'}, v_{\lambda_2'}\}\neq 0$ (as $\lambda_1'\cdot\lambda_2'=1$), it follows that  $C^*_{1\lambda'}\neq 0$ and $C_{2\lambda''}^*\neq 0$ for some $\lambda', \lambda''\notin \underline{\lambda}'$. 

This allows us to define $\underline{\lambda}'':=(\underline{\lambda}'-\{\lambda_1'\})\cup\{\lambda'\}$. Note that $\#p(\underline{\lambda}'')=\#p(\underline{\lambda}')+1$ and the minor $C[\underline{\lambda}, \underline{\lambda}'']$ of $C$ associated to $\underline{\lambda}$ and $\underline{\lambda}''$ is obtained from the minor $C[\underline{\lambda}, \underline{\lambda}']$ by replacing the line associated to $v_{\lambda_1'}$ with the line associated to $v_{\lambda'}$. In the basis $w_1,\dots, w_k$, the minor $C[\underline{\lambda}, \underline{\lambda}'']$ differs from the identity minor $C[\underline{\lambda}, \underline{\lambda}']$ precisely by the replacement of the line associated to $v_{\lambda_1'}$ by the line associated to $v_{\lambda'}$, that is, one of the entries $1$ of  $C[\underline{\lambda}, \underline{\lambda}']$ was replaced by the coefficient $C^*_{1\lambda'}\neq 0$. Thus, the determinant $C^{(k)}_{\underline{\lambda}, \underline{\lambda}''}$ of the the minor $C[\underline{\lambda}, \underline{\lambda}'']$ is
$$C^{(k)}_{\underline{\lambda}, \underline{\lambda}''} = C^*_{1\lambda'}\neq 0\,.$$
Therefore, $\underline{\lambda}''$ has the desired properties. This completes the proof of the lemma.
\end{proof}

Let us now discuss Step 1 in the strategy of proof of Theorem \ref{t.MMY-twisting}.  

\begin{lemma}\label{l.step1} Let $A\in Sp(2d,\mathbb{Z})$ be a Galois-pinching matrix. Suppose that $B\in Sp(2d,\mathbb{Z})$ has the property that $A$ and $B^2$ share no common proper invariant subspace. Then, $\Gamma_1(B)$ is mixing.
\end{lemma}

\begin{proof} In the sequel, our figures are drawn with the convention that two points inside the same ellipse represent a pair of eigenvalues of $A$ of the form $\lambda$, $\lambda^{-1}$.

For $d=1$, the set $\widehat{R}_1$ consists of exactly one pair $\widehat{R}_1=\{\lambda,\lambda^{-1}\}$. Hence, all possible {\em Galois invariant} graphs are described in Figure \ref{f.spcs18}.

\begin{figure}[htb!]
\includegraphics[scale=0.8]{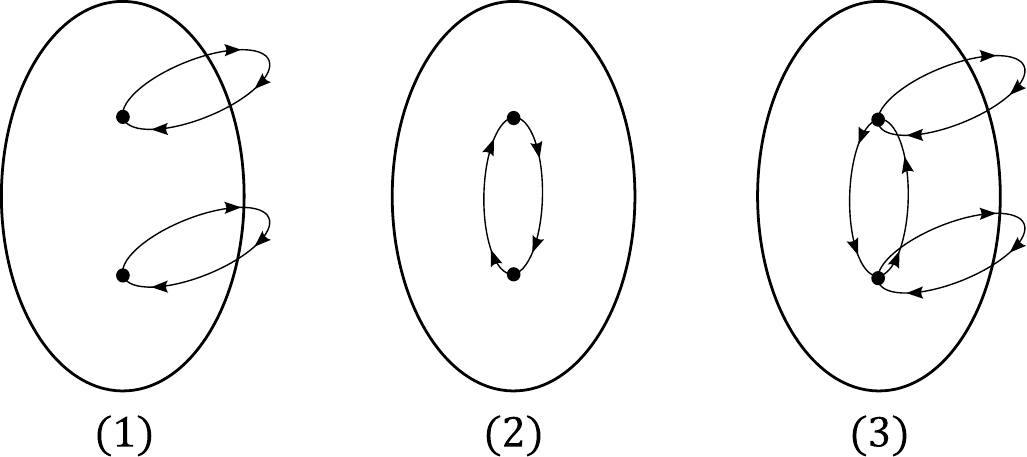}
\caption{All Galois-invariant graphs in the case $d=1$, $k=1$.}\label{f.spcs18}
\end{figure}

In the first situation, by definition, we have that $B(\mathbb{R}v_{\lambda})=\mathbb{R}v_{\lambda}$ (and $B(\mathbb{R}v_{\lambda^{-1}}) = \mathbb{R}v_{\lambda^{-1}}$). Thus,  
$B$ and $A$ share a common invariant subspace, a contradiction with our assumptions. 

In the second situation, by definition, we have that $B(\mathbb{R}v_{\lambda})=\mathbb{R}v_{\lambda^{-1}}$ and $B(\mathbb{R}v_{\lambda^{-1}}) = \mathbb{R}v_{\lambda}$. So,  $B^2(\mathbb{R}v_{\lambda})=\mathbb{R}v_{\lambda}$ and thus $B^2$ and $A$ share a common invariant subspace, a contradiction with our standing hypothesis. 

In the third situation, we have that the graph $\Gamma_1(B)$ is {\em complete}, and, \emph{a fortiori}, mixing. 

This establishes the case $d=1$ of the lemma. After this ``warm up'', let us investigate the general case $d\geq 2$. 

First, let us assume that {\em all} arrows in $\Gamma_1(B)$ have the form $\lambda\to\lambda^{\pm1}$. Then, 
$B(\mathbb{R}v_{\lambda}\oplus\mathbb{R}v_{\lambda^{-1}}) = \mathbb{R}v_{\lambda}\oplus\mathbb{R}v_{\lambda^{-1}}$, and, since\footnote{Of course, this arguments breaks up for $d=1$ and this is why we had a separate argument for this case.} $d\geq 2$, the subspace 
$\mathbb{R}v_{\lambda}\oplus\mathbb{R}v_{\lambda^{-1}}$ is {\em non-trivial}. Thus, in this case, $B$ and $A$ share a common non-trivial subspace, a contradiction with our assumptions. 

Therefore, we can assume (without loss of generality) that $\Gamma_1(B)$ contains an arrow 
$\lambda\to\lambda'$ with $\lambda'\neq\lambda^{\pm1}$. By Remarks~\ref{r.galois-invariance} and~\ref{r.mixing}, this implies that {\em all} arrows of this type belong to $\Gamma_1(B)$ and, moreover, $\Gamma_1(B)$ is mixing whenever it contains a loop of {\em odd} length. Hence, our task is reduced to exhibit loops of odd length in $\Gamma_1(B)$.

It is not hard to construct a loop of length $3$ in $\Gamma_1(B)$ when $d\geq 3$: indeed, this follows immediately from the presence of all arrows $\lambda\to\lambda'$ with $\lambda'\neq\lambda^{\pm1}$ in $\Gamma_1(B)$ (cf. Figure \ref{f.spcs19}).
\begin{figure}[htb!]
\includegraphics[scale=0.7]{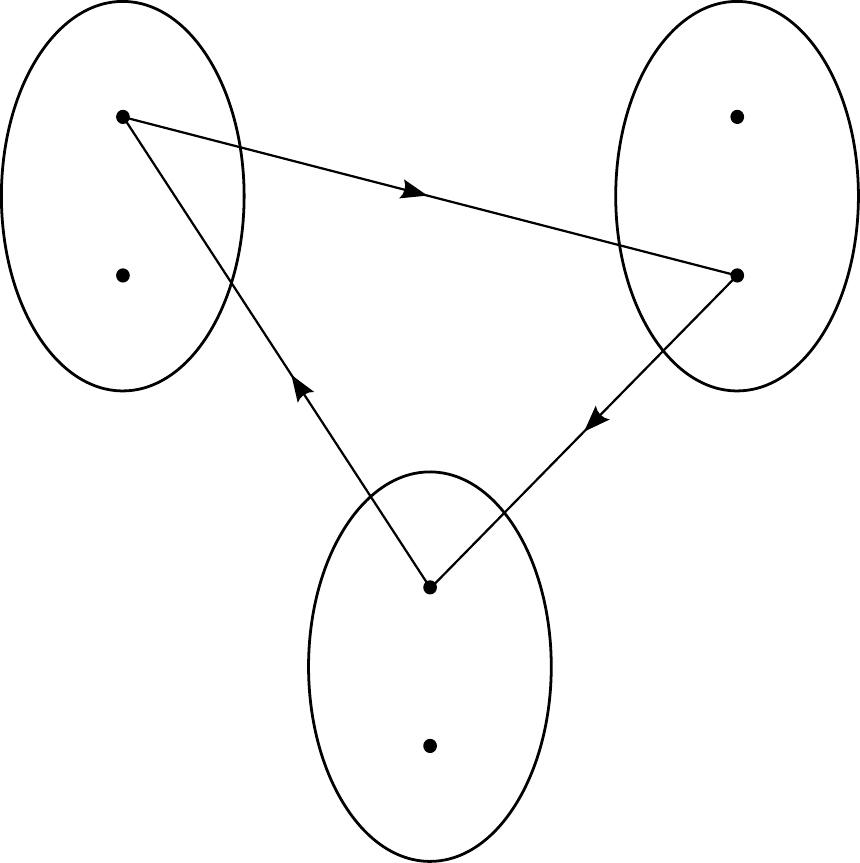}
\caption{A loop of length three in $\Gamma_1(B)$ in the case $d\geq 3$.}\label{f.spcs19}
\end{figure}

On the other hand, for $d=2$, we need a different argument. If $\Gamma_1(B)$ were the non-mixing graph Galois-invariant depicted in Figure \ref{f.spcs20}, we would have $B(\mathbb{R}v_{\lambda_1}\oplus \mathbb{R}v_{\lambda_1}^{-1}) = \mathbb{R}v_{\lambda_2}\oplus \mathbb{R}v_{\lambda_2}^{-1}$ and $B(\mathbb{R}v_{\lambda_2}\oplus \mathbb{R}v_{\lambda_2}^{-1}) = \mathbb{R}v_{\lambda_1}\oplus \mathbb{R}v_{\lambda_1}^{-1}$. So, $B^2$ and $A$ would share a common invariant subspace, a contradiction with our hypothesis.
\begin{figure}[htb!]
\includegraphics[scale=0.75]{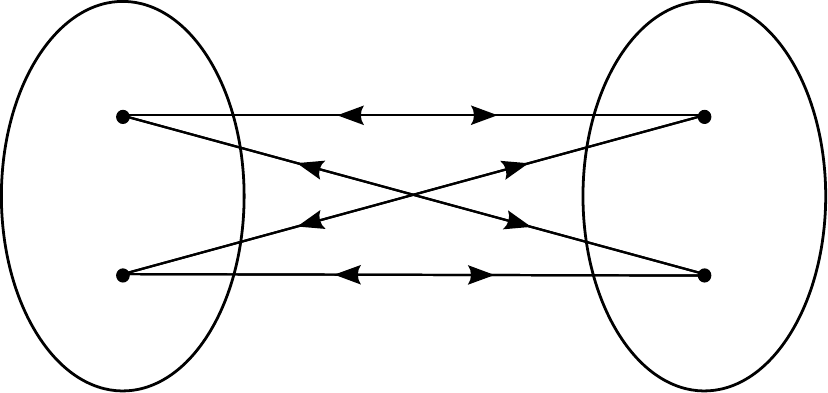}
\caption{Non-mixing Galois-invariant graph $\Gamma_1(B)$ in the case $d=2$.}\label{f.spcs20}
\end{figure}

Thus, $\Gamma_1(B)$ must contain the graph of Figure \ref{f.spcs20} \emph{and} an extra arrow. In this situation, it is not hard to build up loops of lenght $3$. 

In any case, we proved that, under our assumptions, $\Gamma_1(B)$ contains a loop of length $3$ whenever $d\geq 2$. This completes the proof of the lemma.
\end{proof}

We are now ready to implement Step 2 of the strategy of proof of Theorem \ref{t.MMY-twisting}. The first part of this step is implemented by the following lemma: 

\begin{lemma}\label{l.step2i} Let $A\in Sp(2d,\mathbb{Z})$ be a Galois-pinching matrix. Suppose that $d\geq 3$ and $C\in Sp(2d,\mathbb{Z})$ is a matrix such that $\Gamma_1(C)$ is complete. Then, $\Gamma_k(C)$ is mixing for all  $2\leq k<d$. 
\end{lemma} 

\begin{proof} Since $\Gamma_k(C)$ is invariant under the Galois group $Gal$ (see Remark~\ref{r.galois-invariance}), it consists of a certain number of orbits of the action of $Gal$ on $\widehat{R}_k\times \widehat{R}_k$. 

As it turns out, it is not hard to check that all $Gal$-orbits on $\widehat{R}_k\times \widehat{R}_k$ have the form 
$$\mathcal{O}_{\widetilde{\ell}, \ell} = \{(\underline{\lambda}, \underline{\lambda}')\in \widehat{R}_k\times \widehat{R}_k: \#(\underline{\lambda}\cap \underline{\lambda}')=\widetilde{\ell}, \#(p(\underline{\lambda}) \cap p(\underline{\lambda}'))=\ell\}\,,$$
where $0\leq\widetilde{\ell}\leq\ell\leq k$ and $\ell\geq 2k-d$. In particular, the $Gal$-orbits on $\widehat{R}_k\times \widehat{R}_k$ are naturally parametrized by the set 
$$\widetilde{I}=\{(\widetilde{\ell}, \ell): 0\leq\widetilde{\ell}\leq\ell\leq k \textrm{ and } \ell\geq 2k-d \}.$$

It follows that $\Gamma_k(C)=\Gamma_k(\widetilde{J})$ for some $\widetilde{J}:=\widetilde{J}(C)\subset\widetilde{I}$, where $\Gamma_k(\widetilde{J})$ is the graph whose vertices are $\widehat{R}_k$ and whose arrows are 
$$\bigcup\limits_{(\widetilde{\ell}, \ell)\in\widetilde{J}} \mathcal{O}_{\widetilde{\ell}, \ell}\,.$$

We affirm that $\Gamma_k(\widetilde{J})$ is not mixing if and only if 
\begin{itemize}
\item either $k\neq d/2$ and $\widetilde{J}\subset\{(\widetilde{\ell},k): 0\leq \widetilde{\ell}\leq k\}$\,,
\item or $k=d/2$ and $\widetilde{J}\subset\{(\widetilde{\ell},k): 0\leq \widetilde{\ell}\leq k\}\cup\{(0,0)\}$\,.
\end{itemize}
Indeed, suppose that $\widetilde{J}\subset\{(\widetilde{\ell},k): 0\leq\widetilde{\ell}\leq k\}$ for $k\neq d/2$ or $\widetilde{J}\subset\{(\widetilde{\ell},k): 0\leq\widetilde{\ell}\leq k\}\cup\{(0,0)\}$ for $k=d/2$. Then, {\em since} $k<d$, one can show that $\Gamma_k(\widetilde{J})$ is not mixing simply because it is not {\em connected}! For the reciprocal statement, one proceeds as follows (cf. \cite[Proposition 4.19]{MMY} for more details): first, one converts pairs $\{\lambda,\lambda^{-1}\}$ into a single point $p(\lambda)=p(\lambda^{-1})=\lambda+\lambda^{-1}$, so that $\Gamma_k(\widetilde{J})$ becomes a new graph $\overline{\Gamma}_k(\widetilde{J})$; second, one proves that $\overline{\Gamma}_k(\widetilde{J})$ is connected whenever $\widetilde{J}\not\subset \{(\widetilde{\ell},k): 0\leq \widetilde{\ell}\leq k\}\cup\{(0,0)\}$; from the connectedness of $\overline{\Gamma}_k(\widetilde{J})$ one can prove that $\Gamma_k(\widetilde{J})$ is connected; since the connectedness of $\Gamma_k(\widetilde{J})$ allows us to  construct loops of odd length, one obtains (from Remarks \ref{r.galois-invariance} and \ref{r.mixing}) that $\Gamma_k(\widetilde{J})$ is mixing whenever $\widetilde{J}\not\subset \{(\widetilde{\ell},k): 0\leq \widetilde{\ell}\leq k\}\cup\{(0,0)\}$ and this concludes the argument. 

Once we dispose of this characterization of the mixing property for $\Gamma_k(C)$, we can complete the proof of the lemma as follows. Suppose that $\Gamma_1(C)$ is complete but $\Gamma_k(C)$ is not mixing for some $2\leq k<d$ (where $d\geq 3$). The discussion of the previous paragraph implies that 
$$\Gamma_k(C) = \Gamma_k(\widetilde{J})$$ 
for some $\widetilde{J}\subset \{(\widetilde{\ell}, k): 0\leq \widetilde{\ell}\leq k\}$ for $k\neq d/2$ or $\widetilde{J}\subset\{(\widetilde{\ell},k): 0\leq\widetilde{\ell}\leq k\}\cup\{(0,0)\}$ for $k=d/2$.  For the sake of concreteness, we shall discuss only\footnote{The particular case 
$\widetilde{J}=\{(0,0)\}$ when $k=d/2$ is left as an exercise to the reader (see also \cite{MMY}).} the case $\widetilde{J}\subset\{(\widetilde{\ell},k): 0\leq\widetilde{\ell}\leq k\}\cup\{(0,0)\}$. In this situation, there is an arrow $\{\lambda_1, \dots, \lambda_k\}=\underline{\lambda}\to\underline{\lambda}'=\{\lambda_1', \dots, \lambda_k'\}$ of $\Gamma_k(C)$ with $p(\underline{\lambda})=p(\underline{\lambda}')$. This means that $C^{(k)}_{\underline{\lambda}, \underline{\lambda}'}\neq 0$, and hence we can find $w_1, \dots, w_k$ such that $\textrm{span}\{w_1,\dots,w_k\} = \textrm{span}\{v_{\lambda_1}, \dots, v_{\lambda_k}\}$ and 
$$C(w_i)=v_{\lambda_i'}+\sum\limits_{\lambda\notin\underline{\lambda}'} 
C_{i\lambda}^* v_{\lambda}\,.$$
In other words, as we also did in Step 0 (cf. the proof of Lemma \ref{l.step0}), we can use $w_1,\dots, w_k$ to ``convert'' the minor of $C$ associated to $\underline{\lambda}, \underline{\lambda}'$ into the identity. 

We claim that if $\lambda, \lambda^{-1}\notin \underline{\lambda}'$, then $C_{i\lambda}^*=0$ for all $i=1,\dots,k$. Indeed, the same arguments with minors and replacement of lines yield that if this were not true, say $C_{i\lambda}^*\neq 0$, then there would be an arrow from $\underline{\lambda}$ to $\underline{\lambda}''=(\underline{\lambda}'-\{\lambda_i'\})\cup\{\lambda\}$. Since $p(\underline{\lambda})=p(\underline{\lambda}')$, we would have $\#(p(\underline{\lambda})\cap p(\underline{\lambda}''))=k-1$, and, thus, $(\widetilde{\ell}_0, k-1)\in \widetilde{J}\subset \{(\widetilde{\ell}, k): 0\leq \widetilde{\ell}\leq k\}$ for some $\widetilde{\ell}_0$, a contradiction (proving the claim). 

This claim permits to complete the proof of the lemma: in fact, it implies that $C(v_{\lambda_1})$ is a linear combination of $v_{\lambda_i'}$, $1\leq i\leq k<d$, a contradiction with the completeness hypothesis on $\Gamma_1(C)$.
\end{proof}

The second part of Step 2 of the strategy of proof of Theorem \ref{t.MMY-twisting} is the following lemma:

\begin{lemma}\label{l.step2ii} Let $A\in Sp(2d,\mathbb{Z})$ be a Galois-pinching matrix. Suppose that $d\geq 3$ and $D\in Sp(2d,\mathbb{Z})$ is a matrix such that $\Gamma_k(D)$ is complete for all $1\leq k<d$. Then, $\Gamma_d(D)$ is mixing.
\end{lemma}

\begin{proof}
Recall that the $Gal$-orbits on $\widehat{R}_k\times\widehat{R}_k$ are 
$$\mathcal{O}_{\widetilde{\ell},\ell} = \{(\underline{\lambda}, \underline{\lambda}')\in \widehat{R}_k\times \widehat{R}_k: \#(\underline{\lambda}\cap \underline{\lambda}')=\widetilde{\ell}, \#(p(\underline{\lambda}) \cap p(\underline{\lambda}'))=\ell\}$$
with $\ell\leq k$ and $\ell\geq 2k-d$. 

Hence, in the case $k=d$, these orbits are parametrized by the set $I=\{0\leq\widetilde{\ell}\leq d\}$. For the sake of simplicity, let us denote the $Gal$-orbits on $\widehat{R}_d\times\widehat{R}_d$ by 
$$\mathcal{O}(\widetilde{\ell})=\{(\underline{\lambda}, \underline{\lambda}')\in \widehat{R}_d\times \widehat{R}_d: 
\#(\underline{\lambda}\cap \underline{\lambda}')=\widetilde{\ell}\}$$
and let us write 
$$\Gamma_d(D)=\Gamma_d(J)=\bigcup\limits_{\widetilde{\ell}\in J}\mathcal{O}(\widetilde{\ell})\,,$$
where $J=J(D)\subset I=\{0\leq\widetilde{\ell}\leq d\}$.

It is possible to show\footnote{Again by the arguments with ``minors'' described above.} that if $\Gamma_k(D)$ is complete for each $1\leq k < d$, then $J$ contains two consecutive 
integers, say $\widetilde{\ell}$ and $\widetilde{\ell}+1$ (see \cite{MMY} for more details).  

Thus, our task is reduced to prove that $\Gamma_d(D)=\Gamma_d(J)$ is mixing when $J\supset\{\widetilde{\ell},  \widetilde{\ell}+1\}$. 

In this direction, we establish first the \emph{connectedness} of $\Gamma_d(J)$. For this sake, note that it suffices\footnote{The general case of two general vertices $\underline{\lambda}$ and $\underline{\lambda}'$ follows by producing a series of vertices $\underline{\lambda}=\underline{\lambda}_0$, $\underline{\lambda}_1$, $\dots$, $\underline{\lambda}_a=\underline{\lambda}'$ with $\#(\underline{\lambda}_i\cap \underline{\lambda}_{i+1})=d-1$ for $i=0,\dots, a-1$.} to connect two vertices $\underline{\lambda}_0$ and $\underline{\lambda}_1$ with $\#(\underline{\lambda}_0\cap\underline{\lambda}_1)=d-1$. Given such $\underline{\lambda}_0$ and $\underline{\lambda}_1$, let us select $\underline{\lambda}'\subset \underline{\lambda}_0\cap \underline{\lambda}_1$ with $\#\underline{\lambda}'=d-\widetilde{\ell}-1$, and let us consider $\underline{\lambda}''$ obtained 
from $\underline{\lambda}_0$ by replacing the elements of $\underline{\lambda}'$ by their inverses. Since $\#(\underline{\lambda}_0\cap\underline{\lambda}_1)=d-1$, we have  $\#(\underline{\lambda}''\cap\underline{\lambda}_0) = \widetilde{\ell}+1$ and $\#(\underline{\lambda}''\cap\underline{\lambda}_1)=\widetilde{\ell}$. 
From our assumption $J\supset\{\widetilde{\ell}, \widetilde{\ell}+1\}$, we infer the presence of the arrows $\underline{\lambda}_0\to\underline{\lambda}''$ and 
$\underline{\lambda}''\to\underline{\lambda}_1$ in $\Gamma_d(J)$. This proves the connectedness of $\Gamma_d(J)$. 

Next, we show that $\Gamma_d(J)$ is mixing. The Galois invariance of $\Gamma_d(J)$ guarantees that it contains loops of length $2$ (see Remark~\ref{r.galois-invariance}). So, it suffices to exhibit some loop of {\em odd} length in $\Gamma_d(J)$ (see Remark~\ref{r.mixing}). For this sake, let us fix an arrow 
$\underline{\lambda}\to\underline{\lambda}'\in\mathcal{O}(\widetilde{\ell})$ of $\Gamma_d(J)$. The construction ``$\underline{\lambda}_0\to\underline{\lambda}''\to\underline{\lambda}_1$ for $\#(\underline{\lambda}_0\cap\underline{\lambda}_1)=d-1$'' performed in the previous paragraph allows us to connect $\underline{\lambda}'$ to $\underline{\lambda}$ by a path of length $2\widetilde{\ell}$ in $\Gamma_d(J)$. In this way, 
we get a loop (based on $\underline{\lambda}_0$) in $\Gamma_d(J)$ of length $2\widetilde{\ell}+1$. This proves the lemma. 
\end{proof}

Finally, let us comment on Step 3 of the strategy of proof of Theorem \ref{t.MMY-twisting}, i.e, the special case $d=2$ of this theorem. Consider a symplectic form $\{.,.\}:\bigwedge^2\mathbb{R}^4\to\mathbb{R}$. Since $\bigwedge^2\mathbb{R}^4$ has dimension $6$ and $\{.,.\}$ is non-degenerate, 
$K:=\textrm{Ker}\{.,.\}$ has dimension $5$. 

By denoting by $\lambda_1> \lambda_2>\lambda_2^{-1}> \lambda_1^{-1}$ the eigenvalues of a Galois-pinching matrix $A$, we have the following basis of $K$
\begin{itemize}
\item $v_{\lambda_1}\wedge v_{\lambda_2}$, $v_{\lambda_1}\wedge v_{\lambda_2^{-1}}$, $v_{\lambda_1^{-1}}\wedge v_{\lambda_2}$, 
$v_{\lambda_1^{-1}}\wedge v_{\lambda_2^{-1}}$;
\item $v_*=\frac{v_{\lambda_1}\wedge v_{\lambda_1^{-1}}}{\omega_1} - \frac{v_{\lambda_2}\wedge v_{\lambda_2^{-1}}}{\omega_2}$ where $\omega_i = \{v_{\lambda_i}, v_{\lambda_i^{-1}}\}\neq 0$.
\end{itemize}

In general, given $C\in Sp(4,\mathbb{Z})$, we can use $\bigwedge^2C|_K$ to construct a graph $\Gamma_2^*(C)$ whose vertices are 
$\widehat{R}_2\simeq \{ v_{\lambda_1}\wedge v_{\lambda_2}, \dots, v_{\lambda_1^{-1}}\wedge v_{\lambda_2^{-1}}\}$ and $v_*$, and whose arrows connect vertices associated to non-zero entries of $\bigwedge^2C|_K$. 

\begin{remark}\label{r.step3} By definition, $\bigwedge^2A(v_*)=v_*$, so that $1$ is an eigenvalue of $\bigwedge^2A|_K$. In principle, this poses a problem to derive the analog of  Proposition~\ref{p.mixing/twisting} with $\Gamma_2(C)$ replaced by $\Gamma_2^*(C)$. But, as it turns out, the \emph{simplicity} of the eigenvalue $1$ of $\wedge^2A|_K$ can be exploited to rework the proof of Proposition~\ref{p.mixing/twisting} to check that if $\Gamma_2^*(C)$ is mixing, then there are adequate products $D$ of $C$ and powers of $A$ such that $\Gamma_2(D)$ is complete (cf. Proposition 4.26 in \cite{MMY}). 
\end{remark}

In this setting, the third step of our strategy of proof of Theorem \ref{t.MMY-twisting} amounts to show that:
\begin{lemma}\label{l.step3} Let $A\in Sp(4,\mathbb{Z})$ be a Galois-pinching matrix. Suppose that $C\in Sp(4,\mathbb{Z})$ is a matrix such that $\Gamma_1(C)$ is complete. Then, either $\Gamma_2(C)$ or $\Gamma_2^*(C)$ is mixing. 
\end{lemma}

\begin{proof} We write $\Gamma_2(C)=\Gamma_2(J)$ with $J\subset\{0, 1, 2\}$. If $J$ contains two consecutive integers, then the arguments from the proof of Lemma \ref{l.step2ii} can be used to show that  $\Gamma_2(C)$ is mixing.

Hence, we can assume (without loss of generality) that $J$ does not contain two consecutive integers. Since $J\neq\emptyset$ (see Lemma \ref{l.step0}, i.e., Step 0), this means that $J=\{0\}$, $\{2\}, \{1\}$ or $\{0,2\}$. As it turns out, the cases $J=\{0\}, \{2\}$ are ``symmetric'', as well as the cases $J=\{1\}, \{0,2\}$. 

For the sake of exposition, we will deal only\footnote{The  ``symmetric'' cases are left as an exercise for the reader.} with the cases $J=\{2\}$ and $J=\{1\}$. We will show that $J=\{2\}$ is impossible, while $J=\{1\}$ implies that $\Gamma_2^*(C)$ is mixing. 

We begin by $J=\{2\}$. This means that we have an arrow $\underline{\lambda}\to\underline{\lambda}$ with $\underline{\lambda}= \{\lambda_1,\lambda_2\}$. So, we can find $w_1, w_2$ with $\textrm{span}\{w_1,w_2\}=\textrm{span}\{v_{\lambda_1}, v_{\lambda_2}\}$ and 
$$C(w_1) = v_{\lambda_1}+ C^*_{11} v_{\lambda_1^{-1}} + C^*_{12} v_{\lambda_2^{-1}}\,, $$
$$C(w_2) = v_{\lambda_2}+ C^*_{21} v_{\lambda_1^{-1}} + C^*_{22} v_{\lambda_2^{-1}}\,.$$
Since $J=\{2\}$, the arrows $\underline{\lambda}\to\{\lambda_1^{-1}, \lambda_2\}$, $\underline{\lambda}\to\{\lambda_1, \lambda_2^{-1}\}$, and 
$\underline{\lambda}\to\{\lambda_1^{-1}, \lambda_2^{-1}\}$ do not belong $\Gamma_2(J)$. Thus, $C_{11}^*=C_{22}^*=0=C_{12}^*C_{21}^*$. On the other hand, because $C$ is symplectic, $\omega_1 C_{21}^* - \omega_2 C_{12}^* = 0$ (with $\omega_1,\omega_2\neq 0$). It follows that $C^*_{ij}=0$ for all $1\leq i,j\leq 2$, that is, $C$ preserves 
the $A$-invariant subspace $\textrm{span}\{v_{\lambda_1}, v_{\lambda_2}\}$. By exploiting this fact\footnote{As explained in the proof of Proposition 4.27 in \cite{MMY}.}\label{footnote-twisting}, one reaches a contradiction with our assumption of completeness of $\Gamma_1(C)$. Thus, $J=\{2\}$ is impossible. 

Suppose now that $J=\{1\}$ and let us show that $\Gamma_2^*(C)$ is mixing. For this sake, we claim that, in this situation, it suffices to construct arrows from the vertex $v_*$ to $\widehat{R}_2$ {\em and}\footnote{Notice that the action of the Galois group can \emph{not} be used to revert arrows of $\Gamma_2^*(C)$ involving the vertex $v_*$, so that the two previous statements are \emph{independent}.} vice-versa. Assuming the claim, we can use the Galois action to see that once $\Gamma_2^*(C)$ 
contains {\em some} arrows from $v_*$ and {\em some} arrows to $v_*$, it contains {\em all such arrows}. In other words, if the claim is true, we have the situation depicted in Figure \ref{f.spcs22}.
\begin{figure}[htb!]
\includegraphics[scale=0.8]{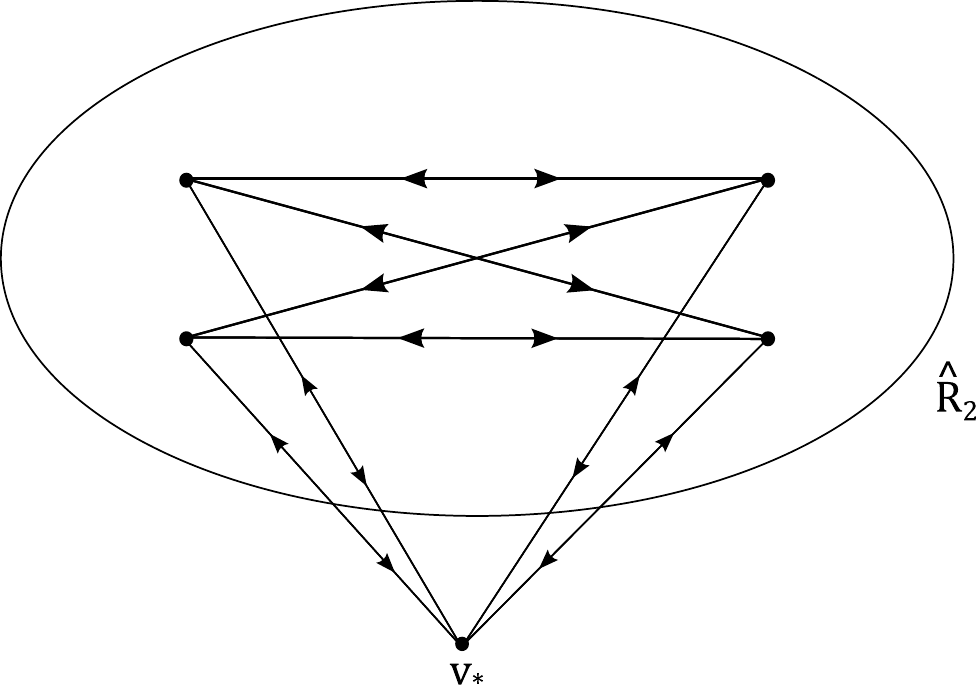}
\caption{The graph $\Gamma_2^*(C)$ when $J=\{1\}$.}\label{f.spcs22}
\end{figure}
Thus, we have loops of length $2$ (in $\widehat{R}_2$), and also loops of length $3$ (based on $v_*$), so that $\Gamma_2^*(C)$ is mixing (cf. Remark \ref{r.mixing}). Hence, it remains only to show the claim. 

The existence of arrows from $\widehat{R}_2$ to $v_*$ follows from the same kind of arguments involving ``minors'' above (i.e., selecting $w_1, w_2$ as above, etc.) and we will not repeat it here. 

Instead, we focus on showing that there are arrows from $v_*$ to $\widehat{R}_2$. The proof is by contradiction: otherwise, one would have $\bigwedge^2C(v_*)\in\mathbb{R}v_*$. 

We want to use this information to determine the image of $\textrm{span}(v_{\lambda_1}, v_{\lambda_1^{-1}})$ under $C$. Here, the following elementary (linear algebra) fact is useful. Given $H$ is a symplectic $2$-plane on $\mathbb{R}^4$ equipped with the standard symplectic form $\{.,.\}$, let $V(H):=e\wedge f - g\wedge h\in\bigwedge^2\mathbb{R}^4$ where $e, f$ is a basis of $H$ and $g,h$ is a basis of the symplectic orthogonal $H^{\perp}$ such that $\{e,f\}=\{g,h\}=1$. Then, $V(H)$ is well-defined (independently on the choices) and $V(H)$ is collinear to $V(H')$ if and only if $H'=H$ or $H^{\perp}$. (See, e.g., \cite[Proposition 4.28]{MMY} for a proof of this fact.)  

Since $v_*:=V(\textrm{span}(v_{\lambda_1}, v_{\lambda_1^{-1}}))$, the fact stated above says that if $\bigwedge^2C(v_*)\in\mathbb{R}v_*$, then 
$$C(\textrm{span}(v_{\lambda_1}, v_{\lambda_1^{-1}})) = \textrm{span}(v_{\lambda_1}, 
v_{\lambda_1^{-1}}) \quad \textrm{ or } \quad \textrm{span}(v_{\lambda_2}, v_{\lambda_2^{-1}}).$$ 
Once again, an argument using ``minors'' says that this is a contradiction with our assumption that $\Gamma_1(C)$ is complete. This completes the proof of the lemma. 
\end{proof}

At this point, the proof of Theorem \ref{t.MMY-twisting} is complete (compare with Steps 0, 1, 2 and 3 of the strategy of proof outlined above). Indeed, the case $d=1$ follows from Lemma \ref{l.step1} and Proposition \ref{p.mixing/twisting}. The generic case $d\geq 3$ follows from Lemmas \ref{l.step1}, \ref{l.step2i}, \ref{l.step2ii} and Proposition \ref{p.mixing/twisting}. Finally, the special case $d=2$ follows from Lemmas \ref{l.step1}, \ref{l.step3}, Proposition \ref{p.mixing/twisting} and Remark \ref{r.step3}.

\subsubsection{Two simplicity criteria for the Lyapunov exponents of origamis}

We are now in position to discuss the simplicity of Lyapunov exponents of square-tiled surfaces. 

Let $\pi: X=(M,\omega)\to (\mathbb{T}^2, dz)$ be a reduced\footnote{This means that the group of relative periods of $X$ is $\mathbb{Z}\oplus i\mathbb{Z}$.} square-tiled surface of genus $g\geq 1$ with a trivial group $\textrm{Aut}(X)=\{\textrm{Id}\}$ of automorphisms. In this case, the group $\textrm{Aff}(X)$ of affine homeomorphisms of $X$ is naturally isomorphic to a finite-index subgroup of $SL(2,\mathbb{Z})$, namely, the Veech group $SL(X)$ of $X$ (i.e., the stabilizer of $X$ in $SL(2,\mathbb{R})$). 

The action of $\textrm{Aff}(X)$ on $H_1(X,\mathbb{R})$ respects the decomposition 
$$H_1(X,\mathbb{R}) = H_1^{st}(X,\mathbb{R})\oplus H_1^{(0)}(X,\mathbb{R})$$
where $H_1^{(0)}(X,\mathbb{R})=\textrm{ker}(\pi_*)$ is the annihilator of $\mathbb{R}\cdot\textrm{Re}(\omega) \oplus \mathbb{R}\cdot \textrm{Im}(\omega)\subset H^1(X,\mathbb{R})$ and $H_1^{st}(X,\mathbb{R})$ is the symplectic orthogonal of $H_1^{(0)}(X,\mathbb{R})$. Note that this decomposition is defined over $\mathbb{Q}$ (because $X$ is a square-tiled surface) and, hence, the elements of $\textrm{Aff}(X)$ act on $H_1^{(0)}(X,\mathbb{R})$ via symplectic matrices in $Sp(H_1^{(0)}(X,\mathbb{Z}))\simeq Sp(2g-2,\mathbb{Z})$. 

Our discussion of the Galois-pinching and twisting properties give the following two simplicity criteria for the Lyapunov exponents of square-tiled surfaces (cf. Theorems 1.1 and 5.4 in \cite{MMY}).

\begin{theorem}\label{t.MMY-1} Let $X=(M,\omega)$ be a reduced square-tiled surface of genus $g\geq 1$ with trivial group of automorphisms. Suppose that $\phi$ and $\psi$ are affine homeomorphisms of $X$ whose actions on $H_1^{(0)}(X,\mathbb{R})$ are given by matrices $A$ and $B$ in $Sp(2g-2,\mathbb{Z})$ such that: 
\begin{itemize} 
\item $A$ is Galois-pinching, and 
\item $B\neq\textrm{Id}$ is unipotent and $(B-\textrm{Id})(H_1^{(0)}(X,\mathbb{R}))$ is not a Lagrangian subspace.
\end{itemize}
Then, the Lyapunov exponents of the KZ cocycle with respect to the unique $SL(2,\mathbb{R})$-invariant probability measure support on $SL(2,\mathbb{R})X$ are simple.
\end{theorem} 

\begin{proof}
By Proposition \ref{p.Galois-pinching}, the matrix $A=\phi|_{H_1^{(0)}(X,\mathbb{R})}$ is pinching. By Proposition \ref{p.unipotent-twisting} and Theorem \ref{t.MMY-twisting}, there is a product $\widetilde{\psi}$ of powers of $\phi$ and $\psi$ such that the matrix $C=\widetilde{\psi}|_{H_1^{(0)}(X,\mathbb{R})}$ is twisting with respect to $A$. Therefore, the desired theorem follows from Corollary \ref{c.EM}. 
\end{proof}

\begin{theorem}\label{t.MMY-2} Let $X=(M,\omega)$ be a reduced square-tiled surface of genus $g\geq 1$ with trivial group of automorphisms. Suppose that $\phi$ and $\psi$ are affine homeomorphisms of $X$ whose actions on $H_1^{(0)}(X,\mathbb{R})$ are given by matrices $A$ and $B$ in $Sp(2g-2,\mathbb{Z})$ such that: 
\begin{itemize} 
\item $A$ is Galois-pinching, and 
\item the minimal polynomial of $B$ has degree $>2$ with no irreducible factor of even degree and a splitting field disjoint from the splitting field of the characteristic polynomial of $A$.
\end{itemize}
Then, the Lyapunov exponents of the KZ cocycle with respect to the unique $SL(2,\mathbb{R})$-invariant probability measure support on $SL(2,\mathbb{R})X$ are simple.
\end{theorem}

\begin{proof} By Propositions \ref{p.Galois-pinching}, \ref{p.Galois-twisting} and Theorem \ref{t.MMY-twisting}, we can apply Corollary \ref{c.EM} to get the result. 
\end{proof}

From the practical point of view, it is sometimes more convenient to apply the following corollary of Theorem \ref{t.MMY-1}. 

\begin{corollary}\label{c.MMY-1} Let $X=(M,\omega)$ be a reduced square-tiled surface of genus $g\geq 1$ with trivial group of automorphisms. Suppose that $\phi\in\textrm{Aff}(X)$ acts on $H_1^{(0)}(X,\mathbb{R})$ by a Galois-pinching matrix $A$ in $Sp(2g-2,\mathbb{Z})$ and there exists a rational direction such that $X$ decomposes into a finite union of cylinders whose waist curves generate a subspace $E$ of $H_1(X,\mathbb{Q})$ of dimension $1<\textrm{dim}(E)<g$.

Then, the Lyapunov exponents of the KZ cocycle with respect to the unique $SL(2,\mathbb{R})$-invariant probability measure support on $SL(2,\mathbb{R})X$ are simple.
\end{corollary}

\begin{proof} In view of Theorem \ref{t.MMY-1}, it suffices to construct $\psi\in\textrm{Aff}(X)$ acting on $H_1^{(0)}(X,\mathbb{R})$ by an unipotent matrix $B\neq\textrm{Id}$ such that $(B-\textrm{Id})(H_1^{(0)}(X,\mathbb{R}))$ is not Lagrangian. 

We may assume that $X$ decomposes into a collection $\mathcal{C}$ of horizontal cylinders $C$ whose waist curves $\sigma_C$ generate an isotropic subspace $E$ of $H_1(X,\mathbb{Q})$ of dimension $1<\textrm{dim}(E)<g$. Note that 
the image of $E$ under $\pi_*:H_1(X,\mathbb{R})\to H_1(\mathbb{T}^2,\mathbb{R})\simeq H_1^{st}(X,\mathbb{Q})$ is one-dimensiona (because $E$ is isotropic), so that 
$$0<\textrm{dim}(E\cap H_1^{(0)}(X,\mathbb{Q}))<g-1$$

Let $K\in\mathbb{N}$ be an integer such that $D\psi:=\left(\begin{array}{cc} 1 & K \\ 0 & 1 \end{array}\right)\in SL(X)$ and the affine homeomorphism $\psi\in\textrm{Aff}(X)$ with linear part $D\psi$ fixes all horizontal separatrices at all points of the set $\Sigma$ of conical singularities of $(M,\omega)$. 

We affirm that $\psi$ is the desired affine homeomorphism. Indeed, for each horizontal cylinder $C$ of $X$, let us fix $v_C\in H_1(X,\Sigma,\mathbb{Z})$ a relative cycle crossing $C$ upwards. By definition, the matrix $B$ of the action of $\psi$ on $H_1(X,\Sigma,\mathbb{Z})$ satisfies 
$$B(v_C) - v_C = m_C \sigma_C$$
for some $m_C\in\mathbb{Z}$. Since $B$ fixes each $\sigma_C$, we have that $B-\textrm{Id}$ is a nilpotent operator of degree two on $H_1(X,\Sigma,\mathbb{Z})$. The image of $H_1^{(0)}(X,\mathbb{R})$ under $B-\textrm{Id}$ is contained in $E\cap H_1^{(0)}(X,\mathbb{R})$, so that 
$$\textrm{dim}((B-\textrm{Id})(H_1^{(0)}(X,\mathbb{R})))\leq \textrm{dim}(E\cap H_1^{(0)}(X,\mathbb{Q}))<g-1$$
In particular, $(B-\textrm{Id})(H_1^{(0)}(X,\mathbb{R}))$ is not Lagrangian. 

It remains only to prove that $B|_{H_1^{(0)}(X,\mathbb{R})}\neq\textrm{Id}$. For this sake, consider the oriented graph $\Gamma$ whose vertices are the connected components of $M-\bigcup\limits_{C\in\mathcal{C}} C$ and whose edges $e_C$, $C\in\mathcal{C}$, go from the component $b(C)$ of $M-\bigcup\limits_{C\in\mathcal{C}} C$ containing the bottom boundary of $C$ to the component $t(C)$ of $M-\bigcup\limits_{C\in\mathcal{C}} C$ containing the top boundary of $C$. 

Note that, for each vertex $v$ of $\Gamma$, we have the relation
$$\sum\limits_{b(C)=v}\sigma_C = \sum\limits_{t(C)=v}\sigma_C$$
because these are the homology classes of the two components of a small neighborhood of $v$ in $M$.

The orbit of the vertical flow on $X$ through a generic point of $C$ provides a way to construct a simple oriented loop in $\Gamma$ containing the edge $e_C$. 

On the other hand, $\Gamma$ contains two distinct simple oriented loops at least: otherwise, $\Gamma$ would consist of a single loop, so that the previous relations would say that all classes $\sigma_C$ are equal, a contradiction with our assumption that the span $E$ of the classes $\sigma_C$ has dimension $\textrm{dim}(E)>1$. 

Let us fix two distinct simple oriented loops $\gamma_1$ and $\gamma_2$ in $\Gamma$. By definition, there is an edge of $\gamma_i$ not contained in $\gamma_{3-i}$ for $i=1, 2$. Consider a loop $\delta_i$ on $X$ obtained by concatenation of $v_C$ with $e_C\in\gamma_i$ and some horizontal saddle connections. Observe that 
$$B(\delta_i)-\delta_i=\sum\limits_{e_C\in\gamma_i} m_C \sigma_C$$ 
because $B$ fixes horizontal saddle connections. 

Denote by $I_{ij}$ the homological intersection between $B(\delta_i)-\delta_i$ and $\delta_j$. By definition, 
$$I_{ii} = \sum\limits_{e_C\in\gamma_i} m_C \quad \textrm{ and } I_{12} = \sum\limits_{e_C\in\gamma_1\cap\gamma_2} m_C = I_{21}$$
Since $m_C>0$ for all $C\in\mathcal{C}$ and we can find edges in $\gamma_i$ not contained in $\gamma_j$ for $i\neq j$, we get that 
$$\min\{I_{11}, I_{22}\}>I_{12}=I_{21}$$
Thus, $(B-\textrm{Id})(c_1\delta_1+c_2\delta_2)$ intersects non-trivially some $\delta_j$ for all $(c_1,c_2)\in\mathbb{R}^2-\{(0,0)\}$. 

In particular, if we choose $(c_1,c_2)\neq (0,0)$ and $\sigma\in E$ such that $c_1\delta_1+c_2\delta_2+\sigma\in H_1^{(0)}(X,\mathbb{R})$, then we obtain a cycle whose image under $B-\textrm{Id}$ is not zero\footnote{Because $(B-\textrm{Id})(c_1\delta_1+c_2\delta+\sigma)=(B-\textrm{Id})(c_1\delta_1+c_2\delta_2)$ intersects some $\delta_j$ in a non-trivial way.}. Hence, $B|_{H_1^{(0)}(X,\mathbb{R})}\neq\textrm{Id}$. 

This completes the proof of the corollary. 
\end{proof}

The simplicity criteria in Corollary \ref{c.MMY-1} was originally applied in our joint paper \cite{MMY} with M\"oller and Yoccoz to study the Lyapunov exponents of square-tiled surfaces in the minimal stratum $\mathcal{H}(4)$ of the moduli space of translation surfaces of genus $3$. 

More precisely, we exhibited many  infinite families of square-tiled surfaces in both connected components of $\mathcal{H}(4)$ fitting the assumptions of Corollary \ref{c.MMY-1} (cf. Theorem 1.3 in \cite{MMY}), and, \emph{conditionally} on a conjecture of Delecroix and Leli\`evre on the classification of $SL(2,\mathbb{Z})$-orbits of square-tiled surfaces in $\mathcal{H}(4)$, we showed that \emph{all but finitely many} square-tiled surfaces in $\mathcal{H}(4)$ have simple Lyapunov spectrum. 

In a nutshell, the idea of the proof of these facts goes as follows. Given a family $X_n=(M_n,\omega_n)\in\mathcal{H}(4)$, $n\in\mathbb{N}$, of square-tiled surfaces and a family of affine homeomorphisms $\phi_n\in\textrm{Aff}(X_n)$, let $A_n\in Sp(4,\mathbb{Z})$ be the matrices of $\phi_n|_{H_1^{(0)}(X_n,\mathbb{R})}$ and denote by $P_n(x)=x^4+a_nx^3+b_n^2+a_nx+1$ the characteristic polynomial of $A_n$. In our paper \cite{MMY}, we identify several\footnote{If the conjecture of Delecroix-Leli\`evre is true, then the families described in \cite{MMY} include all but finitely many $SL(2,\mathbb{Z})$-orbits of square-tiled surfaces in $\mathcal{H}(4)$.} such families with the property that $P_n(x)$ is irreducible and the discriminants $\Delta_1(P_n)$, $\Delta_2(P_n)$ and $\Delta_3(P_n):=\Delta_1(P_n)\Delta_2(P_n)$ introduced in Proposition \ref{p.Galois-degree4} are polynomial functions of high ``reduced'' degree of the parameter $n\in\mathbb{N}$. By Siegel's theorem (on integral points in algebraic curves of positive genus), we have that the discriminants $\Delta_1(P_n)$, $\Delta_2(P_n)$ and $\Delta_3(n)$ are not squares for all $n$ sufficiently large, i.e., for all $n\geq n_0$ where $n_0$ is an \emph{effective} (but \emph{doubly} exponential) function of the coefficients of $\Delta_1(P_n)$, $\Delta_2(P_n)$ and $\Delta_3(P_n)$. By Proposition \ref{p.Galois-degree4}, it follows that $A_n$ is Galois-pinching for all $n\geq n_0$. Since these families $X_n$ are built in such a way that the cylinders in some rational direction span a subspace of dimension two in homology, we can conclude thanks to Corollary \ref{c.MMY-1}.

\begin{remark} After the publication of \cite{MMY}, Bonatti, Eskin and Wilkinson announced an \emph{unconditional} proof of the simplicity of the Lyapunov exponents of all but finitely many square-tiled surfaces in $\mathcal{H}(4)$. Their arguments are based on the continuity of the Lyapunov spectra of $SL(2,\mathbb{R})$-invariant probability measures in moduli spaces of translation surfaces (and, for this reason, the conclusions of Bonatti, Eskin and Wilkinson are \emph{not} effective in the sense explained above).
\end{remark} 

For the sake of brevity, we shall not detail here the application of Corollary \ref{c.MMY-1} metionned above. Instead, we offer in the next subsection an application (due to Delecroix and the author \cite{DeMa}) of this corollary to the construction of a counterexample to the converse of a theorem of Forni. 

\subsection{A counterexample to an informal conjecture of Forni} 

In his paper \cite{Fo11}, Forni obtained a geometrical criterion for the non-uniform hyperbolicity of the KZ cocycle with respect to a given $SL(2,\mathbb{R})$-invariant probability measure $\mu$ on the moduli space of translation surfaces. In particular, Forni showed that if the support of $\mu$ contains a translation surface $X=(M,\omega)$ decomposing completely into parallel cylinders whose waist curves generate a Lagrangian subspace of $H_1(M,\mathbb{R})$, then all Lyapunov exponents of the KZ cocycle with respect to $\mu$ are non-zero. 

During some informal conversations with the author, Forni conjectured that the converse statement to his theorem could be true: if the Lyapunov exponents of the KZ cocycle with respect to $\mu$ are all non-zero, then the support of $\mu$ contains a translation surface completely decomposing into parallel cylinders whose waist curves generate a Lagrangian subspace in homology. 

In our joint work \cite{DeMa} with Delecroix, we found two counterexamples to this informal conjecture. For the sake of exposition, we present only one of them. 

\begin{theorem}\label{t.DeMa} The $SL(2,\mathbb{R})$-orbit of the square-tiled surface $X$ of genus $3$ associated to the pair of permutations $h=(1, 2, 3, 4)(5, 6, 7, 8)$ and $v=(1, 2, 3, 5)(4, 8, 7, 6)$ supports an unique $SL(2,\mathbb{R})$-invariant probability measure $\mu$ such that all Lyapunov exponents of the KZ cocycle with respect to $\mu$ are non-zero but the subspaces generated by complete decompositions in parallel cylinders of surfaces in $SL(2,\mathbb{R})X$ are never Lagrangian. 
\end{theorem}

\begin{proof} 
An elementary calculation shows that $SL(2,\mathbb{Z})$-orbit of $X$ has exactly three square-tiled surfaces $X=X_0$, $X_1$ and $X_2$ (cf. Proposition 2.3 of \cite{DeMa}). A direct inspection of these square-tiled surfaces reveals that the subspace $E_i$ generated by the waists curves of the horizontal cylinders of $X_i$ has dimension $2$ for all $i\in\{0, 1, 2\}$ (cf. Proposition 1.2 of \cite{DeMa}). In particular, this means that the subspaces generated by complete decompositions in parallel cylinders of surfaces in $SL(2,\mathbb{R})X$ are never Lagrangian (because $X$ has genus $3$). 

Therefore, the proof of the theorem will be complete once we show that the Lyapunov exponents of the KZ cocycle with respect to $\mu$ are all non-zero. In this direction, we will actually use Corollary \ref{c.MMY-1} to prove a \emph{stronger} property, namely, the simplicity of these Lyapunov exponents. More concretely, the discussion of the previous paragraph (about the subspaces generated by waist curves of cylinders) says that we can apply Corollary \ref{c.MMY-1} once we exhibit an affine homeomorphism $\phi\in\textrm{Aff}(X)$ acting on $H_1^{(0)}(X,\mathbb{R})$ through a Galois-pinching matrix. Fortunately, this goal is not hard to accomplish: first, one shows that the affine homeomorphisms $L$ and $R$ of $X$ with linear parts $DL=\left(\begin{array}{cc}1&2\\0&1\end{array}\right)$ and $DR = \left(\begin{array}{cc}1&0\\2&1\end{array}\right)$
act on $H_1^{(0)}(X,\mathbb{R})$ (equipped with a certain basis) via the matrices 
$$L|_{H_1^{(0)}(X,\mathbb{R})} = 
\left(\begin{array}{cccc} 
-1 & 0 & -1 & -1 \\ 0 & 1 & 0 & 0 \\ 0 & 0 & 0 & -1 \\ 0 & 0 & -1 & 0 
\end{array}\right) \quad \textrm{ and } \quad 
R|_{H_1^{(0)}(X,\mathbb{R})} = 
\left(\begin{array}{cccc}
0 & 1 & 0 & 0 \\ 1 & 0 & 0 & 0 \\ -1 & 1 & -1 & 0 \\ 0 & 0 & 0 & 1 
\end{array}\right)$$
(cf. Lemma 2.4 in \cite{DeMa}). A straightforward computation reveals that the affine homeomorphism $\phi = L^4 R L R \in\textrm{Aff}(X)$ acts on $H_1^{(0)}(X,\mathbb{R})$ via a matrix with characteristic polynomial 
$$P(x)=x^4 -2 x^3 -30 x^2 -2x +1.$$ 
In particular, $\phi|_{H_1^{(0)}(X,\mathbb{R})}$ is Galois-pinching: indeed, this is a consequence of Proposition \ref{p.Galois-degree4} because $P$ is an irreducible polynomial such that the discriminants $$\Delta_1 = 4 -4 (-30)+8 = 2^2\times 3\times 11, \quad \Delta_2 = (-30+2-2(-2))(-30+2+2(-2)) = 2^8\times 3,$$ 
and $\Delta_1\cdot\Delta_2= 2^{10}\times 3^2\times 11$ are not perfect squares. 

This proves the desired theorem. 
\end{proof}

\newpage

\begin{centering}
\rule{\textwidth}{1.6pt}\vspace*{-\baselineskip}\vspace*{2.5pt}
\rule{\textwidth}{0.4pt}

\section{An example of quaternionic Kontsevich-Zorich monodromy group}\label{s.FFM}

\rule{\textwidth}{0.4pt}\vspace*{-\baselineskip}\vspace{3.2pt}
\rule{\textwidth}{1.6pt}
\end{centering}\\

The features of the $SL(2,\mathbb{R})$-action on the moduli spaces of translation surfaces (and its applications to the study of interval exchange transformations, translation flows and billiards) are intimately related to the properties of the so-called Kontsevich-Zorich (KZ) cocycle. 

In particular, it is not surprising that the KZ cocycle is one of the main actors in the recent groundbreaking work of Eskin and Mirzakhani \cite{EsMi} towards the classification of $SL(2,\mathbb{R})$-invariant measures on moduli spaces of translation surfaces. 

Partly motivated by this scenario, some authors decided to investigate the possible \emph{monodromies} of the KZ cocycle, i.e., the potential Zariski closures of the corresponding groups of matrices. 

\subsection{Filip's classification of Kontsevich-Zorich monodromy groups}

By extending a previous work of M\"oller \cite{Mo06} on the KZ cocycle over Teichm\"uller curves, Filip \cite{Fi13} showed that a version of the so-called \emph{Deligne's semisimplicity theorem} holds for the KZ cocycle in general. 

In plain terms, this means that the KZ cocycle can be completely decomposed into $SL(2,\mathbb{R})$-irreducible pieces, and, furthermore, each piece respects the Hodge structure coming from the Hodge bundle. Equivalently, the Kontsevich-Zorich cocycle is always diagonalizable by blocks and its restriction to each block is related to a variation of Hodge structures of weight $1$.

The previous paragraph seems abstract at first sight, but, as it turns out, it gives geometrical constraints on the possible groups of matrices obtained from the KZ cocycle. 

More precisely, by exploiting the known tables\footnote{See \cite[\S 3.2]{Fi14}.} for monodromy representations coming from variations of Hodge structures of weight $1$ over quasiprojective varieties, Filip \cite{Fi14} classified (up to compact and finite-index factors) the possible Zariski closures of the groups of matrices associated to restrictions of the KZ cocycle to an irreducible piece. In particular, there are at most five types of possible Zariski closures for blocks of the KZ cocycle (cf. Theorems 1.1 and 1.2 in \cite{Fi14}):

\begin{itemize}
\item[(i)] the symplectic group $Sp(2d,\mathbb{R})$ in its standard representation;
\item[(ii)] the (generalized) unitary group $SU_{\mathbb{C}}(p,q)$ in its standard representation;
\item[(iii)] $SU_{\mathbb{C}}(p,1)$ in an exterior power representation;
\item[(iv)] the quaternionic orthogonal group $SO^*(2n)$ (sometimes called $U_{\mathbb{H}}^*(n)$, $SU^*(2n)$ or $SL_n(\mathbb{H})$) of matrices on $\mathbb{C}^{2n}$ respecting a quaternionic structure and an Hermitian (complex) form of signature $(n,n)$ in its standard representation\footnote{In concrete terms, $SO^*(2n)$ is the group of $n\times n$ matrices $A$ with coefficients in the quaternions $\textbf{H}$ such that $A^{\#} A = \textrm{Id}$ where $A^{\#}$ is the transpose of $\sigma(A)$ and $\sigma(a+bi+cj+dk)=a-bi+cj+dk$ is a reversion on $\textbf{H}$.}; 
\item[(v)] the indefinite orthogonal group $SO_{\mathbb{R}}(p,2)$ in a spin representation.
\end{itemize}

Moreover, it is not hard to check that each of these items can be realized as an \emph{abstract} variation of Hodge structures of weight $1$ over abstract curves and/or Abelian varieties. 

\begin{remark} This classification of Kontsevich-Zorich monodromy groups allowed Filip to confirm a conjecture of Forni, Zorich and the author \cite{FMZ} saying that all zero Lyapunov exponents of the KZ cocycle are ``explained'' by its monodromy: see \cite{Fi14} for more details. 
\end{remark}

\subsection{Realizability problem for Kontsevich-Zorich monodromy groups} It is worth to stress out that Filip's classification of the possible blocks of the KZ cocycle comes from a \emph{general} study of variations of Hodge structures of weight $1$. 

Thus, it is \emph{not} clear whether all items above can actually be realized as a block of the KZ cocycle over the closure of some $SL(2,\mathbb{R})$-orbit in the moduli spaces of translations surfaces.

In fact, it was previously known in the literature that (all groups listed in) items (i) and (ii) appear as blocks of the KZ cocycle over closures of $SL(2,\mathbb{R})$-orbits of translation surfaces given by certain cyclic cover constructions. 

On the other hand, it is not obvious that the other 3 items occur in the context of the KZ cocycle: indeed, this realizability question was explicitly posed by Filip in \cite[Question 5.5]{Fi14} (see also \S B.2 in Appendix B of Delecroix-Zorich paper \cite{DZ}).

Filip, Forni and myself \cite{FFM} gave a partial answer to this question by showing that the case $SO^*(6)$ of item (iv) is realizable as a block of the KZ cocycle: 

\begin{theorem}\label{t.FFM} There exists a square-tiled surface $\widetilde{L}$ of genus 11 such that the restriction of the KZ cocycle over $SL(2,\mathbb{R})\cdot\widetilde{L}$ to a certain $SL(2,\mathbb{R})$-irreducible piece acts through a Zariski dense subgroup of $SO^*(6)$ in its standard representation (modulo finite-index subgroups). 
\end{theorem}

\begin{remark}\label{r.SO*6-isomorphism} Thanks to an exceptional isomorphism between the real Lie algebra $\mathfrak{so}^*(6)$ in its standard representation and the second exterior power representation of the real Lie algebra $\mathfrak{su}(3,1)$, this theorem also says that the case of $\wedge^2 SU(3,1)$ of item (iii) is realized. 
\end{remark}

\begin{remark}\label{r.MYZ-examples} The examples constructed by Yoccoz, Zmiaikou and myself \cite{MYZ} of regular origamis associated to the groups $SL(2,\mathbb{F}_p)$ of Lie type might lead to the realizability of all groups $SO^*(2n)$ in item (iv). In fact, what prevents us to show that this is the case is the absence of a systematic method to prove that the natural candidates to blocks of the KZ cocycle over these examples are actually irreducible pieces. 
\end{remark} 

\begin{remark} The realizability of some and/or all groups in items (iii) and (v) might be a delicate problem: indeed, contrary to our previous remark about the realization of all groups in item (iv), we are \emph{not} aware of examples of translation surfaces which could solve this question.
\end{remark}

The remainder of this section is dedicated to the proof of Theorem \ref{t.FFM}. 

\subsection{A quaternionic cover of a L-shaped orgami}

The starting point of the square-tiled surface $\widetilde{L}$ in Theorem \ref{t.FFM} is the following observation. The group $SO^*(2n)$ is related to quaternionic structures on vector spaces. In particular, it is natural to look for translation surfaces possessing an automorphism group admitting representations of quaternionic type.

Note that automorphism groups of translation surfaces (of genus $\geq 2$) are always finite\footnote{E.g., by Hurwitz’s theorem saying that a Riemann surface of genus $g\geq 2$ has $84(g-1)$ automorphisms at most.} and the simplest finite group with representations of quaternionic type is the quaternion group
$$Q:=\{1,-1,i,-i,j,-j,k,-k\}$$
where $i^2=j^2=k^2=-1$, $ij=k$, $jk=i$ and $ki=j$.

This indicates that we should look for translation surfaces whose group of automorphisms is isomorphic to $Q$. A concrete way of building such translation surfaces $S$ is to consider ramified covers $S\rightarrow C$ of ``simple'' translation surfaces $C$ such that the group of deck transformations of $S\rightarrow C$ is isomorphic to $Q$.

The first natural attempt is to take $C=\mathbb{R}^2/\mathbb{Z}^2$ the flat torus, and define $S$ as the translation surface obtained as follows. We let $C_g$, $g\in Q$, be copies of the flat torus $C$. Then, we glue by translation the rightmost vertical, resp. topmost horizontal side, of $C_g$ with the leftmost vertical, resp. bottommost horizontal side, of $C_{gi}$, resp. $C_{gj}$ for each $g\in Q$. In this way, we obtain a translation surface $S$ tiled by eight squares $C_{g}$, $g\in Q$, such that the natural projection $S\rightarrow C$ is a ramified cover (branched only at the origin of $C$) whose group of automorphisms is isomorphic to $Q$ (namely, an element $h\in Q$ acts by translating $C_g$ to $C_{hg}$ for all $g\in Q$).

The translation surface $S$ constructed above is a well-known square-tiled surface: it is the so-called  \emph{Eierlegende Wollmilchsau} in the literature (see e.g. \cite{Fo06} and \cite{HS}).

Unfortunately, the Eierlegende Wollmilchsau is \emph{not} a good example for our current purposes. Indeed, it is known that the KZ cocycle over the $SL(2,\mathbb{R})$-orbit of the Eierlegende Wollmilchsau acts through a \emph{finite} group of matrices: see, e.g., the work \cite{MY} of Yoccoz and the author. In particular, this provides \emph{no} meaningful information towards realizing $SO^*(2n)$ monodromy groups because in Filip's list one always \emph{ignores} compact and/or finite-index factors.

After this frustrated attempt, we are led to look at other translation surfaces distinct from the flat torus. In this direction, it is natural to consider the simplest L-shaped square-tiled surface $L$ in genus $2$ described in Figure \ref{f.L-shaped} above.

Next, we build a ramified cover $\widetilde{L}$ of $L$ in a similar way to the construction of the Eierlegende Wollmilchsau: we take copies $L_g$, $g\in Q$, of this L-shaped square-tiled surface, and we glue by translations the corresponding vertical, resp. horizontal, sides of $L_g$ and $L_{gi}$, resp. $L_{gj}$. Equivalently, we label the sides of $L_g$ as indicated in Figure \ref{f.FFM-1}   
\begin{figure}[htb!] 
\begin{center}
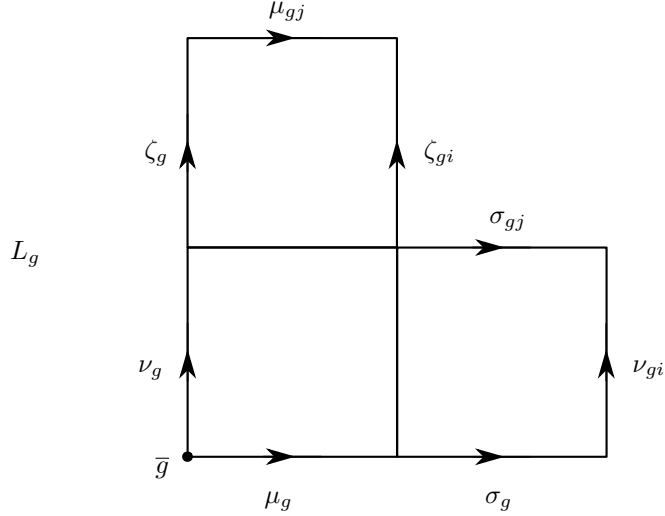
\caption{Construction of $\widetilde{L}$.}\label{f.FFM-1}
\end{center}
\end{figure}
and we glue by translations the pairs of sides with the same labels.

The natural projection $\widetilde{L}\rightarrow L$ is a ramified cover branched only at the unique conical singularity of $L$. Also, the automorphism group of $\widetilde{L}$ is isomorphic to $Q$ and each $h\in Q$ acts on $\widetilde{L}$ by translating each $L_g$ to $L_{hg}$ for all $g\in Q$.

A direct inspection reveals that $\widetilde{L}$ has four conical singularities (at the points labelled $\underline{1}$, $\underline{i}$, $\underline{j}$ and $\underline{k}$ in Figure \ref{f.FFM-1}) whose cone angles are $12\pi$. In particular, $\widetilde{L}\in\mathcal{H}(5,5,5,5)$ is a genus $11$ surface. 

In this setting, the KZ cocycle over $SL(2,\mathbb{R})\widetilde{L}$ is simply the action on $H_1(\widetilde{L},\mathbb{R})$ of the group $\textrm{Aff}(\widetilde{L})$ of affine homeomorphisms of $\widetilde{L}$.

\subsection{Block decomposition of the KZ cocycle over $SL(2,\mathbb{R})\cdot\widetilde{L}$}

Similarly to the so-called \emph{wind-tree models} studied by Delecroix, Hubert and Leli\`evre [REF], the translation surface $\widetilde{L}$ has a rich group of symmetries allowing us to decompose the KZ cocycle into blocks.

More precisely, by taking the quotient of $\widetilde{L}$ by the center $Z=\{1,-1\}$ of its automorphism group $Q$, we obtain a translation surface $M=\widetilde{L}/Z$ of genus $5$ with four conical singularities whose cone angles are $6\pi$. Moreover, by taking the quotient of $S$ by the subgroups $Z\cup\{i,-i\}$, $Z\cup\{j,-j\}$ and $Z\cup \{k,-k\}$ of its automorphism group $Q$, we obtain three genus $3$ surfaces $N_i$, $N_j$ and $N_k$ each having two conical singularities whose cone angles are $6\pi$. In summary, we have intermediate covers $\widetilde{L}\rightarrow M\rightarrow L$ and $M\rightarrow N_{\ast}\rightarrow L$ for $\ast=i, j, k$ such that $M\in\mathcal{H}(2,2,2,2)$ and $N_{\ast}\in\mathcal{H}(2,2)^{\textrm{odd}}$.

These intermediate covers lead us towards a natural candidate for blocks of the KZ cocycle over $SL(2,\mathbb{R})\widetilde{L}$. More concretely, a translation cover $p:S\to C$ induces a decomposition 
$$H_1(S,\mathbb{R}) = \textrm{Ker}(p_*)^{\perp}\oplus \textrm{Ker}(p_*)$$ 
where $\textrm{Ker}(p_*)^{\perp}\simeq H_1(C,\mathbb{R})$ is the symplectic orthogonal of the vector space $\textrm{Ker}(p_*)$ of cycles on $S$ projecting to zero under $p$. Thus, if we take $\textrm{Aff}_{0}(\widetilde{L})$ an adequate finite-index subgroup of $\textrm{Aff}(\widetilde{L})$ whose elements \emph{commute}\footnote{I.e., up to finite-index, the KZ cocycle commutes with the action of $Q$ on $H_1(\widetilde{L},\mathbb{R})$.} with the automorphisms of $\widetilde{L}$, then we obtain that the action of $\textrm{Aff}(\widetilde{L})$ can be virtually diagonalized by (symplectically orthogonal) blocks 
$$H_1(S,\mathbb{R}) = H_1^{st}\oplus E_1\oplus E_i\oplus E_j\oplus E_k\oplus W$$
where: 
\begin{itemize}
\item $H_1^{st}$ is the subspace generated by $\sigma:=\sum\limits_{g\in Q} (\sigma_g+\mu_g)$ and $\zeta:=\sum\limits_{g\in Q}(\zeta_g+\nu_g)$, 
\item $H_1^{st}\oplus E_1\simeq H_1(L,\mathbb{R})$ is induced by the cover $\widetilde{L}\to L$, 
\item $H_1^{st}\oplus E_1\oplus E_{\ast}\simeq H_1(N_{\ast},\mathbb{R})$ is induced by the cover $\widetilde{L}\to N_{\ast}$ for each $\ast=i, j, k$, and 
\item $W$ is the symplectic orthogonal of the direct sum of the other subspaces.
\end{itemize}
See Subsection 5.3 of \cite{FFM} for more details. 

These subspaces have the structure of $Q$-modules, and, by a quick comparison with the character table of $Q$, one can show that $E_1$, $E_i$, $E_j$, $E_k$ and $W$ (resp.) are the isotypical components of the trivial, $i$-kernel, $j$-kernel, $k$-kernel and the unique four-dimensional faithful irreducible representation $\chi_2$ of $Q$ (resp.): for example, $W$ is the isotypical component of $\chi_2$ because $-1\in Q$ acts as $-\textrm{id}$ on $W$ and the sole character of $Q$ to take a negative value at $-1$ is precisely $\chi_2$.

Furthermore, $W$ is $12$-dimensional because $\widetilde{L}$ and $M$ have genera $11$ and $5$ (so that $H_1(\widetilde{L},\mathbb{R})$ and $H_1(M,\mathbb{R})$ have dimensions $22$ and $10$), and $W$ is the symplectic orthogonal of the symplectic subspace $H_1^{st}\oplus E_1\oplus E_i\oplus E_j\oplus E_k\simeq H_1(M,\mathbb{R})$. Hence, $W=3\chi_2$ as a $Q$-module. 

\subsection{Some constraints on Kontsevich-Zorich monodromy group of $W$}

Note that $\textrm{Aff}_0(\widetilde{L})$ acts via symplectic automorphisms of the $Q$-module $W$ (because the actions of $\textrm{Aff}_0(\widetilde{L})$ and the automorphism group $Q$ on $H_1(\widetilde{L},\mathbb{R})$ commute), and $W=3\chi_2$ carries a quaternionic structure. In particular, we are almost in position to apply Filip's classification results to determine the group of matrices through which $\textrm{Aff}_0(\widetilde{L})$ acts on $W$.

Indeed, if we have that $\textrm{Aff}_0(\widetilde{L})$ acts \emph{irreducibly} on $W$, then Filip's list of possible monodromy groups says that $\textrm{Aff}_0(\widetilde{L})$ acts through a (virtually Zariski dense) subgroup of $SO^*(6)$ (because $\textrm{Aff}_0(\widetilde{L})$ preserves a quaternionic structure on $W$).

However, there is \emph{no} reason for the action of the affine homeomorphisms on an isotypical component of the automorphism group to be irreducible in general (as far as we know). Nevertheless, the semi-simplicity theorems of M\"oller \cite{Mo06} and Filip \cite{Fi13} mentioned above tells us that $W$ can split into irreducible pieces in one of the following three ways:
\begin{itemize}
\item[(a)] $W$ is irreducible, i.e., it does not decompose further;
\item[(b)] $W=U\oplus V$ where $U=2\chi_2$ and $V=\chi_2$ are irreducible pieces;
\item[(c)] $W=V'\oplus V''\oplus V'''$ where $V'$, $V''$, $V'''$ are irreducible pieces isomorphic to $\chi_2$.
\end{itemize}

By applying Filip's classification to each of these items, we find that (up to compact and finite-index factors) there are just three possibilities:
\begin{itemize}
\item[(a')] if $W$ is $\textrm{Aff}_0(\widetilde{L})$-irreducible, then $\textrm{Aff}_0(\widetilde{L})$ acts through a Zariski-dense subgroup of $SO^*(6)$;
\item[(b')] if $W=U\oplus V$ with $U=2\chi_2$ and $V=\chi_2$ irreducible pieces, then $\textrm{Aff}_0(\widetilde{L})$ acts through a subgroup of $SO^*(4)\times SO^*(2)$;
\item[(c')] if $W=V'\oplus V''\oplus V'''$ with $V'$, $V''$, $V'''$ irreducible pieces isomorphic to $\chi_2$, then $\textrm{Aff}_0(\widetilde{L})$ acts through a subgroup of $SO^*(2)\times SO^*(2) \times SO^*(2)$.
\end{itemize}

At this point, we reduced the proof of Theorem \ref{t.FFM} to show that we are in situation (a') or, equivalently, the situations (b') and (c') can't occur.

In the next two subsections, we shall rule out (c') and (b') respectively. 

\subsection{Ruling out $SO^*(2)\times SO^*(2) \times SO^*(2)$ monodromy on $W$}

Suppose that $\textrm{Aff}_0(\widetilde{L})$ acts on $W$ as described in item (c'). In this situation, the nature\footnote{A product of three copies of the \emph{compact} group $SO^*(2)$.} of $SO^*(2)\times SO^*(2)\times SO^*(2)$ would force \emph{all} Lyapunov exponents of the restriction of the KZ cocycle to $W$ to \emph{vanish}. Therefore, we can rule out this situation by showing the following proposition:

\begin{proposition}\label{p.EKZ-W} Some Lyapunov exponents of KZ cocycle on $W$ are \emph{not} zero. 
\end{proposition}

The proof of this proposition relies crucially on the formulas of Bainbridge \cite{Ba}, Chen and M\"oller \cite{CM}, and Eskin-Kontsevich-Zorich \cite{EKZ} for the sum of non-negative Lyapunov exponents for the square-tiled surfaces. More concretely, we can derive this proposition as follows.  

First, since $L\in\mathcal{H}(2)$, Bainbridge work \cite{Ba} ensures that the non-negative Lyapunov exponents associated to $H_1^{st}\oplus E_1\simeq H_1(L,\mathbb{R})$ are $1$ and $1/3$. 

Secondly, since $N_{\ast}\in\mathcal{H}(2,2)^{odd}$, $\ast=i, j, k$, the work of Chen and M\"oller \cite{CM} says that the sum of non-negative Lyapunov exponents associated to $H_1^{st}\oplus E_1\oplus E_{\ast}$ is $5/3$. Since we already know that the non-negative Lyapunov exponents of $H_1^{st}\oplus E_1$ are $1$ and $1/3$, we conclude that the non-negative Lyapunov exponents associated to 
$$H_1^{st}\oplus E_1\oplus E_i\oplus E_j\oplus E_k\simeq H_1(M,\mathbb{R})$$ 
are $1$ with multiplicity one and $1/3$ with multiplicity four. 

Thirdly, the fact that $W$ has quaternionic nature forces each of its Lyapunov exponent to appear with multiplicity four (at least). Since $W\simeq 3\chi_2$ is $12$-dimensional, the Lyapunov exponents associated to $W$ have the form 
$$\lambda=\lambda=\lambda=\lambda\geq 0=0=0=0\geq -\lambda=-\lambda=-\lambda=-\lambda$$ 

It follows that the sum of the non-negative Lyapunov exponents of $\widetilde{L}$ is 
$$\frac{7}{3}+4\lambda$$ 

At this point, we will prove that $\lambda$ is \emph{not} zero by using Eskin-Kontsevich-Zorich formula \cite{EKZ} saying that the sum of the non-negative Lyapunov exponents $1=\theta_1>\theta_2\geq\dots\geq\theta_g\geq 0$ of a square-tiled surface $X\in\mathcal{H}(k_1,\dots, k_s)$ of genus $g$ is given by the following expression: 
$$\theta_1+\dots+\theta_g= \frac{1}{12}\sum\limits_{l=1}^s\frac{k_l(k_l+2)}{k_l+1} + 
\frac{1}{\# SL(2,\mathbb{Z})X} \sum\limits_{\substack{Y\in SL(2,\mathbb{Z})X \\ c \textrm{ cycle of } h_Y}} \frac{1}{\textrm{length of } c}$$
where $(h_Y, v_Y)$ is a pair of permutations associated to $Y$. 

In our case, $\widetilde{L}\in\mathcal{H}(5,5,5,5)$ and its $SL(2,\mathbb{Z})\cdot\widetilde{L}$ contains twelve square-tiled surfaces whose flat geometries are explicitely described in \cite[Subsection 5.4]{FFM}. From this, we can show that 
$$\frac{7}{3}+4\lambda = \frac{1}{12}\left(4\times\frac{5\times 7}{6}\right) + \frac{1}{12}\sum\limits_{\substack{Y\in SL(2,\mathbb{Z})\widetilde{L} \\ c \textrm{ cycle of } h_Y}} \frac{1}{\textrm{length of } c} = 3$$ 

This means that $\lambda=1/6\neq 0$, and, thus, the proof of Proposition \ref{p.EKZ-W} is complete. 

\begin{remark} We have just proved that we are in situation (a') or (b'). Hence, we already know that the KZ  cocycle over $SL(2,\mathbb{R})\widetilde{L}$ acts on a irreducible piece $W$ through a Zariski dense subgroup of $SO^*(6)$ or $SO^*(4)$ (modulo compact and/or finite-index factors). Of course, this suffices to deduce that we can realize a non-trivial case ($n=3$ or $2$) of item (iv) in Filip's list, but, for the sake of completeness, we will show in the next subsection how to rule out the situation (b'). 
\end{remark}

\subsection{Ruling out $SO^*(4)\times SO^*(2)$ monodromy on $W$} 

Suppose that $\textrm{Aff}_0(\widetilde{L})$ acts on $W$ as described in item (b'), i.e., we have a $\textrm{Aff}_0(\widetilde{L})$-invariant decomposition $W=U\oplus V$ with $U=2\chi_2$ and $V=\chi_2$. 

In this case, the sole possibility for the subspace $V$ is to be the central subspace of \emph{any} matrix of any element of $\textrm{Aff}_0(\widetilde{L})$ acting on $W=3\chi_2$ with ``simple spectrum'' in the quaternionic sense\footnote{I.e., the matrix has an unstable (modulus $>1$) eigenvalue, a central (modulus $=1$) eigenvalue, and an stable (modulus $<1$) eigenvalue, all of them with multiplicity four.} 

Therefore, we can contradict the existence of $V$ by exhibiting two matrices of the action of $\textrm{Aff}_0(\widetilde{L})$ on $W=3\chi_2$ with ``simple spectrum'' whose central spaces are distinct. 

Unfortunately, we do \emph{not} have an abstract method to produce two matrices with the properties above (compare with Remark \ref{r.MYZ-examples} above). Thus, we are obliged to compute by hands the action of some elements of $\textrm{Aff}_0(\widetilde{L})$. 

In this direction, we observe that $-1\in Q$ acts on $W$ via $-\textrm{Id}$. From this, it is not hard to check  that a basis $\mathcal{B}$ of $W$ is given by the following twelve (absolute) cycles 
$$\{\widehat{\sigma}_1, \widehat{\sigma}_i, \widehat{\sigma}_j, \widehat{\sigma}_k, \widehat{\zeta}_1, \widehat{\zeta}_i, \widehat{\zeta}_j, \widehat{\zeta}_k, \widehat{\mu}_1, \widehat{\mu}_i, \widehat{\nu}_1, \widehat{\nu}_j\}$$ 
where $\widehat{\sigma}_g := \sigma_{g}-\sigma_{-g}$, $\widehat{\zeta}_g := \zeta_{g}-\zeta_{-g}$, $\widehat{\mu}_g := \mu_{g}-\mu_{-g}$, $\widehat{\nu}_g := \nu_{g}-\nu_{-g}$, and $\sigma_g$, $\zeta_g$, $\mu_g$ and $\nu_g$ are the (relative) cycles in Figure \ref{f.FFM-1}. 

Next, we consider the affine homeomorphisms $\underline{A}, \underline{B}, \underline{C}\in\textrm{Aff}(\widetilde{L})$ with linear parts 
$$d\underline{A} = \left(\begin{array}{cc} 4 & -3 \\ 3 & -2 \end{array}\right), \quad d\underline{B} = \left(\begin{array}{cc} 10 & 27 \\ -3 & -8 \end{array}\right), \quad d\underline{C} = \left(\begin{array}{cc} -8 & -3 \\ 27 & 10 \end{array}\right)\in SL(2,\mathbb{Z}),$$
and fixing (pointwise) the conical singularities of $\widetilde{L}$. Geometrically, $\underline{A}, \underline{B}, \underline{C}$ are Dehn multitwists of $\widetilde{L}$ along the cylinders in directions $(1,1), (3,-1), (-1,3)$. 

A straightforward calculation shows that the actions of $\underline{A}, \underline{B}, \underline{C}$ on $W$ with respect to the basis $\mathcal{B}$ are given by the following $12\times 12$ matrices $A, B, C$:

\begin{equation*}
A=\left(\begin{array}{cccccccccccc}
2 & -1 & -1 & 0 & 1 & -1 & -1 & 0 & -1 & 0 & -1 & 0 \\
1 & 0 & 0 & 1 & 1 & -1 & 0 & 1 & 0 & 1 & -1 & -1 \\
0 & 1 & 1 & 0 & 0 & 1 & 0 & 0 & 0 & 0 & 1 & 0 \\
-1 & 0 & 0 & 1 & -1 & 0 & 0 & 0 & 0 & 0 & 0 & 0 \\
-1 & 1 & 1 & 0 & 0 & 1 & 1 & 0 & 1 & 0 & 1 & 0 \\
0 & 0 & -1 & 0 & 0 & 1 & -1 & 0 & -1 & 0 & 0 & 0 \\
-1 & 0 & 1 & 1 & -1 & 0 & 2 & 1 & 1 & 1 & 0 & -1 \\
-1 & 0 & 0 & 0 & -1 & 0 & 0 & 1 & 0 & 0 & 0 & 0 \\
1 & -1 & -1 & -1 & 1 & -1 & -1 & -1 & 0 & -1 & -1 & 1 \\
1 & -1 & 1 & 1 & 1 & -1 & 1 & 1 & 1 & 2 & -1 & -1 \\
-1 & 1 & 1 & -1 & -1 & 1 & 1 & -1 & 1 & -1 & 2 & 1 \\
-1 & -1 & 1 & 1 & -1 & -1 & 1 & 1 & 1 & 1 & -1 & 0
\end{array}\right),
\end{equation*} 

\begin{equation*}
B=\left(\begin{array}{cccccccccccc}
2 & 1 & 1 & 0 & -1 & -1 & -1 & 0 & -1 & 0 & 0 & 1 \\
-1 & 0 & 0 & -1 & 1 & 1 & 0 & 1 & 0 & 1 & 0 & 0 \\
0 & 1 & 1 & 0 & 0 & -1 & 0 & 0 & 0 & 0 & 0 & -1 \\
-1 & 0 & 0 & 1 & 1 & 0 & 0 & 0 & 0 & 0 & -1 & -1 \\
-1 & 1 & -1 & 0 & 2 & -1 & 1 & 0 & 1 & 0 & 0 & -3 \\
0 & 0 & 1 & 0 & 0 & 1 & -1 & 0 & -1 & 0 & -1 & 1 \\
-1 & -2 & -1 & -1 & 1 & 2 & 2 & 1 & 1 & 1 & 1 & 0 \\
1 & 0 & 0 & 0 & -1 & 0 & 0 & 1 & 0 & 0 & 1 & 1 \\
1 & 1 & 1 & -1 & -1 & -1 & -1 & 1 & 0 & 1 & 1 & 1 \\
-1 & -1 & 1 & -1 & 1 & 1 & -1 & 1 & -1 & 2 & -1 & 1 \\
-1 & 1 & -1 & 1 & 1 & -1 & 1 & -1 & 1 & -1 & 0 & -3 \\
-1 & -1 & -1 & -1 & 1 & 1 & 1 & 1 & 1 & 1 & 1 & 0
\end{array}\right),  
\end{equation*}

and 

\begin{equation*}
C=\left(\begin{array}{cccccccccccc}
2 & 1 & -1 & 0 & -1 & -1 & 1 & 0 & 0 & -3 & 1 & 0 \\
1 & 2 & 2 & -1 & -1 & -1 & -2 & 1 & 1 & 0 & 1 & 1 \\
0 & -1 & 1 & 0 & 0 & 1 & 0 & 0 & -1 & 1 & -1 & 0 \\
1 & 0 & 0 & 1 & -1 & 0 & 0 & 0 & -1 & -1 & 0 & 0 \\
-1 & -1 & -1 & 0 & 2 & 1 & 1 & 0 & 0 & 1 & -1 & 0 \\
0 & 0 & -1 & 0 & 0 & 1 & 1 & 0 & 0 & -1 & 0 & 0 \\
1 & 0 & 1 & -1 & -1 & 0 & 0 & 1 & 0 & 0 & 0 & 1 \\
-1 & 0 & 0 & 0 & 1 & 0 & 0 & 1 & 1 & 1 & 0 & 0 \\
1 & 1 & -1 & 1 & -1 & -1 & 1 & -1 & 0 & -3 & 1 & -1 \\
1 & 1 & 1 & -1 & -1 & -1 & -1 & 1 & 1 & 0 & 1 & 1 \\
-1 & -1 & -1 & -1 & 1 & 1 & 1 & 1 & 1 & 1 & 0 & 1 \\
1 & -1 & 1 & -1 & -1 & 1 & -1 & 1 & -1 & 1 & -1 & 2
\end{array}\right) 
\end{equation*}

See Subsections 6.2, 6.3 and 6.4 of \cite{FFM} for more details. 

From these formulas, we can exhibit two elements of $\textrm{Aff}_0(\widetilde{L})$ acting on $W$ through matrices with ``simple spectrum'' and distinct central subspaces. In fact, let us fix $k, l\in\mathbb{N}$ such that $(\underline{A}\circ\underline{B})^k$ and $(\underline{C}\circ\underline{B})^l$ belong to $\textrm{Aff}_0(\widetilde{L})$: this is possible because $\textrm{Aff}_0(\widetilde{L})$ has finite-index in $\textrm{Aff}(\widetilde{L})$. Note that these elements act on $W$ by matrices with ``simple spectrum'': indeed, a computation reveals that $A.B$ and $C.B$ have characteristic polynomials 
\begin{eqnarray*}
P_{A.B}(x) &=& x^{12} - 28 x^{11} + 322 x^{10} -  1964 x^9 + 6895 x^8 - 14392 x^7 \\ &+& 18332 x^6  -14392 x^5 +6895 x^4 -1964 x^3 +322 x^2 -28 x + 1 \\ &=&(x-1)^4 (x^2-6x+1)^2
\end{eqnarray*}
and 
\begin{eqnarray*}
P_{C.B}(x) &=& x^{12} - 44 x^{11} + 770 x^{10} - 6780 x^9 + 31471 x^8 - 76120 x^7 \\ &+& 101404 x^6  - 76120 x^5 + 31471 x^4 - 6780 x^3 + 770 x^2 - 44 x + 1 \\ &=& (x-1)^4 (x^2-10x+1)^2
\end{eqnarray*}
In particular, $A.B$ has three eigenvalues (each of them with multiplicity four), namely,  $3+2\sqrt{2}$, $1$ and $3-2\sqrt{2}$, and $C.B$ has three eigenvalues (each of them with multiplicity four), namely,  $5+2\sqrt{6}$, $1$ and $5-2\sqrt{6}$, so that $(A.B)^k$ and $(C.B)^l$ have ``simple spectrum''.  

Furthermore, it is not hard to see that the central eigenspace $V_{AB}$  of $A.B$ (associated to the eigenvalue $1$) is spanned by the following four vectors:
$$v_{AB}^{(1)} = (-1, 1, -1, 1, 1, -1, 1, 1, 0, 0, 0, 2), \quad v_{AB}^{(2)} = (-1, -1, 1, 1, 1, -1, -1, -1, 0, 0, 2, 0),$$
$$v_{AB}^{(3)} = (0, 0, 0, 0, 0, 0, 0, -1, 0, 1, 0, 0), \quad v_{AB}^{(4)} = (0, 0, 0, 0, 0, 0, -1, 0, 1, 0, 0, 0),$$
and the central eigenspace $V_{CB}$  of $C.B$ (associated to the eigenvalue $1$) is spanned by the following four vectors:
$$v_{CB}^{(1)} = (0, 0, 0, 1, 0, 0, 0, 1, 0, 0, 0, 0), \quad v_{CB}^{(2)} = (0, 0, 1, 0, 0, 0, 1, 0, 0, 0, 0, 0),$$
$$v_{CB}^{(3)} = (0, 1, 0, 0, 0, 1, 0, 0, 0, 0, 0, 0), \quad v_{CB}^{(4)} = (1, 0, 0, 0, 1, 0, 0, 0, 0, 0, 0, 0).$$
Thus, $V_{AB}$ and $V_{CB}$ are distinct\footnote{Actually, $\{v_{AB}^{(n)}, v_{CB}^{(m)}\}_{1\leq n, m\leq 4}$ span a $8$-dimensional vector space}. Since $V_{AB}$, resp. $V_{CB}$, is also the central space of $(A.B)^k$, resp. $(C.B)^l$, this means that we are not in situation (b'), so that the proof of Theorem \ref{t.FFM} is complete. 





\bibliographystyle{amsplain}

\end{document}